\documentclass[11pt]{article}
\usepackage[a4paper]{geometry}

\usepackage{amsmath}
\usepackage{amssymb}
\usepackage{amsthm}
\usepackage{mathrsfs}
\usepackage{bbm}
\usepackage{empheq}
\usepackage[loose]{subfigure}
\usepackage{epsfig}
\usepackage{graphicx}
\usepackage{psfrag}
\usepackage[usenames,dvipsnames]{pstricks}
\usepackage{pst-plot}
\usepackage[colorlinks]{hyperref}
\usepackage{tabls}
\usepackage{paralist}

\usepackage[all]{xypic}

\newcommand{\COMMENT}[1]{}
\newcommand{\E}{\mathbb{E}}
\newcommand{\eins}{\mathbbm{1}}
\newcommand{\G}{\mathbf{G}}
\renewcommand{\P}{\mathbb{P}}

\newcommand{\R}{\mathbb{R}}

\newcommand{\numeq}{\stackrel{\mathrm{num}}{=}}

\newtheorem{thm}{Theorem}[section]
\newtheorem{prop}[thm]{Proposition}
\newtheorem{lem}[thm]{Lemma}

\newtheorem{cor}[thm]{Corollary}
\theoremstyle{definition}
\newtheorem{defn}[thm]{Definition}
\newtheorem{rmk}[thm]{Remark}
\newtheorem{eg}{Example}[section]

\hypersetup{
    linkcolor=blue,
}

\title{Optimal Uncertainty Quantification}

\author{
	H.\ Owhadi, T.\ J.\ Sullivan, M.\ McKerns, M. Ortiz \\
	California Institute of Technology \\
	\textsf{owhadi@caltech.edu} \\
	~ \\
	C.\ Scovel \\
	Los Alamos National Laboratory \\
	\textsf{jcs@lanl.gov}
}
\date{\today}

\makeatletter
\@addtoreset{equation}{section}
\@addtoreset{figure}{section}
\@addtoreset{table}{section}
\makeatother

\renewcommand{\thefigure}{\arabic{section}.\arabic{figure}}

\makeatletter
\renewcommand{\p@subfigure}{\thefigure}
\makeatother

\begin{document}

\maketitle

\begin{abstract}
	We propose a rigorous framework for Uncertainty Quantification (UQ) in which the UQ objectives and  the assumptions/information set are brought to the forefront.  This framework, which we call \emph{Optimal Uncertainty Quantification} (OUQ), is based on the observation that, given a set of assumptions and information about the problem, there exist optimal bounds on uncertainties:  these are obtained as values of well-defined optimization problems corresponding to extremizing probabilities of failure, or of deviations, subject to the constraints imposed by the scenarios compatible with the assumptions and information.  In particular, this framework does not implicitly impose inappropriate assumptions, nor does it repudiate relevant information.

	Although OUQ optimization problems are extremely large, we show that under general conditions they have finite-dimensional reductions.  As an application, we develop \emph{Optimal Concentration Inequalities} (OCI) of Hoeffding and McDiarmid type.  Surprisingly, these results show that
uncertainties in input parameters, which propagate to output uncertainties in the classical sensitivity analysis paradigm, may fail to do so if the transfer functions (or probability distributions) are imperfectly known. We show how, for hierarchical structures, this phenomenon may lead to the non-propagation of uncertainties or information across scales.

	In addition, a general algorithmic framework is developed for OUQ and is tested on the Caltech surrogate model for hypervelocity impact and on the seismic safety assessment of truss structures, suggesting the feasibility of the framework for important complex systems.

The introduction of this paper provides both an overview of the paper and a self-contained mini-tutorial about basic concepts and issues of UQ.
\end{abstract}

\newpage

\tableofcontents

\newpage

\section{Introduction}
\label{sec:Introduction}

\subsection{The UQ problem}

This paper sets out a rigorous and unified framework for the statement and solution of uncertainty quantification (UQ) problems centered on the notion of available information.  In general, UQ refers to any attempt to quantitatively understand the relationships among uncertain parameters and processes in physical processes, or in mathematical and computational models for them;  such understanding may be deterministic or probabilistic in nature.  However, to make the discussion specific, we start the description of the proposed framework as it applies to the certification problem;  Section \ref{sec:comparisons} gives a broader description of the purpose, motivation and applications of UQ in the proposed framework and a comparison with current methods.

By \emph{certification} we mean the problem of showing that, with probability at least $1 - \epsilon$, the real-valued response function $G$ of a given physical system will not exceed a given safety threshold $a$.  That is, we wish to show that
\begin{equation}
	\label{eq:def_cert}
	\P[G(X)\geq a] \leq \epsilon.
\end{equation}
In practice, the event $[ G(X) \geq a ]$ may represent the crash of an aircraft, the failure of a weapons system, or the average surface temperature on the Earth being too high.  The symbol $\P$ denotes the probability measure associated with the randomness of (some of) the input variables $X$ of $G$ (commonly referred to as ``aleatoric uncertainty'').

Specific examples of values of $\epsilon$ used in practice are: $10^{-9}$ in the aviation industry (for the maximum probability of a catastrophic event per flight hour, see \cite[p.581]{Soekkha:1997} and \cite{Boeing:2010}), $0$ in the seismic design of nuclear power plants \cite{Esteva:1970, Drenick:1980} and $0.05$ for the collapse of soil embankments in surface mining \cite[p.358]{Hustrulid:2000}.  In structural engineering \cite{Gillford:2005}, the maximum permissible probability of failure (due to any cause) is $10^{-4} K_s n_d / n_r$ (this is an example of $\epsilon$) where $n_d$ is the design life (in years), $n_r$ is the number of people at risk in the event of failure  and $K_s$ is given by the following values (with $1/\text{year}$ units): $0.005$ for places of public safety (including dams); $0.05$ for domestic, office or trade and industry structures;  $0.5$ for bridges; and $5$ for towers, masts and offshore structures.  In US environmental legislation, the maximum acceptable increased lifetime chance of developing cancer due to lifetime exposure to a substance is $10^{-6}$ \cite{Mantel:1961} (\cite{Kelly:1993} draws attention to the fact that ``there is no sound scientific, social, economic, or other basis for the selection of the threshold $10^{-6}$ as a cleanup goal for hazardous waste sites'').

One of the most challenging aspects of UQ lies in the fact that in practical applications, the measure $\P$ and the response function $G$ are not known a priori.  This lack of information, commonly referred to as ``epistemic uncertainty'', can be described precisely by introducing $\mathcal{A}$, the set of all \emph{admissible scenarios} $(f, \mu)$ for the unknown --- or partially known --- reality $(G, \P)$.  More precisely, in those applications, the available information does not determine $(G, \P)$ \emph{uniquely} but instead determines a set $\mathcal{A}$ such that any $(f, \mu) \in \mathcal{A}$ could a priori be $(G, \P)$.  Hence, $\mathcal{A}$ is a (possibly infinite-dimensional) set of measures and functions defining explicitly \emph{information on and assumptions about $G$ and $\P$}.  In practice, this set is obtained from physical laws, experimental data and expert judgment.  It then follows from $(G, \P) \in \mathcal{A}$ that
\begin{equation}
	\label{eq:def_cert_ouq_bounds}
	\inf_{(f, \mu) \in \mathcal{A}} \mu [f(X) \geq a] \leq \P [G(X) \geq a] \leq \sup_{(f, \mu) \in \mathcal{A}} \mu [f(X) \geq a].
\end{equation}
Moreover, it is elementary to observe that
\begin{compactitem}
	\item The quantities on the right-hand and left-hand of \eqref{eq:def_cert_ouq_bounds} are extreme values of optimization problems and elements of $[0, 1]$.
	\item  Both the right-hand and left-hand inequalities are optimal in the sense that they are the sharpest  bounds for $\P [G(X) \geq a]$ that are consistent with the information and assumptions $\mathcal{A}$.
\end{compactitem}
More importantly, in Proposition \ref{prop_opt}, we show that these two inequalities provide sufficient information to produce an optimal solution to the certification problem.

 \begin{figure}[tp]
	\begin{center}
			\includegraphics[width=0.4\textwidth]{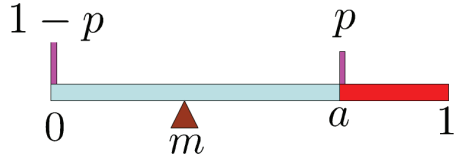}
		\caption{\emph{You are given one pound of play-dough and a seesaw balanced around $m$. How much mass can you put on right hand side of $a$ while keeping the seesaw balanced around $m$?} The solution of this optimization problem can be achieved by placing any mass on the right hand side of $a$, exactly at $a$ (to place mass on $[a,1]$ with minimum leverage towards the right hand side of the seesaw) and any mass on the left hand side of $a$, exactly at $0$ (for maximum leverage towards the left hand side of the seesaw).}\label{fig:playdoh}
	\end{center}
\end{figure}

\begin{eg}
	To give a very simple example of the effect of information and optimal bounds over a class $\mathcal{A}$, consider the certification problem \eqref{eq:def_cert} when $Y := G(X)$ is a real-valued random variable taking values in the interval $[0,1]$ and $a \in (0,1)$;  to further simplify the exposition, we consider only the upper bound problem, suppress dependence upon $G$ and $X$ and focus solely on the question of which probability measures $\nu$ on $\R$ are admissible scenarios for the probability distribution of $Y$.  So far, \emph{any} probability measure on $[0,1]$ is admissible:
	\[
		\mathcal{A} = \{ \nu \mid \text{$\nu$ is a probability measure on $[0,1]$} \}.
	\]
	and so the optimal upper bound in \eqref{eq:def_cert_ouq_bounds} is simply
	\[
		\P [Y \geq a] \leq \sup_{\nu \in \mathcal{A}} \nu [Y \geq a] = 1.
	\]
	Now suppose that we are given an additional piece of information:  the expected value of $Y$ equals $m \in (0,a)$.  These are, in fact, the assumptions corresponding to an elementary Markov inequality, and the corresponding admissible set is
	\[
		\mathcal{A}_{\text{Mrkv}} = \left\{ \nu \,\middle|\,
		\begin{array}{c}
			\text{$\nu$ is a probability measure on $[0,1]$,} \\
			 \E_{\nu}[Y] = m
		\end{array}
		\right\}.
	\]
The least upper bound on $\P [Y \geq a]$ corresponding to the admissible set $\mathcal{A}_{\text{Mrkv}}$ is the solution of the infinite dimensional optimization problem
	\begin{equation}\label{eq:leastupbound}
		\sup_{\nu \in \mathcal{A}_{\text{Mrkv}}} \nu [Y \geq a]
	\end{equation}
Formulating \eqref{eq:leastupbound} as a mechanical optimization problem (see Figure \ref{fig:playdoh}), it is easy to observe that the extremum of \eqref{eq:leastupbound} can be achieved only considering the situation where $\nu$ is the weighted sum of mass a Dirac at $0$ (with weight $1-p$) and a mass of Dirac at $a$ (with weight $p$). It follows that \eqref{eq:leastupbound} can be reduced to the simple (one-dimensional) optimization problem: \emph{Maximize $p$ subject to $a p=m$}. It follows that Markov's inequality is the optimal bound for the admissible set $\mathcal{A}_{\text{Mrkv}}$.
	\begin{equation}
		\P [Y \geq a] \leq \sup_{\nu \in \mathcal{A}_{\text{Mrkv}}} \nu [Y \geq a] = \tfrac{m}{a} \text{.}
	\end{equation}
	In some sense, the OUQ framework that we present in this paper is the the extension of this procedure to situations in which the admissible class $\mathcal{A}$ is complicated enough that a closed-form inequality such as Markov's inequality is unavailable, but optimal bounds can nevertheless be \emph{computed} using reduction properties analogous to the one illustrated in Figure \ref{fig:playdoh}.
\end{eg}

\subsection{Motivating physical example and outline of the paper}
\label{Subsec:motivatingex}

Section \ref{sec:OUQ} gives a formal description of the \emph{Optimal Uncertainty Quantification} framework.  In order to help intuition, we will illustrate and motivate our abstract definitions and results with a practical example: an analytical surrogate model for hypervelocity impact.

The physical system of interest is one in which a 440C steel ball (440C is a standard, i.e.\ a grade of steel) of diameter $D_{\mathrm{p}} = 1.778 \, \mathrm{mm}$ impacts a 440C steel plate of thickness $h$ (expressed in $\mathrm{mm}$) at speed $v$ (expressed in $\mathrm{km \cdot s}^{-1}$) at obliquity $\theta$ from the plate normal.  The physical experiments are performed at the California Institute of Technology SPHIR (Small Particle Hypervelocity Impact Range) facility (see Figure \ref{fig:sphir}).  An analytical surrogate model was developed to approximate the perforation area (in $\mathrm{mm}^{2}$) caused by this impact scenario. The surrogate response function is as follows:
\begin{equation}
	\label{eq:PSAAP_SPHIR_surr}
	H(h, \theta, v) = K \left( \frac{h}{D_{\mathrm{p}}} \right)^{p} (\cos \theta)^{u} \left( \tanh \left( \frac{v}{v_{\mathrm{bl}}} - 1 \right) \right)_{+}^{m},
\end{equation}
where the \emph{ballistic limit velocity} (the speed below which no perforation area occurs) is given by
\begin{equation}
	\label{eq:ballistic_limit}
	v_{\mathrm{bl}} := H_{0} \left( \frac{h }{(\cos \theta)^{n}} \right)^{s}.
\end{equation}
The seven quantities $H_{0}$, $s$, $n$, $K$, $p$, $u$ and $m$ are fitting parameters that have been chosen to minimize the least-squares error between the surrogate and a set of 56 experimental data points;  they take the values
\begin{align*}
	H_{0} &= 0.5794 \, \mathrm{km \cdot s}^{-1}, & s &= 1.4004, & n &= 0.4482, & K &= 10.3936 \, \mathrm{mm^{2}}, \\
	p &= 0.4757, & u &= 1.0275, & m &= 0.4682. &&
\end{align*}

 \begin{figure}[tp]
	\begin{center}
			\includegraphics[width=0.9\textwidth]{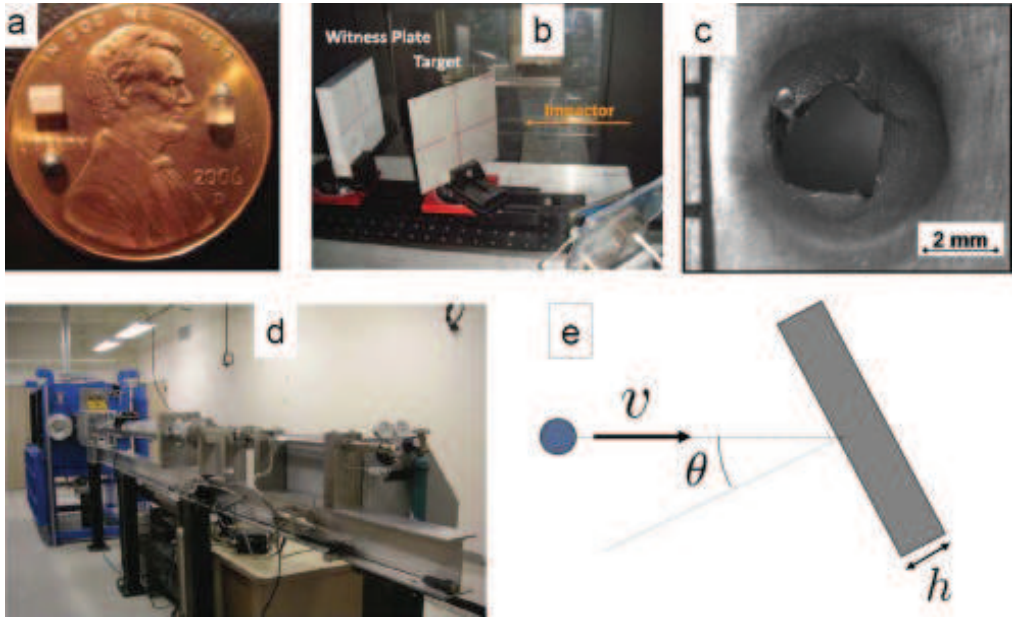}
		\caption{Experimenal set up. (a) Stainless steel spherical projectiles and nylon sabots. (b) Target plate held at the end of the gun barrel (c) Perforation of the target plate  (d) General view of the Small Particle Hypervelocity Impact Range (SPHIR) Facility at Caltech (e) plate thickness $h$, plate obliquity $\theta$ and projectile velocity $v$.}\label{fig:sphir}
	\end{center}
\end{figure}

Hence, in this illustrative example, $H(h, \theta, v)$ will be our response function  $G(X_{1}, X_{2}, X_{3})$ and we will consider  cases in which $H$ is perfectly and imperfectly known.  In Section \ref{Sec:seismic}, we will apply the OUQ framework to the seismic safety assessment of structures and consider a more complex example involving a large number of variables.

\subsubsection{Formulation of the admissible set and reduction theorems.}

For the example considered here, we will assume that the input parameters $h, \theta$ and $v$ are  random variables, of unknown probability distribution $\P$ and of given range
\begin{equation}\label{eq:PSAAP_SPHIR_range}
	\begin{split}
		(h, \theta, v) \in \mathcal{X} & := \mathcal{X}_{1} \times \mathcal{X}_{2} \times \mathcal{X}_{3}, \\
		h \in \mathcal{X}_{1} & := [1.524, 2.667] \, \mathrm{mm} = [60, 105] \, \mathrm{mils}, \\
		\theta \in \mathcal{X}_{2} & := [0, \tfrac{\pi}{6}], \\
		v \in \mathcal{X}_{3} & := [2.1, 2.8] \, \mathrm{km \cdot s}^{-1}.
	\end{split}
\end{equation}
We will measure lengths in both $\mathrm{mm}$ and $\mathrm{mils}$ (recall that $1\,\mathrm{mm}=39.4\,\mathrm{mils}$).

We will adopt the ``gunner's perspective'' that failure consists of not perforating the plate, and therefore seek to obtain an optimal bound on the probability of non-perforation, i.e.\ $\P[H\leq 0]$, with possibly incomplete information on $\P$ and $H$.

Assuming $H$ to be known, if the information on $\P$ is limited to the knowledge that velocity, impact obliquity and plate thickness are independent random variables and that the mean perforation area lies in a prescribed range $[m_{1}, m_{2}] := [5.5, 7.5] \, \mathrm{mm}^{2}$, then this information describes the admissible set $\mathcal{A}_{H}$, where
\begin{equation}
\label{eq:PSAAP_SPHIR_Admissible}
	\mathcal{A}_{H} := \left\{ (H, \mu) \,\middle|\,
	\begin{matrix}
		H \text{ given by \eqref{eq:PSAAP_SPHIR_surr},} \\
		\mu = \mu_{1} \otimes \mu_{2} \otimes \mu_{3}, \\
		m_{1}=5.5 \,\mathrm{mm}^{2} \leq \mathbb{E}_{\mu}[H] \leq m_2=7.5 \,\mathrm{mm}^{2}
	\end{matrix}
	\right\}.
\end{equation}

If the information on $H$ is limited to values of $\operatorname{Osc}_i(H)$,
the component-wise oscillations (defined below, it is a least upper bound on how a change in variable $i$ affects the response function), and if the information on $\P$ is as above, then the corresponding admissible set is $\mathcal{A}_{\mathrm{McD}}$, which corresponds to the assumptions of McDiarmid's inequality \cite{McDiarmid:1989}, and is defined by
\begin{equation}
\label{eq:PSAAP_SPHIR_Admissible_McD}
	\mathcal{A}_{\mathrm{McD}} :=
	\left\{ (f, \mu) \,\middle|\,
	\begin{matrix}
		\mu = \mu_{1} \otimes \mu_{2} \otimes \mu_{3}, \\
			m_{1}=5.5 \,\mathrm{mm}^{2} \leq \mathbb{E}_{\mu}[f] \leq m_2=7.5\, \mathrm{mm}^{2}, \\
		\operatorname{Osc}_{i}(f) \leq \operatorname{Osc}_{i}(H) \text{ for } i = 1, 2, 3
	\end{matrix}
	\right\}.
\end{equation}

\begin{defn}
Let $\mathcal{X} := \mathcal{X}_1 \times \dots \times \mathcal{X}_m$ and consider a function $f \colon \mathcal{X} \to \R$. For $i = 1, \dots, m$, we define the component-wise oscillations
\begin{equation}
	\label{eq:defOsc}
	\operatorname{Osc}_i(f) := \sup_{(x_1,\ldots,x_{m}) \in \mathcal{X}} \sup_{x_i' \in \mathcal{X}_{i}} \left| f(\ldots,x_i,\ldots)- f(\ldots,x_i',\ldots) \right|.
\end{equation}
Thus, $\operatorname{Osc}_i(f)$ measures the maximum oscillation of $f$ in the $i^{\text{th}}$ factor.
\end{defn}
\begin{rmk}
The explicit expression \eqref{eq:PSAAP_SPHIR_surr} of $H$ and the ranges
\eqref{eq:PSAAP_SPHIR_range} allow us to compute the component-wise oscillations $\operatorname{Osc}_i(H)$, which are, respectively, $8.86 \, \mathrm{mm}^2$, $4.17 \, \mathrm{mm}^2$ and $7.20 \, \mathrm{mm}^2$ for thickness, obliquity, and velocity.
\end{rmk}

In general, for any admissible set $\mathcal{A}$ of function/measure pairs for the perforation problem, we define
\begin{equation}
	\label{eq:opta}
	\mathcal{U}(\mathcal{A}):=\sup_{(f, \mu) \in \mathcal{A}} \mu[f(h,\theta,v)\leq 0].
\end{equation}
In this notation, the optimal upper bounds on the probability of non-perforation, given the information contained in $\mathcal{A}_{H}$ and $\mathcal{A}_{\mathrm{McD}}$, are $\mathcal{U}(\mathcal{A}_{H})$ and $\mathcal{U}(\mathcal{A}_{\mathrm{McD}})$ respectively.

In $\mathcal{A}_{H}$ the response function is exactly known whereas in $\mathcal{A}_{\mathrm{McD}}$ it is imperfectly known (the information on the response function is limited to its component-wise oscillations $\operatorname{Osc}_i(H)$). Both $\mathcal{A}_{H}$ and $\mathcal{A}_{\mathrm{McD}}$ describe epistemic uncertainties (since in $\mathcal{A}_{H}$ the probability distributions of thickness, obliquity, and velocity are imperfectly known). $\mathcal{A}_{\mathrm{McD}}$ is the set of response functions $f$ and
probability measures $\mu$ that could be $H$ and $\P$ given the information contained in (i.e.\ the constraints imposed by) $\operatorname{Osc}_i(H)$, the independence of the input variables and the bounds $m_1$ and $m_2$ on the mean perforation area. $\mathcal{U}(\mathcal{A}_{\mathrm{McD}})$ quantifies the worst case scenario, i.e. the largest probability of non-perforation given what $H$ and $\P$ could be.

\paragraph{Reduction theorems.}
The optimization variables associated with $\mathcal{U}(\mathcal{A}_{H})$ are tensorizations of probability measures on thickness $h$, on obliquity $\theta$ and velocity $v$.  This problem is not directly computational tractable since finding the optimum appears to require a search over the spaces of probability measures on the intervals $[1.524, 2.667] \, \mathrm{mm}$,  $[0, \tfrac{\pi}{6}]$ and $[2.1, 2.8] \, \mathrm{km \cdot s}^{-1}$. However, in Section \ref{sec:Reduction} (Theorem \ref{thm:baby_measure} and Corollary \ref{cor:ouqreduce}) we show that, since the constraint $m_{1} \leq \mathbb{E}_{\mu}[H] \leq m_{2}$ is multi-linear in $\mu_1,\mu_2$ and $\mu_3$, the optimum $\mathcal{U}(\mathcal{A}_{H})$ can be achieved by searching among those measures $\mu$ whose marginal distributions on each of the three input parameter ranges have support on at most two points.  That is,
\begin{equation}
\label{eq:Hdeltarecduced}
\mathcal{U}(\mathcal{A}_{H}) = \mathcal{U}(\mathcal{A}_{\Delta}),	
\end{equation}
where the reduced feasible set $\mathcal{A}_{\Delta}$ is given by
\begin{equation}
\label{eq:PSAAP_SPHIR_Admissible_Reduced}
	\mathcal{A}_{\Delta} := \left\{ (H, \mu) \,\middle|\,
	\begin{matrix}
		H \text{ given by \eqref{eq:PSAAP_SPHIR_surr},} \\
		\mu = \mu_{1} \otimes \mu_{2} \otimes \mu_{3}, \\
		\mu_{i} \in \Delta_{1}(\mathcal{X}_{i}) \text{ for } i = 1, 2, 3, \\
		m_{1} \leq \mathbb{E}_{\mu}[H] \leq m_{2}
	\end{matrix}
	\right\},
\end{equation}
where
\[
	\Delta_{1}(\mathcal{X}_i) := \left\{ \alpha \delta_{x^{0}}+ (1-\alpha) \delta_{x^{1}} \,\middle|\, x^{j} \in \mathcal{X}_i, \text{ for } j = 0, 1\text{ and } \alpha \in [0,1] \right\}
\]
denotes the set of binary convex combinations of Dirac masses on $\mathcal{X}_i$.

More generally, although the OUQ optimization problems \eqref{eq:def_cert_ouq_bounds} are extremely large, we show in Section \ref{sec:Reduction} that an important subclass enjoys significant and practical finite-dimensional reduction properties.  More precisely, although the optimization variables $(f, \mu)$ live in a product space of functions and probability measures, for OUQ problems governed by linear inequality constraints on generalized moments, we demonstrate in Theorem \ref{thm:baby_measure} and Corollary \ref{cor:ouqreduce} that the search can be reduced to one over probability measures that are products of finite convex combinations of Dirac masses with explicit upper bounds on the number of Dirac masses.  Moreover, all the results in this paper can be extended to sets of extreme points (extremal measures) more general than Dirac masses, such as those described by  Dynkin \cite{Dynkin:1978}; we have phrased the results in terms of Dirac masses for simplicity.

Furthermore,  when all constraints are generalized moments of functions of $f$,
the search over admissible functions reduces to a search over functions on an $m$-fold product of finite discrete spaces, and the search over $m$-fold products of finite convex combinations of Dirac masses reduce  to the products of probability measures on this $m$-fold product of finite discrete spaces.  This latter reduction, presented in Theorem \ref{thm:valuereduce}, completely eliminates dependency on the coordinate positions of the Dirac masses. With this result, the optimization variables of $\mathcal{U}(\mathcal{A}_{\mathrm{McD}})$ can be reduced to functions and  products of probability measures on $\{ 0, 1 \}^3$.

\subsubsection{Optimal concentration inequalities.}\label{subsec:sphirex}
	Concentration-of-measure inequalities can be used to obtain upper bounds on $\mathcal{U}(\mathcal{A}_{H})$ and $\mathcal{U}(\mathcal{A}_{\mathrm{McD}})$;  in that sense, they lead to sub-optimal methods. Indeed, according to McDiarmid's inequality \cite{McDiarmid:1989, McDiarmid:1998}, for all  functions $f$ of $m$ independent variables, one must have
\begin{equation}
	\label{eq:McDintro}
	\mu\big[ f(X_1,\ldots,X_m)-\E_\mu[f]\geq a\big] \leq \exp \left( -2\frac{a^2}{\sum_{i=1}^m (\operatorname{Osc}_{i}(f))^2 } \right).
\end{equation}
Application of this inequality  to \eqref{eq:PSAAP_SPHIR_Admissible_McD} (using  $\E_\mu[f]\geq m_1= 5.5 \, \mathrm{mm}^2$) yields
the bound
\begin{equation}\label{eq:664}
	 \mathcal{U}(\mathcal{A}_{\mathrm{McD}}) \leq \exp \left( - \frac{2 m_{1}^{2}}{\sum_{i = 1}^{3} \operatorname{Osc}_{i}(H)^{2}} \right) = 66.4\%.
\end{equation}
Note that $\mathcal{U}(\mathcal{A}_{\mathrm{McD}}):=\sup_{(f,\mu)\in \mathcal{A}_{\mathrm{McD}}}\mu[f\leq 0]$ is the least upper bound on the probability of non-perforation $\P[H=0]$ given the information contained in the admissible set \eqref{eq:PSAAP_SPHIR_Admissible_McD}.

In Section \ref{sec:Reduction-mcd}, the reduction techniques of Section \ref{sec:Reduction} are applied to obtain optimal McDiarmid and Hoeffding inequalities, i.e.  optimal concentration-of-measure inequalities with the assumptions of McDiarmid's inequality \cite{McDiarmid:1989} or Hoeffding's inequality \cite{Hoeffding:1963}.  In particular, Theorems \ref{thm:m1}, \ref{thm:m2} and \ref{thm:m3} provide analytic solutions to the McDiarmid problem for dimension $m = 1, 2, 3$, and  Proposition \ref{prop:McD_reduced_explicit_allm} provides a recursive formula for general $m$, thereby providing an optimal McDiarmid inequality in these cases.  In Theorems \ref{thm:Hfdm2} and \ref{thm:Hfdm3}, we give analytic solutions under Hoeffding's assumptions. A noteworthy result is that the optimal bounds associated with McDiarmid's and Hoeffding's assumptions are the same for $m=2$ but may be distinct for $m=3$, and so, in some sense, information about linearity or non-linearity of the response function has a different effect depending upon the dimension $m$ of the problem.

\paragraph{Non-propagation of uncertainties.}
For $m=2$, define $\mathcal{A}_2$ to be the space of all functions $f$ and measure $\mu$ such that $\mu=\mu_1 \otimes \mu_2$ and $\operatorname{Osc}_{i}(f)\leq D_i$. The optimal  concentration-of-measure inequality with the assumptions of McDiarmid's inequality, Theorem \ref{thm:m2}, states that
\begin{equation}
	\label{eq:McDintroreduced_explicit_2d}
	\sup_{(f,\mu) \in \mathcal{A}_2}  \mu\big[ f(X_1,X_2)-\E_\mu[f]\geq a\big] =
	\begin{cases}
		0, & \text{if } D_{1} + D_{2} \geq a, \\
		\dfrac{(D_{1} + D_{2} -a)^{2}}{4 D_{1} D_{2}}, & \text{if } | D_{1} - D_{2} | \leq a \leq  D_{1} + D_{2}, \\
		1 - \dfrac{a}{\max (D_{1}, D_{2})}, & \text{if } 0\leq  a \leq | D_{1} - D_{2} |.
	\end{cases}
\end{equation}
Observe that if $D_{2}+a\leq D_1$, then the optimal bound does not depend on
 $D_{2}$, and therefore, any decrease in $D_{2}$ does not improve the inequality.  These explicit bounds show that, although uncertainties may propagate for the true values of $G$ and $\P$ (as expected from the sensitivity analysis paradigm), they may fail to do so when the information is incomplete on $G$ and $\P$ and the objective is the maximum of $\mu[f\geq a]$ compatible with the given information.  The non-propagation of input uncertainties is a non-trivial observation related to the fact that some of the constraints defining the range of the input variables may not be realized by the worst-case scenario (extremum of the OUQ problem).
We have further illustrated this point in Section \ref{Sec:porous} and shown that for systems characterized by multiple scales or hierarchical structures, information or uncertainties may not propagate across scales. Note that the $m=2$ case does not apply to the SPHIR example  (since \eqref{eq:PSAAP_SPHIR_Admissible_McD} involves three variables, i.e. it requires $m=3$).

\paragraph{Application to the SPHIR facility admissible set \eqref{eq:PSAAP_SPHIR_Admissible_McD}.}
For $m=3$, the ``optimal McDiarmid inequality'' of Theorem \ref{thm:m3} and Remark \ref{rmk:baby_measure_indexing} provides the upper bound
\begin{equation}\label{eq:664b}
	\mathcal{U}(\mathcal{A}_{\mathrm{McD}}) = 43.7\%.	
\end{equation}
Remark \ref{rmk:deihdehdu} also shows that reducing the uncertainty in obliquity (described by the constraint $\operatorname{Osc}_{i}(f) \leq \operatorname{Osc}_{i}(H)$ in \eqref{eq:PSAAP_SPHIR_Admissible_McD} for $i=\text{obliquity}$)
does not affect the \emph{least}
 upper bound \eqref{eq:664b}. Recall that $\mathcal{U}(\mathcal{A}_{\mathrm{McD}})$ is the least upper bound on the probability that the perforation is zero given that the mean perforation area is in between $5.5 \,\mathrm{mm}^{2}$ and $7.5\, \mathrm{mm}^{2}$ and the constraints imposed by the independence, ranges and component-wise oscillations associated with the input random variables.

The difference between \eqref{eq:664} and \eqref{eq:664b} lies in the fact that $66.4\%$ is non-optimal whereas $43.7\%$ is the \emph{least} upper bound on the probability of non perforation given the information contained in
$\mathcal{A}_{\mathrm{McD}}$. $43.7\%$ is a direct function of
the information contained in $\mathcal{A}_{\mathrm{McD}}$ and Section \ref{sec:OUQ} shows how admissible sets with higher information content lead to smaller least upper bounds on the probability of non perforation.  Section \ref{sec:OUQ} also shows how such admissible sets can be constructed, in the OUQ framework, via the optimal selection of experiments.

\begin{figure}[tp]
	\begin{center}
		\subfigure[support points at iteration 0]{
			\includegraphics[width=0.45\textwidth]{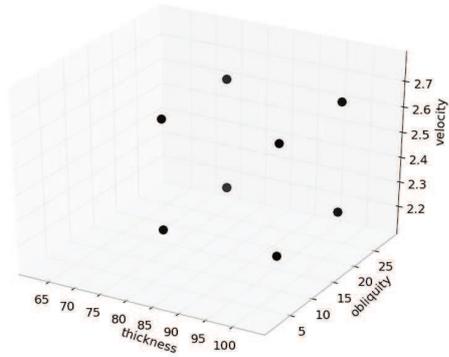}
		}
		\subfigure[support points at iteration 150]{
			\includegraphics[width=0.45\textwidth]{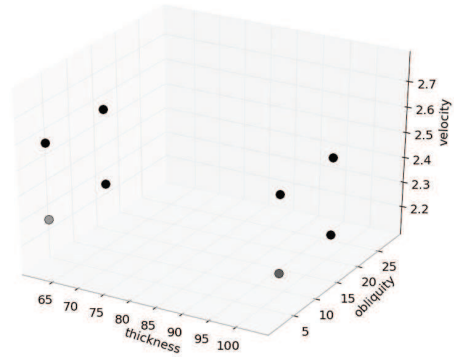}
		}
		\subfigure[support points at iteration 200]{
			\includegraphics[width=0.45\textwidth]{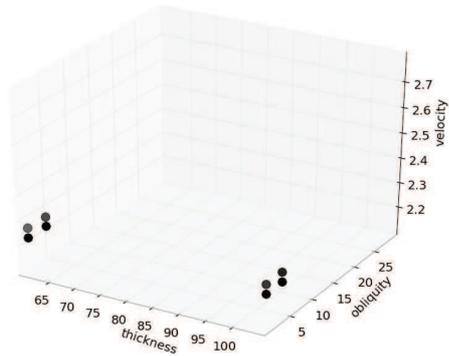}
		}
		\subfigure[support points at iteration 1000]{
			\includegraphics[width=0.45\textwidth]{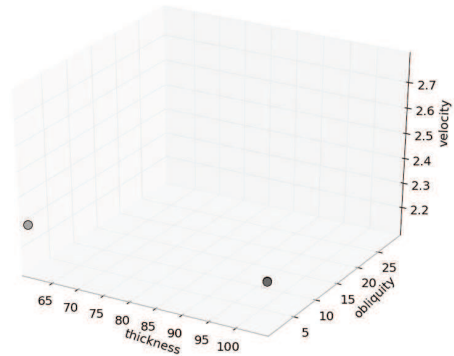}
		}
		\caption{
For  $\# \mathrm{supp}(\mu_{i}) \le 2, \, i = 1, 2, 3$,  the maximizers of the OUQ problem \eqref{eq:Hdeltarecduced} associated with the information set
\eqref{eq:PSAAP_SPHIR_Admissible} collapse to two-point (as opposed to eight-point) support. Velocity and obliquity marginals each collapse to a single Dirac mass, while the plate thickness marginal collapses to have support on the extremes of its range.  Note the perhaps surprising result that the probability of non-perforation is maximized by a distribution supported on the minimal, not maximal, impact obliquity.}
		\label{fig:CollapseSupport2}
	\end{center}
\end{figure}

\begin{figure}[tp]
	\begin{center}
		\caption{Time evolution of the genetic algorithm search for the OUQ problem \eqref{eq:Hdeltarecduced} associated with the information set
\eqref{eq:PSAAP_SPHIR_Admissible} (\eqref{eq:PSAAP_SPHIR_Admissible_Reduced} after reduction) for $\# \mathrm{supp}(\mu_{i}) \leq 2$ for $i = 1, 2, 3$, as optimized by \emph{mystic}. Thickness quickly converges to the extremes of its range, with a weight of 0.621 at $60 \, \mathrm{mils}$ and a weight of 0.379 at $105 \, \mathrm{mils}$. The degeneracy in obliquity at 0 causes the fluctuations seen in the convergence of obliquity weight.  Similarly, velocity converges to a single support point at 2.289 $\mathrm{km \cdot s}^{-1}$, the ballistic limit velocity for thickness $105 \, \mathrm{mils}$ and obliquity $0$ (see \eqref{eq:ballistic_limit}).}
		\label{fig:CollapseConverge2}
		\subfigure[convergence for thickness]{
			\includegraphics[width=0.40\textwidth]{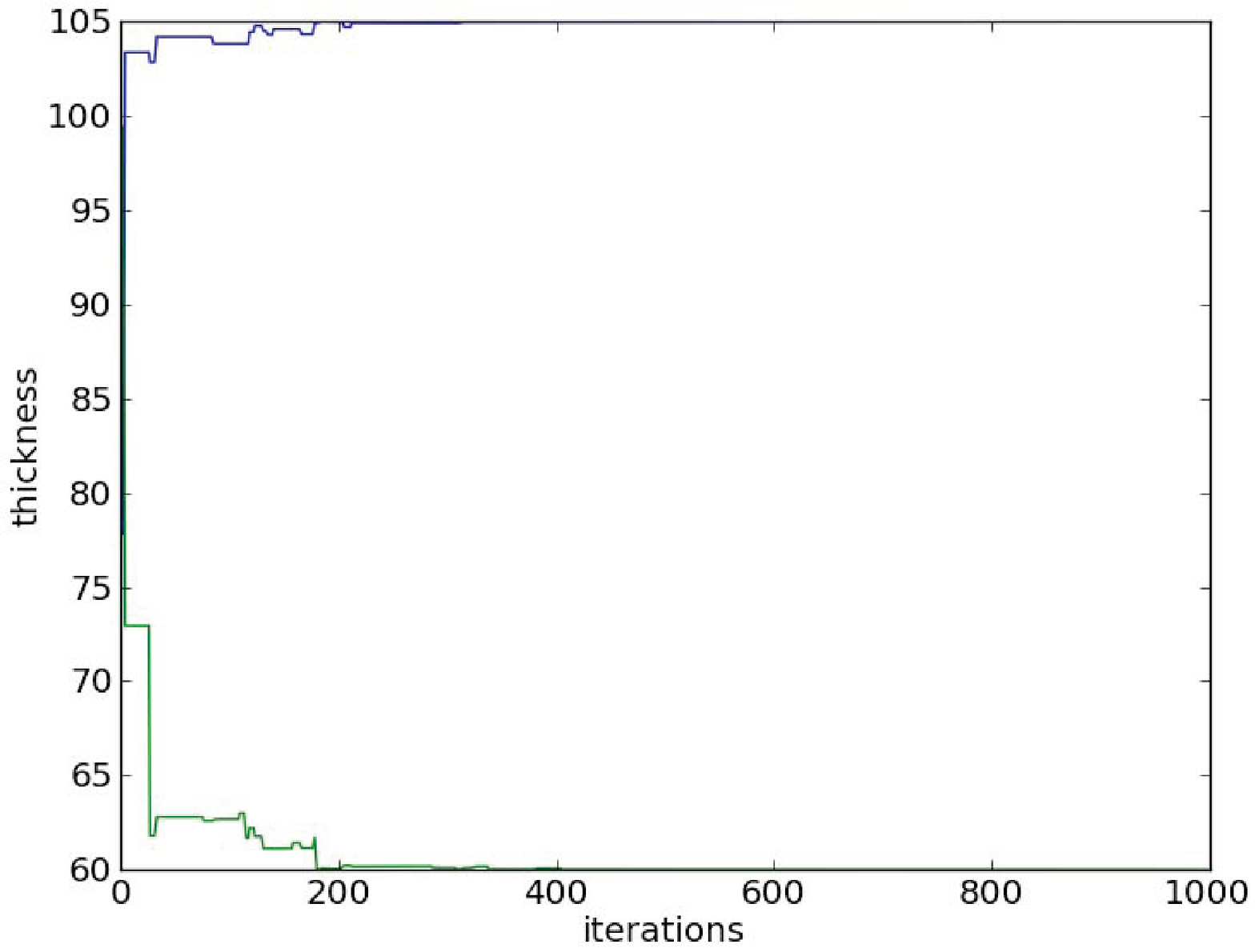}
		}
		\subfigure[convergence for thickness weight]{
			\includegraphics[width=0.40\textwidth]{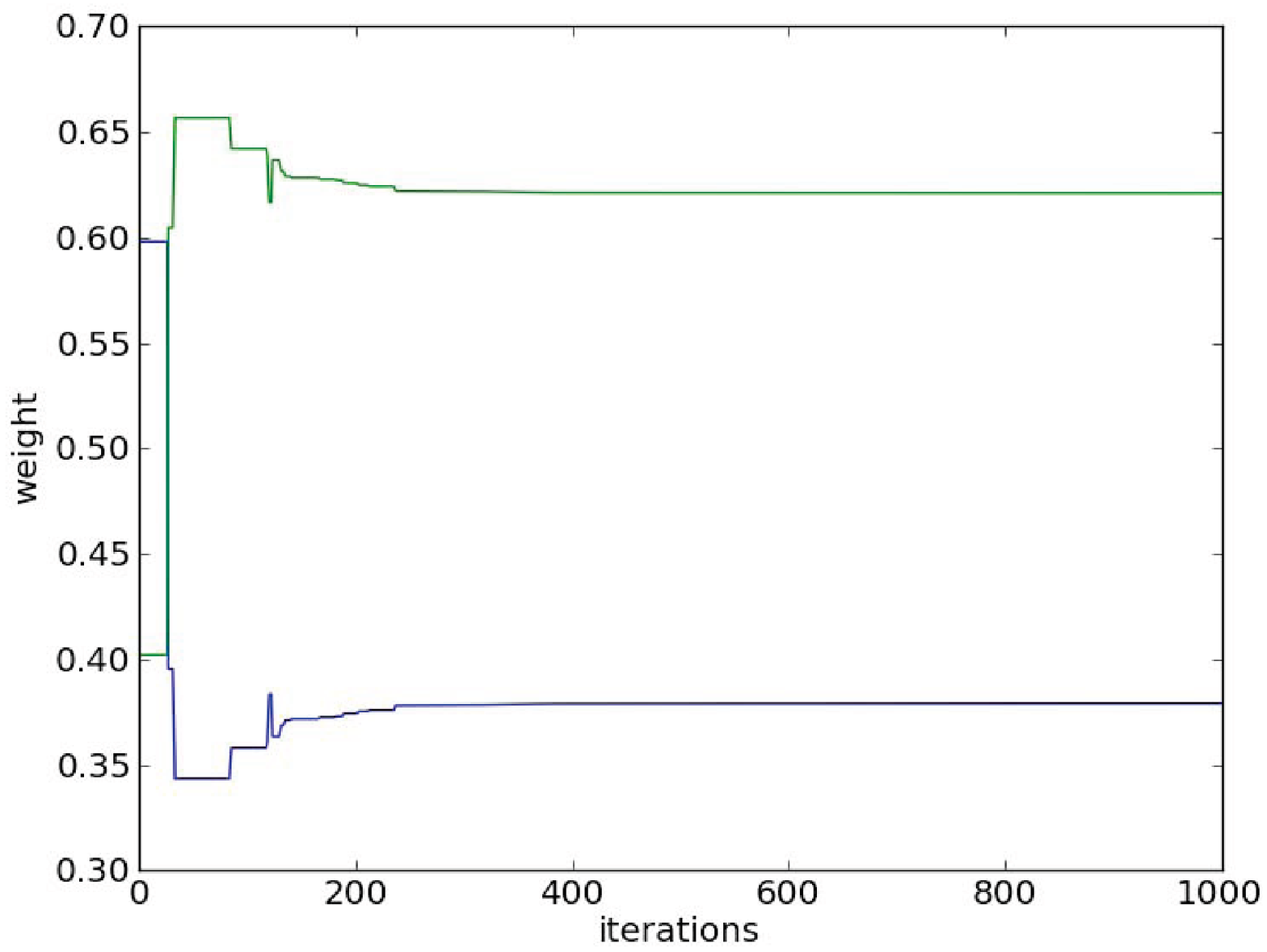}
		}
		\subfigure[convergence for obliquity]{
			\includegraphics[width=0.40\textwidth]{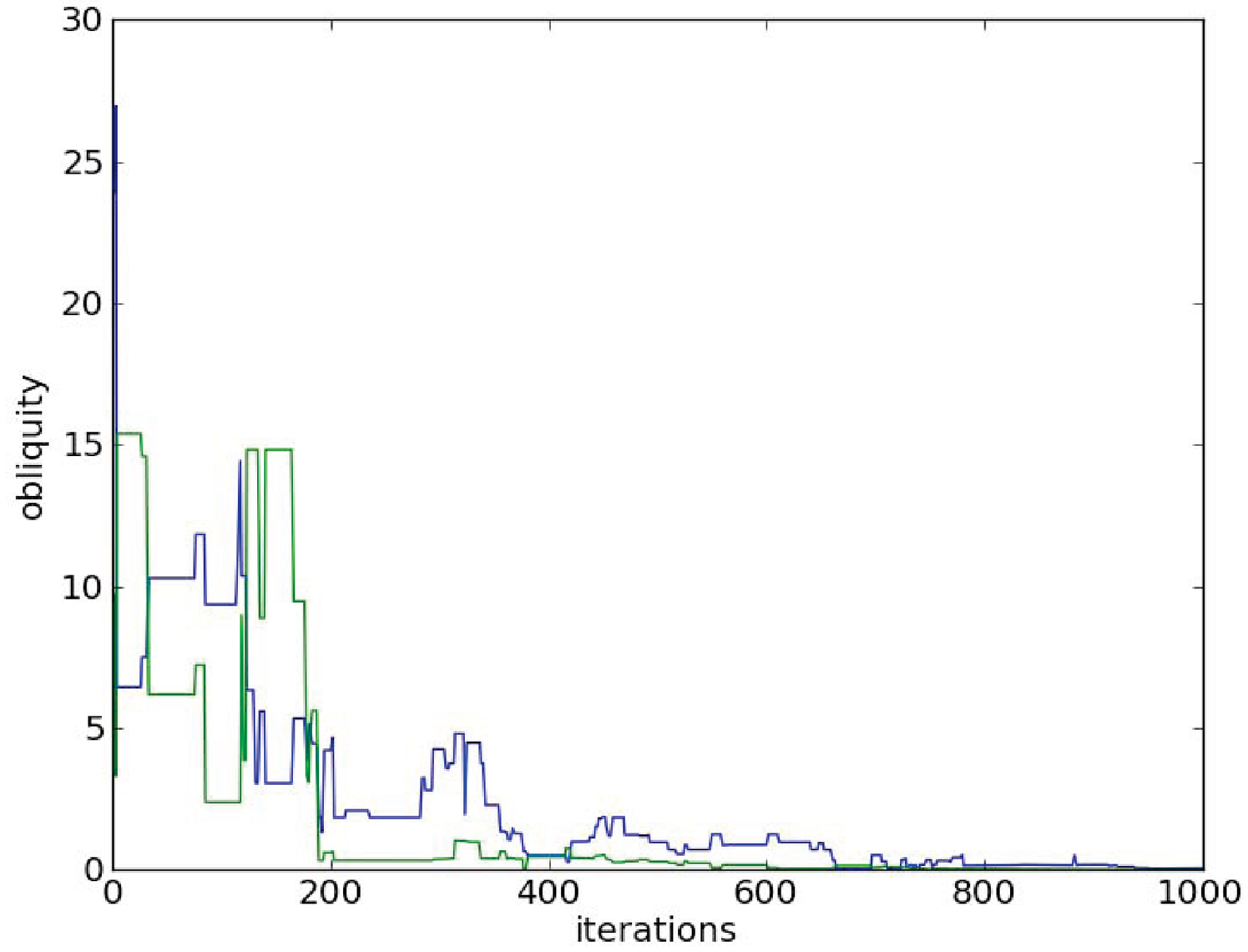}
		}
		\subfigure[convergence for obliquity weight]{
			\includegraphics[width=0.40\textwidth]{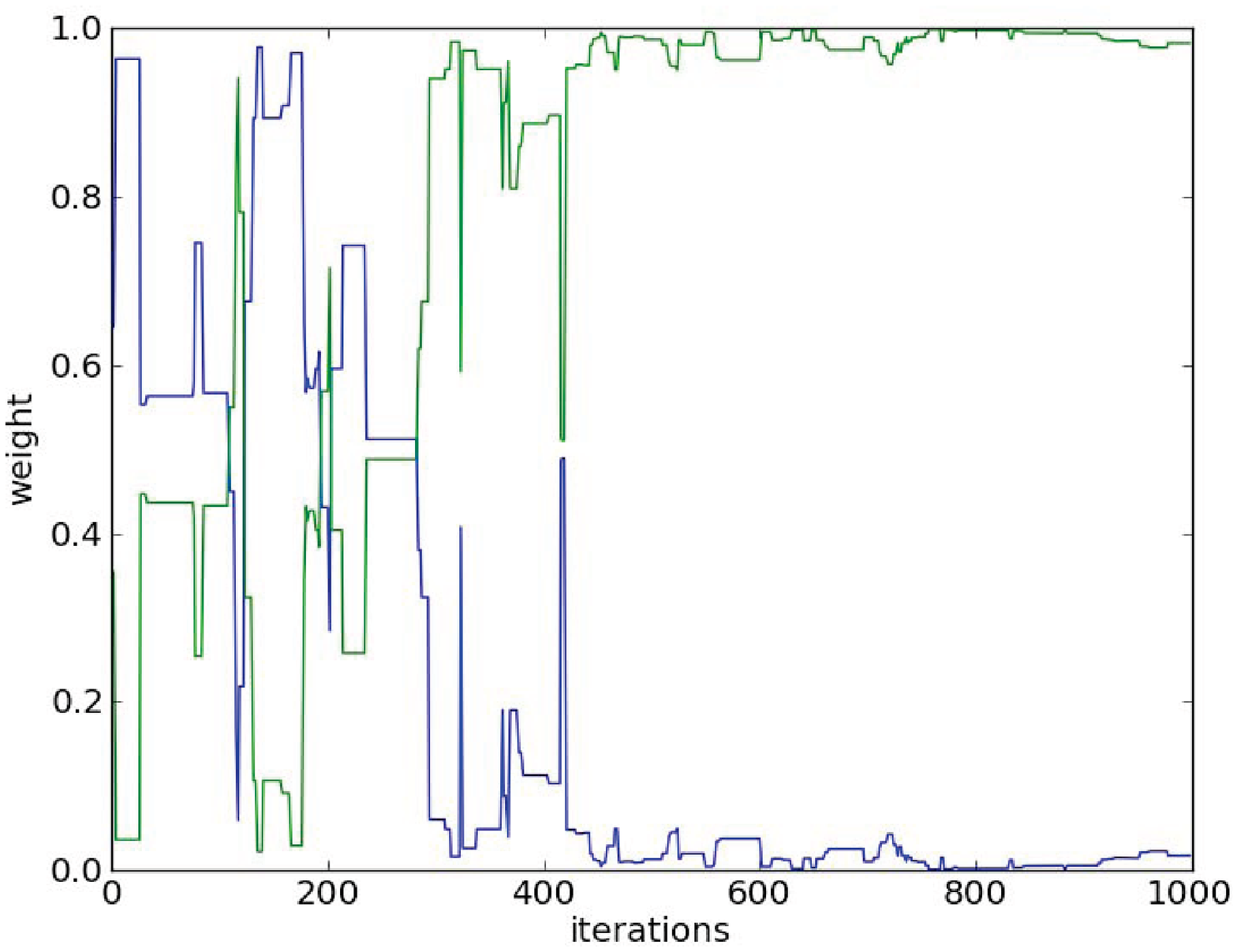}
		}
		\subfigure[convergence for velocity]{
			\includegraphics[width=0.40\textwidth]{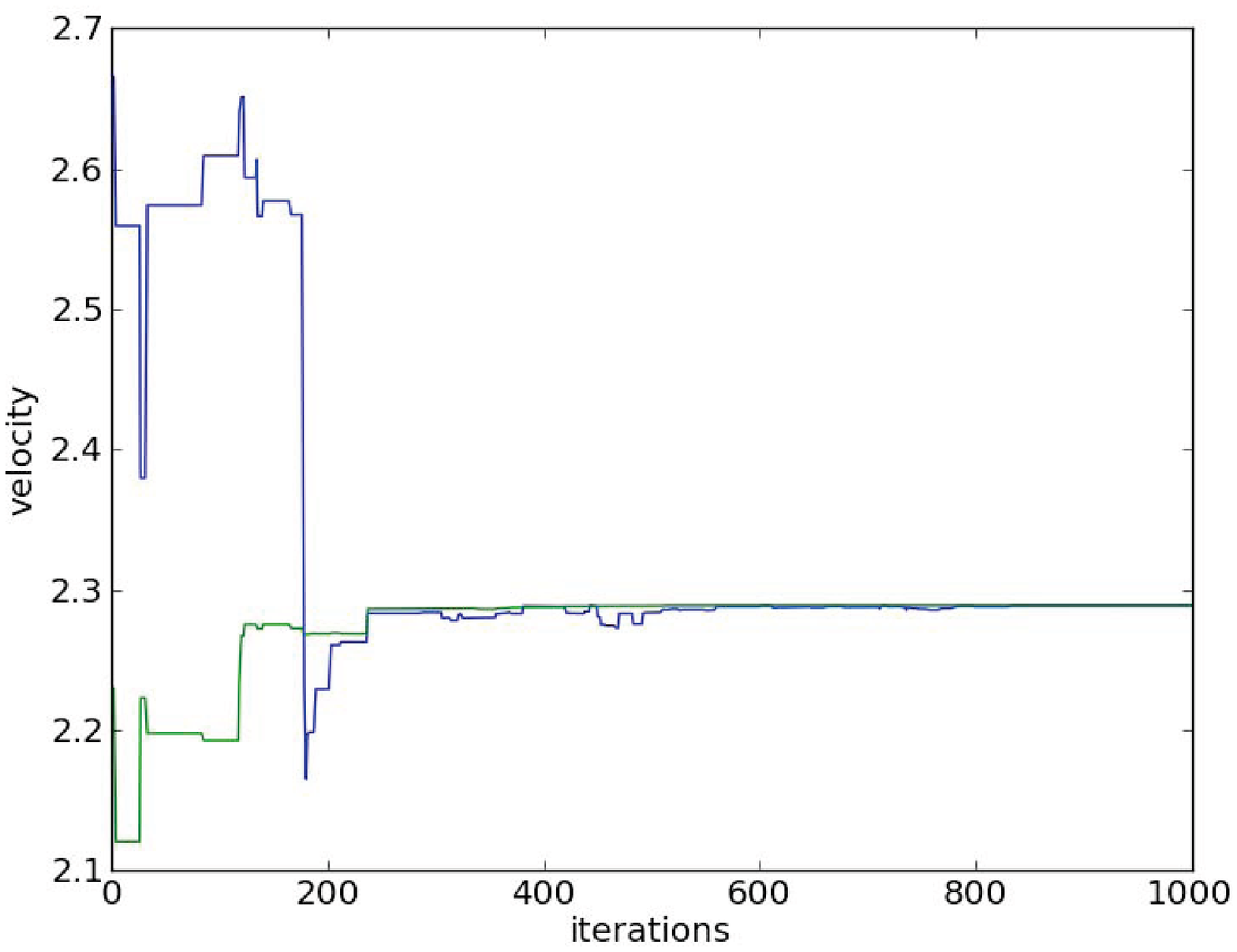}
		}
		\subfigure[convergence for velocity weight]{
			\includegraphics[width=0.40\textwidth]{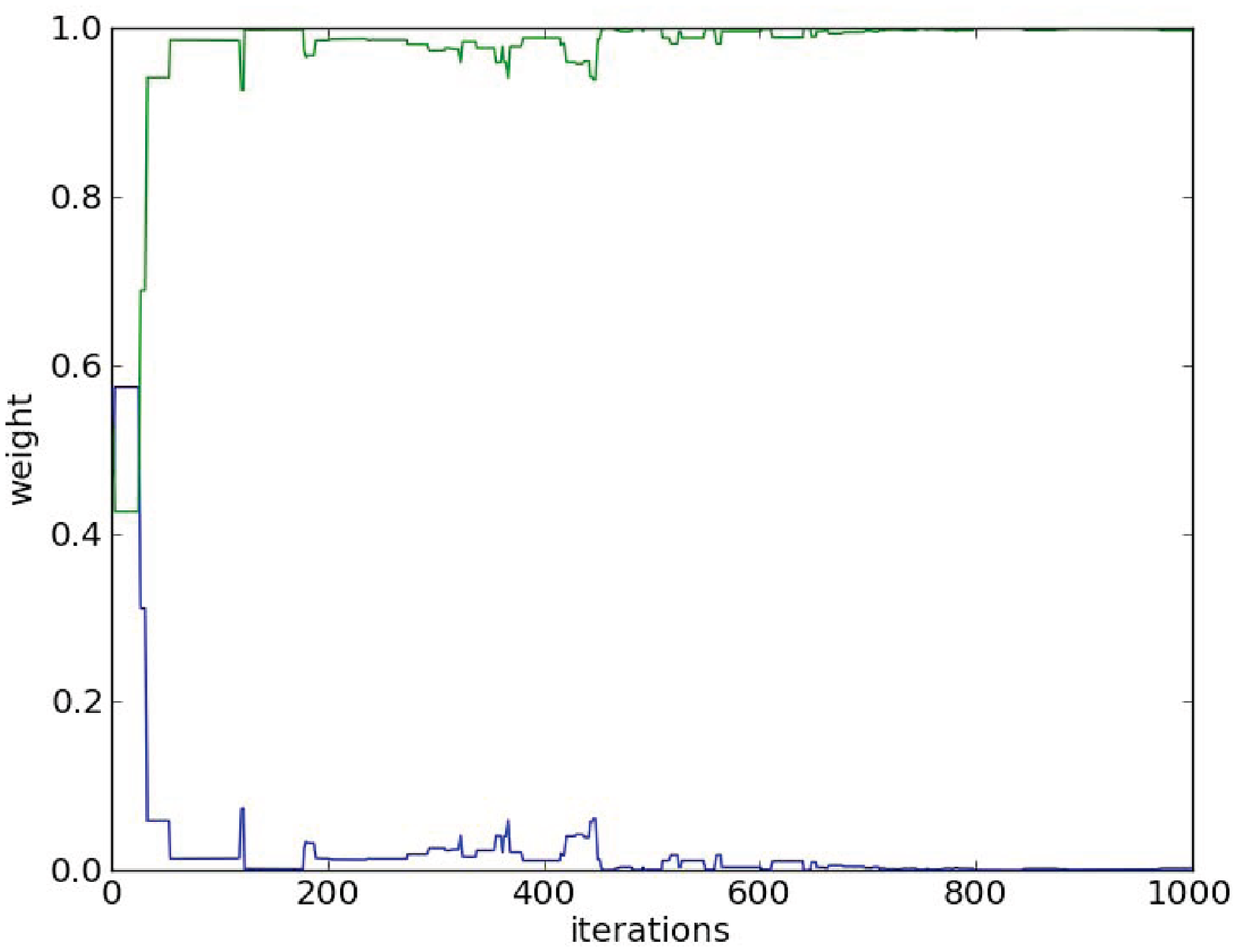}
		}
	\end{center}
\end{figure}

\subsubsection{Computational framework.}
With access to $H$, not just its componentwise oscillations, even sharper bounds on the probability of non-perforation can be calculated.  Although we do not have an analytical formula for $\mathcal{U}(\mathcal{A}_{H})$, its calculation is made possible by the identity \eqref{eq:Hdeltarecduced} derived from  the reduction results of Section \ref{sec:Reduction}.  A numerical optimization over the finite-dimensional reduced feasible set $\mathcal{A}_{\Delta}$ using a Differential Evolution \cite{PriceStornLampinen:2005} optimization algorithm  implemented in the \emph{mystic framework} \cite{McKernsOwhadiScovelSullivanOrtiz:2010}  (see Subsection \ref{sec:Mystic}) yields the following optimal upper bound on the probability of non-perforation:
\[
	\P[H = 0] \leq \mathcal{U}(\mathcal{A}_{H}) = \mathcal{U}(\mathcal{A}_{\Delta}) \numeq 37.9\%.
\]
Observe that ``$\P[H = 0] \leq \mathcal{U}(\mathcal{A}) = \mathcal{U}(\mathcal{A}_{\Delta})$'' is a theorem, whereas ``$\mathcal{U}(\mathcal{A}_{\Delta}) \numeq 37.9\%$'' is the output of an algorithm (in this case, a genetic algorithm implemented in the \emph{mystic framework} described in Subsection \ref{sec:Mystic}). In particular, its validity is correlated with the efficiency of the specific algorithm.  We have introduced the symbol $\numeq$ to emphasize the distinction between mathematical (in)equalities and numerical outputs.

Although we do not have a theorem associated with the convergence of the numerical optimization algorithm,   we have a robust control over its efficiency because  it is applied to the finite dimensional problem $\mathcal{U}(\mathcal{A}_{\Delta})$ instead of the infinite optimization problem associated with $\mathcal{U}(\mathcal{A}_{H})$ (thanks to the reduction theorems obtained in Section \ref{sec:Reduction}).

We also observe that the maximizer $\mathcal{U}(\mathcal{A}_{H})$ can be of significantly smaller dimension than that of the elements of $\mathcal{U}(\mathcal{A}_{\Delta})$. Indeed, for  $\# \mathrm{supp}(\mu_{i}) \le 2, \, i = 1, 2, 3$ (where $\# \mathrm{supp}(\mu_{i})$ is the number of points forming the support of $\mu_i$), Figure \ref{fig:CollapseSupport2} shows that numerical simulations collapse to two-point support:  the velocity and obliquity marginals each collapse to a single Dirac mass, and the plate thickness marginal collapses to have support on the two extremes of its range.  See Figure \ref{fig:CollapseConverge2} for plots of the locations and weights of the Dirac masses forming each marginal $\mu_i$ as  functions of the number of iterations. Note that the lines for \emph{thickness} and \emph{thickness weight} are of the same color if they correspond to the same support point for the measure.

In Section \ref{sec:ComputationalExamples} we observe that, even when a wider search is performed (i.e.\ over measures with more than two-point support per marginal), the calculated maximizers for these problems maintain two-point support.   This observation suggests that the extreme points of the reduced  OUQ problems are, in some sense, attractors  and that adequate numerical implementation of OUQ problems can detect and use ``hidden'' reduction properties even in the absence of theorems proving them to be true.  Based on these observations, we propose, in Section \ref{sec:ComputationalExamples}, an OUQ optimization algorithm for arbitrary constraints based on  a coagulation/fragmentation of probability distributions.

The simulations of Figures \ref{fig:CollapseSupport2} and \ref{fig:CollapseConverge2} show that extremizers are singular and that their support points identify key players, i.e. weak points of the system. In particular, for $\mathcal{U}(\mathcal{A}_{H})$, the location of the two-point support extremizer shows that reducing maximum obliquity and the range of velocity will not decrease the optimal bound on the probability of non perforation, and suggests that reducing the uncertainty in thickness will decrease this bound.  In Section \ref{sec:OUQ}, we will show that the OUQ framework allows the development of an OUQ loop that can be used for experimental design. In particular, we will show that the problem of predicting optimal bounds on the results of experiments under the assumption that the system is safe (or unsafe) is well-posed and benefits from similar reduction properties as the certification problem. Best experiments are then naturally identified as those whose predicted ranges have minimal overlap between safe and unsafe systems.

\subsubsection{Outline of the paper.} Section \ref{sec:OUQ}  formally describes the \emph{Optimal Uncertainty Quantification} framework and what it means to give optimal bounds on the probability of failure in \eqref{eq:def_cert} given (limited) information/assumptions about the system of interest, and hence how to rigorously certify or de-certify that system.  Section \ref{sec:comparisons} shows that many important UQ problems, such as prediction, verification and validation, can be formulated as certification problems. Section \ref{sec:comparisons} also gives a comparison of OUQ with other widely used UQ methods. Section \ref{sec:Reduction} shows that although  OUQ optimization problems \eqref{eq:def_cert_ouq_bounds} are (a priori)
infinite-dimensional, they can (in general) be reduced to equivalent finite-dimensional problems in which the optimization is over the extremal scenarios of $\mathcal{A}$ and that the dimension of the reduced problem is a function of the number of probabilistic inequalities that describe $\mathcal{A}$. Just as a linear program finds its extreme value at the extremal points of a convex domain in $\R^n$, OUQ problems reduce to searches over finite-dimensional families of extremal scenarios.
Section \ref{sec:Reduction-mcd} applies the results of Section \ref{sec:Reduction}  to obtain optimal concentration inequalities under the assumptions of McDiarmid's inequality and Hoeffding's inequality. Those inequalities show that, although uncertainties may propagate for the true values of $G$ and $\P$,  they might not when the information is incomplete on $G$ and $\P$.
Section \ref{sec:ComputationalExamples} discusses the numerical implementation of OUQ algorithms for the analytical surrogate model \eqref{eq:PSAAP_SPHIR_surr} for hypervelocity impact.
Section  \ref{Sec:seismic} assesses the feasibility of the OUQ formalism  by means of an application to the safety assessment of truss structures subjected to ground motion excitation. This application contains many of the features that both motivate and challenge UQ, including imperfect knowledge of random inputs of high dimensionality, a time-dependent and complex response of the system, and the need to make high-consequence decisions pertaining to the safety of the system.  Section \ref{Sec:porous}
applies the OUQ framework and reduction theorems of sections \ref{sec:Reduction} and \ref{sec:Reduction-mcd} to divergence form elliptic PDEs. A striking observation of  Section \ref{Sec:porous} is that with incomplete information on the probability distribution of the microstructure, uncertainties or information do not necessarily propagate across scales.
Section \ref{sec:Conclusions} emphasizes  that the ``UQ problem''(as it is treated in common practice today) appears to have all the symptoms of an ill-posed problem; and that, at the very least, it lacks a coherent general presentation, much like the state of probability theory before its rigorous formulation by Kolmogorov in the 1930s.
It also stresses that  OUQ is not the definitive answer to the UQ problem, but an opening gambit.
Section \ref{sec-appendixproofs} gives the proofs of our main results.

\section{Optimal Uncertainty Quantification}
\label{sec:OUQ}

In this section, we describe more formally the \emph{Optimal Uncertainty Quantification} framework.  In particular, we describe what it means to give optimal bounds on the probability of failure in \eqref{eq:def_cert} given information/assumptions about the system of interest, and hence how to rigorously certify or de-certify that system.

For the sake of clarity, we will start the description of OUQ with deterministic information and assumptions (when $\mathcal{A}$ is a deterministic set of functions and probability measures).

\subsection{First description}
In the OUQ paradigm, information and assumptions lie at the core of UQ:
     the available information and assumptions describe sets of admissible scenarios over which optimizations will be performed.
 As noted by Hoeffding \cite{Hoeffding:1956}, assumptions about the system of interest play a central and sensitive role in any statistical decision problem, even though the assumptions are often only approximations of reality.

A simple example of an information/assumptions set is given by constraining the mean and range of the response function.  For example, let $\mathcal{M}(\mathcal{X})$ be the set of probability measures on the set $\mathcal{X}$, and let $\mathcal{A}_{1}$  be the set of pairs of probability measures $\mu \in \mathcal{M}(\mathcal{X})$ and real-valued measurable functions $f$ on $\mathcal{X}$ such that the mean value of $f$ with respect to $\mu$ is $b$ and the diameter of the range of $f$ is at most $D$;
\begin{equation}
	\label{eq:simple_assumptions}
	\mathcal{A}_{1} := \left\{ (f, \mu) \,\middle|\,
	\begin{matrix}
		f \colon \mathcal{X} \to \R, \\
	  \mu \in \mathcal{M}(\mathcal{X}), \\
	 	\mathbb{E}_{\mu}[f] = b, \\
		(\sup f - \inf f) \leq D
	\end{matrix} \right\}.
\end{equation}
Let us assume that \emph{all} that we know about the ``reality'' $(G, \mathbb{P})$ is that $(G, \mathbb{P}) \in \mathcal{A}_{1}$.  Then any other
pair $(f, \mu) \in \mathcal{A}_{1}$ constitutes an admissible scenario representing a valid possibility for the ``reality'' $(G, \P)$.  If asked to bound $\P[G(X) \geq  a]$,  should we apply different methods and obtain different bounds on $\P[G(X) \geq a]$?  Since some methods will distort this information set and
 others are only using part of it, we instead view set $\mathcal{A}_{1}$ as a feasible set for an optimization problem.

\paragraph{The general OUQ framework.}  In the general case, we regard the response function $G$ as an
unknown measurable function, with some possibly  known characteristics, from one measurable space $\mathcal{X}$ of inputs to a second measurable space $\mathcal{Y}$ of values.  The input variables are generated randomly according to an unknown random variable $X$ with values in $\mathcal{X}$ according to a law $\P \in \mathcal{M}(\mathcal{X})$, also with some possibly  known characteristics.  We let a  measurable subset $\mathcal{Y}_{0} \subseteq \mathcal{Y}$ define the \emph{failure region};  in the example given above, $\mathcal{Y} = \R$ and $\mathcal{Y}_{0} = [a, +\infty)$.  When there is no danger of confusion, we shall simply write $[G \text{ fails}]$ for the event $[G(X) \in \mathcal{Y}_{0}]$.

Let $\epsilon \in [0, 1]$ denote the \emph{greatest acceptable probability of failure}.  We say that the system is  \emph{safe}  if $\P[G \text{ fails}] \leq \epsilon$ and the system is \emph{unsafe} if $\P[G \text{ fails}] > \epsilon$.  By \emph{information}, or a \emph{set of assumptions}, we mean a  subset
\begin{equation}
	\label{eq:A}
	\mathcal{A} \subseteq \left\{ (f, \mu) \,\middle|\, \begin{matrix}
		f \colon \mathcal{X} \to \mathcal{Y} \text{ is measurable,} \\
		\mu \in \mathcal{M}(\mathcal{X})
 \end{matrix} \right\}
\end{equation}
that contains, at the least, $(G, \P)$.  The set $\mathcal{A}$ encodes all the information that we have about the real system $(G, \P)$,   information that may come from known physical laws, past experimental data, and expert opinion.  In the example $\mathcal{A}_{1}$ above, the only information that we have is that the mean response  of the system is $b$ and that the diameter of its range is at most $D$;  any pair $(f, \mu)$ that satisfies these two criteria is an \emph{admissible scenario} for the unknown reality $(G, \P)$. Since some admissible scenarios may be safe (i.e.\ have $\mu[f \text{ fails}] \leq \epsilon$) whereas other admissible scenarios may be unsafe (i.e.\ have $\mu[f \text{ fails}] > \epsilon$), we decompose
   $\mathcal{A}$ into the disjoint union $\mathcal{A} = \mathcal{A}_{\mathrm{safe}, \epsilon} \uplus \mathcal{A}_{\mathrm{unsafe}, \epsilon}$, where
\begin{subequations}
	\label{eq:A_safe_and_unsafe}
	\begin{align}
		\label{eq:A_safe}
		\mathcal{A}_{\mathrm{safe}, \epsilon} & := \{ (f, \mu) \in \mathcal{A} \mid \mu [ f \text{ fails} ] \leq \epsilon \}, \\
		\label{eq:A_unsafe}
		\mathcal{A}_{\mathrm{unsafe}, \epsilon} & := \{ (f, \mu) \in \mathcal{A} \mid \mu [ f \text{ fails} ] > \epsilon \}.
	\end{align}
\end{subequations}

Now observe that, given such an information/assumptions set $\mathcal{A}$, there exist  upper and lower bounds on $\P[ G(X) \geq a ]$ corresponding to the scenarios compatible with assumptions, i.e.\ the values $\mathcal{L}(\mathcal{A})$ and  $\mathcal{U}(\mathcal{A})$ of the optimization problems:
\begin{subequations}
	\label{eq:def_ouq}
	\begin{align}	
		\mathcal{L}(\mathcal{A}) & := \inf_{(f, \mu) \in \mathcal{A}} \mu [f \text{ fails}] \\
		\mathcal{U}(\mathcal{A}) & := \sup_{(f, \mu) \in \mathcal{A}} \mu [f \text{ fails}].
	\end{align}
\end{subequations}
Since $\mathcal{L}(\mathcal{A})$ and $\mathcal{U}(\mathcal{A})$ are well-defined in $[0, 1]$, and approximations
are sufficient for most purposes and are necessary in general, the difference between $\sup$ and $\max$ should not be much of an issue.  Of course, some of the work that follows is concerned with the attainment of maximizers, and whether those maximizers have any simple structure that can be exploited for the sake of computational efficiency, and this is the topic of Section \ref{sec:Reduction}.  For the moment, however, simply assume that $\mathcal{L}(\mathcal{A})$ and $\mathcal{U}(\mathcal{A})$ can indeed be computed on demand.  Now, since $(G, \P) \in \mathcal{A}$, it follows that
\[
	\mathcal{L}(\mathcal{A}) \leq \P[G \text{ fails}] \leq \mathcal{U}(\mathcal{A}).
\]
Moreover, the upper bound $\mathcal{U}(\mathcal{A})$ is  \emph{optimal}  in the sense that
\[
	\mu[f \text{ fails}] \leq \mathcal{U}(\mathcal{A}) \text{ for all } (f,\mu) \in \mathcal{A}
\]
and, if $U' < \mathcal{U}(\mathcal{A})$, then there is an admissible scenario $(f, \mu) \in \mathcal{A}$ such that
\[
	U' <\mu[f \text{ fails}] \leq	 \mathcal{U}(\mathcal{A}).
\]
That is, although $\P[G \text{ fails}]$ may be much smaller than $\mathcal{U}(\mathcal{A})$, there is a pair $(f, \mu)$ which satisfies the same assumptions as $(G, \P)$ such that $\mu[f \text{ fails}]$ is approximately equal to $ \mathcal{U}(\mathcal{A})$.  Similar remarks apply for the lower bound $\mathcal{L}(\mathcal{A})$.

Moreover, the values $\mathcal{L}(\mathcal{A})$ and $\mathcal{U}(\mathcal{A})$, defined in \eqref{eq:def_ouq} can be used to construct a solution to the certification problem.  Let the certification problem be defined by an error function that gives an error whenever  1) the certification process produces ``safe'' and there exists an admissible scenario that is unsafe,  2) the certification  process produces ``unsafe'' and there exists an admissible scenario that is safe,  or 3) the certification process produces ``cannot decide'' and all admissible scenarios are safe or all admissible points are unsafe;  otherwise, the  certification process produces no error.  The following proposition demonstrates that, except in the special case $\mathcal{L}(\mathcal{A})=\epsilon$, that these values determine an optimal solution to this certification problem.

\begin{prop}
	\label{prop_opt}
	If $(G,\P)\in \mathcal{A}$ and
	\begin{itemize}
		\item $\mathcal{U}(\mathcal{A})\leq \epsilon$ then $\P[G \text{ fails}]\leq \epsilon$.
		\item $\epsilon < \mathcal{L}(\mathcal{A})$ then $\P[G \text{ fails}]> \epsilon$.
		\item $ \mathcal{L}(\mathcal{A})<\epsilon <\mathcal{U}(\mathcal{A})$ the there exists $(f_1,\mu_1)\in \mathcal{A}$ and  $(f_2,\mu_2)\in \mathcal{A}$ such that $\mu_1[f_1 \text{ fails} ]<\epsilon <\mu_2[f_2 \text{ fails} ]$.
	\end{itemize}
\end{prop}

In other words, provided that the information set $\mathcal{A}$ is valid (in the sense that  $(G,\P)\in \mathcal{A}$) then if $\mathcal{U}(\mathcal{A}) \leq \epsilon $, then, the system is provably safe;  if $\epsilon < \mathcal{L}(\mathcal{A})$, then  the system is provably unsafe;  and if $\mathcal{L}(\mathcal{A}) < \epsilon < \mathcal{U}(\mathcal{A})$, then the safety of the system cannot be decided due to lack of information.   The corresponding certification process and its optimality are represented in Table \ref{tab:cons_cert_criteria}. Hence,  solving the optimization problems \eqref{eq:def_ouq} determines an optimal solution to the certification problem, under the condition that $\mathcal{L}(\mathcal{A})\neq \epsilon$. When $\mathcal{L}(\mathcal{A})=\epsilon$  we can still produce an optimal solution if we obtain further information.  That is, when $\mathcal{L}(\mathcal{A})=\epsilon = \mathcal{U}(\mathcal{A})$, then the optimal process produces ``safe''.   On the other hand, when $\mathcal{L}(\mathcal{A})=\epsilon <\mathcal{U}(\mathcal{A})$, the optimal solution depends on whether or not there exists a minimizer $(f,\mu)\in \mathcal{A}$ such that $\mu[f \text{ fails} ]=\mathcal{L}(\mathcal{A})$; if so, the optimal process should declare ``cannot decide'', otherwise, the optimal process should declare ``unsafe''.  Observe that, in Table \ref{tab:cons_cert_criteria}, we have classified $\mathcal{L}(\mathcal{A}) = \epsilon < \mathcal{U}(\mathcal{A})$ as ``cannot decide''.  This ``nearly optimal'' solution appears natural and conservative without the knowledge of the existence or non-existence of optimizers.

\begin{table}[tp]
  \begin{center}
		\begin{tabular}{|c||c|c|}
			\hline
			 & $\mathcal{L}(\mathcal{A}) := \displaystyle \inf_{(f, \mu)\in \mathcal{A}} \mu \big[f(X) \geq a \big]$ & $\mathcal{U}(\mathcal{A}) := \displaystyle \sup_{(f, \mu)\in \mathcal{A}} \mu \big[f(X) \geq a \big]$ \\
			\hline \hline
			$\leq \epsilon$ & $\begin{matrix} \textbf{Cannot
 decide} \\ \text{Insufficient Information}  \end{matrix}$ & $\begin{matrix} \textbf{Certify} \\ \text{{\bf Safe} even in the Worst Case} \end{matrix}$ \\
			\hline
			$> \epsilon$ & $\begin{matrix} \textbf{De-certify} \\ \text{{\bf Unsafe} even in the Best Case} \end{matrix}$ & $\begin{matrix} \textbf{Cannot decide} \\ \text{Insufficient Information}  \end{matrix}$ \\
			\hline
		\end{tabular}
		\caption{The OUQ certification process provides a rigorous certification criterion whose outcomes are of three types: ``Certify'', ``De-certify'' and ``Cannot decide''.}
		\label{tab:cons_cert_criteria}
	\end{center}
\end{table}

\begin{eg}
	The bounds $\mathcal{L}(\mathcal{A})$ and $\mathcal{U}(\mathcal{A})$ can be computed exactly --- and are non-trivial --- in the case of the simple example $\mathcal{A}_{1}$ given in \eqref{eq:simple_assumptions}.  Indeed, writing $x_+:=\max(x,0)$, the optimal upper bound is given by
	\begin{equation}
		\label{eq:seesaw_exact_ouq}
		\mathcal{U}(\mathcal{A}_{1}) = p_{\max} := \left( 1-\frac{(a - b)_{+}}{D}\right)_{+},
	\end{equation}
where the maximum is achieved by taking the measure of probability of the random variable $f(X)$ to be the weighted sum of two weighted Dirac delta masses\footnote{$\delta_z$ is the Dirac delta mass on $z$, i.e.\ the measure of probability on Borel subsets $A \subset \mathbb{R}$ such that $\delta_z(A)=1$ if $z\in A$ and $\delta_z(A)=0$ otherwise.  The first Dirac delta mass is located at the minimum of the interval $[a,\infty]$ (since we are interested in maximizing the probability of the event $\mu[f(X)\geq a]$). The second Dirac delta mass is located at $z = a - D$ because we seek to maximize $p_{\max}$ under the constraints $p_{\max} a + (1 - p_{\max}) z \leq b$ and $a - z \leq D$.}
\[
	p_{\max} \delta_{a} + (1 - p_{\max}) \delta_{a - D}.
\]
This simple example demonstrates an extremely important point:  even if the function $G$ is extremely expensive to evaluate, certification can be accomplished without recourse to the expensive evaluations of $G$.  Furthermore, the simple structure of the maximizers motivates the reduction theorems later in Section \ref{sec:Reduction}.
\end{eg}

\begin{eg}
	\label{Ex:subdiameters}
	As shown in Equation \eqref{eq:McDintro}, concentration-of-measure inequalities lead to sub-optimal methods in the sense that they can be used to obtain upper bounds on $\mathcal{U}(\mathcal{A})$ and lower bounds on $\mathcal{L}(\mathcal{A})$. Observe that McDiarmid's inequality \eqref{eq:McDintro} required an information/assumptions set $\mathcal{A}_{\text{McD}}$ where the space $\mathcal{X}$ is a product space with $X = (X_{1}, X_{2}, \dots, X_{m})$, the mean performance satisfies $\mathbb{E}[G(X)] \leq b$, the $m$ inputs $X_{1}, \dots, X_{m}$ are independent, and the component-wise oscillations of $G$, (see \eqref{eq:defOsc}) are bounded $\operatorname{Osc}_i(G)\leq D_i$. It follows from McDiarmid's inequality \eqref{eq:McDintro} that, under the assumptions $\mathcal{A}_{\text{McD}}$,
	\[
		\mathcal{U}(\mathcal{A}_{\text{McD}}) \leq \exp \left( - 2 \frac{(a - b)_{+}^2}{\sum_{i=1}^m D_i^2} \right).
	\]
	This example shows how existing techniques such as concentration-of-measure inequalities can be incorporated into OUQ.  In Section \ref{sec:Reduction}, we will show how to reduce $\mathcal{U}(\mathcal{A}_{\text{McD}})$ to a finite dimensional optimization problem. Based on this reduction, in Section \ref{sec:Reduction-mcd}, we provide analytic solutions to the optimization problem $\mathcal{U}(\mathcal{A}_{\text{McD}})$ for $m=1,2,3$.  In practice, the computation of the bounds $D_i$ require the resolution of an optimization problem, we refer to \cite{LucasOwhadiOrtiz:2008, PSAAPI:2012, PSAAPII:2012} for practical methods.  We refer to \cite{LucasOwhadiOrtiz:2008, PSAAPI:2012, PSAAPII:2012, ScovelSteinwart:2010} for illustrations of UQ through concentration of measure inequalities.  In particular, since $\operatorname{Osc}_i(G)$ is a semi-norm, a (possibly numerical) model can be used to compute bounds on component-wise oscillations of $G$ via the triangular inequality $\operatorname{Osc}_i(G)\leq \operatorname{Osc}_i(F)+\operatorname{Osc}_i(G-F)$ (we refer to \cite{LucasOwhadiOrtiz:2008, PSAAPI:2012, PSAAPII:2012}  for details, the idea here is that an accurate model leads to a reduced number of experiments for the computation of $\operatorname{Osc}_i(G-F)$, while the computation of $\operatorname{Osc}_i(F)$ does not involve experiments). In the sequel we will refer to $D_{i,G}:=\operatorname{Osc}_i(G)$ (for $i=1,\dots,m$) as the \emph{sub-diameters} of $G$ and to
	\begin{equation}
		D_{G}:=\sqrt{\sum_{i=1}^m D_{i,G}^2}
	\end{equation}
	as the \emph{diameter} of $G$. Bounds on $\operatorname{Osc}_i(G)$
are useful because they constitute a form of non-linear sensitivity analysis and, combined with independence constraints, they lead to the concentration of measure phenomenon.
 The OUQ methodology can also handle  constraints of the type $\| G - F \| < C$ (which are not sufficient to observe take advantage of the concentration of measure effect) and $G(x_i)=z_i$ \cite{Sullivan:2012}.
\end{eg}

\begin{table}
	\begin{center}
		\begin{tabular}{|c|c|c|}
			\hline
			Admissible scenarios, $\mathcal{A}$ & $\mathcal{U}(\mathcal{A})$ & Method \\
			\hline \hline
			$\mathcal{A}_{\mathrm{McD}}$:  independence, oscillation and mean & $\leq 66.4 \%$ & McDiarmid's inequality \\
			constraints as given by \eqref{eq:PSAAP_SPHIR_Admissible_McD} & $= 43.7 \%$ & Theorem \ref{thm:m3} \\
			\hline
			$\mathcal{A}_{H}$	as given by \eqref{eq:PSAAP_SPHIR_Admissible} & $\numeq 37.9 \%$ & \emph{mystic}, $H$ known \\
			\hline
			$\mathcal{A}_{H} \cap \left\{ (H, \mu) \,\middle|\, \begin{matrix} \text{$\mu$-median velocity} \\ \text{$= 2.45 \, \mathrm{km \cdot s}^{-1}$} \end{matrix} \right\} $ & $\numeq 30.0 \%$ & \emph{mystic}, $H$ known \\
			\hline
			$\mathcal{A}_{H} \cap \left\{ (H, \mu) \,\middle|\, \text{$\mu$-median obliquity $= \frac{\pi}{12}$} \right\} $ & $\numeq 36.5 \%$ & \emph{mystic}, $H$ known \\
			\hline
			$\mathcal{A}_{H} \cap \left\{ (H, \mu) \,\middle|\, \text{obliquity $= \frac{\pi}{6}$ $\mu$-a.s.} \right\} $ & $\numeq 28.0 \%$ & \emph{mystic}, $H$ known \\
			\hline
		\end{tabular}
		\caption{Summary of the upper bounds on the probability of non-perforation for Example \eqref{eq:PSAAP_SPHIR_Admissible}  obtained by various methods and assumptions.  Note that OUQ calculations using \emph{mystic} (described in Section \ref{sec:ComputationalExamples}) involve evaluations of the function $H$, whereas McDiarmid's inequality and the optimal bound given the assumptions of McDiarmid's inequality use only the mean of $H$ and its McDiarmid subdiameters, not $H$ itself.  Note also that the incorporation of additional information/assumptions, e.g. on impact obliquity, always reduces the OUQ upper bound on the probability of failure, since this corresponds to a restriction to a subset of the original feasible set $\mathcal{A}_{H}$ for the optimization problem.}
		\label{tab:PSAAP_support_collapse}
	\end{center}
\end{table}

\begin{eg}
	For the  set $\mathcal{A}_{H}$  given in Equation \eqref{eq:PSAAP_SPHIR_Admissible}, the inclusion of additional information further reduces the upper bound  $\mathcal{U}(\mathcal{A}_{H})$. Indeed, the addition of assumptions lead to a smaller admissible set $\mathcal{A}_{H} \mapsto \mathcal{A}' \subset \mathcal{A}_{H}$, therefore $\mathcal{U}$ decreases and $\mathcal{L}$ increases.  For example, if the median of the third input parameter (velocity) is known to lie at the midpoint of its range, and this information is used to provide an additional constraint, then the least upper bound on the probability of non-perforation drops to $30.0 \%$.  See Table \ref{tab:PSAAP_support_collapse}  for a summary of the bounds presented in the hypervelocity impact example introduced in Subsection \ref{Subsec:motivatingex}, and for further examples of the effect of additional information/constraints. The bounds given in Table \ref{tab:PSAAP_support_collapse} have been obtained by using the reduction theorems of Section \ref{sec:Reduction} and the computational framework described in Section \ref{sec:ComputationalExamples}.
\end{eg}

\begin{rmk}
The number of iterations and evaluations of $H$ associated with Table \ref{tab:PSAAP_support_collapse} are:
$600$ iterations and $15300$ $H$-evaluations (second row),
 $822$ iterations and  $22700$ $H$-evaluations (third row),  $515$ iterations and $14550$ $H$-evaluations (fourth row), $760$ iterations and $18000$ $H$-evaluations (fifth row). Half of these numbers of iterations are usually sufficient to obtain the extrema with $4$ digits of accuracy (for the third row, for instance, $365$ iterations and $9350$ $H$-evaluations are sufficient
 to obtain the first $4$ decimal points of the optimum).
\end{rmk}

\paragraph{On the selectiveness of the information set $\mathcal{A}$.}
Observe that, except for the bound obtained from McDiarmid's inequality, the bounds obtained in Table \ref{tab:PSAAP_support_collapse} are the best possible given the information contained in $\mathcal{A}$.  If the unknown distribution $\P$ is completely specified, say by restricting to the feasible set $\mathcal{A}_{\mathrm{unif}}$ for which the only admissible measure is the uniform probability measure on the cube $\mathcal{X}$ (in which case the mean perforation area is $\mathbb{E}[H] = 6.58 \, \mathrm{mm}^{2}$), then the probability of non-perforation is $\mathcal{U}(\mathcal{A}_{\mathrm{unif}}) = \mathcal{L}(\mathcal{A}_{\mathrm{unif}}) \numeq 3.8\%$.  One may argue that there is a large gap between the fifth ($28\%$) row of Table \ref{tab:PSAAP_support_collapse} and $3.8\%$  but observe that $3.8\%$  relies on the exact knowledge of  $G$ (called $H$ here) and $\P$ whereas $28\%$ relies on the limited knowledge contained in $\mathcal{A}_{H} \cap \left\{ (H, \mu) \,\middle|\, \text{obliquity $= \frac{\pi}{6}$ $\mu$-a.s.} \right\} $ with respect to which $28\%$ is optimal. In particular, the gap between those two values is not caused by a lack of tightness of the method, but by a lack of selectiveness of the information contained in $\mathcal{A}_{H} \cap \left\{ (H, \mu) \,\middle|\, \text{obliquity $= \frac{\pi}{6}$ $\mu$-a.s.} \right\} $.  The (mis)use of the terms ``tight'' and ``sharp'' without reference to available information (and in presence of asymmetric information) can be the source of much confusion, something that we hope is cleared up by the present work. Given prior knowledge of $G$ and $\P$, it would be an easy task to construct a set $\mathcal{A}_{\P,G}$ containing $(G,\P)$ such that $\mathcal{U}(\mathcal{A}_{\P,G})\approx 4\%$ (if the probability of failure under $(G,\P)$ is $3.8\%$), but doing so would be delaying a honest discussion on one of the issues at the core of UQ: \emph{How to construct $\mathcal{A}$ without prior knowledge of $G$ and $\P$?} In other words, how to improve the ``selectiveness'' of $\mathcal{A}$ or how to design experiments leading to ``narrow'' $\mathcal{A}$s?  We will now show how this question can be addressed within the OUQ framework.

\subsection{The Optimal UQ loop}
\label{Subsec:ouqloop}

In the previous subsection we discussed how the basic inequality
\[
	\mathcal{L}(\mathcal{A}) \leq \P[G\geq a] \leq \mathcal{U}(\mathcal{A})
\]
provides rigorous optimal certification criteria.  The certification process should not be confused with its three possible outcomes  (see Table \ref{tab:cons_cert_criteria}) which we call ``certify'' (we assert that the system is safe), ``de-certify'' (we assert that the system is unsafe) and ``cannot decide'' (the safety or un-safety of the system is undecidable given the information/assumption set $\mathcal{A}$).  Indeed, in the case
\[
	\mathcal{L}(\mathcal{A}) \leq \epsilon < \mathcal{U}(\mathcal{A})
\]
there exist admissible scenarios under which the system is safe, and other admissible scenarios under which it is unsafe.  Consequently, it follows that we can make no definite certification statement for $(G, \P)$ without introducing further information/assumptions.  If no further information can be obtained,  we conclude that we ``cannot decide'' (this state could also be called ``do not decide'', because we could (arbitrarily) decide that the system is unsafe due to lack of information, for instance,  but do not).

However, if sufficient resources exist to gather additional information, then we enter what may be called the \emph{optimal uncertainty quantification loop}, illustrated in Figure \ref{fig:OUQ-Loop}.
\begin{figure}[t]
	\begin{center}
		\scalebox{1}{
			
\begin{pspicture}(0.0,-6.67)(11.0,5.67)
	\psframe[linewidth=0.04,framearc=0.2,dimen=outer](9.0,5.5)(2.0,4.5)
	\rput(5.5,5.0){$\text{Selection of New Experiments}$}
	\psline[linewidth=0.04,arrowsize=0.2 2.0,arrowlength=1.4,arrowinset=0.4]{->}(4.0,4.5)(4.0,3.5)
	
	\psframe[linewidth=0.04,framearc=0.2,dimen=outer](0.75,2.5)(5.25,3.5)
	\rput(3.0,3.0){$\begin{matrix} \text{Experimental Data} \\ \text{(Legacy / On-Demand)} \end{matrix} $}
	\psline[linewidth=0.04,arrowsize=0.2 2.0,arrowlength=1.4,arrowinset=0.4]{->}(5.25,3.0)(6.0,3.0)
	\psline[linewidth=0.04,arrowsize=0.2 2.0,arrowlength=1.4,arrowinset=0.4]{->}(4.0,2.5)(4.0,1.5)
	
	\psframe[linewidth=0.04,framearc=0.2,dimen=outer](6.0,3.5)(10.0,2.5)
	\rput(8.0,3.0){$\text{Expert Judgement}$}
	\psline[linewidth=0.04,arrowsize=0.2 2.0,arrowlength=1.4,arrowinset=0.4]{->}(7.0,4.5)(7.0,3.5)
	\psline[linewidth=0.04,arrowsize=0.2 2.0,arrowlength=1.4,arrowinset=0.4]{<-}(8.0,4.5)(8.0,3.5)
	\psline[linewidth=0.04,arrowsize=0.2 2.0,arrowlength=1.4,arrowinset=0.4]{->}(7.0,2.5)(7.0,1.5)
	
	\psframe[linewidth=0.04,framearc=0.2,dimen=outer](0.67,0.5)(2.67,1.5)
	\rput(1.67,1.0){$\begin{matrix} \text{Physical} \\ \text{Laws} \end{matrix} $}
	\psline[linewidth=0.04,arrowsize=0.2 2.0,arrowlength=1.4,arrowinset=0.4]{->}(2.67,1.0)(3.67,1.0)

	\psframe[linewidth=0.04,framearc=0.2,dimen=outer](3.67,0.5)(9.67,1.5)
	\rput(6.67,1.0){$\text{Assumptions / Admissible Set, $\mathcal{A}$} $}
	\psline[linewidth=0.04,arrowsize=0.2 2.0,arrowlength=1.4,arrowinset=0.4]{->}(5.5,0.5)(5.5,-0.17)
	
	\psframe[linewidth=0.04,framearc=0.2,dimen=outer](2.0,-1.83)(9.0,-0.17)
	\rput(5.5,-1.0){$\begin{matrix} \text{Extreme Scale Optimizer:  Calculate  } \\ \text{$\mathcal{L}(\mathcal{A}):= \inf \{ \mu[f \text{ fails}] \mid (f, \mu) \in \mathcal{A} \}$} \\ \text{$\mathcal{U}(\mathcal{A}):= \sup \{ \mu[f \text{ fails}] \mid (f, \mu) \in \mathcal{A} \}$} \end{matrix} $}
	\psline[linewidth=0.04,arrowsize=0.2 2.0,arrowlength=1.4,arrowinset=0.4]{->}(3.0,-1.83)(3.0,-2.5)
	\psline[linewidth=0.04,arrowsize=0.2 2.0,arrowlength=1.4,arrowinset=0.4]{->}(8.0,-1.83)(8.0,-2.5)
	
	\psframe[linewidth=0.04,framearc=0.2,dimen=outer](0.75,-3.5)(4.75,-2.5)
	\rput(2.75,-3.0){$\begin{matrix} \text{Certification} \\ \text{Process} \end{matrix} $}
	\psline[linewidth=0.04cm]{-}(3.0,-3.5)(3.0,-4.0)
	%\psline[linewidth=0.04cm]{-}(6.5,-4.0)(26.5,-4.0)
	\psline[linewidth=0.04,arrowsize=0.2 2.0,arrowlength=1.4,arrowinset=0.4]{->}(3.0,-4.0)(2.17,-4.0)(2.17,-4.67)
	\psline[linewidth=0.04,arrowsize=0.2 2.0,arrowlength=1.4,arrowinset=0.4]{->}(5.5,-4.0)(5.5,-4.67)
	\psline[linewidth=0.04,arrowsize=0.2 2.0,arrowlength=1.4,arrowinset=0.4]{->}(3.0,-4.0)(8.83,-4.0)(8.83,-4.67)
	
	\psframe[linewidth=0.04,framearc=0.2,dimen=outer](5.5,-3.5)(10.5,-2.5)
	\rput(8.0,-3.0){$\begin{matrix} \text{Sensitivity / Robustness} \\ \text{Analysis w.r.t.~} \mathcal{A} \end{matrix} $}
	\psline[linewidth=0.04,arrowsize=0.2 2.0,arrowlength=1.4,arrowinset=0.4]{->}(10.5,-3.17)(11.0,-3.17)(11.0,5.0)(9.0,5.0)
	\psline[linewidth=0.04,arrowsize=0.2 2.0,arrowlength=1.4,arrowinset=0.4]{->}(10.5,-2.83)(10.67,-2.83)(10.67,1.0)(9.67,1.0)
	\psline[linewidth=0.04,arrowsize=0.2 2.0,arrowlength=1.4,arrowinset=0.4]{->}(11.0,3.0)(10.0,3.0)
	
	%% The Traffic Lights
	\definecolor{darkgreen}{rgb}{0.0,0.5,0.0}
	\definecolor{amber}{rgb}{1.0,0.75,0.0}
	\definecolor{darkgray}{rgb}{0.33,0.33,0.33}
	
	\psframe[linewidth=0.04,framearc=0.2,dimen=outer](0.83,-4.67)(3.5,-6.17)
	\rput(2.17,-5.42){$\begin{matrix} \text{De-Certify} \\ \text{(\emph{i.e.}\ System is} \\ \text{Unsafe)} \end{matrix} $}
	
	\psframe[linewidth=0.04,framearc=0.2,dimen=outer](4.0,-4.67)(7.0,-6.17)
	\rput(5.5,-5.42){$\begin{matrix} \text{Cannot} \\ \text{Decide} \end{matrix} $}
	\psline[linewidth=0.04,arrowsize=0.2 2.0,arrowlength=1.4,arrowinset=0.4]{->}(5.5,-6.17)(5.5,-6.67)(0.0,-6.67)(0.0,5.0)(2.0,5.0)

	\psframe[linewidth=0.04,framearc=0.2,dimen=outer](7.5,-4.67)(10.17,-6.17)
	\rput(8.83,-5.42){$\begin{matrix} \text{Certify} \\ \text{(\emph{i.e.}\ System is} \\ \text{Safe)} \end{matrix} $}

\end{pspicture}

		}
		\caption{Optimal Uncertainty Quantification Loop.}
		\label{fig:OUQ-Loop}
	\end{center}
\end{figure}
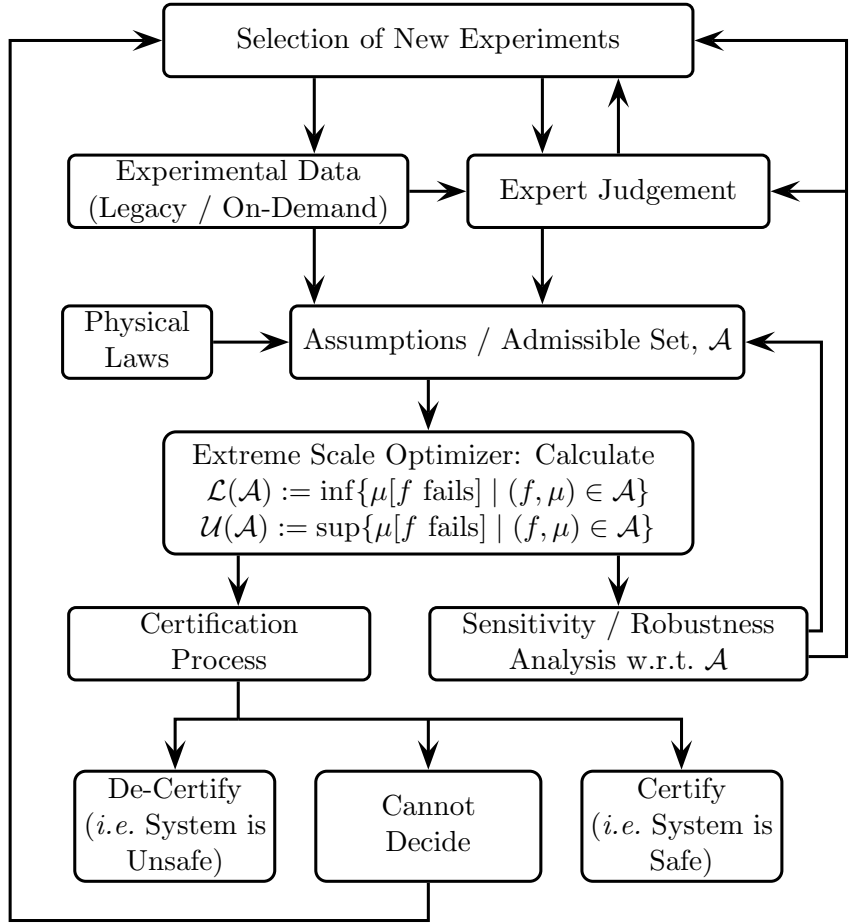
The admissible set $\mathcal{A}$ draws on three principal sources of information:  known physical laws, expert opinion, and experimental data.  Once the set $\mathcal{A}$ has been constructed, the calculation of the lower and upper bounds $\mathcal{L}(\mathcal{A})$ and $\mathcal{U}(\mathcal{A})$ are well-posed optimization problems.  If the results of these optimization problems lead to certification or de-certification, then we are done;  if not, then new experiments should be designed and expert opinion sought in order to validate or invalidate the extremal scenarios that cause the inequality
\[
	\mathcal{L}(\mathcal{A}) \leq \epsilon < \mathcal{U}(\mathcal{A})
\]
to hold.  The addition of information means further constraints on the collection of admissible scenarios; that is,  the original admissible set $\mathcal{A}$ is reduced to a smaller one $\mathcal{A}' \subset \mathcal{A}$, thereby providing sharper bounds on the probability of failure:
\[
	\mathcal{L}(\mathcal{A}) \leq \mathcal{L}(\mathcal{A}') \leq \P[ G(X) \geq a] \leq \mathcal{U}(\mathcal{A}') \leq \mathcal{U}(\mathcal{A}).
\]
The sharper bounds may meet the ``certify'' or ``decertify'' criteria of Table \ref{tab:cons_cert_criteria}.
 If not, and there are still resources for gathering additional information, then the loop should be repeated.  This process is the feedback arrow on the left-hand side of Figure \ref{fig:OUQ-Loop}.

The right-hand side of Figure \ref{fig:OUQ-Loop} constitutes  another aspect of the OUQ loop.  The bounds $\mathcal{L}(\mathcal{A})$ and $\mathcal{U}(\mathcal{A})$ are only useful insofar as the assumptions $\mathcal{A}$ are accurate.  It is possible that the sources of information that informed $\mathcal{A}$ may have been in error:  physical laws may have been extended outside their range of validity (e.g. Newtonian physics may have been applied in the relativistic regime),   there may have been difficulties with the experiments or the results misinterpreted, or  expert opinion may have been erroneous.  Therefore, a vital part of OUQ is to examine the sensitivity and robustness of the bounds $\mathcal{L}(\mathcal{A})$ and $\mathcal{U}(\mathcal{A})$ with respect to the assumption set $\mathcal{A}$.  If the bounds $\mathcal{L}(\mathcal{A})$ and $\mathcal{U}(\mathcal{A})$ are found to depend  sensitively on one particular assumption  (say, the mean performance of one component of the system), then it would be advisable to expend resources investigating this assumption.

The loop illustrated in Figure \ref{fig:OUQ-Loop} differs from the loop used to  solve the numerical optimization problem as described  in Sub-section \ref{sec:Mystic} and Remark \ref{rmk:nestedopt}.

\paragraph{Experimental design and selection of the most decisive experiment.}

An important aspect of the OUQ loop is the selection of new experiments.  Suppose that a number of possible experiments $E_{i}$ are proposed, each of which will determine some functional $\Phi_{i}(G, \P)$ of $G$ and $\P$.  For example, $\Phi_{1}(G, \P)$ could be $\E_{\P}[G]$, $\Phi_{2}(G, \P)$ could be $\P[X \in A]$ for some subset $A \subseteq \mathcal{X}$ of the input parameter space, and so on.  Suppose that there are insufficient experimental resources to run all of these proposed experiments. Let us now consider which experiment should be run for the certification problem.  Recall that the admissible set $\mathcal{A}$ is partitioned into safe and unsafe subsets as in \eqref{eq:A_safe_and_unsafe}.  Define $J_{\mathrm{safe}, \epsilon}(\Phi_{i})$ to be the closed interval spanned by the possible values for the functional $\Phi_{i}$ over the safe admissible scenarios (i.e.\ the closed convex hull of the range of $\Phi_{i}$ on $\mathcal{A}_{\mathrm{safe}, \epsilon}$):  that is, let
\begin{subequations}
	\label{eq:J_intervals}
	\begin{align}
	J_{\mathrm{safe}, \epsilon}(\Phi_{i}) & := \left[ \inf_{(f, \mu) \in \mathcal{A}_{\mathrm{safe}, \epsilon}} \Phi_{i}(f, \mu) ,  \sup_{(f, \mu) \in \mathcal{A}_{\mathrm{safe}, \epsilon}} \Phi_{i}(f, \mu) \right] \\
	J_{\mathrm{unsafe}, \epsilon}(\Phi_{i}) & := \left[ \inf_{(f, \mu) \in \mathcal{A}_{\mathrm{unsafe}, \epsilon}} \Phi_{i}(f, \mu) ,  \sup_{(f, \mu) \in \mathcal{A}_{\mathrm{unsafe}, \epsilon}} \Phi_{i}(f, \mu) \right].
	\end{align}
\end{subequations}
Note that, in general, these two intervals may be disjoint or may have non-empty intersection;  the size of their intersection provides a measure of usefulness of the proposed experiment $E_{i}$.  Observe that if experiment $E_{i}$ were run, yielding the value $\Phi_{i}(G, \P)$, then the following conclusions could be drawn:
\begin{align*}
        \Phi_{i}(G, \P) \in J_{\mathrm{safe}, \epsilon}(\Phi_{i}) \cap J_{\mathrm{unsafe}, \epsilon}(\Phi_{i}) & \implies \text{no  conclusion,} \\
        \Phi_{i}(G, \P) \in J_{\mathrm{safe}, \epsilon}(\Phi_{i}) \setminus J_{\mathrm{unsafe}, \epsilon}(\Phi_{i}) & \implies \text{the system is safe,} \\
        \Phi_{i}(G, \P) \in J_{\mathrm{unsafe}, \epsilon}(\Phi_{i}) \setminus J_{\mathrm{safe}, \epsilon}(\Phi_{i}) & \implies \text{the system is unsafe,} \\
        \Phi_{i}(G, \P) \notin J_{\mathrm{safe}, \epsilon}(\Phi_{i}) \cup J_{\mathrm{unsafe}, \epsilon}(\Phi_{i}) &  \implies \text{faulty assumptions,}
\end{align*}
where the last assertion (\emph{faulty assumptions}) means that $(G, \P) \notin \mathcal{A}$ and follows from the fact that $\Phi_{i}(G, \P) \notin J_{\mathrm{safe}, \epsilon}(\Phi_{i}) \cup J_{\mathrm{unsafe}, \epsilon}(\Phi_{i})$ is a contradiction. The validity of the first three assertions is based on the supposition that $(G, \P) \in \mathcal{A}$.

In this way, the computational optimization exercise of finding $J_{\mathrm{safe}, \epsilon}(\Phi_{i})$ and $J_{\mathrm{unsafe}, \epsilon}(\Phi_{i})$ for each proposed experiment $E_{i}$ provides an objective assessment of which experiments are worth performing:  those for which $J_{\mathrm{safe}, \epsilon}(\Phi_{i})$ and $J_{\mathrm{unsafe}, \epsilon}(\Phi_{i})$ are nearly disjoint intervals are worth performing since they are likely to yield conclusive results \emph{vis-{\`a}-vis} (de-)certification and  conversely, if the intervals $J_{\mathrm{safe}, \epsilon}(\Phi_{i})$ and $J_{\mathrm{unsafe}, \epsilon}(\Phi_{i})$ have a large overlap, then experiment $E_{i}$ is not worth performing since it is unlikely to yield conclusive results.  Furthermore, the fourth possibility above shows how experiments can rigorously establish that one's assumptions $\mathcal{A}$ are incorrect.  See Figure \ref{fig:J_intervals} for an illustration.

\begin{figure}[t]
	\begin{center}
		\scalebox{1}{
			\begin{pspicture}(0,0)(10,8)
	\psline[linewidth=0.04]{->}(0.5,7)(9,7)
	\rput(9.5,7){$\R$}
	\psline[linewidth=0.04,linecolor=red]{*-*}(3,7.2)(4,7.2)
	\rput(3.5,7.5){$J_{\mathrm{unsafe}, \epsilon}(\Phi_{1})$}
	\psline[linewidth=0.04,linecolor=blue]{*-*}(6,6.8)(8,6.8)
	\rput(7,6.5){$J_{\mathrm{safe}, \epsilon}(\Phi_{1})$}
	
	\psline[linewidth=0.04]{->}(0.5,5)(9,5)
	\rput(9.5,5){$\R$}
	\psline[linewidth=0.04,linecolor=red]{*-*}(2.5,5.2)(6,5.2)
	\rput(4.25,5.5){$J_{\mathrm{unsafe}, \epsilon}(\Phi_{2})$}
	\psline[linewidth=0.04,linecolor=blue]{*-*}(5.5,4.8)(8.5,4.8)
	\rput(7,4.5){$J_{\mathrm{safe}, \epsilon}(\Phi_{2})$}
	
	\psline[linewidth=0.04]{->}(0.5,3)(9,3)
	\rput(9.5,3){$\R$}
	\psline[linewidth=0.04,linecolor=red]{*-*}(1,3.2)(7,3.2)
	\rput(4,3.5){$J_{\mathrm{unsafe}, \epsilon}(\Phi_{3})$}
	\psline[linewidth=0.04,linecolor=blue]{*-*}(3,2.8)(8,2.8)
	\rput(5.5,2.5){$J_{\mathrm{safe}, \epsilon}(\Phi_{3})$}
	
	\psline[linewidth=0.04]{->}(0.5,1)(9,1)
	\rput(9.5,1){$\R$}
	\psline[linewidth=0.04,linecolor=red]{*-*}(3,1.2)(7,1.2)
	\rput(5,1.5){$J_{\mathrm{unsafe}, \epsilon}(\Phi_{4})$}
	\psline[linewidth=0.04,linecolor=blue]{*-*}(3,0.8)(7,0.8)
	\rput(5,0.5){$J_{\mathrm{safe}, \epsilon}(\Phi_{4})$}
\end{pspicture}
		}
		\caption{A schematic representation of the intervals $J_{\mathrm{unsafe}, \epsilon}(\Phi_{i})$ (in red) and $J_{\mathrm{safe}, \epsilon}(\Phi_{i})$ (in blue) as defined by \eqref{eq:J_intervals} for four functionals $\Phi_{i}$ that might be the subject of an experiment.  $\Phi_{1}$ is a good candidate for experiment effort, since the intervals do not overlap and hence experimental determination of $\Phi_{1}(G, \P)$ will certify or de-certify the system;  $\Phi_{4}$ is not worth investigating, since it cannot distinguish safe scenarios from unsafe ones;  $\Phi_{2}$ and $\Phi_{3}$ are intermediate cases, and $\Phi_{2}$ is a better prospect than $\Phi_{3}$.}
		\label{fig:J_intervals}
	\end{center}
\end{figure}
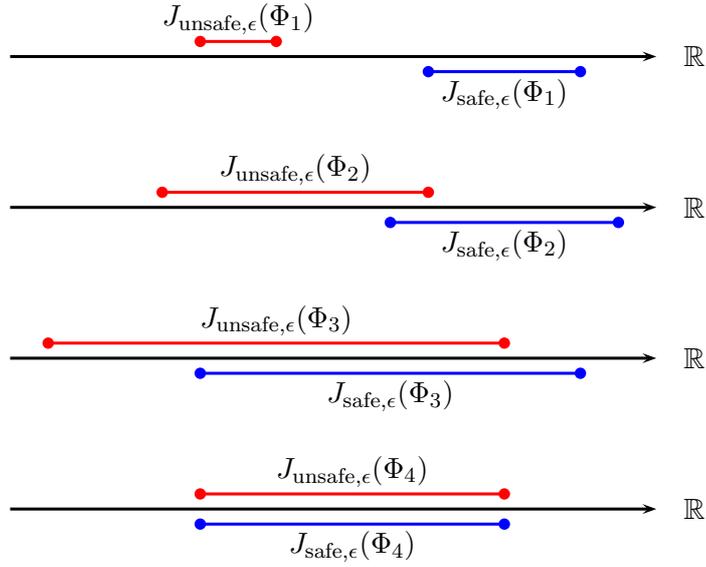

\begin{rmk}
	For the sake of clarity, we have started this description by defining \emph{experiments} as \emph{functionals} $\Phi_i$ of $\P$ and $G$.  In practice, some experiments may not be functionals of $\P$ and $G$  but of related objects. Consider, for instance, the situation where $(X_1, X_2)$ is a two-dimensional Gaussian vector with zero mean  and covariance matrix $C$, $\P$ is the probability distribution of $X_1$, the experiment $E_2$ determines the variance of $X_2$ and the information set $\mathcal{A}$ is $C \in B$, where $B$ is a subset of symmetric positive definite $2 \times 2$ matrices. The outcome of the experiment $E_2$ is not a function of the probability distribution $\P$;  however, the knowledge of $\P$ restricts the range of possible outcomes of $E_2$.  Hence, for some experiments $E_i$, the knowledge of $(G, \P)$ does not determine the outcome of the experiment, but only the set of possible outcomes.  For those experiments, the description given above can be generalized to situations where $\Phi_i$ is a \emph{multivalued} functional of $(G,\P)$ determining the set of possible outcomes of the experiment $E_i$.  This picture can be generalized further by introducing measurement noise, in which case $(G, \P)$ may not determine a deterministic set of possible outcomes, but instead a measure of probability on a set of possible outcomes.
\end{rmk}

\begin{eg}[Computational solution of the experimental selection problem]
We will now consider  again the admissible set $\mathcal{A}_{H}$	as given by \eqref{eq:PSAAP_SPHIR_Admissible}.
The following example shows that the notion of ``best'' experiment depends
on the admissible safety threshold $\epsilon$ for $\P[G\geq a]$.
	Suppose that an experiment $E$ is proposed that will determine the probability mass of the upper half of the velocity range, $[2.45, 2.8] \, \mathrm{km \cdot s}^{-1}$;  the corresponding functional $\Phi$ of study is
	\[
		\Phi(\mu) := \mu[ v \geq 2.45 \, \mathrm{km \cdot s}^{-1} ],
	\]
	and the proposed experiment $E$ will determine $\Phi(\P)$ (approximate determinations including measurement and sampling errors can also be handled, but the exact determination is done here for simplicity). The corresponding intervals $J_{\mathrm{safe}, \epsilon}(\Phi)$ and $J_{\mathrm{unsafe}, \epsilon}(\Phi)$ as defined by \eqref{eq:J_intervals} and \eqref{eq:A_safe_and_unsafe} are reported in Table \ref{tab:PSAAP_exp_selection} for various acceptable probabilities of failure $\epsilon$.  Note that, for larger values of $\epsilon$, $E$ is a ``better'' experiment in the sense that it can distinguish safe scenarios from unsafe ones (see also Figure \ref{fig:J_intervals});  for smaller values of $\epsilon$, $E$ is a poor experiment.  In any case, since the intersection $J_{\mathrm{safe}, \epsilon}(\Phi) \cap J_{\mathrm{unsafe}, \epsilon}(\Phi)$ is not empty, $E$ is not an ideal experiment.
	
	It is important to note that the optimization calculations necessary to compute the intervals $J_{\mathrm{safe}, \epsilon}(\Phi)$ and $J_{\mathrm{unsafe}, \epsilon}(\Phi)$ are simplified by the application of Theorem \ref{thm:baby_measure}:  in this case, the objective function of $\mu$ is $\mu[v \geq 2.45]$ instead of $\mu[H = 0]$, but the constraints are once again linear inequalities on generalized moments of the optimization variable $\mu$.
	
	\begin{table}
		\begin{center}
			\begin{tabular}{|c|c|c|c|c|}
				\hline
				 & \multicolumn{2}{|c|}{$J_{\mathrm{safe}, \epsilon}(\Phi)$} & \multicolumn{2}{|c|}{$J_{\mathrm{unsafe}, \epsilon}(\Phi)$} \\
				 & $\inf$ & $\sup$ & $\inf$ & $\sup$ \\
				\hline \hline
				$\epsilon = 0.100$ & $0.000$ & $1.000$ & $0.000$ & $0.900$ \\
				iterations until numerical convergence & $40$ & $40$ & $40$ & $300$ \\
				total evaluations of $H$ & $1,000$ & $1,000$ & $1,000$ & $8,000$ \\
				\hline
				$\epsilon = 0.200$ & $0.000$ & $1.000$ & $0.000$ & $0.800$ \\
				iterations until numerical convergence & $40$ & $40$ & $40$ & $400$ \\
				total evaluations of $H$ & $1,000$ & $1,000$ & $1,000$ & $12,000$ \\
				\hline
				$\epsilon = 0.300$ & $0.000$ & $1.000$ & $0.000$ & $0.599$ \\
				iterations until numerical convergence & $40$ & $40$ & $40$ & $1000$ \\
				total evaluations of $H$ & $1,000$ & $1,000$ & $1,000$ & $33,000$ \\
				\hline
			\end{tabular}
			\caption{The results of the calculation of the intervals $J_{\mathrm{safe}, \epsilon}(\Phi)$ and $J_{\mathrm{unsafe}, \epsilon}(\Phi)$ for the functional $\Phi(\mu) := \mu[ v \geq 2.45 \, \mathrm{km \cdot s}^{-1} ]$.  Note that, as the acceptable probability of system failure, $\epsilon$, increases, the two intervals overlap less,  so experimental determination of $\Phi(\P)$ would be more likely to yield a decisive conservative certification of the system as safe or unsafe;  the computational cost of this increased decisiveness is a greater number of function evaluations in the optimization calculations.  All computational cost figures are approximate.}
			\label{tab:PSAAP_exp_selection}
		\end{center}
	\end{table}
\end{eg}

\paragraph{On the number of total evaluations on $H$.}
Recall that, for simplicity, we have assumed the response function $G$ to be known and given by $H$.
A large number of evaluations of $H$ has been used in Table \ref{tab:PSAAP_exp_selection} to ensure convergence towards the global optimum.
It is important to observe that those evaluations of $H$ are not (actual, costly) experiments but (cheap) numerical evaluations of equation \eqref{eq:PSAAP_SPHIR_surr}.
 More precisely, the method for selecting new best experiments does not require new experiments; i.e.,
it relies entirely on the information set $\mathcal{A}$ (which contains the information gathered from previous experiments). Hence those evaluations should not be \emph{viewed as} ``information gained from Monte Carlo samples'' but as ``pure CPU processing time''.
In situations where the numerical evaluation of $H$ is expensive, one can introduce its cost in the optimization loop. An investigation of the best algorithm to perform the numerical optimization with the least number of function evaluations is a worthwhile subject but is beyond the scope of the present paper. Observe also that the method proposed in Section  \ref{sec:OUQ} does not rely on the exact knowledge of the response function $G$. More precisely, in  situations where the response function is unknown, the selection of next best experiments is still entirely computational, and based upon previous data/information gathered on $G$ enforced as constraints in a numerical optimization algorithm. More precisely, in those situations, the optimization algorithm may require the numerical evaluation of a large number of admissible functions $f$ (compatible with the prior information available on $G$) but it does not require any new evaluation of $G$.

In situations where $H$ is (numerically) expensive to evaluate, one would have to include the cost of these evaluations in the optimization loop and use fast algorithms exploiting the multi-linear structures associated with the computation of safe and unsafe intervals. Here we have used a genetic algorithm because its robustness. This algorithm typically converges at $10\%$ of the total number of evaluations of $H$ given in the last row of  Table \ref{tab:PSAAP_exp_selection} but we have increased the number of iterations tenfold to guarantee a robust result.
The investigation of efficient optimization algorithms exploiting the multi-linear structures of OUQ optimization problems is of great interest and beyond the immediate scope of this paper.

\begin{figure}[t]
	\begin{center}
			\includegraphics[width=1\textwidth]{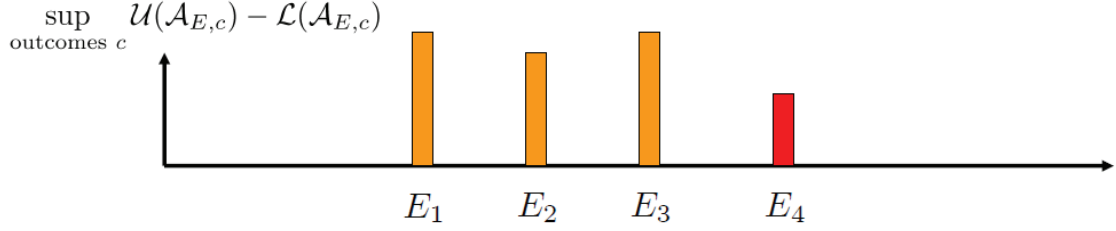}
		\caption{
A schematic representation of the size of the prediction intervals $\sup_{\text{outcomes } c} \big(\mathcal{U}(\mathcal{A}_{E,c})-\mathcal{L}(\mathcal{A}_{E,c})\big)$ in the worst case with respect to outcome $c$. $E_4$ is the most predictive experiment.
}
		\label{fig:MostPredictive}
	\end{center}
\end{figure}

\begin{figure}
	\begin{center}
       	\subfigure[Playing Chess against the universe]{\label{fig:Chess}
			\includegraphics[width=0.45\textwidth]{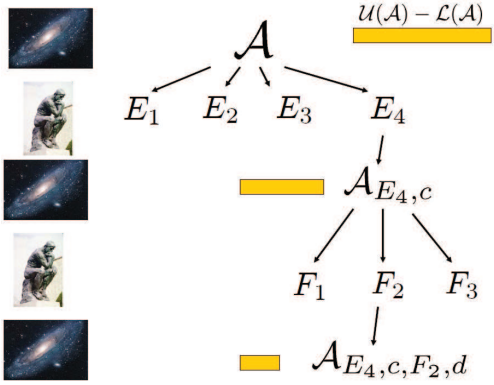}
		}
		\subfigure[Let's play Clue, Round 1]{\label{subround1}
			\includegraphics[width=0.45\textwidth]{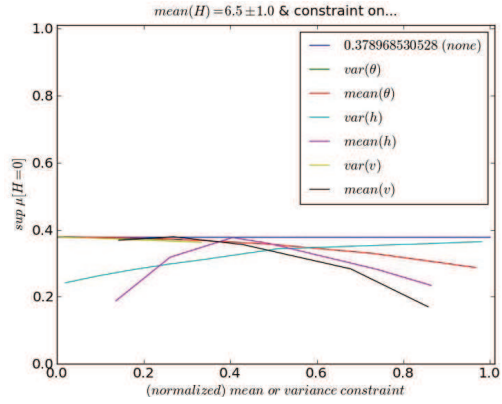}
		}
		\subfigure[Let's play Clue, Round 2]{\label{subround2}
			\includegraphics[width=0.45\textwidth]{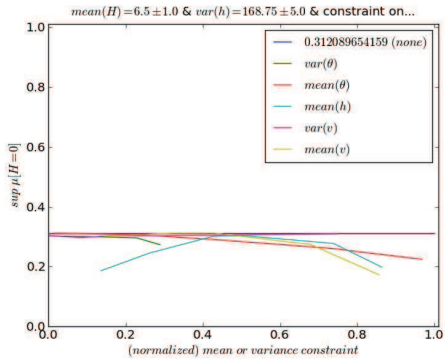}
		}
		\subfigure[Let's play Clue, Round 3]{\label{subround3}
			\includegraphics[width=0.45\textwidth]{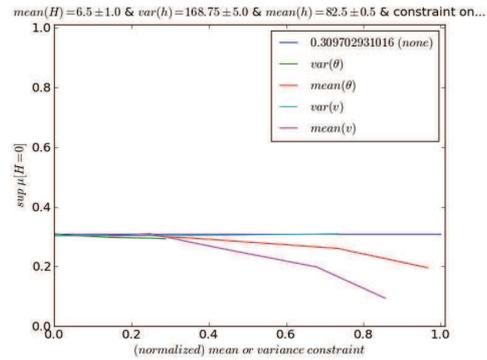}
		}
		\caption{Subfigure \subref{fig:Chess}: Playing Chess Against the Universe. We choose which experiment $E$ to perform and the universe selects the outcome $c$. Our objective is to minimize $\mathcal{U}(\mathcal{A})-\mathcal{L}(\mathcal{A}).$ In the first round our possible moves correspond to a choice between experiments $E_1, E_2, E_3$ and $E_4$. We perform experiment $E_4$, the outcome $c$ of that experiment (selected by the Universe) transforms the admissible into $\mathcal{A}_{E_4,c}$. In the second round, our possible moves correspond to a choice between experiments $F_1, F_2$ and $F_3$. As in the game of Chess, several moves can be planned in advance by solving min max optimization problems, and  the exponential increase of the number branches of the game tree can be kept under control by exploring only a subset of (best) moves. Subfigures \subref{subround1}, \subref{subround2} and \subref{subround3}: Let's play Clue.}
		\label{fig:PlayClueandChess}
	\end{center}
\end{figure}

\paragraph{Selection of the most predictive experiment.}
The computation of safe and unsafe intervals described in the previous paragraph allows of the selection of the most selective experiment. If our objective is to have an ``accurate'' prediction of $\P[ G(X) \geq a]$, in the sense that $\mathcal{U}(\mathcal{A})-\mathcal{L}(\mathcal{A})$ is small, then one can proceed as follows. Let $\mathcal{A}_{E,c}$ denote those scenarios in $\mathcal{A}$  that are
compatible with obtaining outcome $c$ from experiment $E$. An experiment $E^*$ that is most predictive, even in the worst case, is defined by a minmax criterion: we seek (see Figure \ref{fig:MostPredictive})
\begin{equation}
E^* \in \mathop{\operatorname{arg}\operatorname{min}}_{\text{experiments E}}
\Big(\sup_{\text{outcomes } c} \big(\mathcal{U}(\mathcal{A}_{E,c})-\mathcal{L}(\mathcal{A}_{E,c})\big)\Big)
\end{equation}
The idea is that, although we can not predict the precise outcome $c$ of an experiment $E$, we can compute a worst-case scenario with respect to $c$, and obtain an optimal bound for the minimum decrease in our prediction interval for $\P[ G(X) \geq a]$ based on the (yet unknown) information gained from experiment $E$.
Again, the theorems given in this paper can be applied to reduce this kind of problem. Finding $E^*$ is a bigger problem than just calculating $\mathcal{L}(\mathcal{A})$ and $\mathcal{U}(\mathcal{A})$, but the presumption is that computer time is cheaper than experimental effort.

\paragraph{Planning of campaigns of experiments.}
The idea of experimental selection can be extended to plan several
experiments in advance, i.e.\ to plan campaigns of experiments.
This aspect can be used to
assess the safety or the design of complex systems in a minimal number of
experiments (and also to predict bounds on the total number of required
experiments).  Just as a
good chess player thinks several moves ahead, our framework allows for the design of increasingly sophisticated and optimal sequences of experiments that can be performed to measure key system variables. The implementation of this strategy corresponds to a min max game ``played against the Universe'' (Subfigure \ref{fig:Chess}). The well-known games of Clue/Cluedo and Twenty Questions are better analogies than chess for this kind of information
game. In that sense, the planning of campaigns of experiments is an infinite-dimensional Clue, played on spaces of admissible scenarios, against our lack of perfect information about reality, and made tractable by the reduction theorems.
This aspect  calls for  more
investigation since it has the potential to provide a new approach to the
current scientific investigation paradigm, which is based on intuition, expert
judgment, and guessing.

\begin{eg}[Let's play Clue.]
In Subfigures \ref{subround1}, \ref{subround2} and \ref{subround3} we consider  again the admissible set $\mathcal{A}_{H}$	as given by \eqref{eq:PSAAP_SPHIR_Admissible} and select three most predictive experiments, sequentially, choosing the second one after having observed the outcome of the first one.
The list of possible experiments corresponds to measuring the mean or variance of thickness $h$, obliquity $\theta$ or velocity $v$.  Subfigures \ref{subround1}, \ref{subround2} and \ref{subround3} show $\mathcal{U}(\mathcal{A}_{H, E, c})$ for each of these experiments as a function of the re-normalized outcome value $c$. Since, in this example, we always have $\mathcal{L}(\mathcal{A}_{H, E, c})=0$, $\mathcal{U}(\mathcal{A}_{H, E, c})$ corresponds to the size of the prediction interval for the probability of non-perforation given the information that the outcome of experiment $E$ is $c$.
Given the results shown in Subfigure \ref{subround1} we select to measure the variance of thickness as our first best experiment.
 Note that this selection does not require the knowledge of what the experiment \emph{will} yield but only the knowledge of what the experiment \emph{can} yield: we identify as the optimal next experiment the one that is most informative in the minimax sense, i.e.\ of all its possible outcomes (this does not require the knowledge of what the experiment will yield), its least informative outcome is more informative than the least informative outcome of any other candidate experiment. Although not necessary, this selection can, possibly, be guided by a model of reality (i.e.\ in this case a model for the probability distributions of $h,\theta,v$). Used in this manner, an accurate model will reduce the number of experiments required for certification and an inaccurate model will lead to a relatively greater number of experiments (but not to erroneous bounds).
Subfigure \ref{subround2} is based on the information contained in $\mathcal{A}_{H}$ and bounds on the variance of thickness  (obtained from the first experiment). Our selection as  second experiment is to measure the mean of thickness (leading to Subfigure \ref{subround3}).
\end{eg}

\section{Generalizations and Comparisons}
\label{sec:comparisons}

\subsection{Prediction, extrapolation, verification and validation}

In the previous section, the OUQ framework was described as it applies to the the certification problem \eqref{eq:def_cert}. We will now show that many important UQ problems, such as prediction, verification and validation, can be formulated as certification problems.  This is similar to the point of view of \cite{BarlowProshchan:1996}, in which formulations of many problem objectives in reliability are shown to be representable in a unified framework.

A \textbf{prediction problem} can be formulated as, given $\epsilon$ and (possibly incomplete) information on $\P$ and $G$, finding a smallest $b-a$ such that
\begin{equation}
	\label{eq:def_prediction}
	\P[a\leq G(X)\leq b] \geq 1 - \epsilon,
\end{equation}
which, given the admissible set $\mathcal{A}$, is equivalent to solving
\begin{equation}
\inf \left\{ b-a \,\middle|\, \inf_{(f, \mu) \in \mathcal{A}} \mu [a\leq f(X) \leq b]\geq 1 - \epsilon \right\}.
\end{equation}
Observe that $[a,b]$ can be interpreted as an optimal interval of confidence for $G(X)$ (although $b-a$ is minimal, $[a,b]$ may not be unique), in particular,
with probability at least $1-\epsilon$, $G(X)\in [a,b]$.

In many applications the regime where experimental
data can be taken is different than the  deployment regime where prediction or certification is sought, and this is commonly referred to as the \textbf{extrapolation problem}.
 For example, in materials modeling, experimental tests are performed on materials, and the model run for
comparison, but the desire is that these results tell us something where experimental tests are
impossible, or extremely expensive to obtain.

In most applications, the response function $G$ may be approximated via a (possibly numerical) model $F$. Information on the relation between the model $F$ and the response function $G$ that it is designed to represent (i.e.\ information on $(x, F(x), G(x))$) can be used to restrict (constrain) the set $\mathcal{A}$ of admissible scenarios $(G,\P)$. This information may take the form of  a bound on some distance between $F$ and $G$ or a bound on some complex functional of $F$ and $G$ \cite{LucasOwhadiOrtiz:2008,ScovelSteinwart:2010}. Observe that, in the context of the certification problem \eqref{eq:def_cert}, the value of the model can be measured by
changes induced on the optimal bounds   $\mathcal{L}(\mathcal{A})$ and $\mathcal{U}(\mathcal{A})$. The problem of quantifying the relation (possibly the distance) between $F$ and $G$ is commonly referred to as the \textbf{validation problem}. In some situations $F$ may be a numerical model involving millions of lines of code and (possibly) space-time discretization.  The quantification of the uncertainty associated with the possible presence of bugs and discretization approximations is commonly referred to as the \textbf{verification problem}. Both, the validation and the verification problem, can be addressed in the OUQ framework by introducing information sets describing  relations between $G$, $F$ and the code.

\subsection{Comparisons with other UQ methods}

We will now compare OUQ with other widely used UQ methods and consider the certification problem \eqref{eq:def_cert} to be specific.

\begin{itemize}

\item Assume that $n$ independent samples $Y_1,\dots,Y_n$ of the random variable $G(X)$ are available (i.e.\ $n$ independent observations of the random variable $G(X)$, all distributed according to the measure of probability $\P$).  If $\eins[ Y_i \geq a ]$ denotes the random variable equal to one if $Y_i \geq a$ and equal to zero otherwise, then
\begin{equation}\label{eq:nsamples}
p_n:=\frac{\sum_{i=1}^n \eins[ Y_i \geq a ]}{n}
\end{equation}
is an unbiased estimator of $\P[G(X)\geq a]$. Furthermore, as a result of Hoeffding's concentration inequality \cite{Hoeffding:1956b}, the probability that $p_n$ deviates from $\P[G(X)\geq n]$  (its mean) by at least $\epsilon/2$ is bounded from above by  $\exp(-\frac{n}{2} \epsilon^2)$. It follows that if the number of samples $n$ is large
enough (of the order of $\frac{1}{\epsilon^{2}} \log \frac{1}{\epsilon}$), then the certification of \eqref{eq:def_cert} can be obtained through a Monte Carlo estimate (using $p_n$). As this example shows, \textbf{Monte Carlo strategies} \cite{Liu:2008} are simple to implement and do not necessitate  prior information on the response function $G$ and the measure $\P$ (other than the i.i.d.\ samples). However, they require a large number of (independent) samples of $G(X)$ which is a severe limitation for the certification of rare events (the $\epsilon=10^{-9}$ of the aviation industry would \cite{Soekkha:1997, Boeing:2010} necessitate $O(10^{18})$ samples).
Additional information  on $G$ and $\P$ can, in principle, be included (in a limited fashion) in Monte Carlo strategies via
importance and weighted sampling \cite{Liu:2008} to reduce the number of required samples.

\item The number of required samples can also be reduced to $\frac{1}{\epsilon}(\ln \frac{1}{\epsilon})^d$ using \textbf{Quasi-Monte Carlo Methods}. We refer in particular to  the Koksma--Hlawka inequality \cite{Niederreiter:1992}, to \cite{SloanJoe:1994} for multiple integration based on lattice rules and to \cite{Sloan2010} for a recent review. We observe that these methods require some regularity (differentiability) condition  on the response function $G$ and the possibility of sampling $G$ at pre-determined points $X$. Furthermore, the number of required samples blows-up at an exponential rate with the dimension $d$ of the input vector $X$.

\item If $G$ is regular enough and can be sampled at  at pre-determined points, and if  $X$ has a known distribution, then \textbf{stochastic expansion methods} \cite{GhanemDham:1998, Ghanem:1999, Xiu:2009, BabuskaNobRa:2010, EldredWebsterConstantine:2008, DoostanOwhadi:2010} can reduce the number of required samples even further (depending on the regularity of $G$) provided that the dimension of $X$ is not too high \cite{Todor07a, Bieri09a}.   However, in most  applications, only incomplete information on $\P$ and $G$ is available and the number of available samples on $G$ is small or zero.  $X$ may be of high dimension, and may include uncontrollable variables and unknown unknowns (unknown input parameters of the response function $G$).  $G$ may not be the solution of a PDE and may involve interactions between singular and complex processes such as (for instance) dislocation, fragmentation, phase transitions, physical phenomena in untested regimes, and even human decisions. We observe that in many applications of Stochastic Expansion methods
 $G$ and $\P$ are assumed to be perfectly known and UQ reduces to  computing the push forward of the measure $\P$ via the response (transfer) function $I_{\geq a}\circ G$ (to a measure on two points, in those situations $\mathcal{L}(\mathcal{A})=\P[G\geq a]=\mathcal{U}(\mathcal{A})$).

\item The investigation of variations of the response function $G$ under variations of the input parameters $X_i$, commonly referred to as
\textbf{sensitivity analysis} \cite{Saltelli:2000, Saltelli:2008}, allows for the identification of  critical input parameters. Although helpful in estimating the robustness of conclusions made based on specific assumptions on input parameters, sensitivity analysis, in its most general form, has not been targeted at obtaining rigorous upper bounds on probabilities of failures associated with certification problems \eqref{eq:def_cert}.
However, single parameter oscillations of the function $G$ (as defined by \eqref{eq:defOsc}) can be seen as a form of non-linear sensitivity analysis leading to bounds on $\P[G\geq a]$ via McDiarmid's concentration inequality \cite{McDiarmid:1989, McDiarmid:1998}. These bounds can be made sharp by partitioning the input parameter space along maximum oscillation directions and computing sub-diameters on sub-domains \cite{Sullivan:2010}.

\item If $\mathcal{A}$ is expressed probabilistically through a prior (an a priori measure of probability) on the set possible scenarios $(f,\mu)$  then \textbf{Bayesian inference} \cite{Leonard:1999, BeckKatafygiotis:1998} could in principle be used to estimate $\P[G\geq a]$ using the posterior measure of probability on $(f,\mu)$. This combination between OUQ and Bayesian methods avoids the necessity to solve the possibly large optimization problems \eqref{eq:def_ouq} and it  also greatly simplifies the incorporation of sampled data thanks to the Bayes rule.
However, oftentimes, priors are not available or their choice involves some degree of arbitrariness that is incompatible with the certification of rare events.  Priors may become asymptotically irrelevant (in the limit of large data sets) but, for small $\epsilon$, the number of required samples can be of the same order as the number required by Monte-Carlo methods \cite{ShenWasserman:2001}.\\
When unknown parameters are estimated using priors and sampled data, it is important to observe that the convergence of the Bayesian method may fail if the underlying probability mechanism allows an infinite number of possible outcomes (e.g., estimation of an unknown probability on $\mathbb{N}$, the set of all natural numbers) \cite{DiaconisFreedman:1986}.  In fact, in these infinite-dimensional situations, this lack of convergence (commonly referred to as inconsistency) is the rule rather than the exception \cite{DiaconisFreedman:1998}. As emphasized in \cite{DiaconisFreedman:1986}, \emph{as more data comes in, some Bayesian statisticians will become more and more convinced of the wrong answer}.\\
We also observe that, for complex systems, the computation of posterior probabilities has been made possible thanks to advances in computer science. We refer to \cite{Stuart:2010} for a (recent) general (Gaussian) framework for  Bayesian inverse problems and \cite{Beck:2010} for a rigorous UQ framework based on probability logic with Bayesian updating.
Just as Bayesian methods would have been considered computationally infeasible 50 years ago but are now common practice, OUQ methods are now becoming feasible and will only increase in feasibility with the passage of time and advances in computing.

\item The combination of structural optimization (in various fields of engineering) to produce an optimal design given the (deterministic) worst-case scenario has been referred to as \textbf{Optimization and Anti-Optimization} \cite{ElishakofOhsaki:2010} (we also refer to \textbf{critical excitation} in  seismic engineering \cite{Drenick:1973}). The main difference between OUQ and Anti-optimization lies in the fact that the former is based on an optimization over (admissible) functions and measures $(f,\mu)$, while the latter only involves an optimization over $f$.
Because of its robustness, many engineers have adopted the (deterministic) worst-case scenario approach to UQ (we refer to chapter 10 of \cite{ElishakofOhsaki:2010}) when a high  reliability is required. It is noted in \cite{ElishakofOhsaki:2010} that the reason why \emph{probabilistic methods do not find appreciation among theoreticians and practitioners alike} lies in the fact that ``probabilistic reliability studies involve assumptions on the probability densities, whose knowledge regarding relevant input quantities is central''. It is also observed that strong assumptions on $\P$ may lead to GIGO (``garbage in--garbage out'') situations for small certification thresholds $\epsilon$ when reliability estimates and probabilities of failure are very sensitive to small deviations in probability densities. On the other hand, UQ methods based on deterministic worst-case scenarios are oftentimes ``too pessimistic to be practical'' \cite{Drenick:1973, ElishakofOhsaki:2010}. We suggest that by allowing for very weak assumptions on probability measures, OUQ methods can lead to bounds on probabilities of failure that are both reliable and practical. Indeed, when applied to complex systems involving a large number of variables, deterministic worst-case methods do not take into account the \emph{improbability} that these (possibly independent or weakly correlated) variables conspire to produce a failure event.

\end{itemize}

The certification  problem \eqref{eq:def_cert} exhibits one of the main difficulties that face UQ practitioners:  many theoretical methods are available, but they require assumptions or conditions that, oftentimes, are not satisfied by the application.  More precisely, the characteristic elements distinguishing these different methods are the assumptions upon which they are based, and some methods will be more efficient than others depending on the validity of those assumptions.  UQ applications are also characterized by a set of assumptions/information on the response function $G$ and measure $\P$, which varies from application to application.  Hence, on the one hand, we have a list of theoretical methods that are applicable or efficient under very specific assumptions;  on the other hand, most applications are characterized by an information set or assumptions that, in general, do not match those required by these theoretical methods. It is hence natural to pursue the development of a rigorous framework  that does not add inappropriate assumptions or discard information.

We also observe that the effectiveness of different UQ methods cannot be compared without reference to the available information (some methods will be more efficient than others depending on those assumptions).
For the hypervelocity impact example of Subsection \ref{Subsec:motivatingex}, none of the methods mentioned above can be used without adding (arbitrary) assumptions on probability densities or discarding information on the mean value or independence of the input parameters.
We also observe that it is by placing information at the center of UQ that the proposed OUQ framework allows for the identification of best experiments.
Without focus on the available information, UQ methods are faced with the risk of propagating inappropriate assumptions and
producing a sophisticated answer to the wrong question. These distortions of the information set may be of limited impact on certification of common events but they are also of critical importance for the certification of rare events.

\subsection{OUQ with random sample data}

For the sake of clarity, we have started the description of OUQ with deterministic information and assumptions (i.e.\ when $\mathcal{A}$ is a deterministic set of functions and probability measures).  In many applications, however, some of the information arrives in the form of random samples.  The addition of such sample data to the available information and assumptions leads to non-trivial theoretical questions that are of practical importance beyond their fundamental connections with information theory and nonparametric statistics.  In particular, while the notion of an optimal bound \eqref{eq:def_ouq} is transparent and unambiguous, the notion of an optimal bound on $\P[G(X) \geq a]$ in presence of sample data is not immediately obvious and should be treated with care.  This is a very delicate topic, a full treatment of which we shall defer to a future work.

\section{Reduction of OUQ Optimization Problems}
\label{sec:Reduction}

In general, the  lower and upper values
\begin{eqnarray*}
	\mathcal{L}(\mathcal{A})& :=& \inf_{(f,\mu) \in \mathcal{A}} \mu [f(X) \geq a] \\
 \mathcal{U}(\mathcal{A})& :=& \sup_{(f,\mu) \in \mathcal{A}} \mu [f(X) \geq a]
\end{eqnarray*}
are each defined by a non-convex and infinite-dimensional optimization problem, the solution of which poses significant computational challenges.  These optimization problems can be considered to be a generalization of Chebyshev inequalities. The history of the classical inequalities can be found in \cite{KarlinStudden:1966}, and some generalizations in \cite{BertsimasPopescu:2005} and \cite{VandenbergheBoydComanor:2007};  in the latter works, the connection between Chebyshev inequalities and optimization theory is developed  based on the work of Mulholland and Rogers \cite{MulhollandRogers:1958}, Godwin \cite{Godwin:1955}, Isii \cite{Isii:1959, Isii:1960, Isii:1963}, Olkin and Pratt \cite{OlkinPratt:1958}, Marshall and Olkin \cite{MarshallOlkin:1960}, and the classical Markov--Krein Theorem \cite[pages 82 \&~157]{KarlinStudden:1966}, among others.  The Chebyshev-type inequalities defined by $\mathcal{L}(\mathcal{A})$ and $\mathcal{U}(\mathcal{A})$ are a further generalization to independence assumptions, more general domains, more general systems of moments, and the inclusion of classes of functions, in addition to the probability measures, in the optimization problem.  Moreover, although our goal is the computation of these values, and not an analytic expression for them, the study of probability inequalities should be useful in the reduction and approximation of these values.  Without providing a survey of this large body of work, we mention the field of majorization, as discussed in Marshall and Olkin \cite{MarshallOlkin:1979}, the inequalities of Anderson \cite{Anderson:1955}, Hoeffding \cite{Hoeffding:1956b}, Joe \cite{Joe:1987}, Bentkus \emph{et al.}\ \cite{BentkusGeuzeZuijlen:2006}, Bentkus \cite{Bentkus:2002, Bentkus:2004}, Pinelis \cite{Pinelis:2007, Pinelis:2008}, and Boucheron, Lugosi and Massart \cite{BoucheronLugosiMassart:2000}. Moreover, the solution of the resulting nonconvex optimization problems should benefit from duality theories for nonconvex optimization problems such as Rockafellar \cite{Rockafellar:1974} and the development of convex envelopes for them, as can be found, for example, in Rikun \cite{Rikun:1997} and Sherali \cite{Sherali:1997}. Finally, since Pardalos and Vavasis \cite{ParVav:1991} show that  quadratic
programming with one negative eigenvalue is NP-hard, we expect that some OUQ problems may be difficult to solve.

Let us now return to the earlier simple example of an admissible set $\mathcal{A}_{1}$ in \eqref{eq:simple_assumptions}:  the (non-unique) extremizers of the OUQ problem with the admissible set $\mathcal{A}_{1}$ all have the property that the support of the push-forward measure $f_{\ast} \mu$ on $\R$ contains at most two points, i.e.\ $f_{\ast} \mu$ is a convex combination of at most two Dirac delta measures (we recall that  $\# \mathrm{supp}(f_{\ast} \mu)$ is the number of points forming the support of $f_{\ast} \mu$):
\[
	\sup_{(f,\mu) \in \mathcal{A}_{1}} \mu[f(X) \geq a] = \sup_{\substack{ (f,\mu) \in \mathcal{A}_{1} \\ \# \mathrm{supp}(f_{\ast} \mu) \leq 2}} \mu[f(X) \geq a].
\]
The optimization problem on the left-hand side is an infinite-dimensional one, whereas the optimization problem on the right-hand side is amenable to finite-dimensional parametrization for each $f$.
Furthermore, for each $f$, only the two values of $f$ at the support points of the two Dirac measures are relevant to the problem.
The aim of this section is to show that a large class of OUQ problems --- those governed by independence and linear inequality constraints on the moments, --- are amenable to a similar finite-dimensional reduction, and that a priori upper bounds can be given on the number of Dirac delta masses that the reduction requires.

To begin with, we first show that an important class of optimization problems over the space of $m$-fold product measures can be reduced to optimization over products of finite convex combinations of Dirac masses ($m$ is the number of random input variables).  Consequently, we then show in Corollary \ref{cor:ouqreduce} that OUQ optimization problems where the admissible set is defined as a subset of function-measure pairs $(f,\mu)$ that satisfy generalized moment constraints $\G_{f}(\mu) \leq 0$ can also be reduced from the space of measures to the products of finite convex combinations of Dirac masses.  Theorem \ref{thm:valuereduce} shows that, when all the constraints are generalized moments of functions of $f$,  the search space $\mathcal{G}$ of functions can be further reduced to a search over functions on an $m$-fold product of finite discrete spaces, and the search over $m$-fold products of finite convex combinations of Dirac masses can be reduced to a search over the products of probability measures on this $m$-fold product of finite discrete spaces.  This latter reduction completely eliminates dependency on the coordinate positions in $\mathcal{X}$.  Theorem \ref{thm:valuereduce} is then used in Proposition \ref{prop:mcd} to obtain an optimal McDiarmid inequality through the formulation of an appropriate OUQ optimization problem followed by the above-mentioned reductions to an optimization problem on the product of functions on $\{ 0, 1 \}^{m}$ with the $m$-fold products of measures  on $\{ 0, 1 \}^{m}$. This problem is then further reduced, by Theorem \ref{thm:C}, to an optimization problem on the product of the space of subsets (power set) of $\{ 0, 1 \}^{m}$ with the product measures on $\{ 0, 1 \}^{m}$.  Finally, we obtain analytic solutions to this last problem for $m = 1, 2, 3$, thereby obtaining an optimal McDiarmid inequality in these cases.  We also obtain an asymptotic formula for general $m$.  Moreover, the solution for $m = 2$ indicates important information regarding the diameter parameters $D_{1}$ and $D_{2}$ (we refer to Example \ref{Ex:subdiameters}).  For example, if $D_{2}$ is sufficiently smaller than $D_{1}$, then the optimal bound only depends on $D_{1}$ and therefore, any decrease in $D_{2}$ does not improve the inequality.
See Subsection \ref{sec-reduceproofs} for the proofs of the results in this section.

\subsection{Reduction of OUQ}

For a topological space $\mathcal{X}$, let $\mathcal{F}_{\mathcal{X}}$ (or simply $\mathcal{F}$) denote the space of real-valued (Borel) measurable functions on $\mathcal{X}$, and let $\mathcal{M}(\mathcal{X})$ denote the set of Borel probability measures on $\mathcal{X}$.  Denote the process of integration with respect to a measure $\mu$ by $\E_{\mu}$, and let
\[
	\Delta_{k}(\mathcal{X}) := \left\{\sum_{j = 0}^{k} \alpha_{j} \delta_{x^{j}} \,\middle|\, x^{j} \in \mathcal{X}, \alpha_{j} \geq 0 \text{ for } j = 0, \dots, k \text{ and } \sum_{j = 0}^{k} \alpha_{j} = 1 \right\}
\]
denote the set of $(k+1)$-fold convex combinations of Dirac masses.  When $\mathcal{X} = \prod_{i=1}^{m} \mathcal{X}_{i}$ is a product of topological spaces, and we speak of measurable functions on the product $\mathcal{X}$, we mean measurable with respect to the product $\sigma$-algebra and not the Borel $\sigma$-algebra of the product.  For more discussion of this delicate topic, see e.g.~\cite{Johnson:1982}.  The linear equality and inequality constraints on our optimization problems will be encoded in the following measurable functions:
\[
	g'_{j} \colon \mathcal{X} \to \R \text{ for } j = 1, \dots, n',
\]
and, for each $i = 1, \dots, m$,
\[
	g^{i}_{j} \colon \mathcal{X}_{i} \to \R \text{ for } j = 1, \dots, n_{i}.
\]
Let $\mathcal{M}^{\G} \subseteq \mathcal{M}^{m}(\mathcal{X})$ denote the set of products of Borel measures for which these all these functions are integrable with finite integrals.  We use the compact notation $\G(\mu) \leq 0$ to indicate that $\mu \in \mathcal{M}^{\G}$ and that
\begin{align*}
 \E_{\mu}[g'_{j}] \leq 0 ,&\quad \text{ for all } j = 1, \dots, n', \\
 \E_{\mu}[g^{i}_{j}] \leq 0 , &\quad  j = 1, \dots, n_{i},  \text{ for all } i = 1, \dots, m \,.
\end{align*}
Moreover, let $r \colon \mathcal{X} \to \R$ be integrable for all  $\mu \in \mathcal{M}^{\G}$  (possibly with values $+\infty$ or $-\infty$).  For any set $\mathcal{M} \subseteq \mathcal{M}^{\G}$, let
\[
	\mathcal{U}(\mathcal{M}) := \sup_{\mu \in \mathcal{M}} \E_{\mu}[r],
\]
with the convention that the supremum of the empty set is $-\infty$.

For a measurable function $f$, the map $\mu \mapsto \E_{\mu}[f]$ may not be defined, since $f$ may not be absolutely integrable with respect to $\mu$.  If it is defined, then it is continuous in the strong topology on measures;  however, this topology is too strong to provide any compactness.  Moreover, although \cite[Theorem 14.5]{AliprantisBorder:2006} shows that if $f$ is a bounded upper semi-continuous function on a metric space, then integration is upper semi-continous in the weak-$\ast$ topology, we consider the case in which $\mathcal{X}$ may not be metric or compact, and the functions $f$ may be unbounded and lack continuity properties.  The following results heavily use results of Winkler \cite{Winkler:1988} --- which follow from an extension of Choquet Theory (see e.g.~\cite{Phelps:2001}) by von Weizs\"{a}cker and Winkler \cite[Corollary 3]{WeizsackerWinkler:1979} to sets of probability measures with generalized moment constraints --- and a result of Kendall \cite{Kendall:1962} characterizing cones, which are lattice cones in their own order.  These results generalize a result of Karr \cite{Karr:1983} that requires $\mathcal{X}$ to be compact, the constraint functions be bounded and continuous, and the constraints to be equalities.  The results that follow are remarkable in that they make extremely weak assumptions on $\mathcal{X}$ and no assumptions on the functions $f$. Recall that a Suslin space is the continuous image of a Polish space.

\begin{thm}
	\label{thm:baby_measure}
	Let $\mathcal{X} = \prod_{i=1}^{m}\mathcal{X}_{i}$ be a product of Suslin spaces and let
	\[
		\mathcal{M}^{m}(\mathcal{X}) := \bigotimes_{i=1}^{m}{\mathcal{M}(\mathcal{X}_{i}})
	\]
	denote the set of products of Borel probability measures on the spaces $\mathcal{X}_{i}$.  As above, consider the generalized moment functions $\G$ and the corresponding finite moment set $\mathcal{M}^{\G}$.  Suppose that $r \colon \mathcal{X} \to \R$ is  integrable for all $\mu \in \mathcal{M}^{\G}$ (possibly with values $+\infty$ or $-\infty$).  Define the reduced admissible set
	\[
		\mathcal{M}_{\Delta} := \left\{ \mu \in \bigotimes_{i=1}^{m} \Delta_{n_{i}+n'}(\mathcal{X}_{i}) \,\middle|\, \G(\mu) \leq 0 \right\}.
	\]
	Then, it holds that
	\[
		\mathcal{U}(\mathcal{M}^{\G}) = \mathcal{U}(\mathcal{M}_{\Delta}).
	\]
\end{thm}

Theorem \ref{thm:baby_measure} says that, on a product $\mathcal{X}$ of very general spaces $\mathcal{X}_{i}$, optimization problems constrained by $n'$ linear moment constraints on $\mathcal{X}$ and $n_{i}$ linear moment constraints on each factor space $\mathcal{X}_{i}$ achieve their optima among those product measures whose $i^{\mathrm{th}}$ marginal has support on at most $n' + n_{i} + 1$ points of $\mathcal{X}_{i}$.

\begin{rmk}
	\label{rmk:baby_measure_indexing}
	Using \cite[Corollary 3]{WeizsackerWinkler:1979}, this theorem and its consequences below easily generalize from the situation where $\E_{\mu}[g_{k}] \leq 0$ for each $k$ to that in which  $\E_{\mu}[g_{k}] \in I_{k}$  for each $k$, where $k$ indexes the constraint functions, and where each $I_{k}$ is a closed interval.  Consequently, such pairs of linear constraints introduce a requirement for only one Dirac mass, not the two masses that one might expect.  Moreover, observe that the condition that the function $r$ is integrable (possibly with values $+\infty$ or $-\infty$) for all $\mu \in \mathcal{M}^{\G}$ is satisfied if $r$ is non-negative.  In particular, this holds when $r$ is an indicator function of a set, which is our main application in this paper.
\end{rmk}
\begin{rmk}
	\label{rmk:other_extreme_measures}
	Theorem \ref{thm:baby_measure} and its consequents below can be expressed more generally in terms of extreme points of sets of measures, whereas in the above case, the extreme points are the Dirac masses. To that end, Dynkin \cite{Dynkin:1978} describes more general sets of measures and their extreme points, which can be useful in applications.  In particular, one could consider
	\begin{compactenum}
		\item sets of measures that are invariant under a transformation (the extreme points are the ergodic measures);
		\item symmetric measures on an infinite product space (the extreme points are the simple product measures);
		\item the set of stationary distributions for a given Markov transition function;
		\item the set of all Markov processes with a given transition function.
	\end{compactenum}
\end{rmk}

We now apply Theorem \ref{thm:baby_measure} to obtain the same type of reduction for an admissible set $\mathcal{A} \subseteq \mathcal{F} \times \mathcal{M}^{m}(\mathcal{X})$ consisting of pairs of functions and product measures --- this is the case for the OUQ optimization problems $\mathcal{L}(\mathcal{A})$ and $\mathcal{U}(\mathcal{A})$.  Let $\mathcal{G} \subseteq \mathcal{F}$ denote a subset of real-valued measurable functions on $\mathcal{X}$ and consider an admissible set $\mathcal{A} \subseteq \mathcal{G} \times \mathcal{M}^{m}(\mathcal{X})$ defined in the following way.  For each $f \in \mathcal{G}$, let $\G(f, \cdot)$ denote a family of constraints as in Theorem \ref{thm:baby_measure} and Remark \ref{rmk:baby_measure_indexing}.  For each $f \in \mathcal{G}$, let $\mathcal{M}^{\G_{f}} \subseteq \mathcal{M}^{m}(\mathcal{X})$ denote those product probability measures $\mu$ such that the moments $\G(f, \mu)$ are well-defined.  Moreover, for each $f \in \mathcal{G}$, let $r_{f} \colon \mathcal{X} \to \R$ be integrable for all $\mu \in \mathcal{M}^{\G_{f}}$ (possibly with values $+\infty$ or $-\infty$).  Define the admissible set
\begin{equation}\label{eq:firstAtobereduced}
	\mathcal{A} := \left\{ (f, \mu) \in \mathcal{G} \times \mathcal{M}^{m}(\mathcal{X}) \,\middle|\, \G(f,\mu) \leq 0 \right\}
\end{equation}
and define the OUQ optimization problem to be
\begin{equation}
	\label{eq:def_ouq_for_reduction}
	\mathcal{U}(\mathcal{A}) := \sup_{(f,\mu) \in \mathcal{A}} \E_{\mu}[r_{f}].
\end{equation}

\begin{cor}
	\label{cor:ouqreduce}
	Consider the OUQ optimization problem \eqref{eq:def_ouq_for_reduction} and define the reduced admissible set $\mathcal{A}_{\Delta} \subseteq \mathcal{A}$ by
	\begin{equation}\label{eq:reducedtoAdelta}
		\mathcal{A}_{\Delta} := \left\{ (f, \mu) \in \mathcal{G} \times \bigotimes_{i=1}^{m} \Delta_{n_{i}+n'}(\mathcal{X}_{i}) \,\middle|\, \G(f, \mu) \leq 0 \right\}.
	\end{equation}
	Then, it holds that
	\[
		\mathcal{U}(\mathcal{A}) = \mathcal{U}(\mathcal{A}_{\Delta}).
	\]
\end{cor}

\begin{rmk} Corollary \ref{cor:ouqreduce} is easily generalized to the case where for each $f \in \mathcal{G}$,  $i$,  and fixed $\mu_{j}, j\neq i$, $ \G(f,\mu_{1},..,\mu_{i},..,\mu_{m})$ has affine dimension at most $m_{i}$ as $\mu_{i}$ varies. In this case
\[
                \mathcal{A}_{\Delta} := \left\{ (f,\mu)  \in \mathcal{G} \times  \bigotimes_{i=1}^{m}
 \Delta_{m_{i}}(\mathcal{X}_{i}) \,\middle|\, \G(f,\mu) \leq 0 \right\}.        \]
\end{rmk}
\begin{rmk}
Linear moment constraints on the factor spaces $\mathcal{X}_i$ allow to consider information sets with independent random variables $X_1,\ldots,X_m$ and weak constraints on the probability measure of the variables $X_i$. An example of such an admissible set is the one associated with Bernstein inequalities \cite{Bernstein:1964}, in which a priori bounds are given on the variances of the variables $X_i$.
\end{rmk}

\subsection{Generalized moments of the response function}

We now consider the case where the function $r_{f} := r \circ f$ is defined through composition with a measurable function $r$, and all $n$ constraints are determined by compositions $g'_{j} := g_{j} \circ f$, with $j = 1, \dots, n$, of the function $f$.  Hence, the symbol $\G(f, \mu)$ will mean that all functions $g_{j} \circ f$ are $\mu$ integrable and will represent the values $\E_{\mu}[g_{j} \circ f]$ for $j = 1, \dots, n$.  That is, we have the admissible set
\begin{equation}
	\label{eq:def_amoments}
	\mathcal{A} := \left\{ (f,\mu) \in \mathcal{G} \times \mathcal{M}^{m}(\mathcal{X}) \,\middle|\, \G(f,\mu) \leq 0 \right\}
\end{equation}
and the optimization problem
\begin{equation}
	\label{eq:def_ouqmoments}
	\mathcal{U}(\mathcal{A}) := \sup_{(f,\mu) \in \mathcal{A}} \E_{\mu}[r\circ f]
\end{equation}
as in \eqref{eq:def_ouq_for_reduction}.  However, in this case, the fact that the criterion function $r \circ f$ and the constraint functions $g_{j} \circ f$ are compositions of the function $f$ permits a finite-dimensional reduction of the space of functions $\mathcal{G}$ to a space of functions on $\{ 0, \dots, n \}^{m}$ and a reduction of the space of $m$-fold products of finite convex combinations of Dirac masses to the space of product measures on $\{ 0, \dots, n \}^{m}$.  This reduction completely eliminates dependency on the coordinate positions in $\mathcal{X}$.

Formulating this result precisely will require some additional notation.
By the Well-Ordering Theorem,  there exists a well-ordering of each $\mathcal{X}_{i}$.  Suppose that a total ordering of the elements of the spaces $\mathcal{X}_{i}$ for $i = 1, \dots, m$ is specified.  Let $\mathcal{N} := \{ 0, \dots, n \}$ and $\mathcal{D} := \{0, \dots, n\}^{m} = \mathcal{N}^{m}$.  Every element $\mu\in \bigotimes_{i = 1}^{m} \Delta_{n}(\mathcal{X}_{i})$ is a product $\mu = \bigotimes_{i=1}^{m} \mu_{i}$ where each factor $\mu_{i}$ is a convex sum of $n+1$ Dirac masses indexed according to the ordering;  that is,
\[
	\mu_{i} = \sum_{k=0}^{n} \alpha^{i}_{k} \delta_{x_{i}^{k}}
\]
for some $\alpha^{i}_{1}, \dots, \alpha^{i}_{n} \geq 0$ with unit sum and some $x_{i}^{1}, \dots, x_{i}^{n} \in \mathcal{X}_{i}$ such that, with respect to the given ordering of $\mathcal{X}_{i}$,
\[
	x_{i}^{1} \leq x_{i}^{2} \leq \dots \leq x_{i}^{n}.
\]
Let $\mathcal{F}_\mathcal{D}$ denote the real linear space of real functions on $\mathcal{D} = \{ 0, \dots, n \}^{m}$ and consider the mapping
\[
	\mathbb{F} \colon \mathcal{F} \times \bigotimes_{i=1}^{m} \Delta_{n}(\mathcal{X}_{i}) \to \mathcal{F}_\mathcal{D}
\]
defined by
\[
	\left( \mathbb{F}(f, \mu) \right) (i_{1}, i_{2}, \dots, i_{m}) = f(x_{1}^{i_{1}}, x_{2}^{i_{2}}, \dots, x_{m}^{i_{m}}), \quad i_{k} \in \mathcal{N}, k = 1, \dots, m.
\]
$\mathbb{F}$ represents the values of the function $f$ at the Dirac masses in $\mu$, but does not carry information regarding the positions of the Dirac masses or their weights.

\begin{thm}
	\label{thm:valuereduce}
	Consider the admissible set $\mathcal{A}$ and optimization problem $\mathcal{U}(\mathcal{A})$ defined in \eqref{eq:def_amoments} and \eqref{eq:def_ouqmoments} where $r \circ f$ is integrable (possibly with values $+\infty$ or $-\infty$) for all product measures.  For a subset $\mathcal{G}_\mathcal{D} \subseteq \mathcal{F}_\mathcal{D}$, define the admissible set
	\begin{equation}\label{eq:adgen}
		\mathcal{A}_\mathcal{D} = \left\{ (h, \alpha) \in \mathcal{G}_\mathcal{D} \times \mathcal{M}^{m}(\mathcal{D}) \,\middle|\, \E_{\alpha}[g_{i} \circ h] \leq 0 \text{ for all } j = 1, \dots, n \right\}
	\end{equation}
	and the optimization problem
	\[
	\mathcal{U}(\mathcal{A}_\mathcal{D}) := \sup_{(h,\alpha) \in \mathcal{A}_\mathcal{D}} \E_{\alpha}[r\circ h].
	\]
	If
	\[
		\mathbb{F} \left( \mathcal{G} \times \bigotimes_{i=1}^{m}{\Delta_{n}(\mathcal{X}_{i})} \right) = \mathcal{G}_\mathcal{D},
	\]
	then it holds that
	\[
		\mathcal{U}(\mathcal{A}) = \mathcal{U}(\mathcal{A}_\mathcal{D}).
	\]
\end{thm}

When the constraint set also includes functions which are not compositions with $f$, then Theorem \ref{thm:valuereduce} does not apply.  Although it does appear that results similar to Theorem \ref{thm:valuereduce} can be obtained, we leave that as a topic for future work.

\subsection{Application to McDiarmid's inequality}
\label{sec_appmcd}

Theorem \ref{thm:valuereduce} can be applied to the situation of McDiarmid's inequality in order to obtain an optimal solution for that problem.  Let $D_{i} \geq 0$ for $i = 1, \dots, m$ and define
\begin{equation}
	\label{eq:diamparameters}
	\mathcal{G} := \{ f\in \mathcal{F} \mid \operatorname{Osc}_i(f) \leq D_{i} \text{ for each } i=1, \dots, m \},
\end{equation}
where
\[
	\operatorname{Osc}_i(f) := \sup_{(x_1,\ldots,x_{m}) \in \mathcal{X}} \sup_{x_i' \in \mathcal{X}_{i}} \left| f(\ldots,x_i,\ldots)- f(\ldots,x_i',\ldots) \right|.
\]
We have a product probability measure $\P$ on $\mathcal{X}$ and a measurable function $H \colon \mathcal{X} \to \R$ such that $H \in \mathcal{G}$.  Suppose that we have an upper bound
\begin{equation}
	\label{eq:bound}
	\P[H - \E_{\P}[H] \geq a] \leq \mathbb{H}(a, \mathcal{G}) \text{ for all } H \in \mathcal{G}.
\end{equation}
It follows that if $H \in \mathcal{G}$ and $\E_{\P}[H] \leq 0$, then
\[
	\P[H \geq a] \leq \P[H - \E_{\P}[H] \geq a] \leq \mathbb{H}(a,\mathcal{G}) \text{ for all } H \in \mathcal{G} \text{ with }
\E_{\P}[H] \leq 0.
\]
On the other hand, suppose that
\begin{equation}
	\label{eq:bound2}
	\P[H \geq a] \leq \mathbb{H}'(a,\mathcal{G}) \text{ for all } H \in \mathcal{G} \text{ with } \E_{\P}[H] \leq 0.
\end{equation}
It follows that
\[
	\P[H \geq a] \leq \mathbb{H}'(a,\mathcal{G}) \text{ for all } H \in \mathcal{G} \text{ with } \E_{\P}[H] = 0 .
\]
Since the constraints $\mathcal{G}$ and the event $H-\E_{\P}[H] \geq a$ are invariant under scalar translation $H \mapsto H + c$ it follows that
\[
	\P[H-\E_{\P}[H] \geq a] \leq \mathbb{H}'(a,\mathcal{G}) \text{ for all } H \in \mathcal{G}.
\]
That is, the inequalities \eqref{eq:bound} and \eqref{eq:bound2} are equivalent.

McDiarmid's inequality \cite{McDiarmid:1989, McDiarmid:1998} provides the bound $\mathbb{H}(a,\mathcal{G}) := \exp ( -\frac{2a^{2}}{D^2} )$ for \eqref{eq:bound} and its equivalent \eqref{eq:bound2}, with
\begin{equation}
	\label{defD2}
	D^2:=\sum_{i=1}^m D_i^2.
\end{equation}
Define the admissible set corresponding to the assumptions of McDiarmid's inequality:
\begin{equation}
	\label{eq:MDassumptions}
	\mathcal{A}_{\mathrm{McD}} = \left\{ (f,\mu) \in \mathcal{G} \times \mathcal{M}^{m}(\mathcal{X}) \,\middle|\, \E_{\mu}[f] \leq 0 \right\},
\end{equation}
and define the optimization problem
\begin{equation}\label{eq:sjjshgdjhgejhge}
	\mathcal{U}(\mathcal{A}_{\mathrm{McD}}) := \sup_{(f,\mu) \in \mathcal{A}_{\mathrm{McD}}} \mu[f \geq a].
\end{equation}
Since $(H,\P) \in \mathcal{A}_{\mathrm{McD}}$ and McDiarmid's inequality $\mu[f \geq a] \leq \exp ( -\frac{2a^{2}}{D^{2}} )$ is satisfied for all $(f, \mu) \in \mathcal{A}_{\mathrm{McD}}$, it follows that
\[
	\P[H \geq a] \leq \mathcal{U}(\mathcal{A}_{\mathrm{McD}}) \leq \exp \left( - \frac{2a^{2}}{D^{2}} \right).
\]
Moreover, the inequality on the left is optimal in the sense that, for every $\varepsilon > 0$, there exists a McDiarmid-admissible scenario $(f, \mu)$ satisfying the same assumptions as $(H, \P)$ such that $\mu[f \geq a] \geq \mathcal{U}(\mathcal{A}_{\mathrm{McD}}) - \varepsilon$.

To apply the previous results to computing $\mathcal{U}(\mathcal{A}_{\mathrm{McD}})$, let $\mathcal{D} := \{ 0, 1 \}^{m}$ and define
\[
	\mathcal{G}_\mathcal{D} := \{ h \in \mathcal{F}_\mathcal{D} \mid \operatorname{Osc}_{k}(h) \leq D_{k} \text{ for each }  k = 1, \dots, m \},
\]
where the inequality $\operatorname{Osc}_{k}(h) \leq D_{k}$ for $h \in \mathcal{F}_\mathcal{D}$ means that
\begin{equation}\label{eq:Osc}
	| h(s_{1},\dots,s_{k},\dots,s_{m}) -h(s_{1},\dots,s_{k'},\dots,s_{m}) | \leq D_{i},
\end{equation}
for all $s_{j} \in \{ 0, 1 \}$, $j = 1,\dots, m$, and all $s_{k'} \in \{ 0, 1 \}$.  Define the corresponding admissible set
\begin{equation}\label{eq:adfromcd}
	\mathcal{A}_\mathcal{D} = \left\{ (h,\alpha) \in \mathcal{G}_\mathcal{D} \times \mathcal{M}(\{ 0, 1 \})^{m} \,\middle|\, \E_{\alpha}[h] \leq 0 \right\}
\end{equation}
and the optimization problem
\begin{equation}
	\mathcal{U}(\mathcal{A}_\mathcal{D}) := \sup_{(h,\alpha) \in \mathcal{A}_{\mathcal{D}}} \alpha[h \geq a].
\end{equation}

\begin{prop}
	\label{prop:mcd}
	It holds that
	\begin{equation}
		\label{eq:jgdjshghege}
		\mathcal{U}(\mathcal{A}_{\mathrm{McD}}) = \mathcal{U}(\mathcal{A}_\mathcal{D}).
	\end{equation}
\end{prop}

We now provide a further reduction of $\mathcal{U}(\mathcal{A}_{\mathrm{McD}})$ by reducing $\mathcal{U}(\mathcal{A}_\mathcal{D})$.  To that end, for two vertices $s$ and $t$ of $\mathcal{D} = \{ 0, 1 \}^{m}$, let $I(s,t)$ be the set of indices $i$ such that $s_i \not= t_i$.  For $s \in \mathcal{D}$, define the function $h^s \in \mathcal{F}_\mathcal{D}$ by
\[
	h^s(t) = a - \sum_{i \in I(s,t)} D_i.
\]
For $C \subseteq \mathcal{D}$, define $h^C \in \mathcal{F}_\mathcal{D}$ by
\begin{equation}
	\label{eq:jhdsjdgjhsgdgh}
	h^C(t) := \max_{s \in C} h^s(t) = a - \min_{s\in C} \sum_{i \in I(s,t)} D_i.
\end{equation}
Let $\mathcal{C} := \{ C \mid C \subseteq \mathcal{D} \}$ be the power set of $\mathcal{D}$ (the set of all subsets of $\mathcal{D}$), define the admissible set $\mathcal{A}_{\mathcal{C}}$ by
\begin{equation}\label{eq:acequ}
	\mathcal{A}_{\mathcal{C}} := \left\{ (C, \alpha) \in  \mathcal{C} \times \mathcal{M}(\{ 0, 1 \})^{m} \,\middle|\, \E_{\alpha}[h^{C}] \leq 0 \right\}
\end{equation}
and consider the optimization problem
\begin{equation}
	\label{eq:jkshdjshdjheer}
	\mathcal{U}(\mathcal{A}_{\mathcal{C}}) := \sup_{(C,\alpha) \in \mathcal{A}_\mathcal{C}} \alpha(h^{C} \geq a).
\end{equation}

\begin{thm}
	\label{thm:C}
	It holds that
	\begin{equation}
		\label{eq:equivMcDC}
		\mathcal{U}(\mathcal{A}_\mathcal{D}) = \mathcal{U}(\mathcal{A}_\mathcal{C}).
	\end{equation}
\end{thm}

\begin{rmk}
	The proof of this reduction theorem utilizes the standard lattice structure of the space of functions
	$\mathcal{F}_{\mathcal{D}}$ in a substantial way.  To begin with, the reduction to  $\max h = a$  is attained through lattice invariance.  Moreover, we have a lattice $\mathcal{F}_{\mathcal{D}}$, with sub-lattice $\mathcal{G}_{\mathcal{D}}$, and for each $C \in \mathcal{C}$,  the set $C_{\mathcal{D}} := \{ h \in \mathcal{F}_{\mathcal{D}} \mid \{ s \mid h(s) = a \} = C \} \}$ of functions with value $a$ precisely on the set $C$ is a sub-lattice.  For a clipped $h$, let $C(h) := \{ s \in \mathcal{D} \mid h(s) = a \}$ be the set where $h$ has the value $a$.  If for each $C$ the set
	\[
		\bigcap_{h : C(h) = C} \left\{ f \leq h \right\}  \cap  C_{\mathcal{D}} \cap \mathcal{G}_{\mathcal{D}}
	\]
	is nonempty, then we obtain a reduction.  However, not only is the set nonempty, but the map $C \mapsto h^{C}$ is a simple algorithm that produces a point in this intersection, and therefore an explicit reduction. We suspect that the existence of a simple reduction algorithm in this case is due to the lattice structures, and that such structures may be useful in the more general case.  Indeed, the condition $f \leq h$ implies that $\E_{\alpha}[f] \leq \E_{\alpha}[h]$ for any $\alpha$, and the
	the condition that $\E_{\alpha}[f] \leq \E_{\alpha}[h]$ for all $\alpha$ implies that $f \leq h$, so that the above condition is equivalent to the non-emptiness of
	\[
		\bigcap_{h : C(h)=C} \left\{ \bigcap_{\alpha} \left\{ \E_{\alpha}[f] \leq \E_{\alpha}[h] \right\}\right\}
	\cap  C_{\mathcal{D}} \cap \mathcal{G}_{\mathcal{D}}.
	\]
	For the more general constraints, we would instead have to solve (i.e.\ find an element of)
	\[
		\bigcap_{h:C(h)=C} \left\{ \bigcap_{\alpha} \left\{\G(f,\alpha) \leq \G(h,\alpha) \right\} \right\} \cap C_{\mathcal{D}} \cap \mathcal{G}_{\mathcal{D}}.
	\]
\end{rmk}

\begin{rmk}
We refer to the following diagram for a summary of the relationships between admissible sets $\mathcal{A}$, $\mathcal{A}_{\Delta}$, $\mathcal{A}_{\mathcal{D}}$, $\mathcal{A}_{\mathcal{C}} $, $\mathcal{A}_{\mathrm{McD}}$, the reduction theorems and their assumptions.

  \begin{center}
 \xymatrix{
\mathcal{A}_{\mathrm{McD}} \eqref{eq:MDassumptions}\ar@{->}[dd]  & \mathcal{A} \eqref{eq:def_amoments} \ar@{->}[d]|-{\txt{\footnotesize
Corollary \ref{cor:ouqreduce} }} &  \txt{\footnotesize $(f,\mu)  \in \mathcal{G} \times  \bigotimes_{i=1}^{m}
 \mathcal{M}(\mathcal{X}_{i}),\, \mu=\otimes_{i=1}^m \mu_i$}\ar@{.>}[d]|-{\txt{\footnotesize
$n$ and $n_i$ generalized moment constraints on $\mu$ and $\mu_i$ }} \\
& \mathcal{A}_{\Delta} \eqref{eq:reducedtoAdelta} \ar@{->}[d]|-{\txt{\footnotesize
Theorem \ref{thm:valuereduce} }} & \txt{\footnotesize $\mu_i$ reduces to the weighed sum of $n+n_i+1$ Diracs \\ \footnotesize $(f, \mu) \in \mathcal{G} \times \bigotimes_{i=1}^{m} \Delta_{n_{i}+n'}(\mathcal{X}_{i})$}\ar@{.>}[d]|-{\txt{\footnotesize
Quantity of interest $r\circ f$, $n$ constraints $\E_{\mu}[g_j \circ f]\leq 0$ }}   \\
\txt{\footnotesize Proposition \ref{prop:mcd} }\ar@{->}[r]   & \eqref{eq:adgen} \mathcal{A}_{\mathcal{D}} \eqref{eq:adfromcd}\ar@{->}[d]|-{\txt{\footnotesize
Theorem \ref{thm:C} }} & \txt{\footnotesize  $f$ and $\mu$ reduce to a function and a measure on a finite set\\ \footnotesize $(h, \alpha) \in \mathcal{G}_\mathcal{D} \times \mathcal{M}^{m}(\mathcal{D}),\, \mathcal{D}=\{ 0, \dots, n \}^{m}$ } \ar@{.>}[d]|-{\txt{\footnotesize
The space of functions $\mathcal{G}_\mathcal{D}$ has a lattice structure  }} \\
& \mathcal{A}_{\mathcal{C}} \eqref{eq:acequ} &  \txt{\footnotesize Functions $h$ can be parameterized by a finite set\\ \footnotesize  $(C, \alpha) \in  \mathcal{C} \times \mathcal{M}(\{ 0, 1 \})^{m}$,\, $\mathcal{C}=\{C \mid C \subseteq \{0,1\}^m\}$}
  }
\end{center}

\end{rmk}

\section{Optimal Concentration Inequalities}
\label{sec:Reduction-mcd}

In this section, the results of Section \ref{sec:Reduction} will be applied to obtain optimal concentration inequalities under the assumptions of McDiarmid's inequality and Hoeffding's inequality.  The following subsection gives explicit concentration results under the assumptions of  McDiarmid's inequality, and Subsection \ref{subsecHoeffding} gives explicit concentration results under the assumptions of Hoeffding's inequality.

Surprisingly, these explicit results show that, although uncertainties may propagate for the true value of $G$ and $\P$,  they might not when the information is incomplete on $G$ and $\P$.

We refer to Subsection \ref{sec-mcdproofs} for the proofs of the results in this section.

\subsection{Explicit solutions under the assumptions of McDiarmid's inequality}

In this subsection, we will apply Theorem \ref{thm:C} to obtain explicit formulae for the OUQ problem $\mathcal{U}(\mathcal{A}_{\mathrm{McD}})$ (defined in Equation \eqref{eq:sjjshgdjhgejhge}) under the assumptions of McDiarmid's inequality  \eqref{eq:MDassumptions}.  More precisely, we will compute $\mathcal{U}(\mathcal{A}_{\mathcal{C}})$ defined by equation \eqref{eq:jkshdjshdjheer} and use equalities \eqref{eq:equivMcDC} and \eqref{eq:jgdjshghege} to obtain $\mathcal{U}(\mathcal{A}_{\mathrm{McD}})=\mathcal{U}(\mathcal{A}_{\mathcal{C}})$.  Observe that all the following optimization problems possess solutions because they involve the optimization of a continuous function (with respect to $\alpha$) in a compact space.

Since the inequalities \eqref{eq:bound} and \eqref{eq:bound2} are equivalent, it follows that
\[
	\mathcal{U}(\mathcal{A}_{\mathrm{McD}}) = \sup_{(f, \mu) \in \mathcal{G} \times \mathcal{M}_{m} } \mu\big[f \geq a + \E_{\mu}[f]\big].
\]
In particular, if $\E_{\mu}[f] \leq 0$ is replaced by $\E_{\mu}[f] \leq b$ or $\E_{\mu}[f] = b$ in McDiarmid's inequality assumptions \eqref{eq:MDassumptions}, then the results given in this section remain valid by replacing $a$ by $M:=a-b$ (observe that $M$ plays the role of a margin).

Those results should be compared with McDiarmid's inequality \cite{McDiarmid:1989, McDiarmid:1998}, which provides the bound
\begin{equation}
	\label{eeihuhcussumptions}
	\sup_{(f, \mu) \in \mathcal{G} \times \mathcal{M}_{m} } \mu\big[f \geq a + \E_{\mu}[f]\big]\leq \exp \left( - \frac{2a^{2}}{\sum_{i=1}^m D_i^{2}} \right).
\end{equation}

The statements of the theorem will be given assuming that $a\geq 0$;  in the complementary case of $a < 0$, the solution is simply $\mathcal{U}(\mathcal{A}_{\mathrm{McD}}) = 1$.

To the best of the authors' knowledge, the optimal bounds given here are new.  There is a substantial literature relating to optimization of concentration bounds and de-randomization algorithms (see for instance \cite{StrivastavStangier:1995} and references therein) but, as far as the authors know, those bounds were suboptimal because they were obtained through the moment generating function technique.

\subsubsection{Explicit solutions in dimensions one and two}

\begin{thm}[Explicit solution for $m = 1$]
	\label{thm:m1}
	For $m = 1$, $\mathcal{U}(\mathcal{A}_{\mathrm{McD}})$ is given by
	\begin{equation}
		\label{eq:McD_reduced_explicit_1d}
		\mathcal{U}(\mathcal{A}_{\mathrm{McD}}) =
		\begin{cases}
			0, & \text{if }  D_{1} \leq a, \\
			1 - \dfrac{a}{D_1}, & \text{if } 0 \leq a \leq D_{1}.
		\end{cases}
	\end{equation}
\end{thm}

\begin{thm}[Explicit solution for $m = 2$]
\label{thm:m2}
	For $m = 2$, $\mathcal{U}(\mathcal{A}_{\mathrm{McD}})$ is given by
	\begin{equation}
		\label{eq:McD_reduced_explicit_2d}
		\mathcal{U}(\mathcal{A}_{\mathrm{McD}}) =
		\begin{cases}
			0, & \text{if } D_{1} + D_{2} \leq a, \\
			\dfrac{(D_{1} + D_{2} -a)^{2}}{4 D_{1} D_{2}}, & \text{if } | D_{1} - D_{2} | \leq a \leq  D_{1} + D_{2}, \\
			1 - \dfrac{a}{\max (D_{1}, D_{2})}, & \text{if } 0\leq  a \leq | D_{1} - D_{2} |.
		\end{cases}
	\end{equation}
\end{thm}

See Sub-figures \ref{fig:mcd1}, \ref{fig:mcd2} and \ref{fig:mcd3} for illustrations comparing the McDiarmid and OUQ bounds for $m=2$ (as functions of $(D_1,D_2)$, with mean performance $0$ and failure threshold $a = 1$, the OUQ bound is calculated using the explicit solution \eqref{eq:McD_reduced_explicit_2d}).
Observe that
\begin{itemize}
	\item If $a \leq  D_{1} - D_{2}$,  then a decrease in $D_2$ does not lead to a decrease in the OUQ bound $\mathcal{U}(\mathcal{A}_{\mathrm{McD}})$.  In other words, if most of the uncertainty is contained in the first variable ($a+D_{2} \leq  D_{1}$), then the uncertainty associated with the second variable does not affect the global uncertainty;  a reduction of the global uncertainty requires a reduction in $D_1$.
	\item For $D_1+D_2=2a$, the ratio between the OUQ bound and the  McDiarmid bound is minimized near the diagonal.
\end{itemize}
\begin{rmk}
The maximum of \eqref{eq:McD_reduced_explicit_2d} over $D_1,D_2$ under the constraints $D_1+D_2=D$ and $D_1\geq D_2$ is achieved at $D_2=0$ and is equal to $1-a/D$.
The minimum of \eqref{eq:McD_reduced_explicit_2d} over $D_1,D_2$ under the constraints $D_1+D_2=D$ and $D_1\geq D_2$ is achieved on the diagonal $D_1=D_2$ and is equal to $(1-a/D)^2$.
\end{rmk}

\begin{figure}[tp]
	\begin{center}
		\caption{Comparison of the McDiarmid and OUQ bounds with zero mean performance and
failure threshold $a = 1$.}
		\label{fig:McDiarmidVsOUQm2}
		\subfigure[McDiarmid upper bound, $m=2$]{\label{fig:mcd1}
			\includegraphics[width=0.45\textwidth]{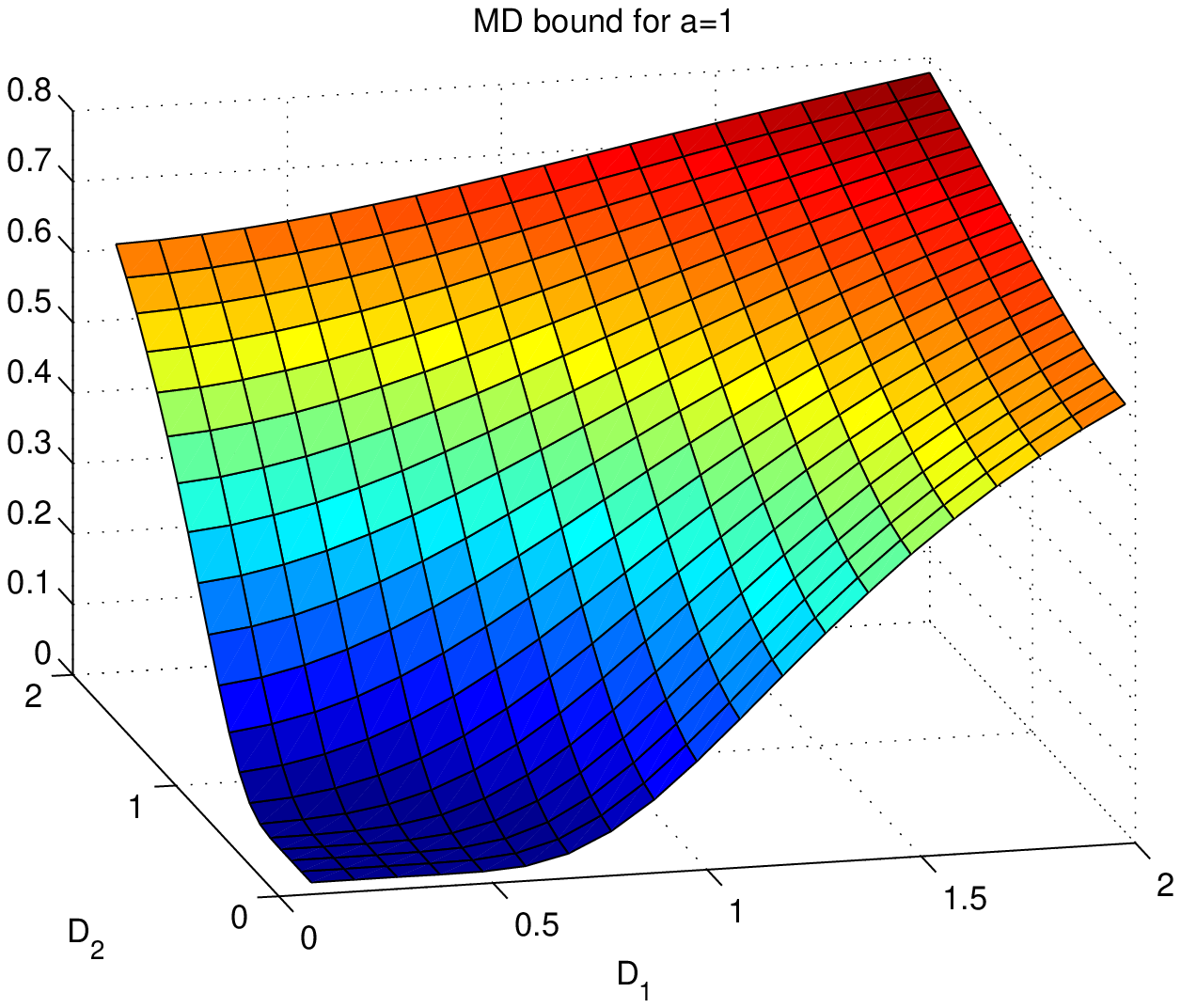}
		}
		\subfigure[OUQ upper bound, $m=2$]{\label{fig:mcd2}
			\includegraphics[width=0.45\textwidth]{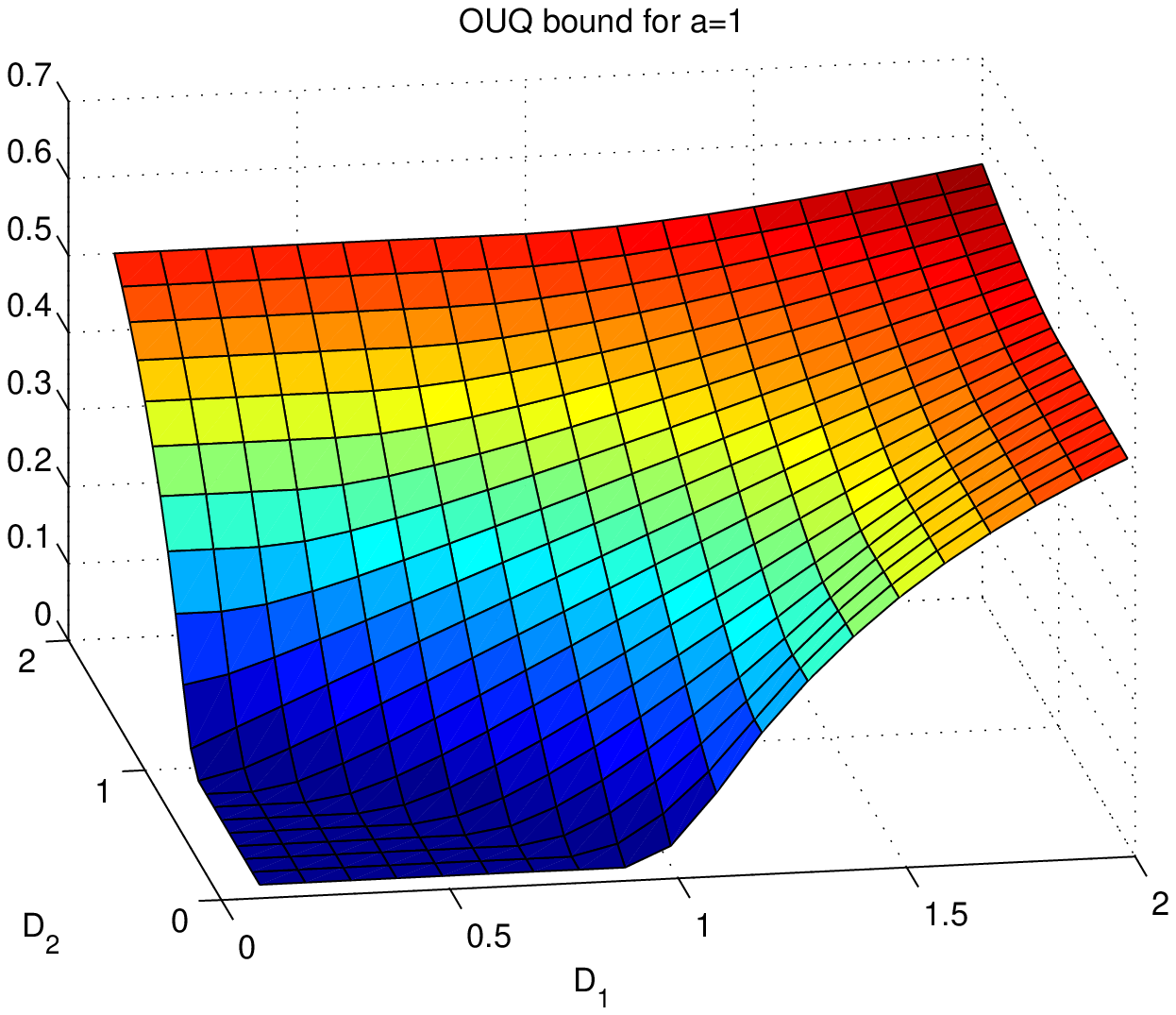}
		}
		\subfigure[Ratio of the two bounds:  OUQ bound divided by McDiarmid bound, $m=2$]{\label{fig:mcd3}
			\includegraphics[width=0.45\textwidth]{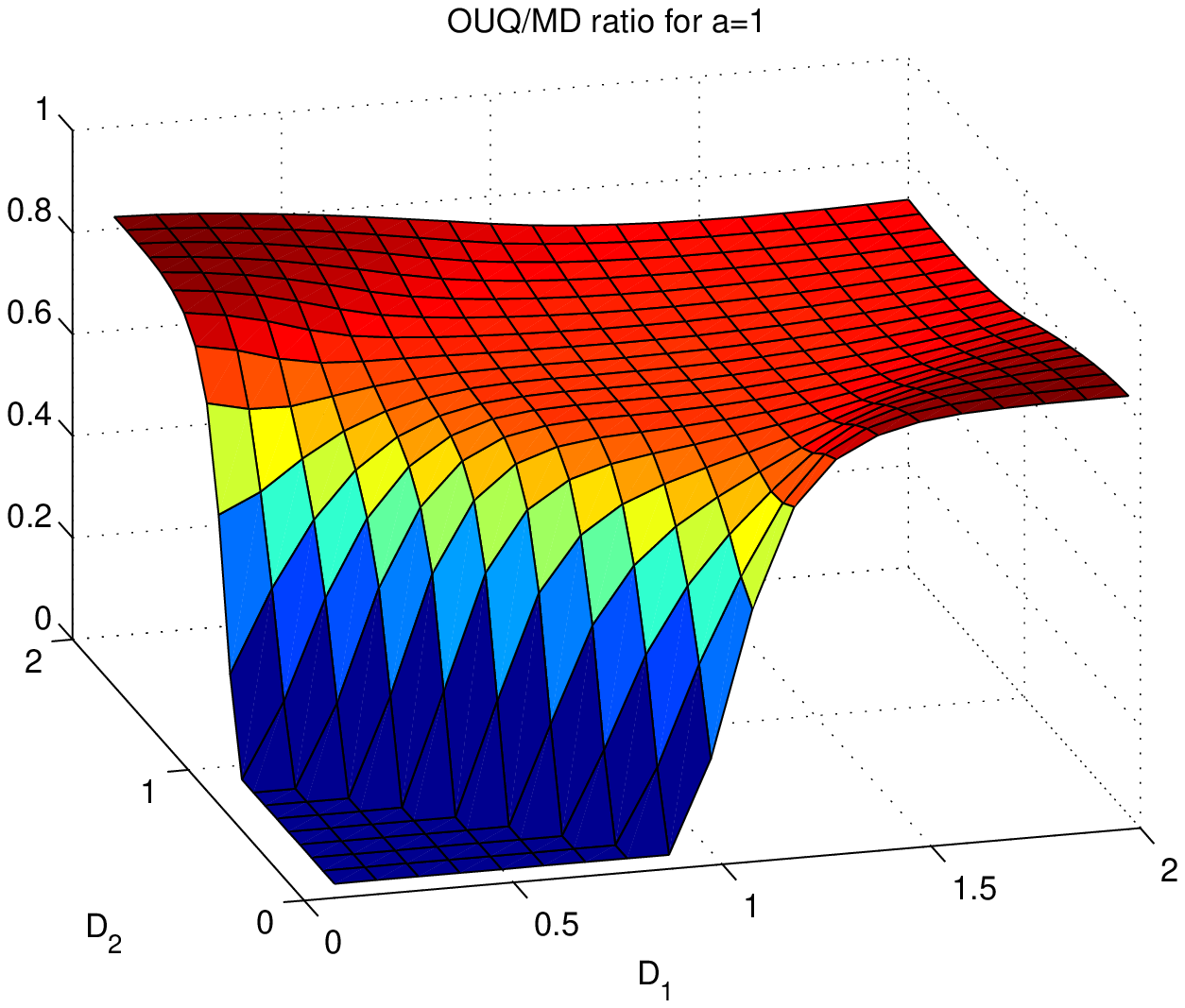}
		}
		\subfigure[McDiarmid vs OUQ bound, $m=3$ and $D_1=D_2=D_3$]{\label{fig:mcd4}
			\includegraphics[width=0.45\textwidth]{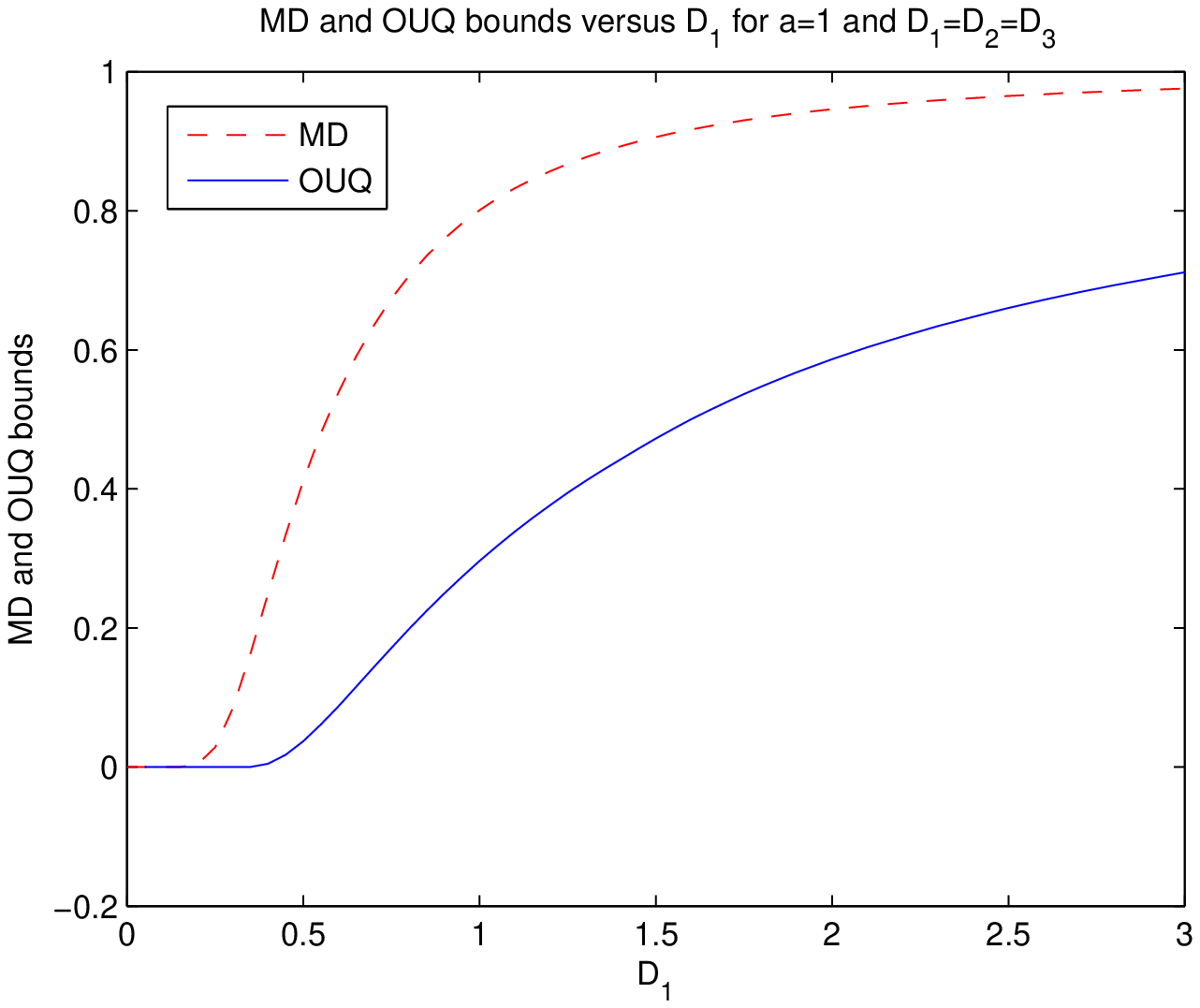}
		}
		\subfigure[$\mathcal{F}_1$ vs $\mathcal{F}_2$, $D_1=D_2=D_3$]{\label{fig:mcd5}
			\includegraphics[width=0.45\textwidth]{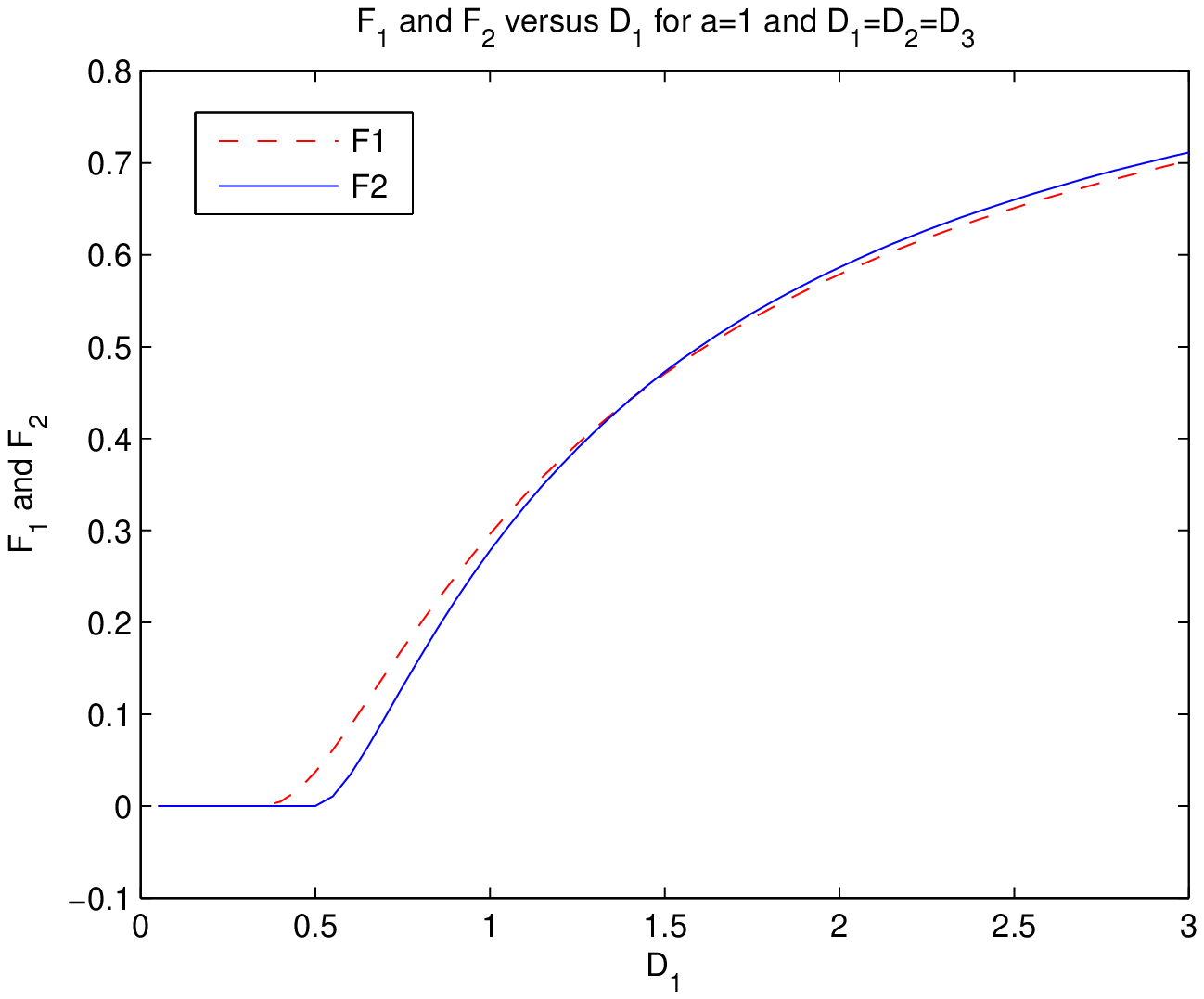}
		}
		\subfigure[$\mathcal{F}_1$ vs $\mathcal{F}_2$, $D_1=D_2=\frac{3}{2}D_3$]{\label{fig:mcd6}
			\includegraphics[width=0.45\textwidth]{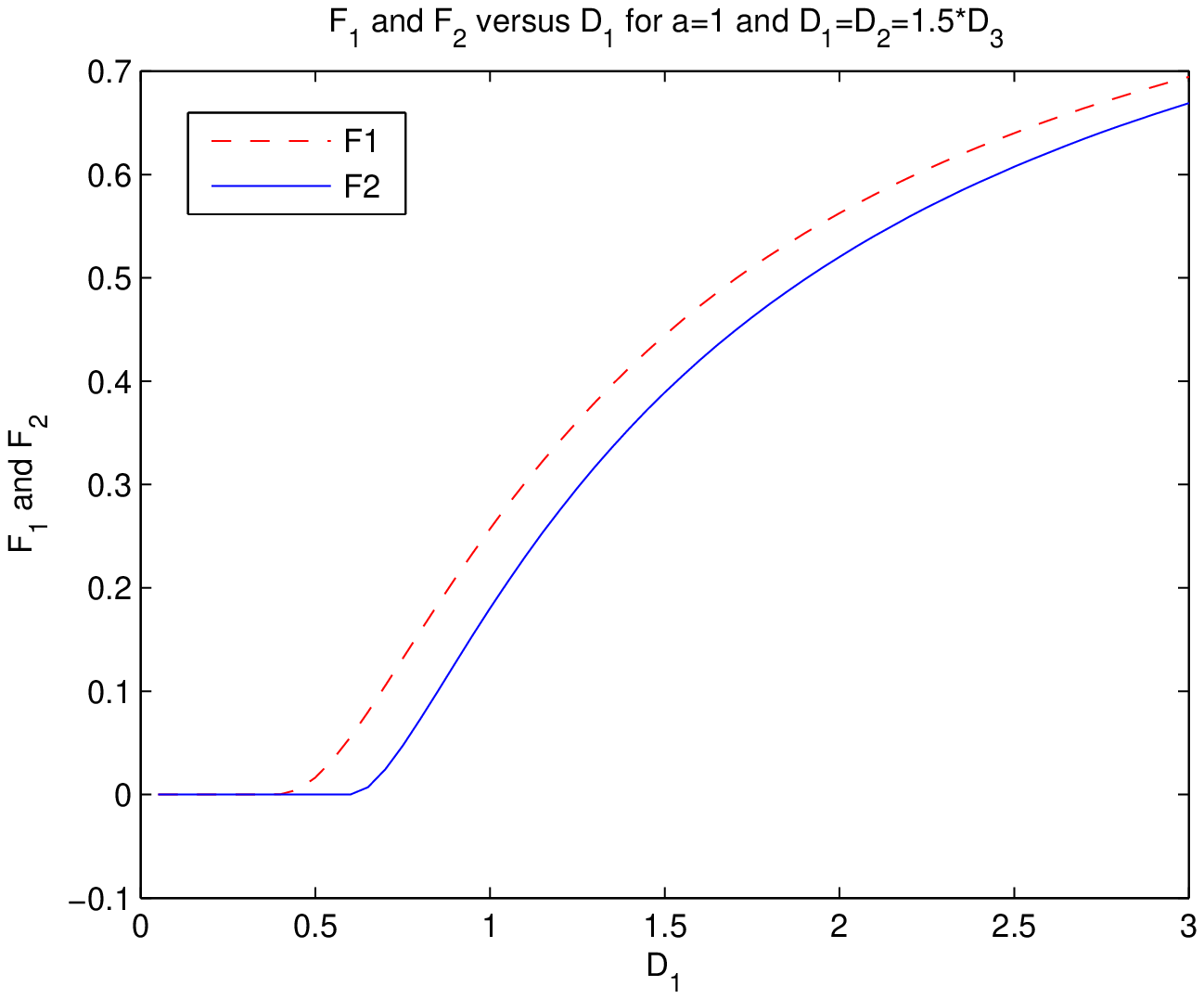}
		}
	\end{center}
\end{figure}

\subsubsection{Explicit solution in dimension three}

Assume that $D_1 \geq D_2 \geq D_3$.  Write
\begin{equation}
	\label{eq:McD_reduced_explicit_n3_1}
	\mathcal{F}_1 :=
	\begin{cases}
		0, & \text{if } D_1+D_2+D_3 \leq a, \\
		\dfrac{(D_1+D_2+D_3-a)^{3}}{27 D_1 D_2 D_3}, & \text{if } D_1+D_2-2D_3 \leq a\leq D_1+D_2+D_3, \\
		\dfrac{(D_1+D_2-a)^{2}}{4 D_1 D_2}, & \text{if } D_1-D_2 \leq a\leq D_1+D_2-2D_3,\\
         1 - \dfrac{a}{D_1}, & \text{if } 0\leq  a \leq D_{1} - D_{2}.
	\end{cases}
\end{equation}
and
\begin{equation}
	\label{eq:McD_reduced_explicit_n3_2}
	\mathcal{F}_2:=\max_{i\in \{1,2,3\}} \phi(\gamma_i)\psi(\gamma_i)
\end{equation}
where
\[
	\psi(\gamma) := \gamma^2 \left( 2 \frac{D_2}{D_3} - 1 \right) - 2 \gamma \left( 3 \frac{D_2}{D_3} - 1 \right) + \frac{\gamma}{1+\gamma} \left( 8 \frac{D_2}{D_3} - 2 \frac{a}{D_3} \right)
\]
and $\gamma_1, \gamma_2, \gamma_3$ are the roots (in $\gamma$) of the  cubic polynomial
\begin{equation}
	\label{eq:magic_cubic}
	(1 + \gamma)^3 - A (1 + \gamma)^2 + B = 0,
\end{equation}
where
\[
	A := \frac{5 D_2 - 2 D_3}{2 D_2 - D_3} \quad \text{and}\quad B := \frac{4 D_2 -a}{2 D_2 - D_3}.
\]
Define a function $\phi$ by
\[
	\phi(\gamma) :=
	\begin{cases}
		1, & \text{if $\gamma \in (0,1)$ and $\theta(\gamma)\in (0,1)$,} \\
		0, & \text{otherwise,}
	\end{cases}
\]
where
\[
	\theta(\gamma) := 1-\frac{a}{D_3 (1 - \gamma^2)} + \frac{D_2}{D_3} \frac{1 - \gamma}{1 + \gamma}.
\]

By the standard formula for the roots of a cubic polynomial, the roots of \eqref{eq:magic_cubic} are given by
\[
	 \gamma_1 := - 1 - \frac{1}{3} \left( - A + \kappa_1 + \kappa_2 \right),
\]
\[
	 \gamma_2 := - 1 - \frac{1}{3} \left( - A + \omega_2 \kappa_1 + \omega_1 \kappa_2 \right),
\]
\[
	 \gamma_3 := - 1 - \frac{1}{3} \left( - A + \omega_1 \kappa_1 + \omega_2 \kappa_2 \right),
\]
where
\[
	\omega_1 := - \frac{1}{2} + \frac{\sqrt{3}}{2} i, \quad \omega_2 := - \frac{1}{2} - \frac{\sqrt{3}}{2} i,\quad \kappa_1 := \left( \frac{\beta_1 + \sqrt{\beta_2}}{2} \right)^\frac{1}{3},
\]
\[
\kappa_2 := \left( \frac{\beta_1 - \sqrt{\beta_2}}{2} \right)^\frac{1}{3},\quad \beta_1 := - 2 A^3 + 27 B \quad \text{and}\quad \beta_2 := \beta_1^2 - 4 A^6.
\]
Since there are 3 possible values for each cube root, $\kappa_1$ and $\kappa_2$ must be taken so that they satisfy $\kappa_1 \kappa_2 = A^2$.

\begin{thm}[Explicit solution for $m = 3$]
	\label{thm:m3}
	For $m = 3$ with $D_1 \geq D_2 \geq D_3$, $\mathcal{U}(\mathcal{A}_{\mathrm{McD}})$ is given by
	\begin{equation}
		\label{eq:McD_reduced_explicit_3d}
		\mathcal{U}(\mathcal{A}_{\mathrm{McD}}) =\max(\mathcal{F}_1, \mathcal{F}_2).
	\end{equation}
\end{thm}

\begin{rmk}
Sub-figure \ref{fig:mcd4} compares the McDiarmid and OUQ bounds for $m = 3$, with zero mean performance,
$D_1=D_2=D_3$, and failure threshold $a = 1$.  Sub-figure \ref{fig:mcd5} shows that
 that $\mathcal{F}_2>\mathcal{F}_1$ for $D_1$ large enough. Sub-figure \ref{fig:mcd6} shows that if
 $D_1=D_2=\frac{3}{2}D_3$, then $\mathcal{F}_2<\mathcal{F}_1$ for all $D_1$. Therefore,
	Sub-figures \ref{fig:mcd5} and \ref{fig:mcd6} suggest that the inequality $\mathcal{F}_2 > \mathcal{F}_1$ holds only if $D_3 \approx D_2$, and $D_2$ is large enough relative to $D_{1}$.
\end{rmk}

\begin{rmk}\label{rmk:deihdehdu}
For the application to the (SPHIR facility) admissible set \eqref{eq:PSAAP_SPHIR_Admissible_McD} (described in Subsection \ref{subsec:sphirex}), the sub-diameters of the surrogate $H$ are: $8.86 \, \mathrm{mm}^2$ for thickness ($D_1$),  $7.20 \, \mathrm{mm}^2$ for velocity ($D_2$), and $4.17 \, \mathrm{mm}^2$ for obliquity ($D_3$). These values have been obtained by solving the optimization problems defined by \eqref{eq:defOsc} with $f=H$ and $i=1,2,3$.
The application of Theorem \ref{thm:m3} with these sub-diameters and $a=5.5\, \mathrm{mm}^{2}$ leads to $\mathcal{F}_2=0.253$ and $\mathcal{F}_1=0.437$ (see \eqref{eq:McD_reduced_explicit_n3_1} and \eqref{eq:McD_reduced_explicit_n3_2} for the definition and interpretation of $\mathcal{F}_1$ and $\mathcal{F}_2$). In particular,  since $D_1-D_2\leq a \leq D_1+D_2-2 D_3$, it follows from  \eqref{eq:McD_reduced_explicit_n3_1} that
the obliquity sub-diameter does not impact $\mathcal{F}_1$ (decreasing $D_3$ down to zero does not change the optimal bound $43.7\%$ obtained from the third line of \eqref{eq:McD_reduced_explicit_n3_1}).
\end{rmk}

\subsubsection{Solution in dimension $m$}

For $C_0 \in \mathcal{C}$, write
\begin{equation}
	\label{eq:sup_over_AC0}
	\mathcal{U}(\mathcal{A}_{C_0}) = \sup_{\alpha\,:\, (C_0,\alpha) \in \mathcal{A}_\mathcal{C}} \alpha[h^{C_0} \geq a],
\end{equation}
where $h^{C_0}$ is defined by equation \eqref{eq:jhdsjdgjhsgdgh}.

\begin{prop}
	\label{prop:McD_reduced_explicit_allm}
	Assume that $D_1\geq \cdots \geq D_{m-1}\geq D_m$. For $C_0:=\{(1,1,\ldots,1,1)\}$, it holds that
	\begin{equation}
		\label{eq:McD_reduced_explicit_allm}
		\mathcal{U}(\mathcal{A}_{C_0}) =
		\begin{cases}
			0, & \text{if } \sum_{j=1}^m D_j \leq a, \\
			\dfrac{(\sum_{j=1}^m D_j-a)^{m}}{m^m \prod_{j=1}^m D_j}, & \text{if } \sum_{j=1}^{m} D_j - mD_m \leq a\leq \sum_{j=1}^{m} D_j, \\
			\dfrac{(\sum_{j=1}^k D_j-a)^{k}}{k^k \prod_{j=1}^k D_j}, & \text{if, for }k\in\{1,\ldots,m-1\},\\& \sum_{j=1}^{k} D_j - kD_k \leq a\leq \sum_{j=1}^{k+1}D_j - (k+1)D_{k+1}.
		\end{cases}
	\end{equation}
\end{prop}

\begin{rmk}
The maximum of \eqref{eq:McD_reduced_explicit_allm} over $D_1,\ldots,D_m$ under the constraints $D_1+\cdots+D_m=D$ and $D_1\geq \cdots\geq D_m$ is achieved at $D_1=D$ and is equal to $1-a/D$.\\
The minimum of \eqref{eq:McD_reduced_explicit_allm} over $D_1,\ldots,D_m$ under the constraints $D_1+\cdots+D_m=D$ and $D_1\geq \cdots\geq D_m$ is achieved on the diagonal $D_1=\cdots=D_m$ and is equal to $(1-a/D)^m$.
\end{rmk}

\begin{prop}
	\label{jhgjhgejeerrsddeedffr}
	Assume that $D_1\geq \cdots \geq D_{m-1}\geq D_m$.  If $a\geq \sum_{j=1}^{m-2} D_j +D_m$, then $\mathcal{U}(\mathcal{A}_{\mathrm{McD}})$ is given by equation \eqref{eq:McD_reduced_explicit_allm}.
\end{prop}

\begin{rmk}
It follows from the previous proposition that, in arbitrary dimension $m$, the tail of $\mathcal{U}(\mathcal{A}_{\mathrm{McD}})$ with respect to $a$ is given by \eqref{eq:McD_reduced_explicit_allm}. Although we do not have an analytic solution  for $m\geq 4$ and $a < \sum_{j=1}^{m-2} D_j +D_m$, a numerical solution can be obtained by solving the finite-dimensional optimization problem \eqref{eq:jkshdjshdjheer} with variables $(C,\alpha)$. Observe that the range of $\alpha$ is $[0,1]^m$. Although the range of $C$ is the set of subsets of $\{0,1\}^m$, we conjecture (based on symmetry and monotonicity arguments)  that the extremum of \eqref{eq:jkshdjshdjheer} can be achieved by restricting $C$ to sets $C_q$ defined by $\{s\in [0,1]^m \mid  \sum_{i=1}^m s_i \geq q$ (with $q\in \{1,\dots,m\}$).
\end{rmk}

\subsection{Explicit solutions under the assumptions of Hoeffding's inequality}\label{subsecHoeffding}

This subsection treats a further special case of OUQ, where the assumptions
 are those of Hoeffding's inequality \cite{Hoeffding:1963}.  Define the admissible set
\begin{equation}
	\label{eq:Hfassumptions}
	\mathcal{A}_{\mathrm{Hfd}} := \left\{ (f, \mu) \,\middle|\,
	\begin{matrix}
		f = X_1 + \dots + X_m, \\
	  \mu \in \bigotimes_{i=1}^m \mathcal{M}([b_i - D_i, b_i]), \\
	 	\mathbb{E}_{\mu}[f] \leq 0
	\end{matrix}
	\right\},
\end{equation}
and define the optimization problem
\[
	\mathcal{U}(\mathcal{A}_{\mathrm{Hfd}}) := \sup_{(f,\mu) \in \mathcal{A}_{\mathrm{Hfd}}} \mu[f \geq a].
\]
By Hoeffding's inequality, for $a \geq 0$,
\[
	\mathcal{U}(\mathcal{A}_{\mathrm{Hfd}}) \leq \exp \left( - 2 \frac{a^{2}}{\sum_{i}^{m}D_i^{2}} \right).
\]

\begin{thm}
	\label{thm:Hfdm2}
	If $m=2$, then
	\begin{equation}
		\label{eq:McD_reduhdfd03}
		\mathcal{U}(\mathcal{A}_{\mathrm{Hfd}})=\mathcal{U}(\mathcal{A}_{\mathrm{McD}}).
	\end{equation}
\end{thm}

\begin{rmk}
Another proof of  Theorem \ref{thm:Hfdm2} can be obtained using entirely different methods than presented in Section \ref{sec-mcdproofs}.  Although  omitted for brevity, this method may be useful in higher dimensions, so we describe an outline of it here. We begin at the reduction obtained through Proposition \ref{prop:mcd} to the hypercube.  Whereas the proof of Theorem \ref{thm:Hfdm2} first applies the reduction of Theorem \ref{thm:C} to subsets of the hypercube, here we instead fix the oscillations in each direction to be $ 0 \leq d_{i}\leq D_{i}$, and  solve the fixed $d:=(d_{1},d_{2})$ case, not using a Langrangian-type analysis but a type of spectral reduction. We then show that the resulting value $\mathcal{U}(d)$ is increasing in $d$ with respect to the standard (lexicographic) partial order on vectors.  The result then easily follows by taking the supremum over all vectors $0 \leq d \leq D$.
\end{rmk}

\begin{thm}
	\label{thm:Hfdm3}
	Let $m=3$, and define $\mathcal{F}_1$ and $\mathcal{F}_2$ as in Theorem \ref{thm:m3}.
	If $\mathcal{F}_1\geq \mathcal{F}_2$, then
	\begin{equation}
		\label{eq:McD_reduhdm3njfd03}
		\mathcal{U}(\mathcal{A}_{\mathrm{Hfd}}) = \mathcal{U}(\mathcal{A}_{\mathrm{McD}}).
	\end{equation}
	If $\mathcal{F}_1< \mathcal{F}_2$, then
	\begin{equation}
		\label{eq:McD_reduhdm3njfd03bis}
		\mathcal{U}(\mathcal{A}_{\mathrm{Hfd}}) < \mathcal{U}(\mathcal{A}_{\mathrm{McD}}).
	\end{equation}
\end{thm}

Under the assumptions of Hoeffding's inequality, each variable $X_i$ is bounded from below and from above.
Without the upper bounds on the variables $X_i$, it is possible to use additional reduction properties and
conjecture an explicit form for the optimal inequality on $\mu[X_1+\cdots+X_m\geq a]$. Here we refer to the work and conjecture of Samuels \cite{Samuels:1966} (see also \cite[p.542]{KarlinStudden:1966}), which has been proven true for $m=1,2,3$.

\begin{rmk}
The optimal Hoeffding inequality can be used for additive models (with response functions of the form $X_1+\cdots +X_m$) but also to obtain optimal probabilities of deviations for empirical means.
Furthermore the fact that the optimal concentration inequalities corresponding to Hoeffding's or McDiarmid's assumptions are the same for $m=2$ and possibly distinct for $m=3$ is a simple but fundamental result analogous to Stein's paradox \cite{Efron:1977}.
\end{rmk}

\section{Computational Implementation}
\label{sec:ComputationalExamples}

In this section, we discuss the numerical implementation of OUQ algorithms for the analytical surrogate model for hypervelocity impact introduced in Subsection \ref{Subsec:motivatingex}.

\subsection{Extreme points of reduced  OUQ problems are attractors}
We consider again the computation of the optimal bound $\mathcal{U}(\mathcal{A}_{H})$ (where $\mathcal{A}_{H}$ is the
 information set  given by Equation \eqref{eq:PSAAP_SPHIR_Admissible}) via the identity \eqref{eq:Hdeltarecduced} derived from  the reduction results of Section \ref{sec:Reduction}.  For  $\# \mathrm{supp}(\mu_{i}) \le 2, \, i = 1, 2, 3$,  Figure \ref{fig:CollapseSupport2} has shown  that numerical simulations collapse to
 two-point support.  Figure \ref{fig:CollapseSupport5} shows that, even when a wider search is performed (i.e., over measures $\mu \in \bigotimes_{i = 1}^{3} \Delta_{k}(\mathcal{X}_{i})$ for $k > 1$), the calculated maximizers for these problems maintain two-point support:  the velocity and obliquity marginals each collapse to a single Dirac mass, and the plate thickness marginal collapses to have support on the two extremes of its range.  As expected, optimization over a larger search space is more computationally intensive and takes longer to perform.  This observation suggests that the extreme points of the reduced  OUQ problems are, in some sense, attractors --- this point will be revisited in the next subsection.

We also refer to Figures \ref{fig:CollapseConverge2} and \ref{fig:CollapseConverge5} for plots of the locations and weights of the Dirac masses forming each marginal $\mu_i$ as  functions of the number of iterations. Note that the lines for \emph{thickness} and \emph{thickness weight} are of the same color if they correspond to the same support point for the measure.
In particular, Figure \ref{fig:CollapseConverge5} shows that at iteration number $3500$ the \emph{thickness} support point at $62.5\, \mathrm{mils}$ (shown in Figure \ref{fig:CollapseSupport5}) has zero weight.

\COMMENT{
	\begin{figure}[tp]
		\subfigure[support points at iteration 0]{
			\includegraphics[width=0.45\textwidth]{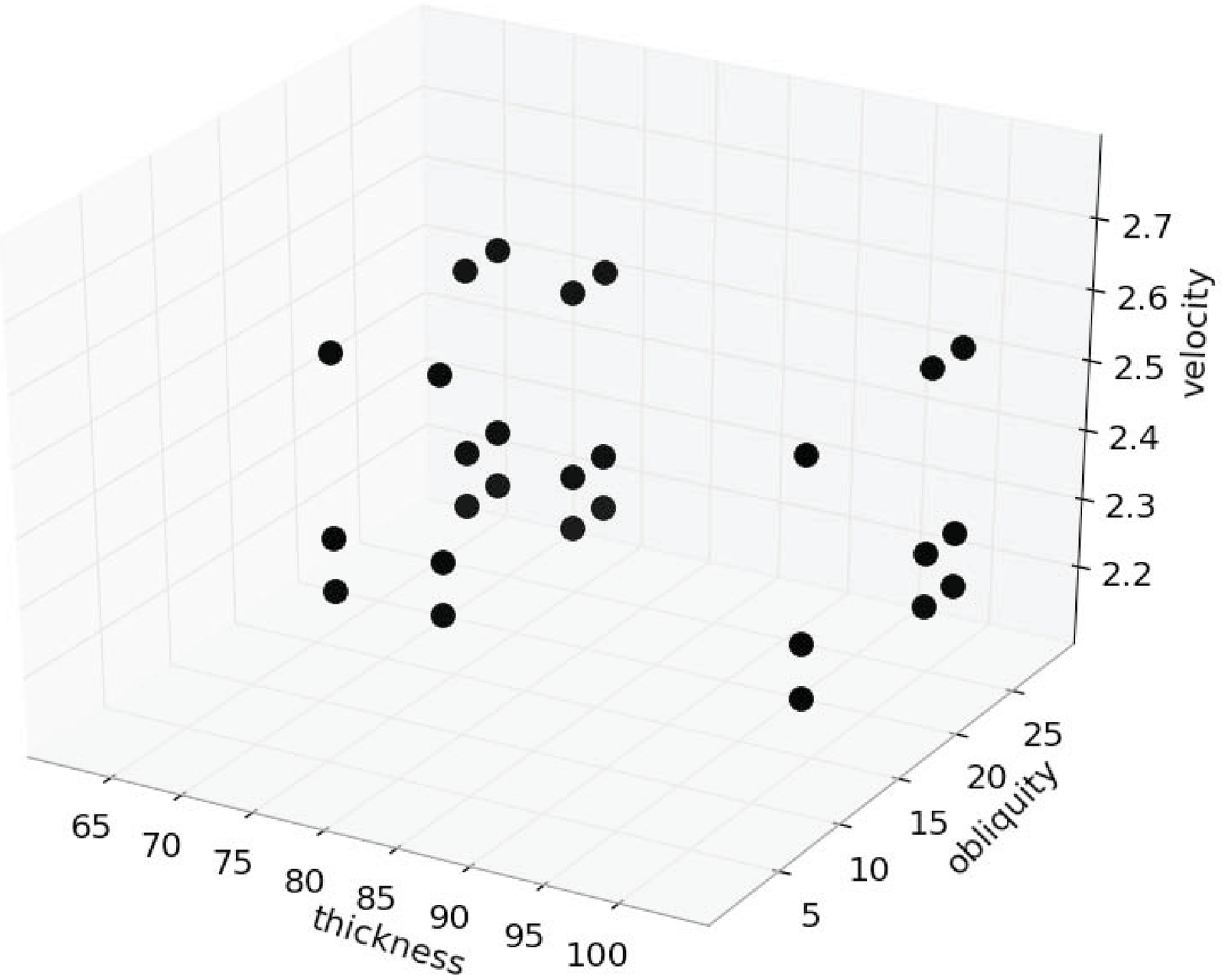}
		}
		\subfigure[support points at iteration 500]{
			\includegraphics[width=0.45\textwidth]{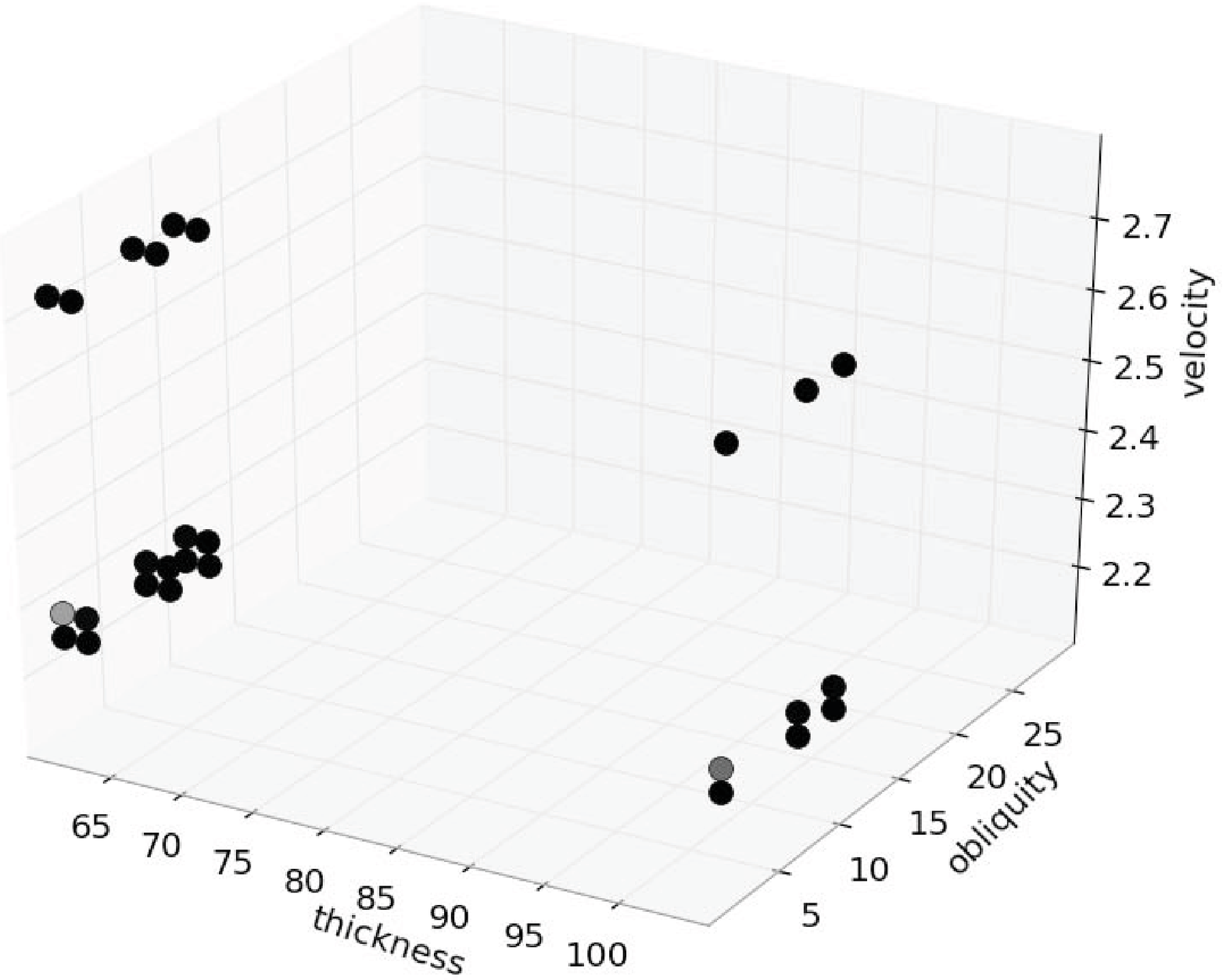}
		}
		\subfigure[support points at iteration 1000]{
			\includegraphics[width=0.45\textwidth]{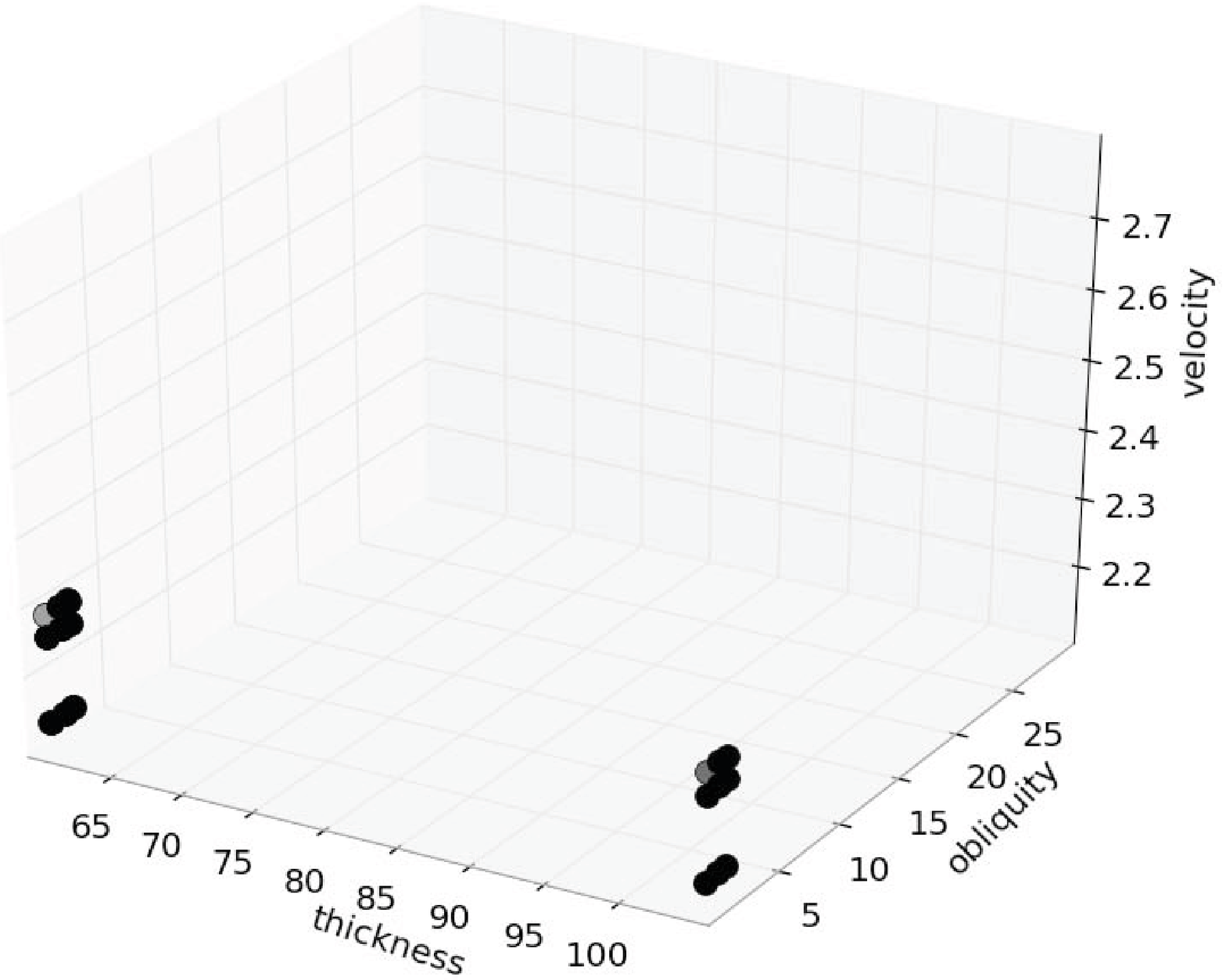}
		}
		\subfigure[support points at iteration 2155]{
			\includegraphics[width=0.45\textwidth]{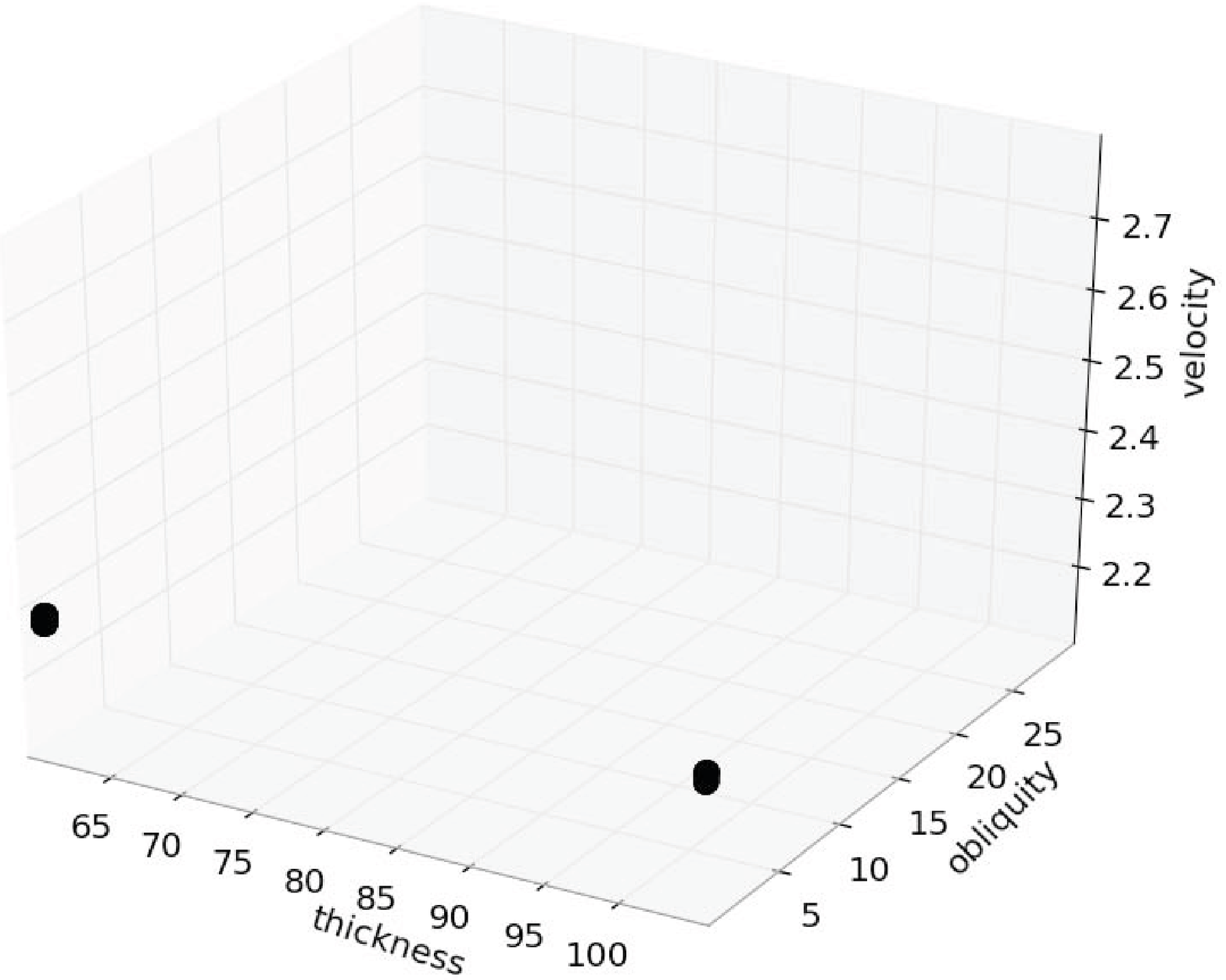}
		}
		\caption{For  $\# \mathrm{supp}(\mu_{i}) \le 3, \, i = 1, 2, 3$,  the maximizers of the OUQ problem \eqref{eq:Hdeltarecduced} associated with the information set
	\eqref{eq:PSAAP_SPHIR_Admissible} collapse to two-point support.
	Velocity, obliquity and plate thickness marginals collapse as in Figure \ref{fig:CollapseSupport2}.}
		\label{fig:CollapseSupport3}
	\end{figure}
}
	
\begin{figure}[tp]
	\subfigure[support points at iteration 0]{
		\includegraphics[width=0.45\textwidth]{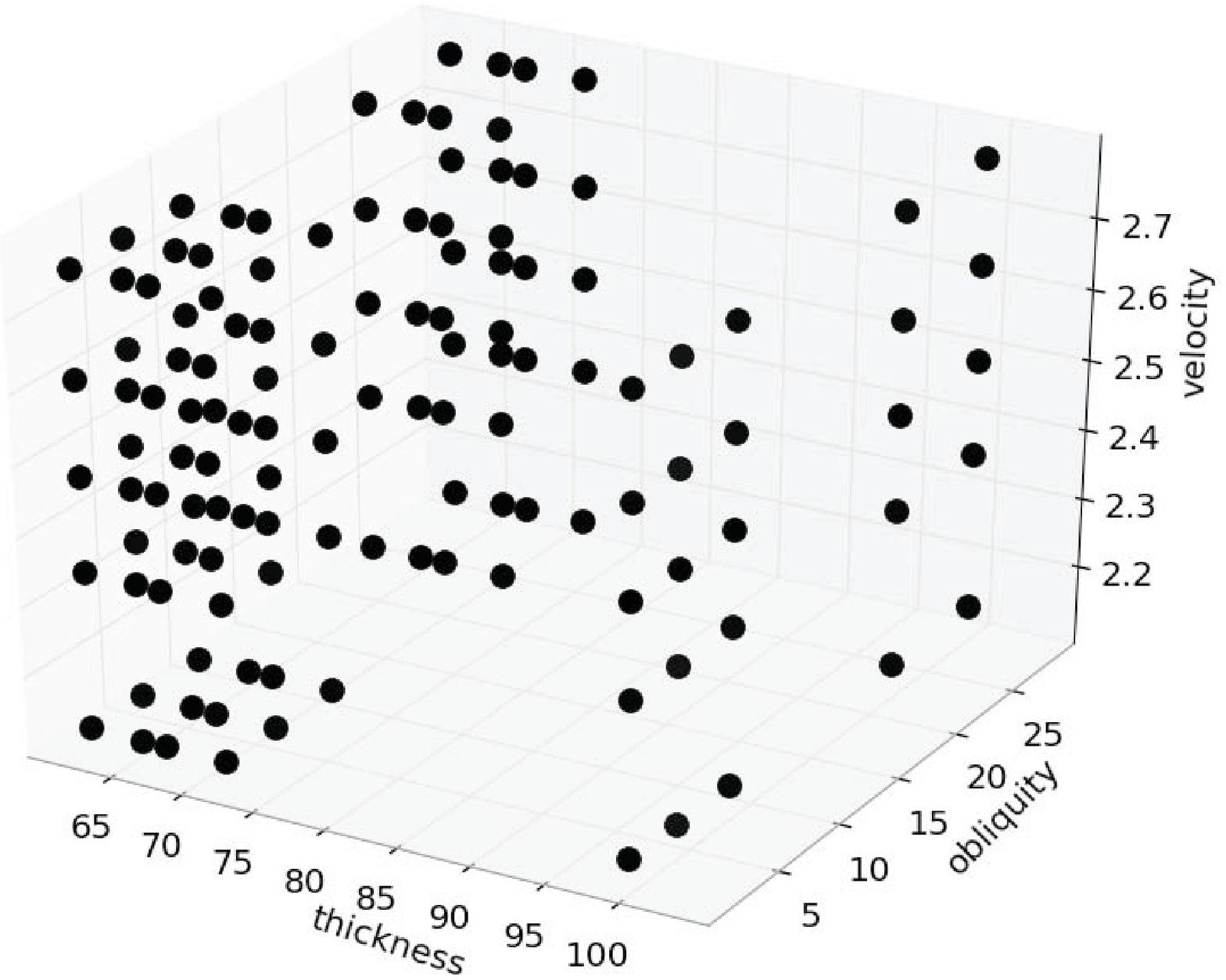}
	}
	\subfigure[support points at iteration 1000]{
		\includegraphics[width=0.45\textwidth]{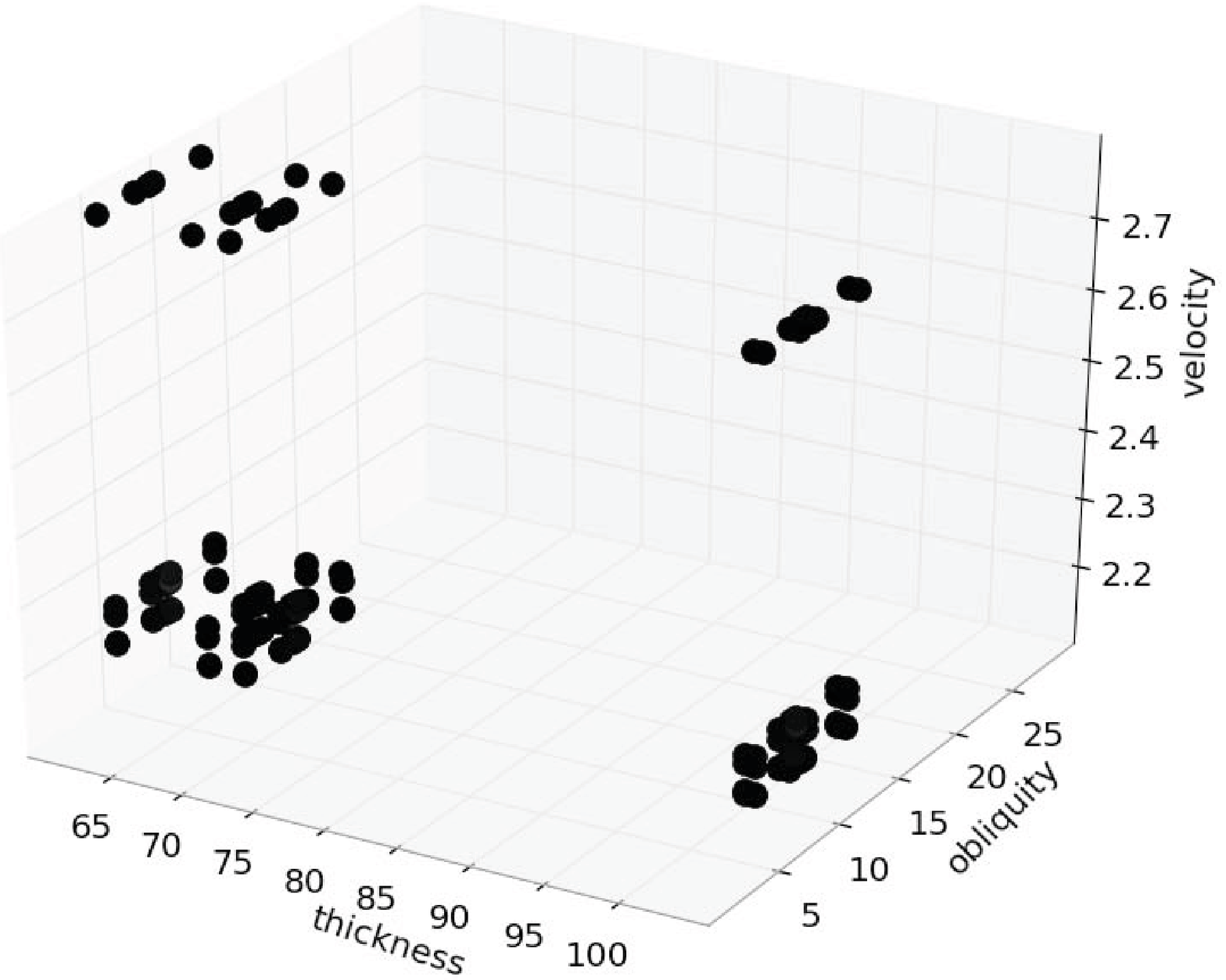}
	}
	\subfigure[support points at iteration 3000]{
		\includegraphics[width=0.45\textwidth]{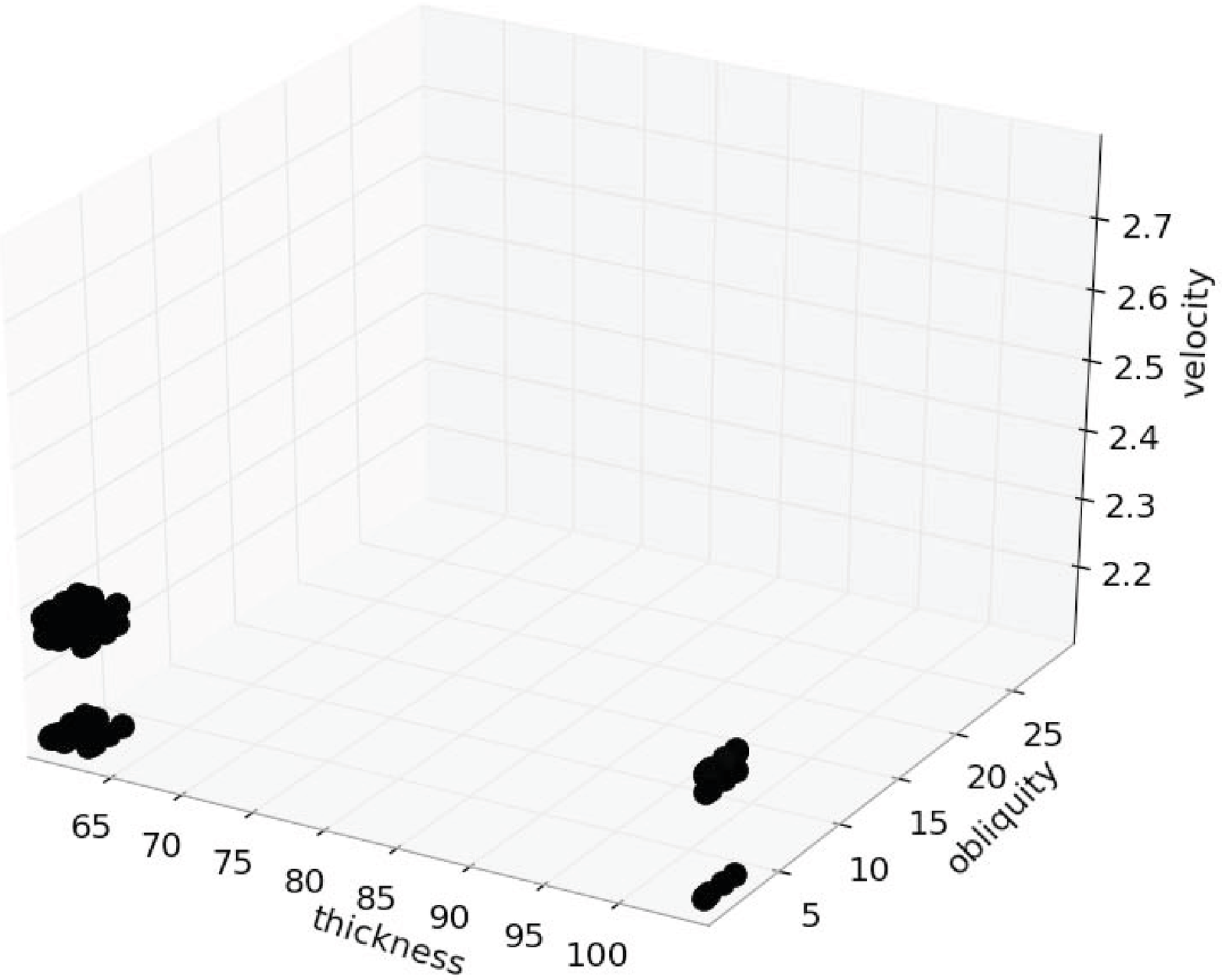}
	}
	\subfigure[support points at iteration 7100]{
		\includegraphics[width=0.45\textwidth]{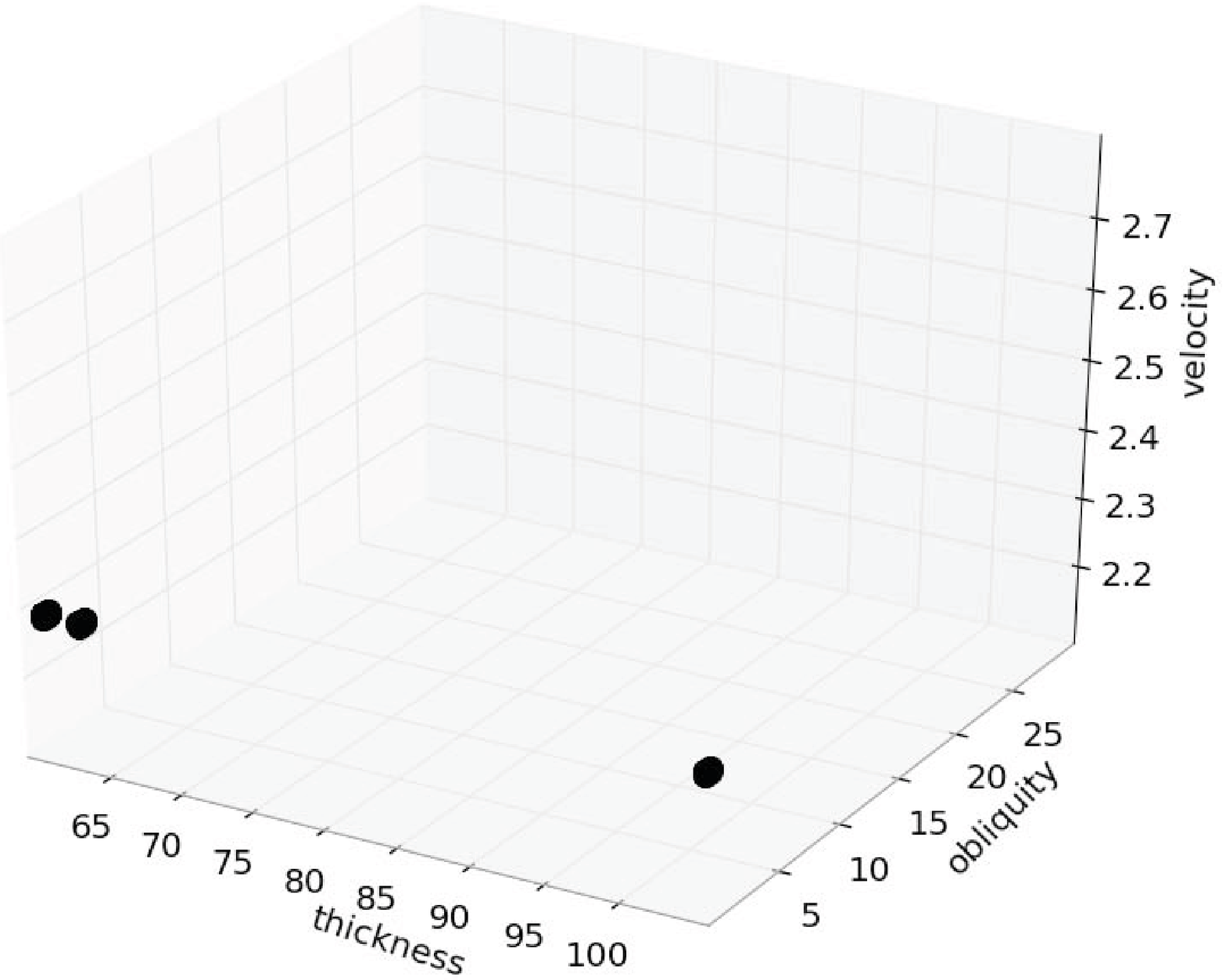}
	}
	\caption{For  $\# \mathrm{supp}(\mu_{i}) \le 5, \, i = 1, 2, 3$,  the maximizers of the OUQ problem \eqref{eq:Hdeltarecduced} associated with the information set
\eqref{eq:PSAAP_SPHIR_Admissible} collapse to two-point support. Velocity, obliquity and plate thickness marginals collapse as in Figure \ref{fig:CollapseSupport2}. At iteration 7100, the thickness support point at $62. 5 \, \mathrm{mils}$ has zero weight, as can be seen in Figure \ref{fig:CollapseConverge5}.}
	\label{fig:CollapseSupport5}
\end{figure}

\begin{figure}[tp]
\begin{center}
	\caption{Time evolution of the genetic algorithm search for the OUQ problem \eqref{eq:Hdeltarecduced} associated with the information set
\eqref{eq:PSAAP_SPHIR_Admissible} for $\# \mathrm{supp}(\mu_{i}) \leq 5$ for $i = 1, 2, 3$, as optimized by \emph{mystic}. Four of the five thickness support points quickly converge to the extremes of its range, with weights 0.024, 0.058, and 0.539 at $60 \, \mathrm{mils}$ and weight 0.379 at $105 \, \mathrm{mils}$. The thickness support point that does not converge to an extremum has zero weight.  Obliquity and velocity each collapse to a single support point, again with the corresponding weights demonstrating fluctuations due to degeneracies. }
	\label{fig:CollapseConverge5}
	\subfigure[convergence for thickness]{
		\includegraphics[width=0.4\textwidth]{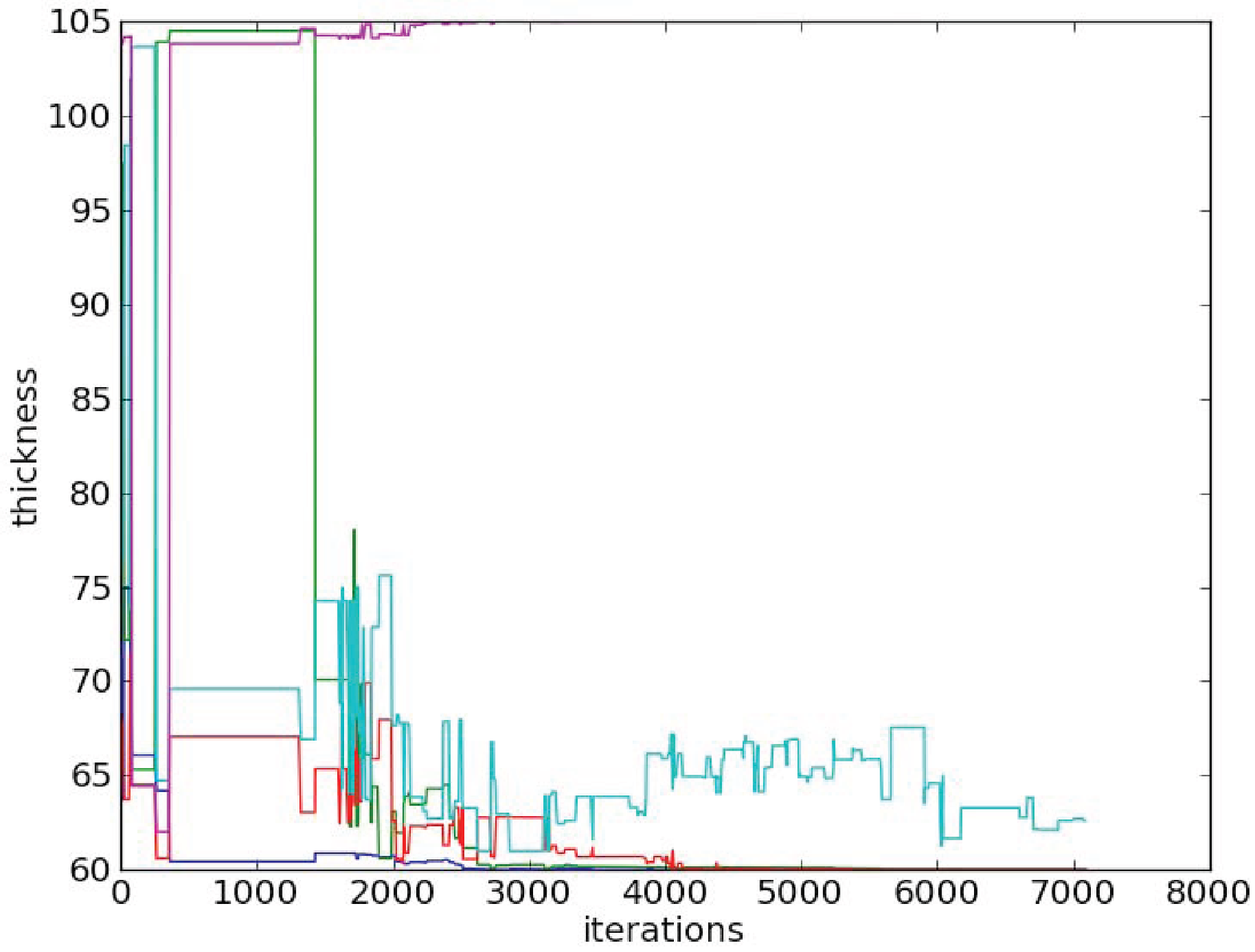}
	}
	\subfigure[convergence for thickness weight]{
		\includegraphics[width=0.4\textwidth]{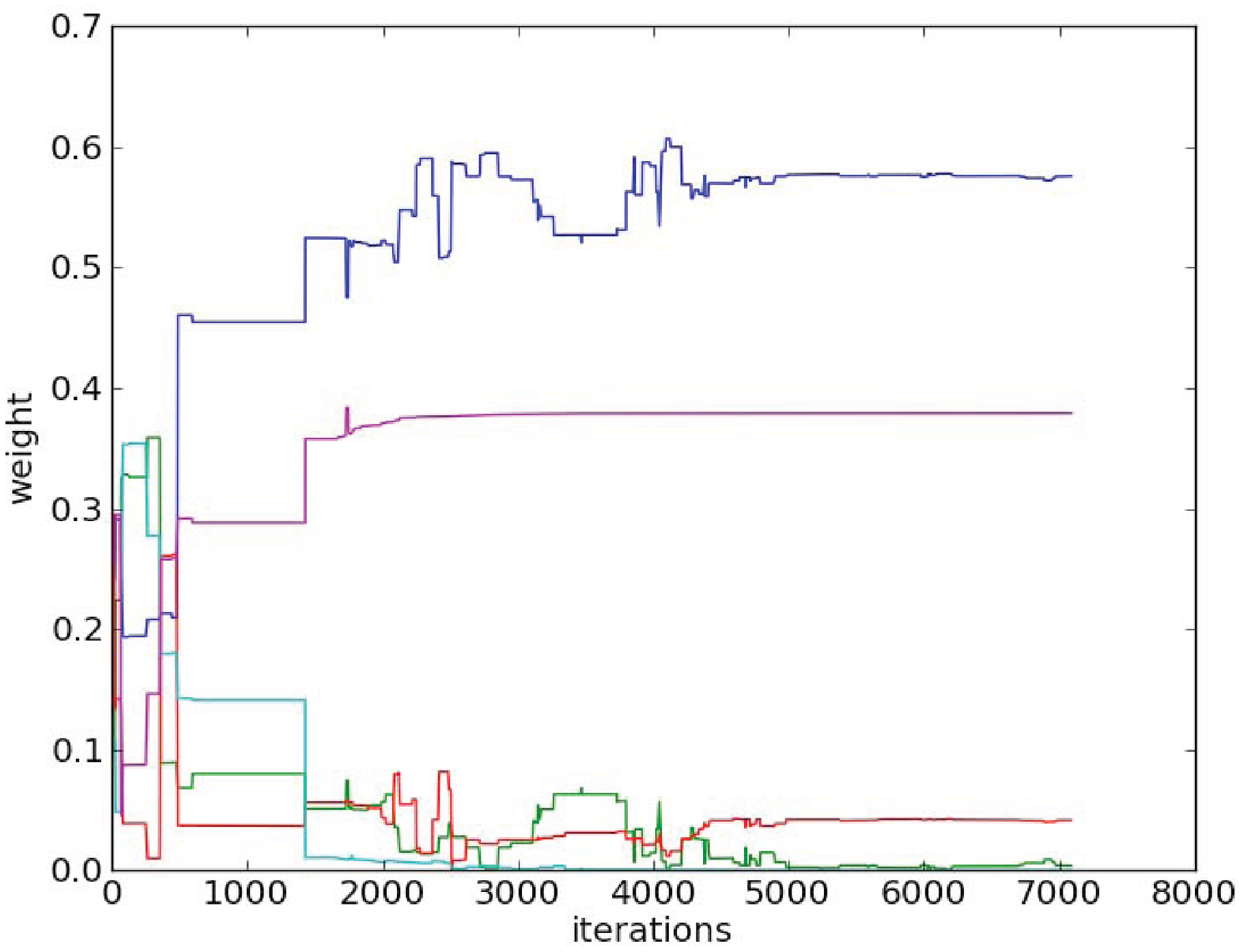}
	}
	\subfigure[convergence for obliquity]{
		\includegraphics[width=0.4\textwidth]{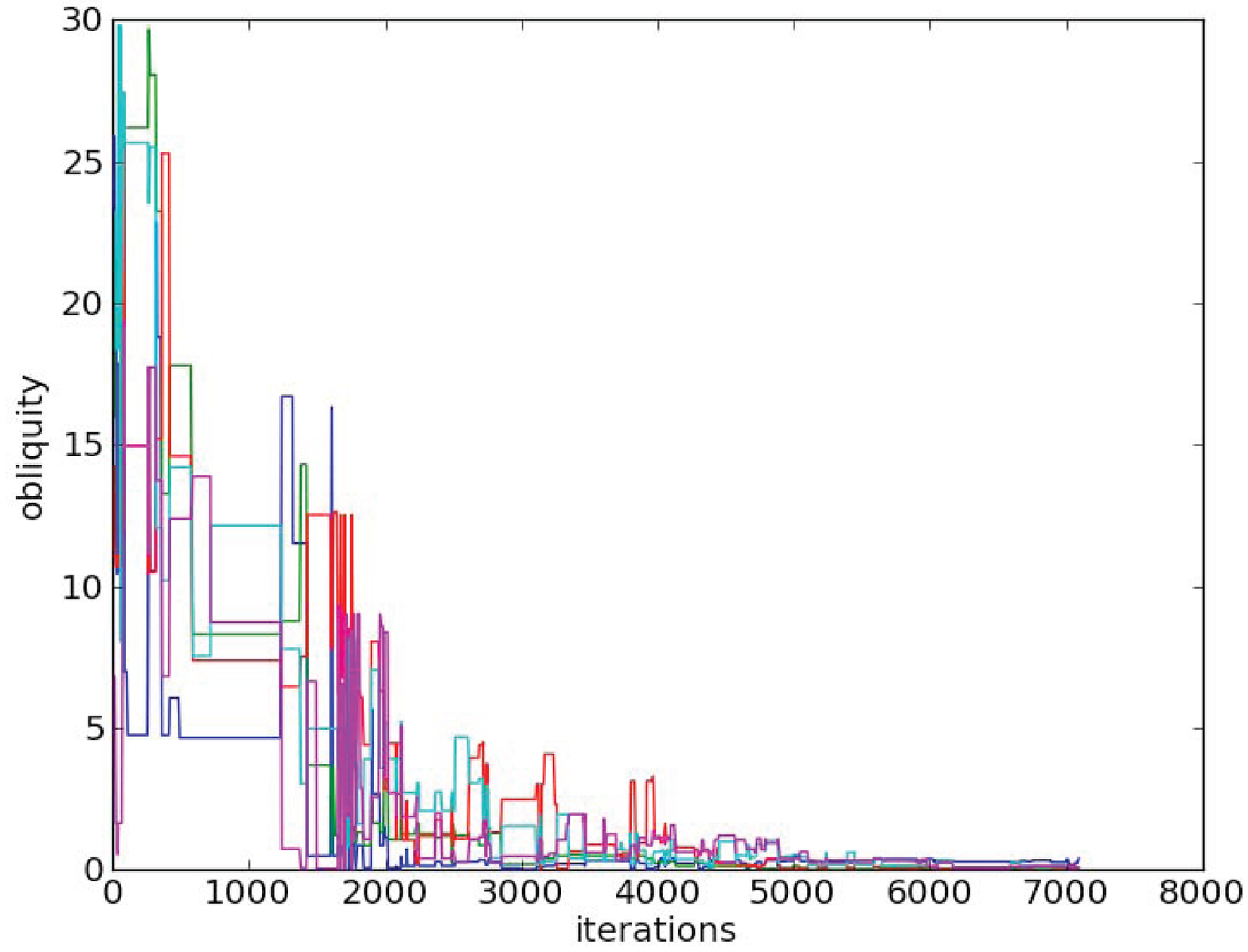}
	}
	\subfigure[convergence for obliquity weight]{
		\includegraphics[width=0.4\textwidth]{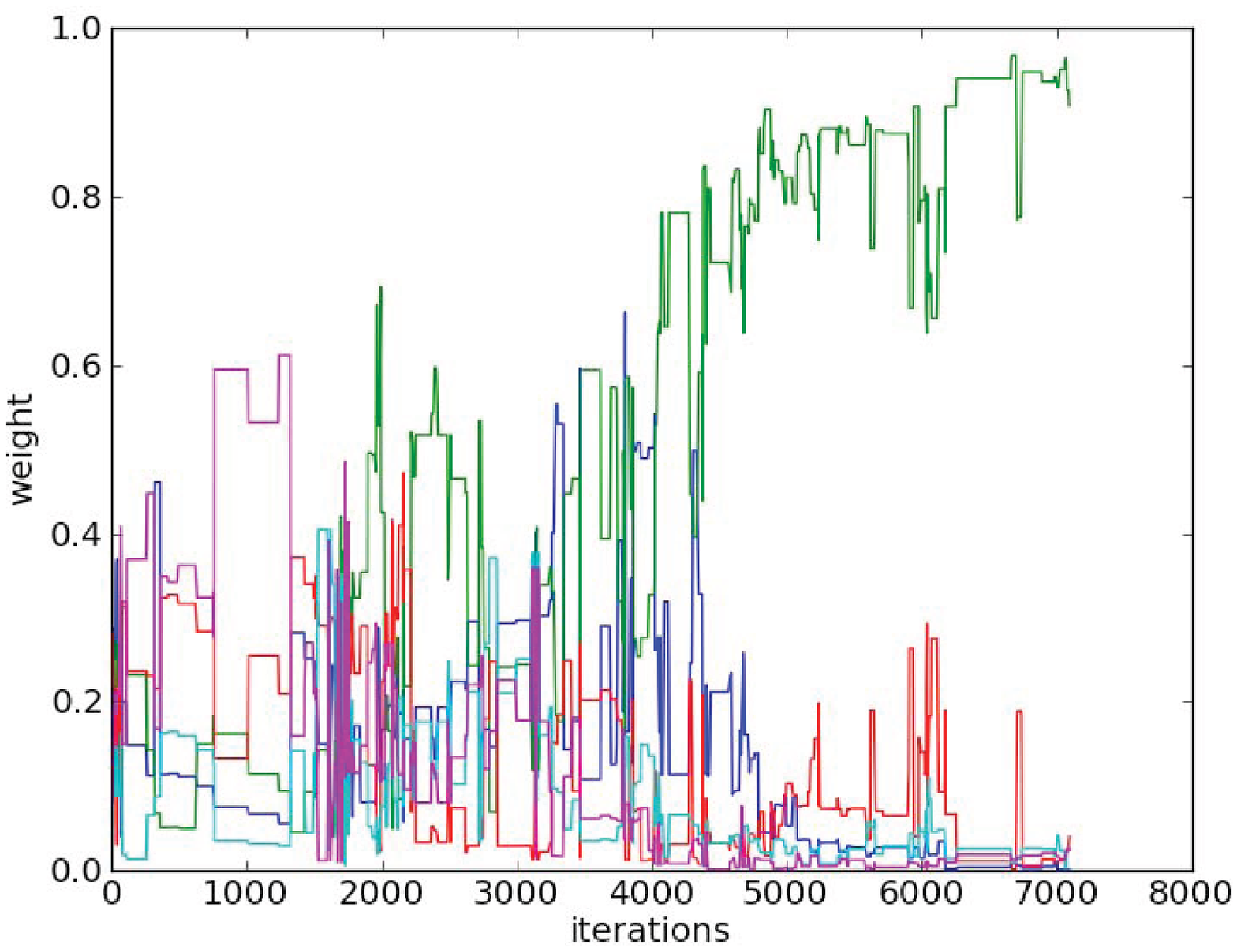}
	}
	\subfigure[convergence for velocity]{
		\includegraphics[width=0.4\textwidth]{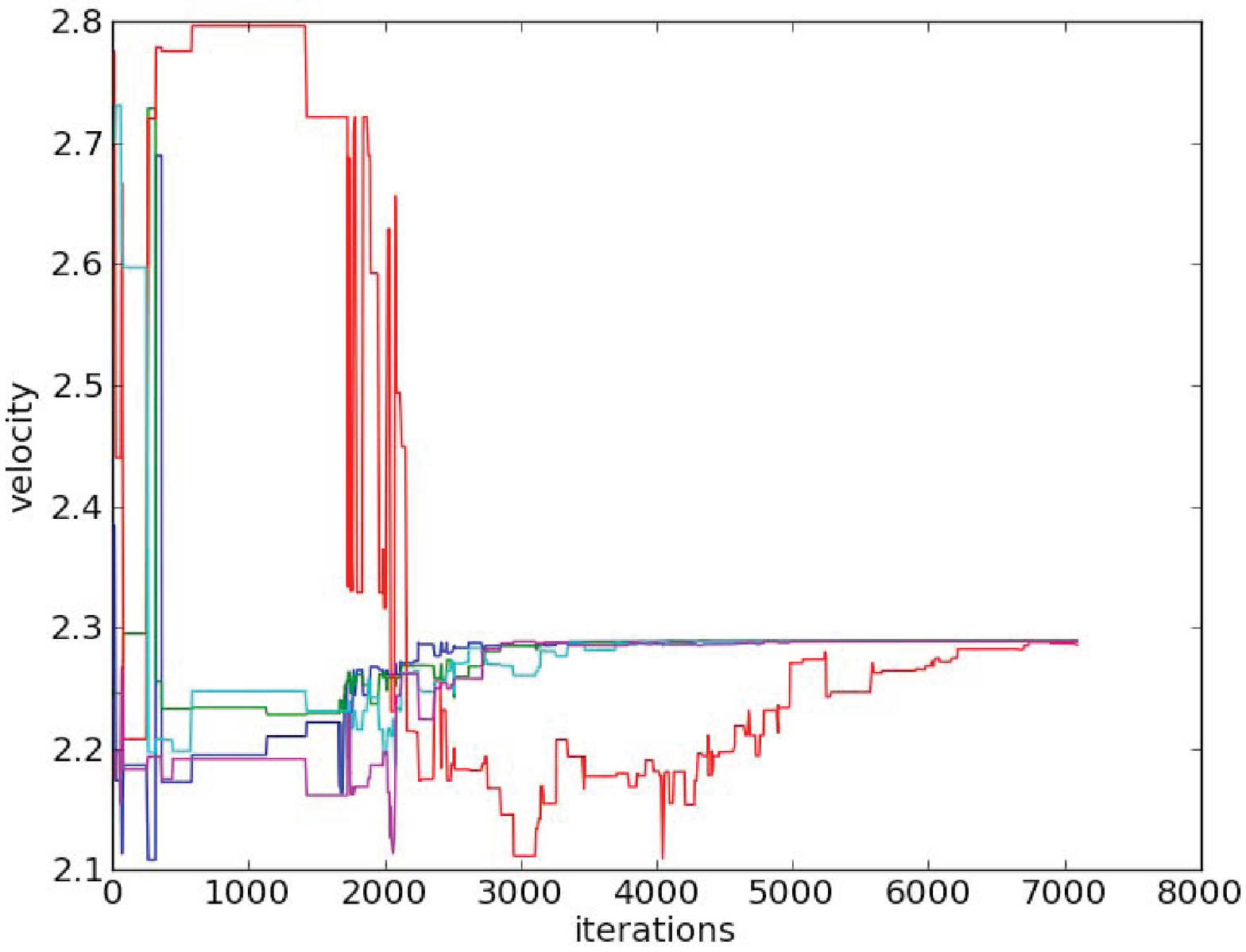}
	}
	\subfigure[convergence for velocity weight]{
		\includegraphics[width=0.4\textwidth]{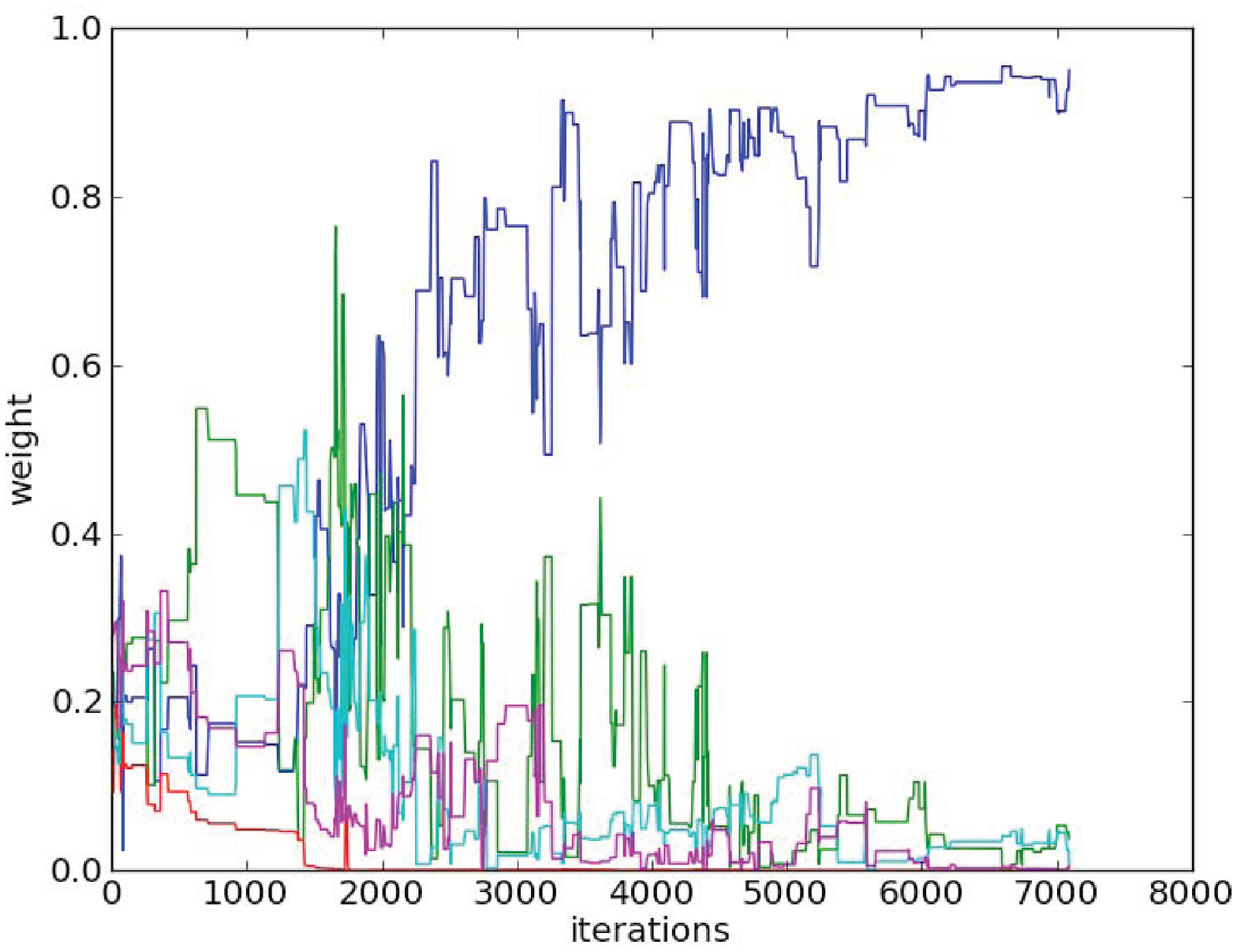}
	}
\end{center}
\end{figure}

\subsection{Coagulation--Fragmentation algorithm for OUQ}

The results of Sections \ref{sec:Reduction} and \ref{sec:Reduction-mcd} give explicit a priori bounds on the number of Dirac masses sufficient to find the lower and upper bounds $\mathcal{L}(\mathcal{A})$ and $\mathcal{U}(\mathcal{A})$ when the admissible set $\mathcal{A}$ is given by independence and linear inequality constraints.  However, it is  possible that reduction properties are present for more general admissible sets $\mathcal{A}$.  Can such ``hidden'' reduction properties be detected by computational means, even in the absence of theorems that prove their existence?

Consider again the results of the previous subsection.  Theorem \ref{thm:baby_measure} provides an a priori guarantee that, to find $\mathcal{U}(\mathcal{A})$, it is sufficient to search the reduced feasible set $\mathcal{A}_{\Delta}$, which consists of those $\mu \in \mathcal{A}$ whose marginal distributions each have support on at most two points.  However, Figure \ref{fig:CollapseSupport5} provides numerical evidence that something much stronger is true:  even if we search among measures $\mu \in \bigotimes_{i = 1}^{3} \Delta_{k}(\mathcal{X}_{i})$ for $k \geq 1$, the measures collapse to an optimizer
\[
	\mu^{\ast} \in \Delta_{1}(\mathcal{X}_{1}) \otimes \Delta_{0}(\mathcal{X}_{2}) \otimes \Delta_{0}(\mathcal{X}_{3})
\]
(that is, two-point support on the thickness axis, and one-point support on the obliquity and velocity axes).  In some sense, the (small support) optimizers are attractors for the optimization process even when the optimization routine is allowed to search over measures
with larger support than that asserted by Theorem \ref{thm:baby_measure}.

Therefore, we propose the following general algorithm for the detection of \emph{hidden} reduction properties.  Let an admissible set $\mathcal{A}$ be given; for $k \in \mathbb{N}$, let
\[
	\mathcal{A}_{k} := \{ (f, \mu) \in \mathcal{A} \mid \mu \in \Delta_{k}(\mathcal{X}) \}
\]
be the collection of admissible scenarios such that $\mu$ has support on at most $k + 1$ points of $\mathcal{X}$.
\begin{compactenum}
	\item Fix any initial value of $k \in \mathbb{N}$.
	\item Numerically calculate $\mathcal{U}(\mathcal{A}_{k})$ and obtain a numerical (approximate) maximizer $\mu^{\ast} \in \mathcal{A}_{k}$.
	\item Calculate $\# \mathop{\mathrm{supp}}(\mu^{\ast})$ and proceed as follows:
	\begin{compactitem}
		\item If $\# \mathop{\mathrm{supp}}(\mu^{\ast}) < k + 1$, then the measure has \emph{coagulated} to have less-than-maximally-sized support and we terminate the algorithm.
		\item If $\# \mathop{\mathrm{supp}}(\mu^{\ast}) = k + 1$, then no coagulation/reduction has yet been observed.  We enter a \emph{fragmentation} phase:  replace $k$ by any $k' > k$ and return to step 2.
	\end{compactitem}
\end{compactenum}

\begin{rmk}
	It would be more accurate to say that the above algorithm is a sketch of an algorithm, and that its details should be adjusted to fit the circumstances of application.  For example, if the admissible set $\mathcal{A}$ includes an independence constraint, then it would be appropriate to base decisions upon the cardinality of the support of the \emph{marginal} distributions of $\mu^{\ast}$, not on the cardinality of the support of $\mu^{\ast}$ itself.
The termination of the algorithm if $\# \mathop{\mathrm{supp}}(\mu^{\ast}) < k + 1$ is motivated by supposition that a \emph{hidden} reduction property has been found and that $\mathcal{U}(\mathcal{A})$ has an (approximate) optimizer in $\mathcal{A}_{k}$.
\end{rmk}

\begin{rmk}
	We reiterate the point made in Remark \ref{rmk:other_extreme_measures} that these methods apply to more general situations than finite convex combinations of Dirac measures;  finite convex combinations of Dirac measures are simply a well-known class of geometrically extreme probability measures (with respect to which numerical integration happens to be very easy), and can be replaced by the extremal points of any class of probability measures as required by the situation of study.  For example, if the OUQ problem of interest involved the invariant measures for some measurable transformation $T \colon \mathcal{X} \to \mathcal{X}$, then each occurence of $\Delta_{k}(\mathcal{X})$ above would be replaced by
	\[
		\mathcal{E}_{k}^{T}(\mathcal{X}) := \left\{ \sum_{j = 0}^{k} \alpha_{j} \mu_{j} \,\middle|\, \begin{matrix} \text{for each } j = 0, \dots, k, \alpha_{j} \geq 0, \\ \text{$\mu_{j} \in \mathcal{M}(\mathcal{X})$ is ergodic with respect to $T$,} \\ \text{ and } \sum_{j = 0}^{k} \alpha_{j} = 1 \end{matrix} \right\}.
	\]
\end{rmk}

\subsection{The OUQ algorithm in the \emph{mystic} framework}\label{sec:Mystic}

As posed above, OUQ at the high level is a global optimization of a cost function that satisfies a set of constraints.
This optimization is performed in \emph{mystic} using the differential evolution algorithm of Price \& Storn \cite{PriceStornLampinen:2005, StornPrice:1997}, with constraints satisfied through a modified Lagrange multiplier method \cite{McKernsOwhadiScovelSullivanOrtiz:2010}.

The \emph{mystic} optimization framework \cite{McKernsHungAivazis:2009} provides a collection of optimization algorithms and tools that lowers the barrier to solving complex optimization problems. Specifically, \emph{mystic} provides flexibility in specifying the optimization algorithm, constraints, and termination conditions.
For example, \emph{mystic} classifies constraints as either ``bounds constraints'' (linear inequality constraints that involve precisely one input variable) or ``non-bounds constraints'' (constraints between two or more parameters),
where either class of constraint modifies the cost function accordingly in attempt to maximize algorithm accuracy and efficiency.
Every \emph{mystic} optimizer provides the ability to apply bounds constraints generically and directly to the cost function, so that the difference in the speed of bounds-constrained optimization and unconstrained optimization is minimized.  \emph{Mystic} also enables the user to impose an arbitrary parameter constraint function on the input of the cost function, allowing non-bounds constraints to be generically applied in any optimization problem.

The \emph{mystic} framework was extended for the OUQ algorithm.
A modified Lagrange multiplier method was added to \emph{mystic}, where an internal optimization is used to satisfy the constraints at each iteration over the cost function \cite{McKernsOwhadiScovelSullivanOrtiz:2010}.
Since evaluation of the cost function is commonly the most expensive part of the optimization, our implementation of OUQ in \emph{mystic} attempts to minimize the number of cost function evaluations required to find an acceptable solution.
By satisfying the constraints within some tolerance at each iteration, our OUQ algorithm will (likely) numerically converge much more quickly than if we were to apply constraints by invalidating generated results (i.e.\ filtering) at each iteration.
In this way, we can use \emph{mystic} to efficiently solve for rare events, because the set of input variables produced by the optimizer at each iteration will also be an admissible point in problem space --- this feature is critical in solving OUQ problems, as tens of thousands of function evaluations may be required to produce a solution.
We refer to \cite{McKernsOwhadiScovelSullivanOrtiz:2010} for a detailed description of the implementation of the OUQ algorithm in the \emph{mystic} framework (we also refer to \cite{McKernsStrand:2011}).

\begin{rmk}\label{rmk:nestedopt}
Our implementation of the OUQ algorithm in \emph{mystic} utilizes a nested optimization (an inner loop) to solve an arbitrary set of parameter constraints at each evaluation of the cost function. We use evolutionary algorithms because they are robust and especially suited to the inner loop (i.e., at making sure that the constraints are satisfied, local methods and even some global method are usually not good enough for this). We also note that the outer loop can be relaxed to other methods (leading to a reduction in the total number of function evaluations by an order of magnitude).
Finally, although we observe approximate extremizers that are
``computationally'' distinct (Figure \ref{fig:CollapseConverge5} shows that mass is traded wildly between practically coincident points),
we have not observed yet ``mathematically'' distinct extrema.
\end{rmk}

\paragraph{Measures as data objects.} Theorem \ref{thm:baby_measure} states that a solution to an OUQ problem, with linear constraints on marginal distributions, can be expressed in terms of products of convex linear combinations of Dirac masses.
In our OUQ algorithm, the optimizer's parameter generator produces new parameters each iteration, and hence produces new product measures to be evaluated within the cost function. For instance, the response function $H$, as defined by $H(h,\theta,v)$ in \eqref{eq:PSAAP_SPHIR_surr}, requires a product measure of dimension $n=3$ for support. In Example \eqref{eq:PSAAP_SPHIR_Admissible}, the mean perforation area is limited to $[m_{1},m_{2}] = [5.5,7.5] \, \mathrm{mm}^{2}$, the parameters $h,\theta,v$ are bounded by the range provided by \eqref{eq:PSAAP_SPHIR_range}, and products of convex combinations of Dirac masses are used as the basis for support. The corresponding OUQ code parameterizes the Dirac masses by their weights and positions.

More generally, it is worth noting that our
computational implementation of OUQ is expressed in terms of methods that act on a hierarchy of parameterized measure data objects. Information is thus passed between the different elements of the OUQ algorithm code as a list of parameters (as required by the optimizer) or as a parameterized measure object. \emph{Mystic} includes methods to automate the conversion of measure objects to parameter lists and vice versa, hence the bulk of the OUQ algorithm code (i.e.\ an optimization on a product measure) is independent of the selection of basis of the product measure itself.
In particular, since the measure data objects can be decoupled from the rest of the algorithm, the product measure representation can be chosen to best provide support for the model, whether
it be a convex combination of Dirac masses as required by Example \eqref{eq:PSAAP_SPHIR_Admissible},
or measures composed of another basis such as Gaussians.
More precisely, this framework can naturally be extended to Gaussians merely by adding covariance matrices as
data object variables and by estimating integral moments equations (with a Monte Carlo method for instance).

\section{Application to the Seismic Safety Assessment of Structures}
\label{Sec:seismic}

In this section, we assess the feasibility of the OUQ formalism  by means of an application to the safety assessment of truss structures subjected to ground motion excitation. This application contains many of the features that both motivate and challenge UQ, including imperfect knowledge of random inputs of high dimensionality, a time-dependent and complex response of the system, and the need to make high-consequence decisions pertaining to the safety of the system.  The main objective of the analysis is to assess the safety of a structure knowing the maximum magnitude and focal distance of the earthquakes that it may be subjected to, with limited information and as few
assumptions as possible.

\subsection{Formulation in the time domain}

\subsubsection{Formulation of the problem}
\label{subsec:formprobseismic}

For definiteness, we specifically consider truss structures undergoing a purely elastic response, whereupon the vibrations of the structure are governed by the structural dynamics equation
\begin{equation}\label{eq:Seismo:EoM}
    M\ddot{u}(t) + C\dot{u}(t) + K u(t) = f(t),
\end{equation}
where $u(t)\in\mathbb{R}^N$ collects the displacements of the joints, $M$ is the mass matrix, $C$ is the damping matrix, $K$ is the stiffness matrix and $f(t)\in\mathbb{R}^N$ are externally applied forces, such as dead-weight loads, wind loads and others. The matrices $M$, $C$ and $K$ are of dimension $N\times N$, symmetric and strictly positive definite.  Let $T$ be an $N\times 3$ matrix such that: $T_{ij} = 1$ if the $i$th degree-of-freedom is a displacement in the $j$th coordinate direction; and $T_{ij} = 0$ otherwise. In addition, let $u_0(t) \in \mathbb{R}^3$ be a ground motion. Then, $T u_0(t)$ represents the motion obtained by translating the entire structures rigidly according to the ground motion.
We now introduce the representation
\begin{equation}
	\label{eq:Seismo:U}
  u(t) = T u_0(t) + v(t),
\end{equation}
where $v(t)$ now describes the vibrations of the structure relative to its translated position. Inserting \eqref{eq:Seismo:U} into \eqref{eq:Seismo:EoM} and using $KT  = 0$ and $CT  = 0$ (implied by translation invariance), we obtain
\begin{equation}
	\label{eq:Seismo:EoM2}
	M \ddot{v}(t) + C \dot{v}(t) + K v(t) = f(t) - M T \ddot{u}_0(t),
\end{equation}
where $- M T \ddot{u}_0(t)$ may be regarded as the effective forces induced in the structure by the ground motion (we start from rest).
We shall assume that the structure is required to remain in the elastic domain for purposes of certification. Suppose that the structure has $J$ members and that all the external loads are applied to the joints of the structure. Let $L$ be a $J\times N$ matrix such that the entries of the vector $Lv$ give the axial strains of the members. The certification condition is, therefore,
\begin{equation}\label{eq:Seismo:CC}
    \|L_i v\|_\infty < S_i, \quad i=1,\dots,J,
\end{equation}
where $S_i$ is the yield strain of the $i$th member and  $\| f \|_\infty := \mathop{\text{ess\,sup}} |f|$
is the $L^\infty$-norm of a function $f \colon \mathbb{R} \to \mathbb{R}$.
In what follows we will write
 \begin{equation}\label{eq:Seismo:Y}
    Y_i = L_i v \quad i=1,\dots,J,
\end{equation}
for the member strains. Due to the linearity of the structure, a general solution of \eqref{eq:Seismo:EoM2} may be formally obtained by means of a modal analysis. Thus, let $q_\alpha\in\mathbb{R}^N$ and $\omega_\alpha > 0$, $\alpha=1,\dots,N$, be the eigenvectors and eigenfrequencies corresponding to the symmetric eigenvalue problem  $(K - \omega_\alpha^2 M) q_\alpha = 0$,
normalized by $q_\alpha^T M q_\alpha = 1$.
Let
\begin{equation}\label{eq:Seismo:VEig}
    v(t) = \sum_{\alpha=1}^N v_\alpha(t) q_\alpha
\end{equation}
be the modal decomposition of $v(t)$. Using this representation, the equation of motion \eqref{eq:Seismo:EoM2} decomposes into the modal equations
\begin{equation}\label{eq:Seismo:EoMEig}
    \ddot{v}_\alpha(t)
    +
    2 \zeta_\alpha \omega_\alpha \dot{v}_\alpha(t)
    +
    \omega_\alpha^2 v_\alpha(t)
    =
    q_\alpha^T\big(f(t) - M T \ddot{u}_0(t)),
\end{equation}
where we have assumed that the eigenmodes $q_\alpha$ are also eigenvectors of $C$ and $\zeta_\alpha$ is the damping ratio for mode $i$.
 The solution of \eqref{eq:Seismo:EoMEig} is given by the hereditary integral
\begin{equation}\label{eq:Seismo:HI}
    v_\alpha(t)
    =
    - \int_0^t
    {\rm e}^{-\zeta_\alpha\omega_\alpha(t-\tau)}
    \sin[\omega_\alpha(t-\tau)]
    \big( q_\alpha^TM T \ddot{u}_0(\tau) \big)
    \, \frac{\mathrm{d} \tau}{\omega_\alpha},
\end{equation}
where, for simplicity, we set $f=0$ and assume that the structure starts from rest and without deformation at time $t=0$.
We can now regard structures oscillating under the action of a ground motion as systems that take the ground motion acceleration $\ddot{u}_0(t)$ as input and whose performance measures of interest are the member strains $Y_i$, $i=1,\dots,J$. The response function $F$ mapping the former to the latter is given by composing \eqref{eq:Seismo:HI}, \eqref{eq:Seismo:VEig} and \eqref{eq:Seismo:Y}.

\subsubsection{Formulation of the information set}
\label{subsec:firstform}

In order to properly define the certification problem we proceed to define constraints on the inputs, i.e.\ the information set associated with the ground motion acceleration. As in \cite{SteinWysession:1991}, we regard the ground motion at the site of the structure as a combination of two factors: the earthquake source $s$ and the earth structure  through which the seismic waves propagate;  this structure is characterized by
a transfer function $\psi$.  Let $\star$ denote the convolution operator;  the ground motion acceleration is then given by
\begin{equation}
	\label{eq:Sesmo:convolution}
	\ddot{u}_0(t) := (\psi \star s)(t).
\end{equation}

We assume that $s$ is a sum of boxcar time impulses (see \cite{SteinWysession:1991} page 230) whose amplitudes and durations  are random, independent, not identically distributed and of unknown distribution. More precisely, we assume that
\begin{equation}
	\label{eq:Sesmo:source}
	s(t) := \sum_{i=1}^{B} X_i \,s_i(t),
\end{equation}
where $X_1,\ldots,X_{B}$ are independent (not necessarily identically distributed) random variables with unknown distribution with support in $[-a_{\max},a_{\max}]^3$ ($s_i$, $B$ and $a_{\max}$ are defined below)
 and such that $\E[X_i]=0$. We also assume the components $(X_{i,1}, X_{i,2}, X_{i,3})$ of the vectors $X_i$ to be independent.
 Since we wish to bound the probability that a structure will fail when it is struck by an earthquake of magnitude $M_{\mathrm{L}}$ in the Richter (local magnitude) scale and hypocentral distance $R$, we adopt the semi-empirical expression proposed by Esteva \cite{Esteva:1970} (see also \cite{NewmarkRosenblueth:1971}) for the maximum ground acceleration
\begin{equation}\label{eq:maxgroundacc}
    a_{\rm max}
    :=
    \frac{a_0 {\rm e}^{\lambda M_{\mathrm{L}}}}{(R_0 + R)^2},
\end{equation}
where $a_0$, $\lambda$ and $R_0$ are constants. For earthquakes on firm ground, Esteva \cite{Esteva:1970} gives $a_0 = 12.3\cdot 10^6 \, \text{m}^3 \cdot \text{s}^{-2}$, $\lambda = 0.8$ and $R_0 = 25\cdot 10^3 \, \text{m}$.

The functions $s_{i}$ are step functions, with $s_{i}(t)$ equal to
 one for $\sum_{j=1}^{i-1} \tau_j \leq t < \sum_{j=1}^{i} \tau_j$ and equal to zero elsewhere, where the durations
$\tau_1,\ldots,\tau_{B}$ are independent (not necessarily identically distributed) random variables with unknown distribution with support in $[0,\tau_{\max}]$ and such that $\bar{\tau}_1 \leq \E[\tau_i]\leq \bar{\tau}_2$.
Observing that the average duration of the earthquake is  $\sum_{i=1}^B \E[\tau_i]$ and keeping in mind the significant effect
 of this duration on structural reliability \cite{LindtGoh:2004}, we select
 $\bar{\tau}_1=1\,\text{s}$, $\bar{\tau}_2=2\,\text{s}$, $\tau_{\max}=6\,\text{s}$, and $B=20$.

The propagation through the earth structure gives rise to
 focusing, de-focusing, reflection, refraction and anelastic attenuation (which is caused by
the conversion of wave energy to heat) \cite{SteinWysession:1991}. We do not assume the earth structure to be known, henceforth we assume that
$\psi$ is a random transfer function of unknown distribution. More precisely, we assume that the transfer function is given by
\begin{equation}
	\label{eq:Sesmo:transfer}
	\psi(t):=\frac{\sqrt{q}}{\tau'} \sum_{i=1}^q  c_i\, \varphi_i(t),
\end{equation}
where $q:=20$, $\tau' = 10 \, \text{s}$, $c$ is a random vector of unknown distribution with support in $\{x \in [-1,1]^q\,\mid\, \sum_{i=1}^q x_i^2 \leq 1\text{ and } \sum_{i=1}^q x_i=0\}$ and $\varphi_i$ is a piecewise linear basis nodal element on the discretization $t_1,\ldots,t_q$ of $[-\tau'/2,\tau'/2]$ with $t_{i+1}-t_i=\tau'/q$ ($\varphi_i(t_j)=\delta_{ij}$,  with $\delta_{ij}=1$ if $i=j$ and zero otherwise).  $\psi$ has the dimension of $1/\text{time}$ and the constraint  $\sum_{i=1}^q c_i^2 \leq 1$ is equivalent to the assumption that $\big(\frac{1}{\tau'} \int_{-\tau'/2}^{\tau'/2}|\psi|^2(t)\,\mathrm{d}t\big)^\frac{1}{2}$ is, with probability one, bounded by a constant of order $1/\tau'$.
Analogously to the Green function of the wave operator, $\psi$ can take both positive and negative values (in time, for a fixed site and source). Observe also that the constraint on the time integral of  $\psi^2$ leads to a bound on the Arias intensity (i.e., the time integral of $(\ddot{u}_0)^2$), which is a popular measure of ground motion strength used as a predictor of the likelihood of damage to short-period structures \cite{StaffordBP:2009}. The constraint $\sum_{i=1}^q c_i=0$ ensures that the residual velocity is zero at the end of the earthquake.
Observe also that, since the maximum amplitude of $s$ already contains the dampening factor  associated with the distance $R$ to the center of the earthquake (in $1/(R_0 + R)^2$, via \eqref{eq:maxgroundacc}), $\psi$ has to be interpreted as a normalized transfer function.
Since propagation in anisotropic structures can lead to changes in the direction of displacements, the coefficients $c_i$ should, for full generality, be assumed to be tensors. Although we have assumed those coefficients to be scalars for the clarity and conciseness of the presentation, the method and reduction theorems proposed in this paper still apply when those coefficients are tensors.

\subsubsection{The OUQ optimization problem}

The optimal bound on the probability that the structure will fail is therefore the solution of the following optimization problem
\begin{equation}
	\label{eq:Seismo:OUQopt}
	\mathcal{U}(\mathcal{A}) := \sup_{(F,\mu) \in \mathcal{A}} \mu[F\leq 0],
\end{equation}
where $\mathcal{A}$ is the set of pairs $(F,\mu)$ such that
 {\bf (1)} $F$ is mapping of the ground acceleration $t \mapsto \ddot{u}_0(t)$ onto the margin $\min_{i=1,\ldots,J}(S_i-\|Y_i\|_{\infty})$ via equations \eqref{eq:Seismo:HI}, \eqref{eq:Seismo:VEig} and \eqref{eq:Seismo:Y}.
	 {\bf (2)} $\mu$ is a probability measure on the ground acceleration $t\rightarrow \ddot{u}_0(t)$ with support on accelerations defined by \eqref{eq:Sesmo:convolution}, \eqref{eq:Sesmo:source}, \eqref{eq:Sesmo:transfer} (with $B=20$). Under this measure,  $X_1,\ldots,X_{B}$, $\tau_1,\ldots, \tau_B$, $c$ are independent (not necessarily identically distributed) random variables.
For $i=1,\ldots,B$, $X_i$ has zero mean and independent (not necessarily identically distributed) components  $(X_{i,1}, X_{i,2}, X_{i,3})$ with
  support in $[-a_{\max},a_{\max}]$, the measure of $\tau_i$ is constrained by
$\bar{\tau}_1 \leq \E[\tau_i]\leq \bar{\tau}_2$ and has support in $[0,\tau_{\max}]$. The support of the measure on $c$  is a subset of $\{x \in [-1,1]^q\,:\, \sum_{i=1}^q x_i^2 \leq 1\,\&\, \sum_{i=1}^q x_i=0\}$.

\subsubsection{Reduction of the optimization problem}

Problem \eqref{eq:Seismo:OUQopt} is  not  computationally tractable  since the optimization variables take values in infinite-dimensional  spaces of measures. However, thanks to  Corollary  \ref{cor:ouqreduce}, we know that the optimum of Problem \eqref{eq:Seismo:OUQopt} can be achieved by
 {\bf (1)}  Handling $c$ as a deterministic optimization variable taking values in $\{x \in [-1,1]^q\,:\, \sum_{i=1}^q x_i^2 \leq 1\,\&\, \sum_{i=1}^q x_i=0\}$ (since no constraints are given on the measure of $c$)
 {\bf (2)}  Assuming that the measure on each $X_{i,j}$ ($X_i=(X_{i,1}, X_{i,2}, X_{i,3})$) is the tensorization of two Dirac masses
in  $[-a_{\max},a_{\max}]$ (since $\E[X_{i,j}]=0$ is one linear constraint)
 {\bf (3)}  Assuming that the measure on each $\tau_i$ is the convex linear combination of $2$ Dirac masses in  $[0,\tau_{\max}]$ ($\bar{\tau}_1 \leq \E[\tau_i]\leq \bar{\tau}_2$ counts as one linear constraint).

Observe that this reduced problem is of finite dimension ($8B+q=180$) (counting the scalar position of the Dirac masses, their weights and subtracting the number of scalar equality constraints).

\setcounter{subfigure}{0}
\begin{figure}[tp]
	\begin{center}
		\subfigure[The truss structure]{\label{fig:tower2}
			\includegraphics[width=0.45\textwidth]{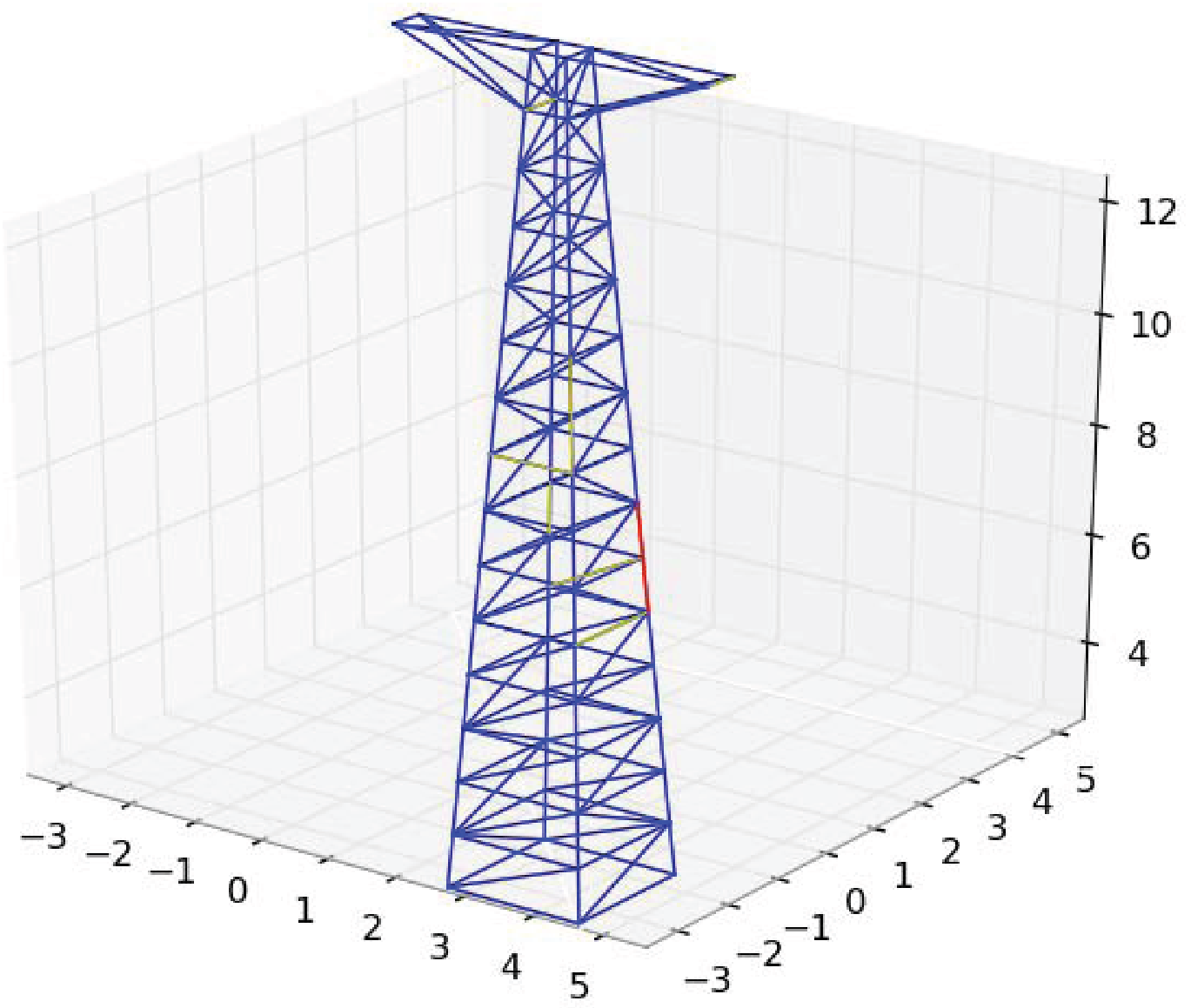}
		}
		\subfigure[Maximum PoF vs  $M_{\mathrm{L}}$]{\label{fig:PoFvsML}
			\includegraphics[width=0.45\textwidth]{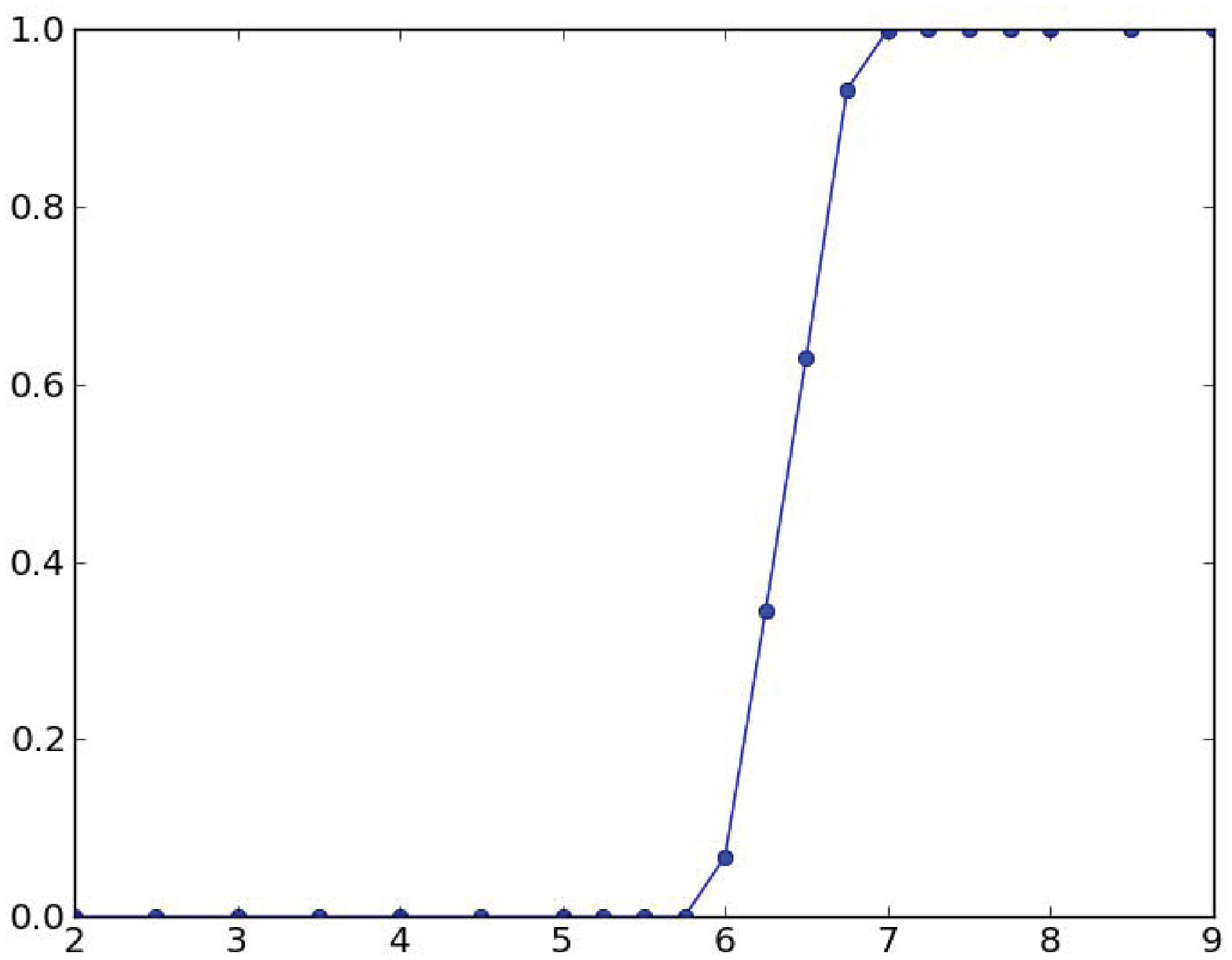}}
		\caption{Numerical results associated with the information set defined in Sub-section \ref{subsec:firstform}.}
		\label{fig:towerphysical}
	\end{center}
\end{figure}

\begin{figure}[tp]
	\begin{center}
		\subfigure[Maximum PoF vs $a_{\rm max}$]{\label{fig:PoFvsGroundAcc}
			\includegraphics[width=0.45\textwidth]{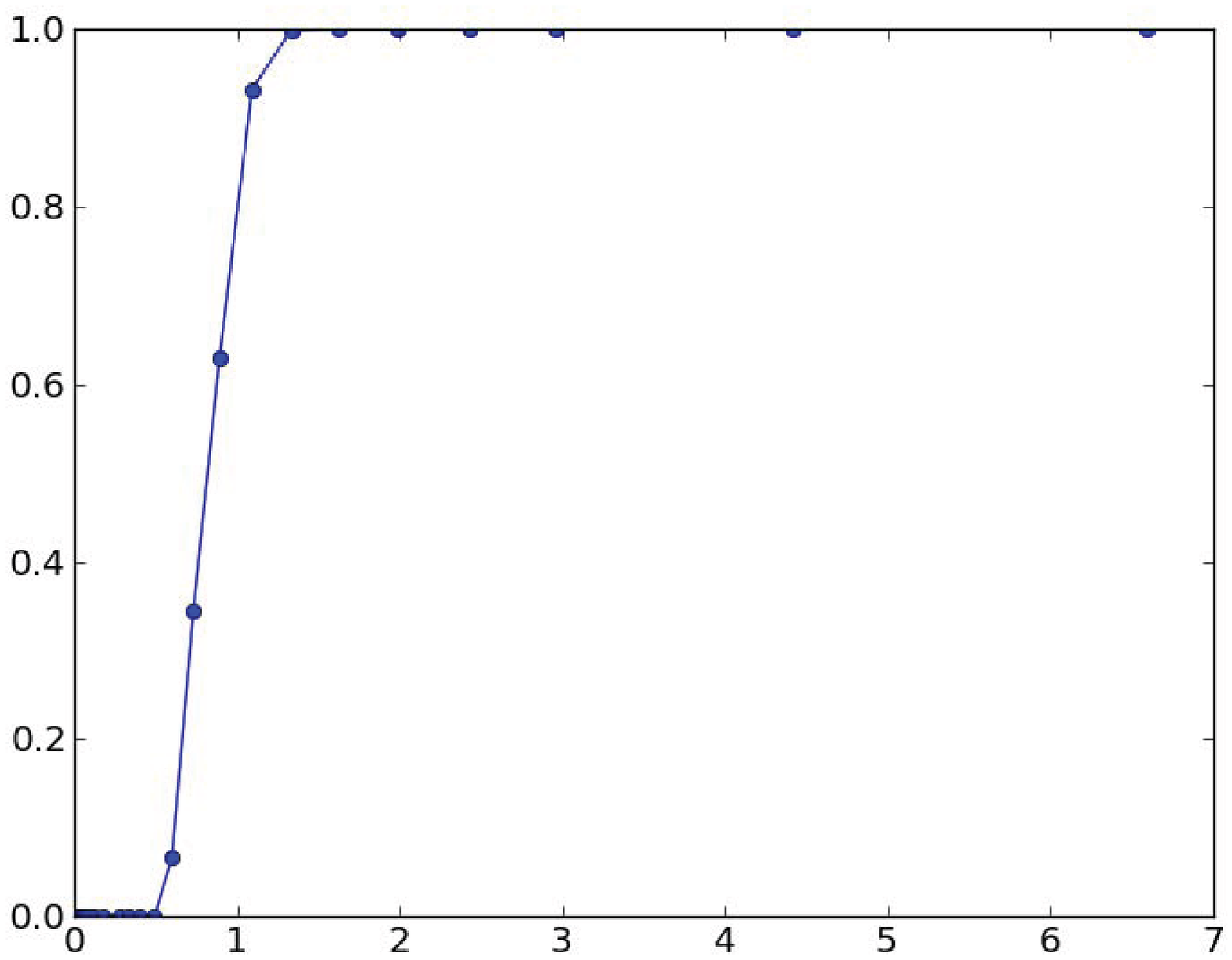}}
		\subfigure[Transfer function $\psi$]{\label{fig:transferfunction}
			\includegraphics[width=0.45\textwidth]{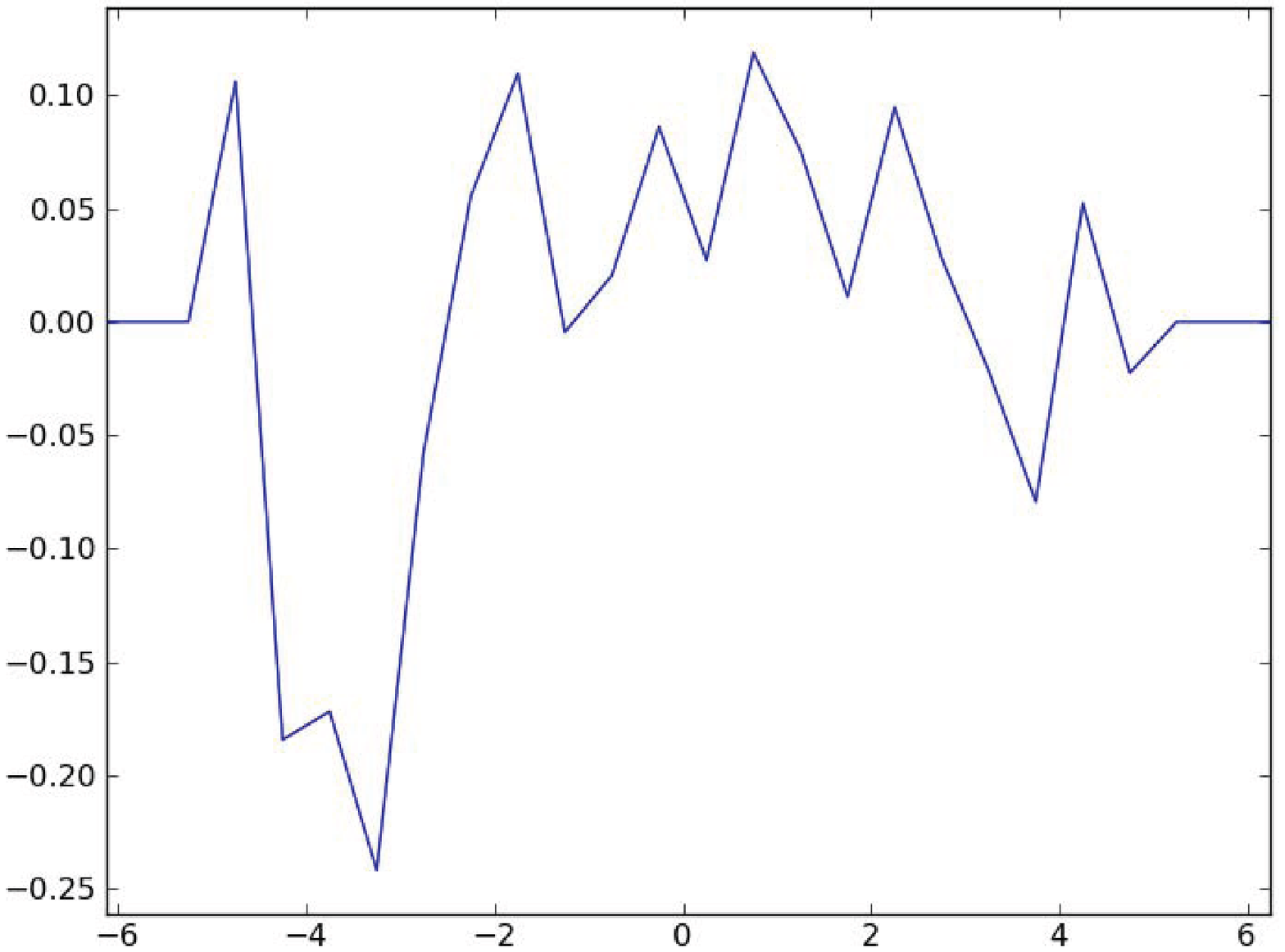}}
       \subfigure[Earthquake source $s(t)$]{\label{fig:tauidistribution}
			\includegraphics[width=0.45\textwidth]{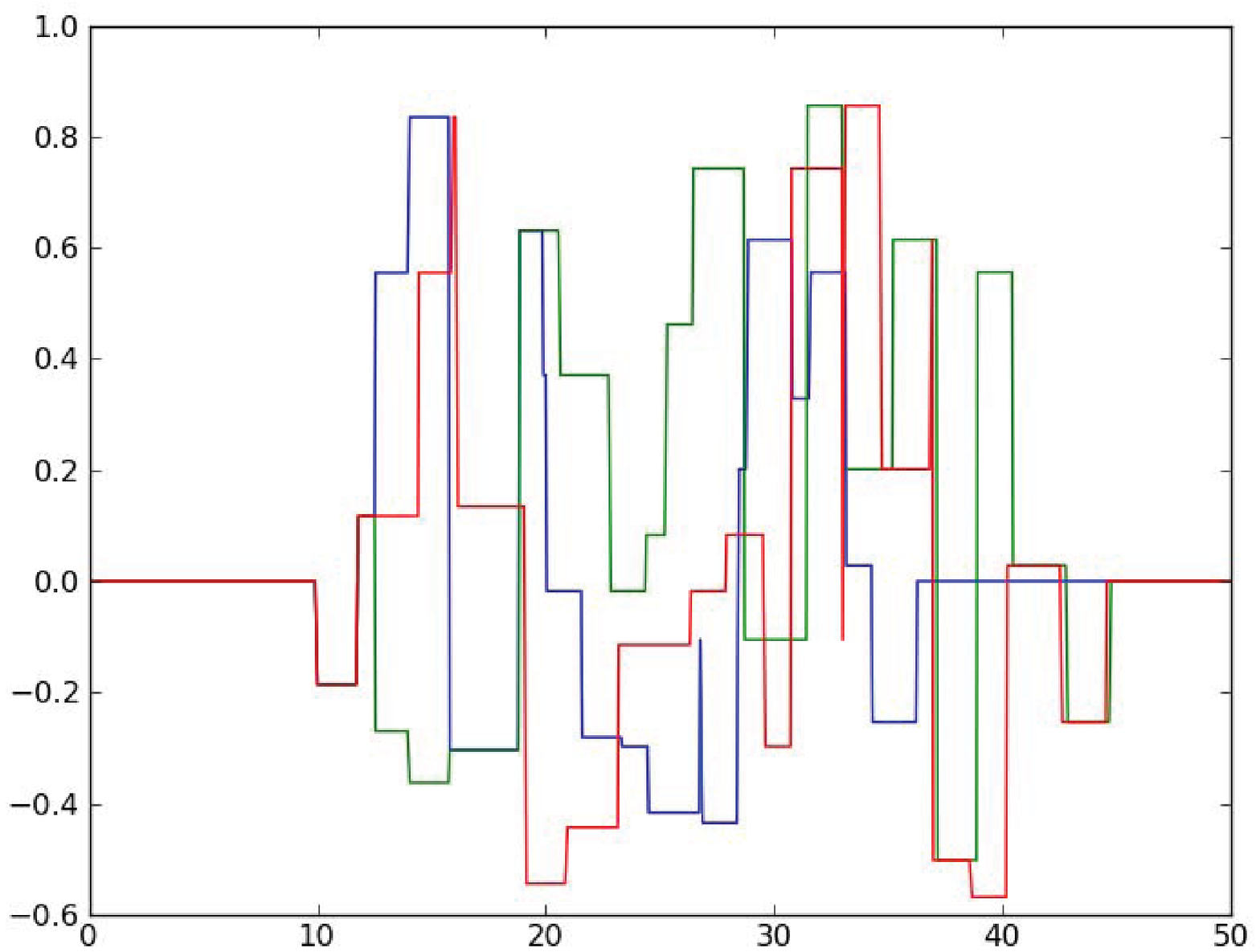}	}
        \subfigure[Ground acceleration]{\label{fig:Horizontalga}
			\includegraphics[width=0.45\textwidth]{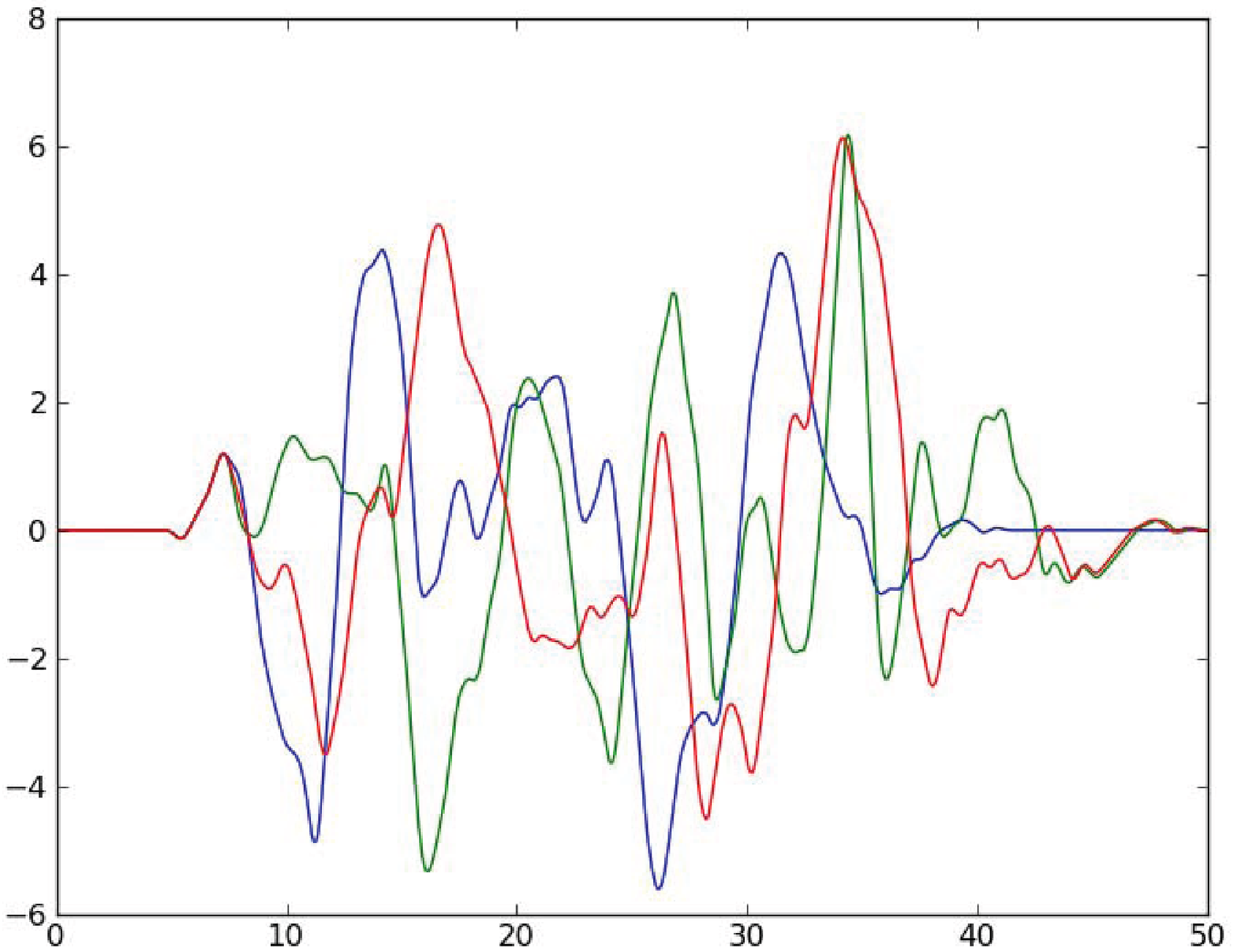}}
       \subfigure[Power Spectrum]{\label{fig:powerspec1}
			\includegraphics[width=0.45\textwidth]{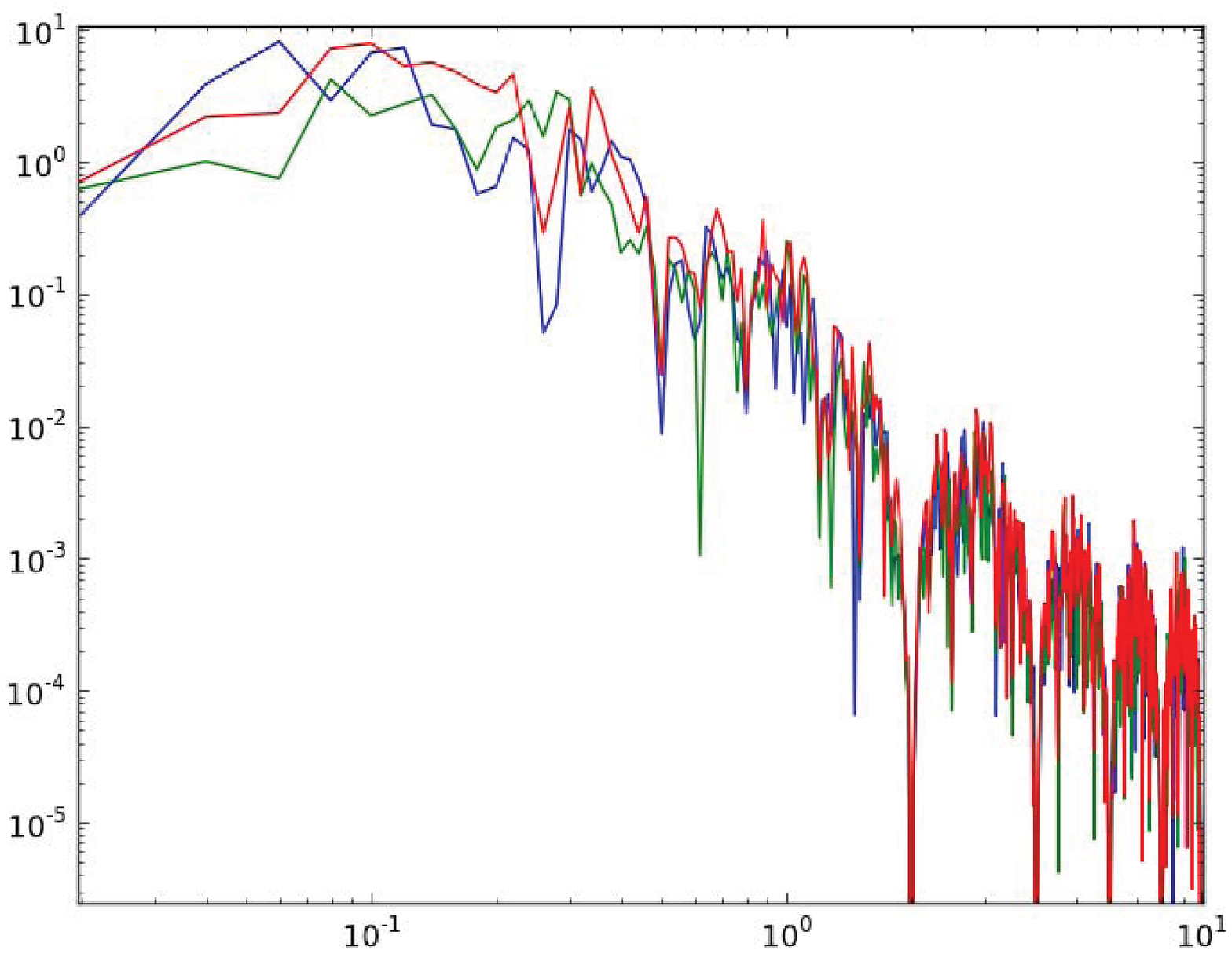}	}
       \subfigure[Elements strain]{\label{fig:elements}
			\includegraphics[width=0.45\textwidth]{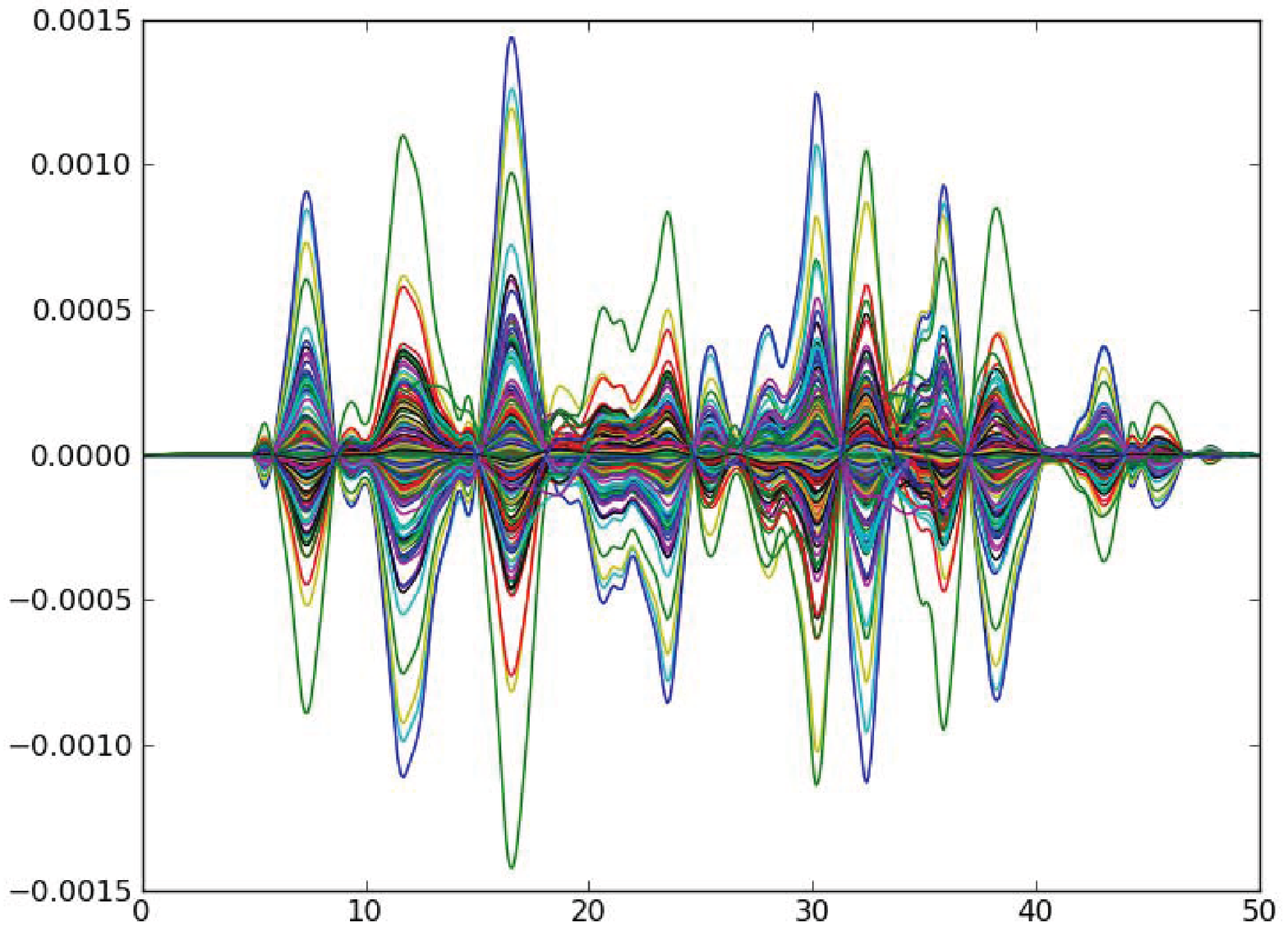}	}
		\caption{Numerical results associated with the information set defined in Sub-section \ref{subsec:firstform}.}
		\label{fig:towerphysicaldetails}
	\end{center}
\end{figure}

\subsubsection{Numerical results}\label{subsec:seismicnumresphysdom}

The truss structure is the electrical tower shown in Sub-figure \ref{fig:tower2}. This structure has $198$ elements and we refer to \cite{MATruss:2008} for precise numerical values associated with its geometry.
The material used for this structure is steel. The corresponding material properties are
 $7860\,\text{kg}/\text{m}^3$ for density,
 $2.1 \cdot 10^{11} \,\text{N}/\text{m}^2$ for the Young's modulus,  $2.5 \cdot 10^{8} \,\text{N}/\text{m}^2$ for the yield stress and $\zeta=0.07$ for the (uniform) damping ratio.  Calculations were performed with  time-step  $\Delta t := 5.0\cdot 10^{-2}\, \text{s}$.
We refer to Sub-figure \ref{fig:PoFvsGroundAcc}  for a graph of the optimal bound on the probability of failure  \eqref{eq:Seismo:OUQopt} versus
the maximum ground acceleration \ref{eq:maxgroundacc} (in $\text{m} \cdot \text{s}^{-2}$). Using Esteva's semi empirical formula  \eqref{eq:maxgroundacc} with a hypocentral distance $R$ equal to $25 \, \text{km}$ we obtain
  Sub-figure \ref{fig:PoFvsML}, the graph of the optimal bound on the probability of failure  \eqref{eq:Seismo:OUQopt} versus
the earthquake of magnitude $M_{\mathrm{L}}$ in the Richter (local magnitude) scale at hypocentral distance $R$ (the difference  $\Delta M_{\mathrm{L}}$ between two consecutive points is $0.25$).
  The ``S'' shape of the graph is typical of vulnerability curves \cite{LopezRocha:2009}.
   We select one of the points in the transition region for further analysis --- the point corresponding to a probability of failure of $0.631$, a maximum ground acceleration of $0.892 \, \text{m} \cdot \text{s}^{-2}$ and an earthquake of magnitude $6.5$.
The vulnerability curve undergoes a sharp transition (from small probabilities of failures to unitary probabilities of failures) around maximum ground acceleration $a_{\rm max} = 0.892  \, \text{m} \cdot \text{s}^{-2}$. This transition becomes smoother as the number of independent variables in the description of the admissible set is increased (results not shown).

For $a_{\rm max} = 0.892 \, \text{m} \cdot \text{s}^{-2}$ ($M_{\mathrm{L}}=6.5$),
Sub-figures \ref{fig:transferfunction} and \ref{fig:tauidistribution} show the
(deterministic) transfer function $\psi$ (the units in the x-axis are seconds) and  $3$ independent realizations of the earthquake source $s(t)$ sampled from the measure $\mu_{0.892}$ maximizing the probability of failure.
 For this measure, Sub-figure \ref{fig:elements} shows the axial strain of all elements versus time (in seconds) and
 Sub-figure \ref{fig:tower2} identifies the ten weakest elements for the most probable earthquake (the axial strain of these  elements are: $0.00142317$, $0.00125928$, $0.00099657$, $0.00081897$, $0.00076223$, $0.00075958$, $0.00072190$, $0.00068266$, $0.00062919$, and $0.00061361$) --- the weakest two elements exceed the yield strain of $0.00119048$ (shown in red in the figure).
Sub-figures \ref{fig:Horizontalga} and \ref{fig:powerspec1} show $3$ independent horizontal ground acceleration and a power spectrum
sampled from $\mu_{0.892}$. The units in Sub-figure \ref{fig:powerspec1} are cycles per seconds for the $x$ axis and $\text{m} \cdot \text{s}^{-2}$ for the $y$ axis. The units in Sub-figure \ref{fig:Horizontalga} are seconds for the $x$ axis and $\text{m} \cdot \text{s}^{-2}$ for the $y$ axis.

An quantitative analysis of the numerical results also show that all the constraints are active at the extremum (i.e. the generalized moments inequality  constraints on $\mu$ defining the information set introduced in Sub-section \ref{subsec:firstform} are equalities or near equalities at the extremum) . The positions and weights of the Dirac masses associated with durations and transfer coefficients do not appear to show any discernible trend. However, the positions and weights of the Dirac masses associated with the amplitudes $X_1,\ldots,X_M$  show a trend (as function of the earthquake magnitude $M_{\mathrm{L}}$) illustrated in Figure \ref{fig:seismicamplitudes}. This trend suggests that for strong earthquakes, probabilities of failures are maximized via (the possibility of) large amplitude impulses.

\begin{figure}[tp]
	\begin{center}
		\subfigure[$M_{\mathrm{L}}=6$]{\label{fig:seis600}
			\includegraphics[width=0.31\textwidth]{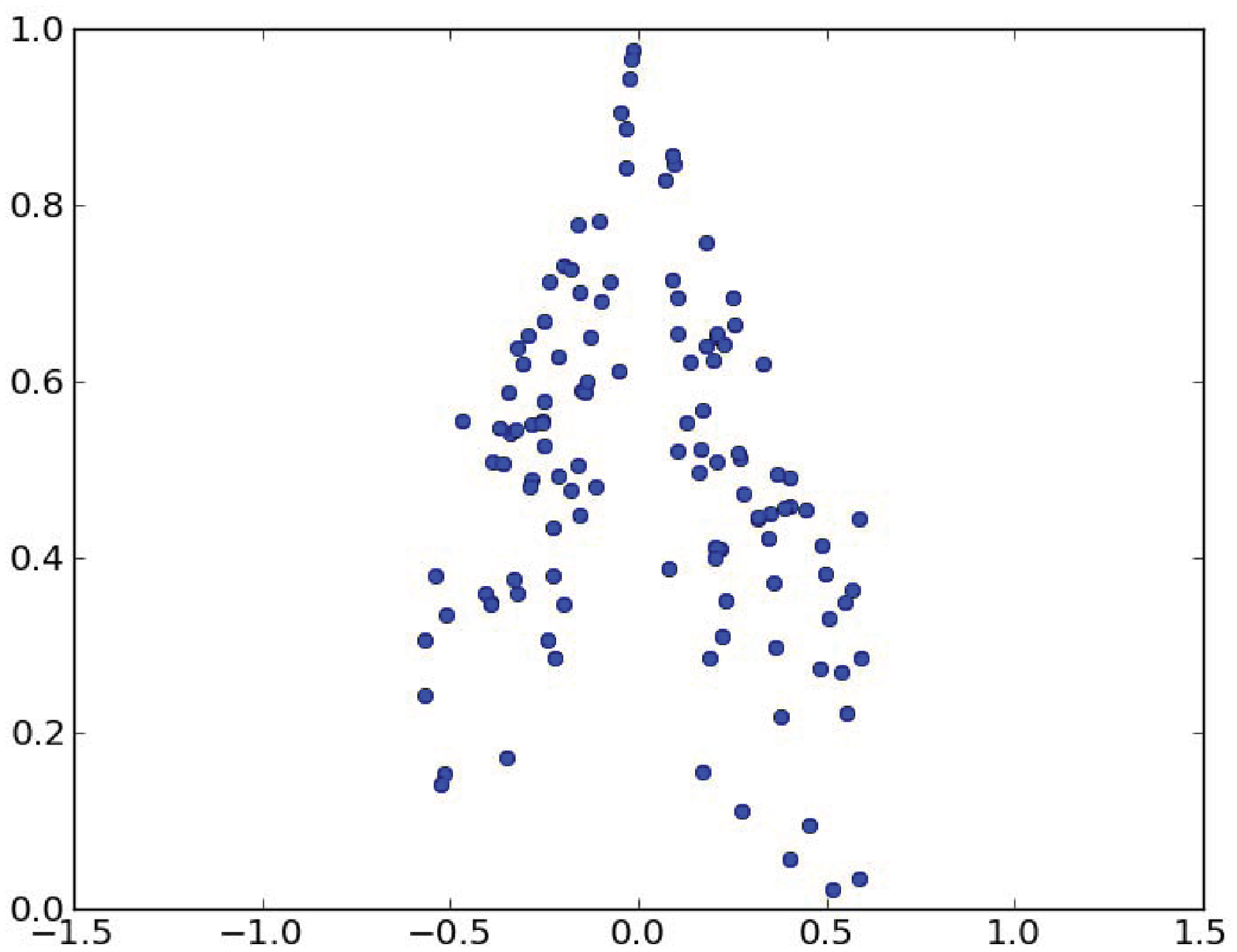}}
		\subfigure[$M_{\mathrm{L}}=6.5$]{\label{fig:seis650}
			\includegraphics[width=0.31\textwidth]{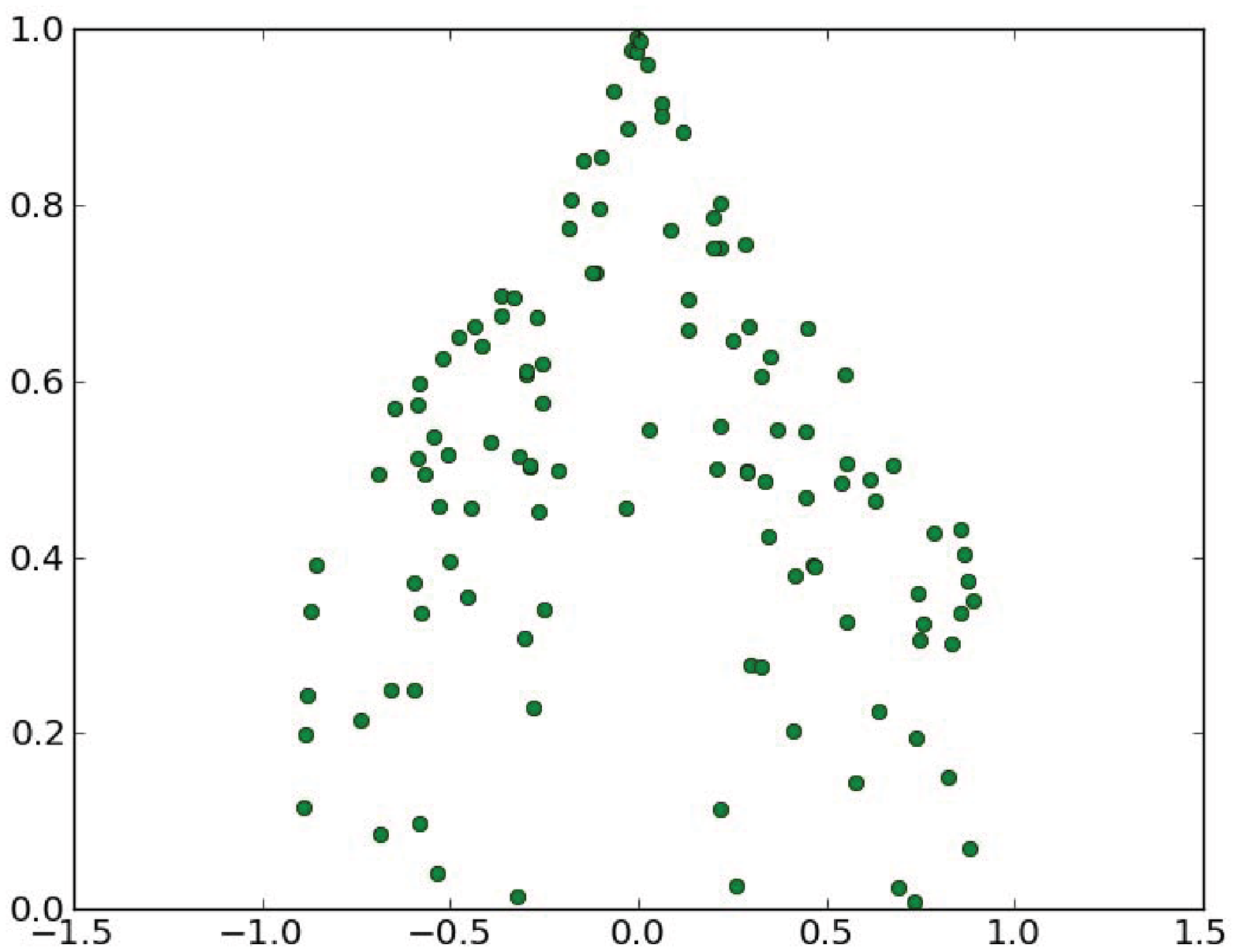}}
		\subfigure[$M_{\mathrm{L}}=7$]{\label{fig:seis700}
			\includegraphics[width=0.31\textwidth]{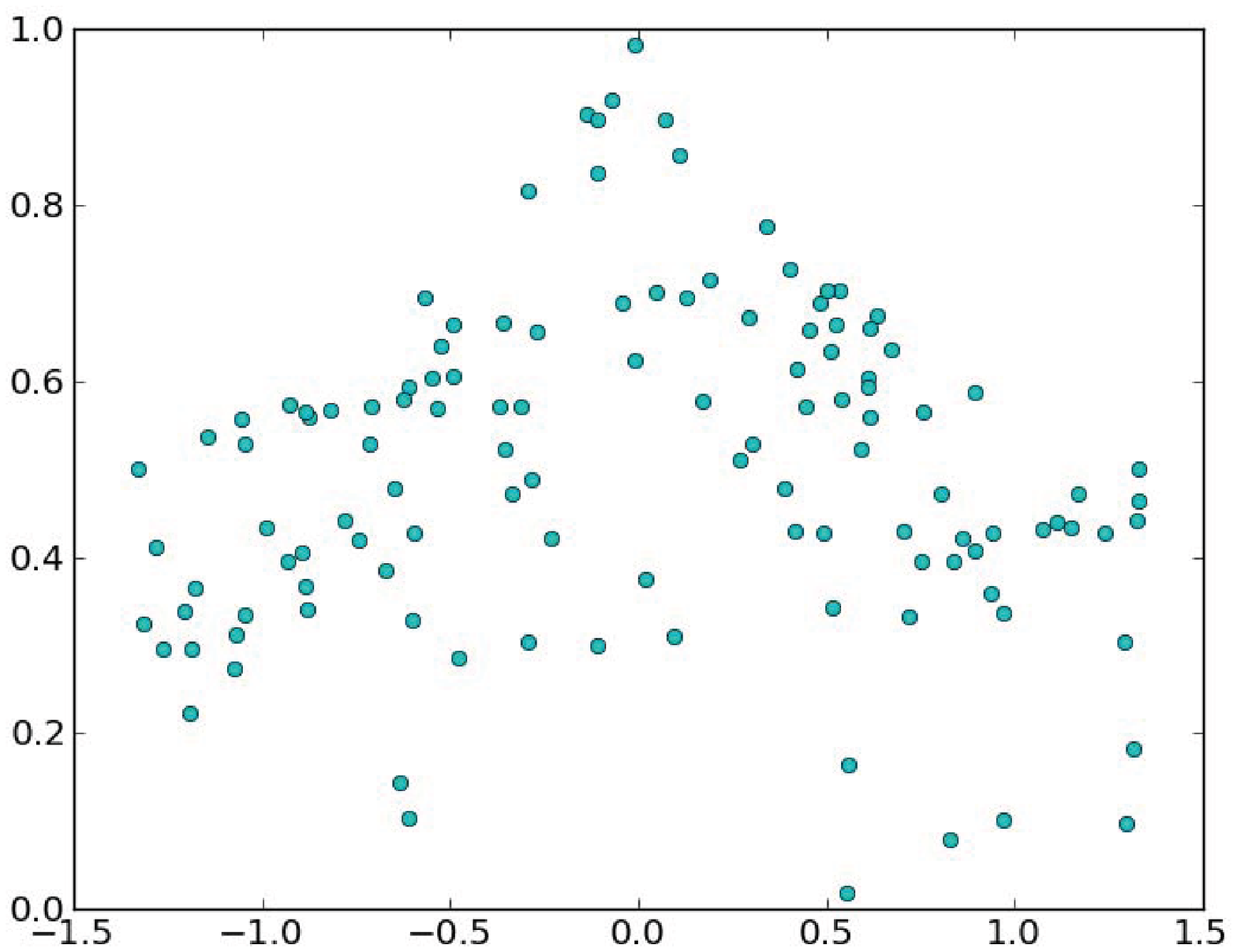}}
		\caption{Positions (abscissa, in $\text{m} \cdot \text{s}^{-2}$) and weights (ordinates) of the Dirac masses associated with the measure of probability on $X_1,\ldots,X_B$ at the extremum for earthquakes of magnitude $M_{\mathrm{L}}=6$, $M_{\mathrm{L}}=6.5$ and $M_{\mathrm{L}}=7$. Note that the positions in abscissa correspond to the possible amplitudes of the impulses $X_i$.
}
		\label{fig:seismicamplitudes}
	\end{center}
\end{figure}

\begin{figure}[tp]
	\begin{center}
		\subfigure[Estimated Maximum PoF vs iterations]{\label{fig:convpoftruss}
			\includegraphics[width=0.45\textwidth]{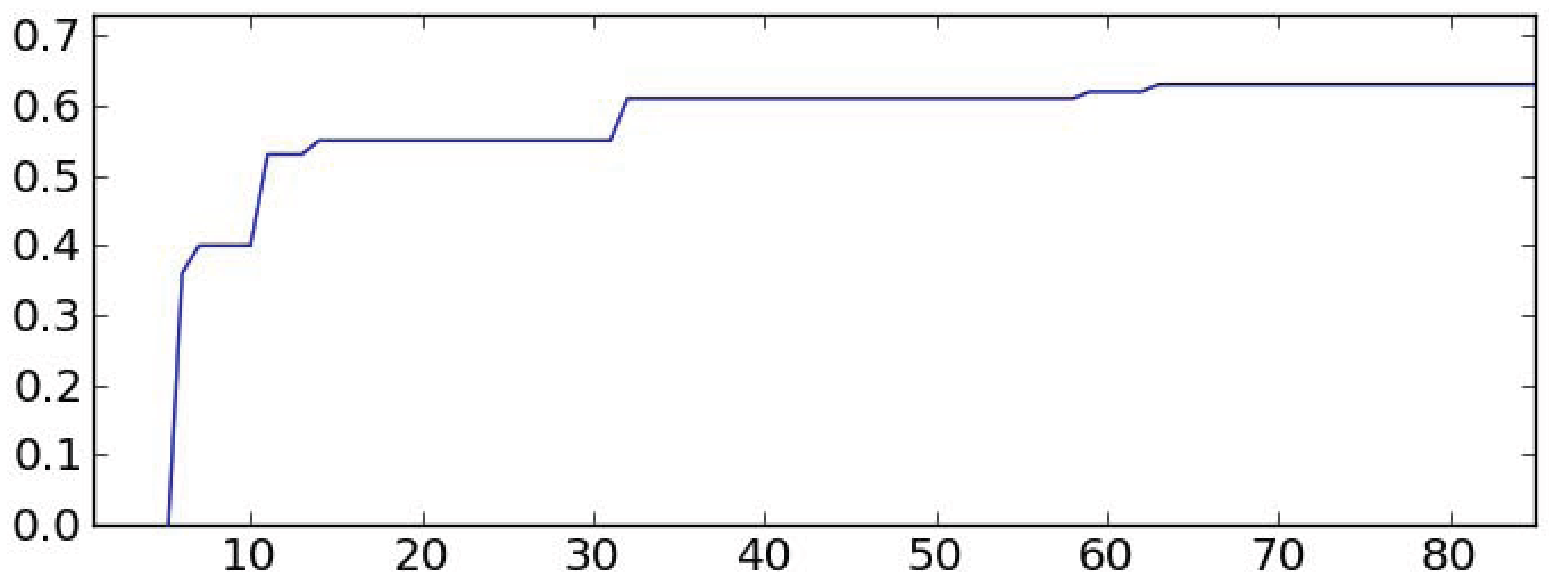}}
		\subfigure[Dirac Positions vs iterations]{\label{fig:convdiracstruss}
			\includegraphics[width=0.45\textwidth]{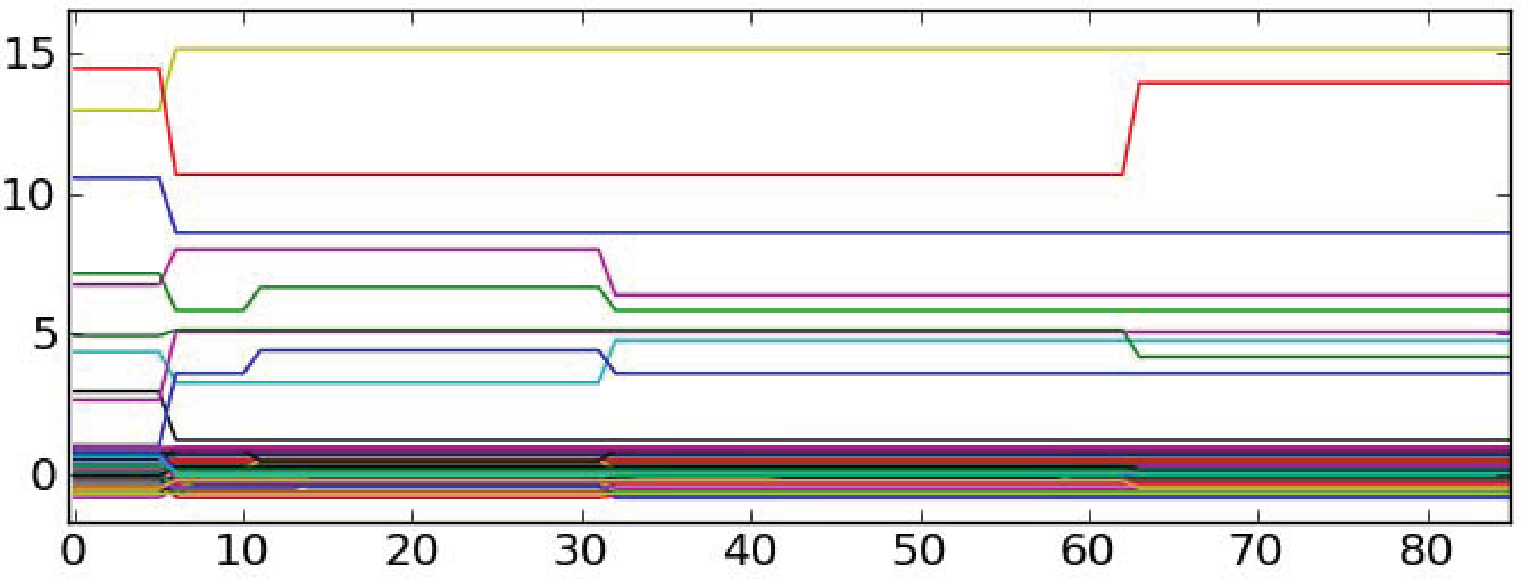}}
		\caption{\subref{fig:convpoftruss}: Estimated maximum probability of failure versus number of iterations for an earthquake of magnitude $M_{\mathrm{L}}=6.5$ (this corresponds to the point in transition region of Sub-figure \ref{fig:PoFvsML}). \subref{fig:convdiracstruss}: re-normalized positions of the Dirac masses for  $M_{\mathrm{L}}=6.5$.
}
		\label{fig:convergetrussdiracspof}
	\end{center}
\end{figure}

\paragraph{On the numerical optimization algorithm.}
Global search algorithms often require hundreds of iterations and thousands of function evaluations, due to their stochastic nature, to find a global optimum.
Local methods, like Powell's method \cite{Powell:1994}, may require orders of magnitude fewer iterations and evaluations, but do not generally converge to a global optimum in a complex parameter space.
To compute the probability of failure, we use a Differential Evolution algorithm \cite{PriceStornLampinen:2005, StornPrice:1997} that has been modified to utilize large-scale parallel computing resources \cite{McKernsHungAivazis:2009}.
Each iteration, the optimizer prepares $m$ points in parameter space, with each new point derived through random mutations from the 'best' point in the previous iteration.
We select $m = 40$, which is of modest size compared to the dimensionality of the problem --- however, we chose this modest size because populations larger than $m = 40$ only modestly increase the efficiency of the algorithm. Each of these $m$ evaluations are performed in parallel on a computer cluster, such that the time required for a single iteration equals the time required for a single function evaluation.
After $n$ iterations complete, the optimal probability of failure for the product measure is returned (convergence is observed around $n\approx 200$ and we select $n\approx 2000$ for the robustness of the result).

Only one iteration is required for values of ground acceleration on the extremes of the range (such as $M_{\mathrm{L}}= 2$ and  $M_{\mathrm{L}}=9$). The number of iterations required for convergence for points in the transition region (around $M_{\mathrm{L}}=6.5$) is between $30$ and $50$ (which corresponds to $2,400$ to $4,000$ function evaluations). We refer to Figure \ref{fig:convergetrussdiracspof} for an illustration of the convergence of the optimization algorithm for $M_{\mathrm{L}}=6.5$.

Each function evaluation  takes approximately $0.5\,\text{s}$ on a high-performance computing cluster (such as the high-performance computing clusters used at the National Labs). With each iteration utilizing $m = 40$ parallel processors, the OUQ calculation takes roughly $24\,\text{hrs}$.

Approximately $1000$ time steps are required for accuracy in the strain calculations, each function evaluation requires two convolutions over time.
Because of the size of the truss structure ($198$ elements), eigenvalues have to be computed with high accuracy.
Because of the size of the product measure associated with the numerical optimization iterates, the probability of failure (associated with these iterates) should be estimated with a controlled  (and adapted) tolerance rather than computed exactly --- we use a sampling size of $5000$ points.

\subsection{OUQ and critical excitation.}

Without  constraints on ground acceleration, the ground motion yielding the maximum peak response (maximum damage in a deterministic setting) has been referred to as the \emph{critical excitation} \cite{Drenick:1970}. Drenick himself  pointed out that a seismic design based on critical excitation could be ``far too pessimistic to be practical'' \cite{Drenick:1973}. He  later  suggested that the combination of probabilistic approaches with worst-case analysis should be employed to make the seismic resistant design robust \cite{Drenick:1977a}.  Practical
application and extension of critical excitation methods have then been made extensively and we refer to \cite{Takewaki:2002} and \cite{Takewaki:2007} for recent reviews.
The probabilities of failures obtained from stochastic approaches depend on particular choices of probability distribution functions.
Because of the scarcity of recorded time-histories, these choices involve some degree of arbitrariness \cite{SteinWysession:1991, Takewaki:2002} that may be incompatible with the certification of critical structures and rare events \cite{Drenick:1980}.
We suggest that by allowing for very weak assumption on probability measures, the reduction theorems associated with the OUQ framework could
lead to certifications methods that are both robust (reliable) and practical (not overly pessimistic).
Of course this does require the identification of a reliable and narrow information set.
The set $\mathcal{A}$ used in this paper does not include all the available information on earthquakes. We also suggest that the method of selecting next best experiments could help in this endeavor.

Observe also that without constraints, worst-case scenarios correspond to focusing the energy of the earthquake in modes of resonances of the structure.
Without correlations in the ground motion these scenarios correspond to rare events where independent random variables must conspire to strongly excite a specific resonance mode.  The lack of information on the transfer function $\psi$ and the mean values $\E[\tau_i]$ permits scenarios characterized by strong correlations in ground motion where  the energy of the earthquake can be focused in the above mentioned modes of resonance.

\subsection{Alternative formulation in the frequency domain}

A popular method for modeling and synthesizing seismic ground motion is to use (deterministic) shape functions and envelopes in the frequency domain (see  \cite{TrainssonKiremidjianWinterstein:2000} for a review).

In this sub-section, we will evaluate the safety of the electrical tower shown in Sub-figure \ref{fig:tower2} using an admissible set $\mathcal{A}_{F}$ defined from weak information on the probability distribution of the power spectrum of the seismic ground motion.

\subsubsection{Formulation of the information set}

We assume that the (three dimensional) ground motion acceleration is given by
\begin{equation}
	\label{eq:Sesmo:fourier}
	\ddot{u}_0(t) := \sum_{k=1}^{W} \big((A_{6k-5},A_{6k-4},A_{6k-3})  \cos(2\pi \omega_k t)+(A_{6k-2},A_{6k-1},A_{6k}) \sin(2\pi \omega_k t)\big),
\end{equation}
where the Fourier coefficients $A_j$ are random variables (in $\R$) of unknown distribution.
We assume that $W:=100$ and that $\omega_k:= k/\tau_d$ with $\tau_d=20\,\text{s}$. Writing $A:=(A_{1},\ldots, A_{6W})$,
 we assume that
\begin{equation}\label{eq:proboneA}
\P\left[A \in B(0, a_{\max})\setminus B(0,\tfrac{a_{\max}}{2})\right]=1,
\end{equation}
 where $a_{\max}$ is given by Esteva's semi-empirical expression \eqref{eq:maxgroundacc} and $B(0, a_{\max})\setminus B(0,\frac{a_{\max}}{2})$ is the Euclidean ball of $\R^{6W}$ of center $0$ and radius $a_{\max}$ minus the Euclidean ball of center $0$ and radius $\frac{a_{\max}}{2}$.

Although different earthquakes have different power spectral densities it is empirically observed that ``on average'', their power spectra follow specific shape functions that may depend on the ground structure of the site where the earthquake is occurring \cite{LamaWilsonaHutchinsona:2000}. Based on this observation, synthetic seismograms are produced by filtering the Fourier spectrum of white noise with these specific shape functions \cite{LamaWilsonaHutchinsona:2000}.  In this sub-section, our information on the distribution of $A$ will be limited to the shape of the mean value of its power spectrum. More precisely, we will assume that, for $k\in \{1,\ldots,W\}$ and $j\in \{0,\ldots,5\}$,
\begin{equation}
	\label{eq:Sesmo:average}
	\E[A_{6k-j}^2]=\frac{a_{\max}^2}{12} \frac{ s(\omega_k)}{s_0},
\end{equation}

where $s$ is the Matsuda--Asano shape function \cite{MatsudaAsano:2006} given by:
\begin{equation}\label{eq:matsano}
s(\omega):=\frac{\omega_g^2 \omega^2}{(\omega_g^2-\omega^2)^2+4 \xi_g^2 \omega_g^2, \omega^2},
\end{equation}
where $\omega_g$ and $\xi_g$ are the natural frequency and natural damping factor of the site and
\begin{equation}\label{jhsgsdhgdsjgd}
s_0:=\sum_{k=1}^{W} s(\omega_k).
\end{equation}
We will use the numerical values $\omega_g=6.24 \,\text{Hz}$ and $\xi_g=0.662$  associated with
the January 24, 1980 Livermore earthquake (see \cite{LiAn:2008}, observe that we are measuring frequency in cycles per seconds instead of radians per seconds).
The purpose of the normalization factor \eqref{jhsgsdhgdsjgd} is to enforce the following mean constraint:
\begin{equation}
\E\left[\frac{1}{\tau_d}\int_0^{\tau_d}|\ddot{u}_0(t)|^2 \,\mathrm{d}t \right]=\frac{1}{2}\E\left[|A|^2\right]=\frac{a_{\max}^2}{4}.
\end{equation}
Observe also that \eqref{eq:proboneA}  implies that, with probability one,
\begin{equation}
\frac{a_{\max}^2}{8} \leq \frac{1}{\tau_d}\int_0^{\tau_d}|\ddot{u}_0(t)|^2 \,\mathrm{d}t \leq \frac{a_{\max}^2}{2}.
\end{equation}
We write $\mathcal{A}_{F}$ the set of probability measures $\mu$ on $A$ satisfying \eqref{eq:proboneA} and \eqref{eq:Sesmo:average}.

\subsubsection{OUQ objectives}
Let $(Y_1,\ldots,Y_J)$ and $(S_1,\ldots,S_J)$ be the axial and yield strains
introduced in Sub-section \ref{subsec:formprobseismic}.
Writing $S:=[-S_1,S_1]\times \cdots \times [-S_J,S_J]$ (this is the safe domain for the axial strains),
we are interested in computing optimal (maximal and minimal with respect to measures $\mu\in \mathcal{A}_{F}$) bounds on the probability (under $\mu$) that
$Y(t) \not \in S$ for some $t\in [0,\tau_d]$ (defined as the probability of failure).
From the linearity of equations \eqref{eq:Seismo:EoM2},
the strain of member $i$ ($i\in \{1,\ldots,J\}$) at time $t$ can be written
\begin{equation}\label{eq:oiwohdh}
Y_i(t)=\sum_{j=1}^{6W} \Psi_{ij}(t) A_{j}.
\end{equation}
Let  $\Psi(t)$ be the $J\times (6W)$ tensor $(\Psi_{ij}(t))$ and
observe that Equation \eqref{eq:oiwohdh} can be also be written $Y(t)=\Psi(t) A$.
Let $\mathcal{F}$ be the subset of $\R^{6W}$ defined as the elements $x$ of $B(0,a_{\max})\setminus B(0,\frac{a_{\max}}{2})$ such that $\Psi(t) x \notin [-S_1,S_1]\times \cdots\times [-S_J,S_J]$ for some $t \in [0,\tau_d]$, i.e.
\begin{equation}
\mathcal{F}:=\left\{x\in B(0, a_{\max})\setminus B(0,\frac{a_{\max}}{2}) \,\middle|\, \Psi(t) x \not \in S\quad\text{for some}\quad t \in [0,\tau_d]\right\}.
\end{equation}
Observe $\mathcal{F}$ corresponds to the set of vectors $A$ (in \eqref{eq:Sesmo:fourier}) that lead to a failure of the structure.
Henceforth, our objective can be formulated as computing
\begin{equation}\label{eq:objspectrum}
\sup_{\mu \in \mathcal{A}_{F}} \mu\big[ A \in \mathcal{F} \big]\quad \text{and} \quad \inf_{\mu \in \mathcal{A}_{F}} \mu\big[A \in \mathcal{F} \big],
\end{equation}
where $\mathcal{A}_{F}$ is the set of probability measures $\mu$ such that\\	$\mu\big[A \in B(0, a_{\max})\setminus B(0,\frac{a_{\max}}{2})\big]=1$, and that
\begin{equation}\label{eq:constA}
	\E_\mu[A_j^2]= b_j\quad \text{with}\quad b_j:=\frac{a_{\max}^2}{12} \frac{ s(\omega_{\lfloor (j+5)/6 \rfloor})}{s_0}.
\end{equation}

In other words, $\mathcal{A}_F$ an infinite-dimensional polytope defined as the set of probability measures on ground acceleration that have the Matsuda--Asano average power spectra \eqref{eq:matsano}. It is important to observe that that with the filtered white noise method the safety of the structure is assessed for a single measure $\mu_0 \in \mathcal{A}_F$ whereas in the proposed OUQ framework we compute best and worst-case scenarios with respect to all measures in $\mathcal{A}_F$.

\subsubsection{Reduction of the optimization problem with Dirac masses}
\label{subsec:dirmassalt}

Since \eqref{eq:constA} corresponds to $6W$ global linear constraints on $\mu$, Theorem \ref{thm:baby_measure} implies that the extrema of problem \eqref{eq:objspectrum} can be achieved by assuming $\mu$
 to be a weighted sum of Dirac masses $\sum_{j=1}^{6W+1} p_j \delta_{Z_{.,j}}$
where $Z_{.,j}\in B(0, a_{\max})\setminus B(0,\frac{a_{\max}}{2})$, $p_j\geq 0$ and $\sum_{j=1}^{6W+1} p_j=1.$
The constraints \eqref{eq:constA} can then be written: for $i\in \{1,\ldots, 6W\}$,
$\sum_{j=1}^{6W+1} Z_{i,j}^2 p_j=b_i$.
Furthermore,
$ \mu\big[ A \in \mathcal{F}  \big]=\sum_{j\,:\,Z_{.,j}\in \mathcal{F} } p_j$.

\subsubsection{Reduction of the optimization problem based on strong duality}
\label{subsec:dirmassalt2}

Since the information contained in $\mathcal{A}_{F}$ is limited to constraints on the moments of $A$, strong duality can be employed to obtain an alternative reduction of problems \eqref{eq:objspectrum}. Indeed, Theorem 2.2 of \cite{BertsimasPopescu:2005} implies that
\begin{equation}\label{eq:maxrhssup}
\sup_{\mu \in \mathcal{A}_{F}} \mu\big[ A \in \mathcal{F} \big]
=\inf_{(H_0,H)\in \R^{6W+1}} H_0 + \sum_{i=1}^{6W} H_{i} b_i,
\end{equation}
where the minimization problem (over the vector $(H_0,H):=(H_0,H_1,\ldots,H_{6W})\in \R^{6W+1}$) in the right hand side of \eqref{eq:maxrhssup} is subject to
\begin{equation}
\sum_{i=1}^{6W} H_{i} x_i^2 +H_0 \geq \chi(x) \quad \text{on}\quad B(0, a_{\max})/B(0,\frac{a_{\max}}{2}),
\end{equation}
where $\chi(x)$ is the function equal to $1$ on $\mathcal{F}$ and $0$ on $(\mathcal{F})^c$ (we note
$(\mathcal{F})^c$  the complement of $\mathcal{F}$, i.e.\ the set of $x$ in $\R^{6W}$ that are not elements of $\mathcal{F}$).
Similarly,
\begin{equation}\label{eq:maxrhsinf}
\inf_{\mu \in \mathcal{A}_{F}} \mu\big[A\in \mathcal{F} \big]
=\sup_{(H_0,H)\in \R^{6W+1}} H_0 + \sum_{i=1}^{6W} H_{i} b_i,
\end{equation}
where the maximization problem in the right hand side of \eqref{eq:maxrhsinf} is subject to
\begin{equation}
\sum_{i=1}^{6W} H_{i} x_i^2 +H_0 \leq \chi(x) \quad \text{on}\quad B(0, a_{\max})/B(0,\frac{a_{\max}}{2}).
\end{equation}
We conclude from these equations (by optimizing first with respect to $H_0$) that the optimal upper bound on the probability of failure (defined as the probability that the displacement $Y(t)$ does not belong to the safe region $S$ for all time $t$ in the interval $[0,\tau_d]$) is
\begin{equation}\label{eq:redsimpseisu}
\sup_{\mu \in \mathcal{A}_{F}} \mu\big[ A \in \mathcal{F} \big]=\inf_{H\in \R^{6W}} \sup_{x\in B(0, a_{\max})/B(0,\frac{a_{\max}}{2})} \chi(x) + \sum_{i=1}^{6W} H_{i} (b_i-x_i^2),
\end{equation}
whereas the optimal lower bound is
\begin{equation}\label{eq:redsimpseisl}
\inf_{\mu \in \mathcal{A}_{F}} \mu\big[ A \in \mathcal{F} \big]=\sup_{H\in \R^{6W}} \inf_{x\in B(0, a_{\max})/B(0,\frac{a_{\max}}{2})}\chi(x) + \sum_{i=1}^{6W} H_{i} (b_i-x_i^2).
\end{equation}
Observe that problem \eqref{eq:redsimpseisu} is convex in $H\in \R^{6W}$ whereas problem \eqref{eq:redsimpseisl} is concave.

\begin{figure}[tp]
	\begin{center}
			\includegraphics[width=0.6\textwidth]{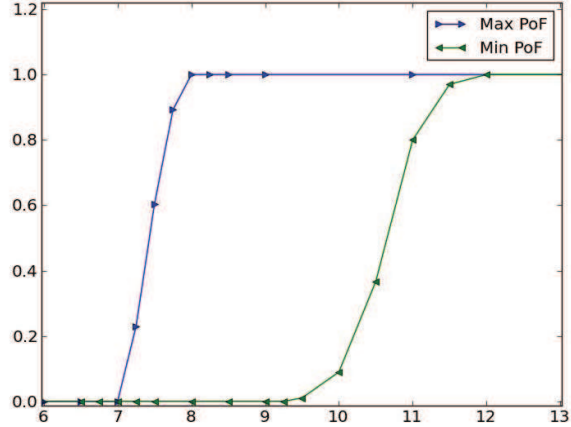}
		\caption{Maximum and minimum probability of failure of the structure (as defined in \eqref{eq:objspectrum}) versus
the earthquake of magnitude $M_{\mathrm{L}}$ in the Richter (local magnitude) scale at hypocentral distance $R=25 \, \text{km}$ ($a_{\max}$ is given by Esteva's semi-empirical expression \eqref{eq:maxgroundacc} as a function of $M_{\mathrm{L}}$). The curve corresponding to maximum probability of failure is not the same as the one given in Sub-figure \ref{fig:PoFvsML} because it is based on a different information set.
}
		\label{fig:seismicminmax}
	\end{center}
\end{figure}

\subsubsection{Numerical results}

The optimal bounds \eqref{eq:objspectrum} can be computed using the reduction to masses to Dirac described in Sub-section \ref{subsec:dirmassalt} or strong duality as described in Sub-section \ref{subsec:dirmassalt2}. While the latter does not identify the extremal measures it
leads to a smaller optimization problem than the former (i.e. to optimization variables in $\R^{12W}$, instead of $\R^{(6W+1)\times (6W+1)}$).
The simplification is allowed by the facts that the response function is well identified, that there are no independence constraints, and that the information on $A$ is limited to $6W$ (scalar) moment constraints. The vulnerability curves of Figure \ref{fig:seismicminmax} have been computed using strong duality as described in Subsection \ref{subsec:dirmassalt2} (an identification of extremal measures would require using the method described in Subsection \ref{subsec:dirmassalt}). Observe that to decrease the gap between the maximum probability of failure and the minimum probability of failure, one would have to refine the information on the probability distribution of ground motion (by, for instance, adding constraints involving the correlation between the amplitudes $A_i$ different Fourier modes).
 To solve optimization problems \eqref{eq:redsimpseisu} and \eqref{eq:redsimpseisl}
we use the modified Differential Evolution algorithm described in Subsection \ref{subsec:seismicnumresphysdom}.
Equation \eqref{eq:redsimpseisu} is implemented as a minimization over $H$, where a nested maximization over $x$ is used to solve for $\sup_{x\in B(0, a_{\max})/B(0,\frac{a_{\max}}{2})} \chi(x) - \sum_{i=1}^{6W} H_{i} x_i^2$ at each function evaluation. Both the minimization over $H$, and the maximization over $x$ use the Differential Evolution algorithm described above, where the optimizer configuration itself differs only in that for the nested optimization termination occurs when the maximization over $x$ does not improve by more than
$10^{-6}$ in $20$ iterations, while the outer optimization is terminated when there is not more than $10^{-6}$ improvement over $100$ iterations.
The optimization over $H$ is performed in parallel, as described in Subsection \ref{subsec:seismicnumresphysdom}, where each of the nested optimizations over $x$ are distributed across nodes of a high-performance computing cluster. Each of the (nested) optimizations over $x$ require only a few seconds on average, and thus are performed serially.  Convergence, on average takes about $15$ hours, and is obtained in roughly $2000$ iterations (over H), corresponding to $35000$ to $50000$ function evaluations.  Each function evaluation is a nested optimization over $x$, which takes a few seconds on a high-performance computing cluster.

\section{Application to Transport in Porous Media}
\label{Sec:porous}

We  now apply the OUQ framework and reduction theorems to divergence form elliptic PDEs and consider the situation where coefficients (corresponding to microstructure and source terms) are random and have imperfectly known probability distributions.  Treating those distributions as optimization variables (in an infinite-dimensional space) we obtain optimal bounds on probabilities of deviation of solutions.  Surprisingly, explicit and optimal bounds show that, with incomplete information on the probability distribution of the microstructure, uncertainties or information do not necessarily propagate across scales.

To make this more precise in a simple setting, let $\mathcal{D} \subseteq \R$ be a bounded domain and consider $u(x,\omega)$, the solution of the following stochastic elliptic PDE:
\begin{equation}
	\label{eq:stochasticellipticPDE}
  \begin{cases}
    -\operatorname{div}(\kappa(x,\omega) \nabla u(x,\omega)) = f(x,\omega), \enspace & x\in\mathcal{D}\\
    u(x,\omega) = 0, & x\in\partial\mathcal{D}
  \end{cases}
\end{equation}
with random microstructure $\kappa$ and random (positive) source term $f$.  Physically, $u$ can be interpreted as the pressure (or head) in a medium of permeability $\kappa$ with source $f$;  the fluid velocity is given by $\nabla u$.  For a given point $x_{0}$ in the interior of $\mathcal{D}$, we are interested in computing the least upper bound on the probability of an unsafe supercritical pressure at $x_{0}$:
\begin{equation}
	\mathcal{U}(\mathcal{A}):=\sup_{\mu\in \mathcal{A}} \mathbb{E}_\mu\big[\log u(x_0,\omega)\geq \mathbb{E}_\mu[\log u(x_0,\omega)]+a\big],
\end{equation}
where $\mathcal{A}$ is a set of probability measures on $(\kappa,f)$.  In this section we will focus on the two admissible sets $\mathcal{A}$ described below.

Let $D_1,D_2\geq 0$, $K, F \in L^\infty(\mathcal{D})$ such that
$\operatorname{essinf}_{\mathcal{D}} K >0$, $F\geq 0$, and $\int_{\mathcal{D}} F >0$. Define
\begin{equation}
	\mathcal{A}_{\kappa,f}:= \left\{ \mu\,\middle|\,
	\begin{matrix}
		\kappa,f \text{ independent under $\mu$,} \\
		K(x) \leq \kappa(x,\omega) \leq  e^{D_1} K(x),\\
		F(x) \leq f(x,\omega) \leq e^{D_2} F(x)
	\end{matrix} \right\}.
\end{equation}
We say that a function $g$ defined on $\mathcal{D}$ is periodic of period $\delta$ if for all $x\in \mathcal{D}$, it holds that $g(x)=g(x+\delta)$ whenever $x+\delta \in \mathcal{D}$.  We now define
\begin{equation}\label{eq:infosetA2k1k2}
	\mathcal{A}_{\kappa_1 \kappa_2}:= \left\{ \mu\,\middle|\,
		\begin{matrix}
			\kappa=\kappa_1 \kappa_2, \\
			\kappa_1,\kappa_2 \text{ independent under $\mu$,} \\
			\|\nabla\kappa_1\|_{L^\infty}\leq e^{D_1} \|\nabla K_1\|_{L^\infty}  \\
			\kappa_2 \text{ periodic of period } \delta   \\
			K_1(x) \leq \kappa_1(x,\omega) \leq  e^{D_1} K_1(x),\\
			K_2(x) \leq \kappa_2(x,\omega) \leq  e^{D_2} K_2(x),\\
		\end{matrix} \right\},
\end{equation}
where $0<\delta\ll 1$, $K_2 \in L^\infty(\mathcal{D})$ is uniformly elliptic over  $\mathcal{D}$ and periodic of period $\delta$, and $K_1$ is smooth and uniformly elliptic over $\mathcal{D}$.

PDEs of the form \eqref{eq:stochasticellipticPDE} have become a benchmark for stochastic expansion methods \cite{Ghanem:1999, Xiu:2009, BabuskaNobRa:2010, EldredWebsterConstantine:2008, DoostanOwhadi:2010, Todor07a, Bieri09a} and we also refer to \cite{GhanemDham:1998} for their importance for transport in porous media.

These PDEs have also been studied as classical examples in the UQ literature on the basis that the randomness in the coefficients (with a perfectly known probability distribution on the coefficients $(\kappa,f)$) is an adequate model of the lack of information on the microstructure $\kappa$. In these situations the quantification of uncertainties is equivalent to a push forward of the measure probability on $(\kappa,f)$.

However, in  practical situations the probability distribution on the coefficients $(\kappa,f)$ may not be known a priori and the sole randomness in coefficients may not constitute a complete characterization of uncertainties. This is our motivation for considering the problem described in this section.
We have also introduced the admissible set \eqref{eq:infosetA2k1k2} as a simple illustration of uncertainty quantification with multiple scales and incomplete information on probability distributions. To relate this example to classical homogenization \cite{BensoussanLionsPapanicolaou:1978} we
have assumed  $\kappa_2$ to be periodic of small period $\delta\ll 1$.

\begin{thm}\label{thmporous}
	We have
	\begin{equation}
		\mathcal{U}(\mathcal{A}_{\kappa,f})=\mathcal{U}(\mathcal{A}_{\kappa_1 \kappa_2})=\mathcal{U}(\mathcal{A}_{\mathrm{McD}}),
	\end{equation}
	with
		\begin{equation}
			\label{eq:McD_reduced_explicit_2dporous}
			\mathcal{U}(\mathcal{A}_{\mathrm{McD}}) =
			\begin{cases}
				0, & \text{if } D_{1} + D_{2} \leq a, \\
				\dfrac{(D_{1} + D_{2} -a)^{2}}{4 D_{1} D_{2}}, & \text{if } | D_{1} - D_{2} | \leq a \leq  D_{1} + D_{2}, \\
				1 - \dfrac{a}{\max (D_{1}, D_{2})}, & \text{if } 0\leq  a \leq | D_{1} - D_{2} |.
			\end{cases}
		\end{equation}
\end{thm}

Before giving the proof of Theorem \ref{thmporous}, we make a few important observations:

It follows from Theorem \ref{thmporous} that if $D_2 \geq a+D_1$,  then $\mathcal{U}(\mathcal{A}_{\kappa,f})(a,D_1,D_2)=\mathcal{U}(\mathcal{A}_{\kappa,f})(a,0,D_2)$. In other words, if the uncertainty on the source term $f$ is dominant, then the uncertainty associated with the microstructure, $\kappa$, does not propagate to the uncertainty corresponding to the probability of deviation of $\log u(x_0,\omega)$ from its mean.

Now consider $\mathcal{A}_{\kappa_1 \kappa_2}$.  Since $\kappa_1$ is constrained to be smooth and $\kappa_2$ periodic with period $\delta \ll 1$, one would expect the microstructure $\kappa_2$ to appear in the probability of deviation in a homogenized form.  However, Theorem \ref{thmporous} shows that if
$D_1\geq a+D_2$, then $\mathcal{U}(\mathcal{A}_{\kappa_1 \kappa_2})(a,D_1,D_2)=\mathcal{U}(\mathcal{A}_{\kappa_1 \kappa_2})(a,D_1,0)$.  More precisely, if the uncertainty associated with the background $\kappa_1$ is dominant, then the uncertainty associated with the microstructure $\kappa_2$ does not propagate to the uncertainty corresponding to the probability of deviation of $\log u(x_0,\omega)$ from its mean.

This simple but generic example suggests that for structures characterized by multiple scales or multiple modules, information or uncertainties may not propagate across modules or scales. This phenomenon can be explained by the fact that, with incomplete information, scales or modules may not communicate certain types of information. Henceforth, the global uncertainty of a modular system cannot be reduced without decreasing local dominant uncertainties. In particular, for modular or multi-scale systems, one can identify  (possibly large) accuracy thresholds (in terms of numerical solutions of PDEs or SPDEs) below which the global uncertainty of the system does not decrease.

\begin{proof}[Proof of Theorem \ref{thmporous}.]
	Let us now prove Theorem \ref{thmporous} with the admissible set $\mathcal{A}_{\kappa,f}$ (the proof with the set $\mathcal{A}_{\kappa_1 \kappa_2}$ is similar).  It follows from Theorem 2.11 and Proposition 2.13 of \cite{Owhadi:2003} that the maximum oscillation of $\log u(x_0,\omega)$ with respect to $\kappa$ and $f$ are bounded by $D_1$ and $D_2$ we obtain that
	\begin{equation}
		\mathcal{U}(\mathcal{A}_{\kappa,f})\leq \mathcal{U}(\mathcal{A}_{\mathrm{McD}}),
	\end{equation}
	where $\mathcal{U}(\mathcal{A}_{\mathrm{McD}})$ is defined in equation \eqref{eq:sjjshgdjhgejhge} (we consider the case $m=2$).
	
	Next, from the proof of Theorem \ref{thm:m2}, we observe that the bound $\mathcal{U}(\mathcal{A}_{\mathrm{McD}})$ can be achieved by $\mathcal{A}_{\kappa,f}$ by considering measures $\mu$ that are tensorizations of two weighted Dirac masses in $\kappa$ (placed at $K$ and $e^{D_1}K$) and two weighted Dirac masses in $f$ (placed at $F$ and $e^{D_1}F$). This concludes the proof.
\end{proof}

\section{Conclusions}
\label{sec:Conclusions}

\paragraph{The UQ Problem --- A Problem with UQ?}
The 2003 \emph{Columbia} space shuttle accident and the 2010 Icelandic volcanic ash cloud crisis have demonstrated two sides of the same problem:  discarding information may lead to disaster, whereas over-conservative safety certification may result in unnecessary economic loss and supplier-client conflict.  Furthermore, while everyone agrees that UQ is a fundamental component of objective science (because, for instance, objective assertions of the validity of a model or the certification of a physical system require UQ), it appears that not only is there no universally accepted notion of the objectives of UQ, there is also no universally accepted framework for the communication of UQ results.  At present, the ``UQ problem'' appears to have all the symptoms of an ill-posed problem;  at the very least, it lacks a coherent general presentation, much like the state of probability theory before its rigorous formulation by Kolmogorov in the 1930s.
\begin{compactitem}
	\item  At present, UQ is an umbrella term that encompasses a large spectrum of methods: Bayesian methods, Latin hypercube sampling, polynomial chaos expansions, stochastic finite-element methods, Monte Carlo, \emph{etc.}  Most (if not all) of these methods are characterized by a list of assumptions required for their application or efficiency.  For example, Monte Carlo methods require a large number of samples to estimate rare events;  stochastic finite-element methods require the precise knowledge of probability density functions and some regularity (in terms of decays in spectrum) for their efficiency; and concentration-of-measure inequalities require uncorrelated (or weakly correlated) random variables.
	\item  There is a disconnect between theoretical UQ methods and  complex systems of importance requiring UQ in the sense that the assumptions of the methods do not match the assumption/information set of the application.  This disconnect means that often a specific method adds inappropriate implicit or explicit assumptions (for instance, when the knowledge of probability density functions is required for polynomial chaos expansions, but is unavailable) and/or the repudiation of relevant information (for instance, the monotonicity of a response function in a given variable) that the method is not designed to incorporate.
\end{compactitem}

\paragraph{OUQ as an opening gambit.}
OUQ is not the definitive answer to the UQ problem, but we hope that it will help to stimulate a discussion on the development of a rigorous and well-posed UQ framework analogous to that surrounding the development of probability theory.
The reduction theorems of Section \ref{sec:Reduction}, the Optimal Concentration Inequalities and non-propagation of input uncertainties of Section \ref{sec:Reduction-mcd}, the possibility of the selection of optimal experiments described at the end of Section \ref{sec:OUQ}, and the numerical evidence of Section \ref{sec:ComputationalExamples} that (singular, i.e.\ low-dimensional) optimizers are also attractors, suggest that such a discussion may lead to non-trivial and worthwhile questions and results at the interface of optimization theory, probability theory, computer science and statistics.

In particular, many questions and issues raised by the OUQ formulation remain to be investigated. A few of those questions and issues are as follows:
\begin{compactitem}
\item Any (possibly numerical) method that finds admissible states $(f,\mu)\in \mathcal{A}$ leads to rigorous lower bounds on $\mathcal{U}(\mathcal{A})$. It is known that duality techniques  lead to upper bounds on $(f,\mu)\in \mathcal{A}$ provided that the associated Lagrangians can be computed. Are there interesting classes of problems for which those Lagrangians can rigorously be estimated or bounded from above?
\item  The reduction theorems of Section \ref{sec:Reduction} are limited to linear constraints on probability distribution marginals and the introduction of sample data may lead to other situations of interest (for instance, relative-entropy type constraints).
\item Although general in its range of application,  the algorithmic framework introduced in Section \ref{sec:ComputationalExamples} is still lacking general convergence theorems.
\item  The introduction of sample data appears to render the OUQ optimization problem even more complex. Can this  optimization problem be made equivalent to applying the deterministic setting to an information set $\mathcal{A}$ randomized by the sample data?
\item In the presence of sample data, instead of doing theoretical analysis to describe the optimal statistical test, one formulation of the OUQ approach provides an optimization problem that must be solved to determine the test. Is this optimization problem reducible?
\end{compactitem}

\section*{Acknowledgements}
\addcontentsline{toc}{section}{Acknowledgements}

The authors gratefully acknowledge  portions of this work supported by the Department of Energy National Nuclear Security Administration under Award Number DE-FC52-08NA28613 through Caltech's ASC/PSAAP Center for the Predictive Modeling and Simulation of High Energy Density Dynamic Response of Materials.  Calculations for this paper were performed using the \emph{mystic} optimization framework \cite{McKernsHungAivazis:2009}.  We thank the Caltech PSAAP Experimental Science Group --- Marc Adams, Leslie Lamberson, Jonathan Mihaly, Laurence Bodelot, Justin Brown, Addis Kidane, Anna Pandolfi, Guruswami  Ravichandran and Ares Rosakis --- for Formula \eqref{eq:PSAAP_SPHIR_surr} and figures \ref{fig:sphir}. We thank Sydney Garstang and Carmen Sirois for proofreading the manuscript. We thank Ilse Ipsen and four anonymous referees for detailed comments and suggestions.

\section{Appendix: Proofs}\label{sec-appendixproofs}

\subsection{Proofs for Section \ref{sec:Reduction}}
\label{sec-reduceproofs}

\begin{proof}[Proof of Theorem \ref{thm:baby_measure}.]
In this proof, we use $(\mu_{1}, \dots, \mu_{m})$ as a synonym for the product $\mu_{1} \otimes \dots \otimes \mu_{m}$.
	For $\mu =\bigotimes_{i=1}^{m}{\mu_{i}} \in \mathcal{M}^{\G}$, consider the optimization problem
	\begin{align*}
		\text{maximize: } & \E_{(\mu_{1}',\mu_{2},\dots,\mu_{m})}[r], \\
		\text{subject to: } & \mu_{1}'\in \mathcal{M}(\mathcal{X}_{1}), \\
		& \G(\mu_{1}',\mu_{2},\dots,\mu_{m})\leq 0.
	\end{align*}
	By Fubini's Theorem,
	\[
		\E_{(\mu_{1}',\mu_{2},\dots,\mu_{m})}[r] = \E_{\mu_{1}'} \left[ \E_{(\mu_{2},\dots,\mu_{m})} [r] \right],
	\]
	where $\E_{(\mu_{2},\dots,\mu_{m})}[r]$ is a Borel-measurable function on $\mathcal{X}_{1}$ and, for $j = 1, \dots, n$, it holds that
	\[
		\E_{(\mu_{1}',\mu_{2},\dots,\mu_{m})} [g'_{j}] = \E_{\mu_{1}'} \left[ \E_{(\mu_{2},\dots,\mu_{m})} [g'_{j}] \right],
	\]
	where $\E_{(\mu_{2},\dots,\mu_{m})}{{g'}_{j}}$ is a Borel-measurable function on $\mathcal{X}_{1}$.  In the same way, we see that
	\[
		\E_{(\mu_{1}',\mu_{2},\dots,\mu_{m})} [g^{1}_{j}] = \E_{\mu_{1}'} [g^{1}_{j}],
	\]
	and, for $k = 2, \dots, m$ and $j = 1, \dots, n_{k}$, it holds that
	\[
		\E_{(\mu_{1}',\mu_{2},\dots,\mu_{m})} [g^{k}_{j}] = \E_{\mu_{k}} [g^{k}_{j}],
	\]
	which are constant in $\mu_{1}'$.

  Since each $\mathcal{X}_{i}$ is Suslin, it follows that all the measures in $\mathcal{M}(\mathcal{X}_{i})$ are regular.  Consequently, by \cite[Theorem 11.1]{Topsoe:1970}, the extreme set of $\mathcal{M}(\mathcal{X}_{i})$ is the set of Dirac masses. For fixed $(\mu_{2},\dots,\mu_{m})$, let $\G_{1} \subseteq \mathcal{M}(\mathcal{X}_{1})$
 denote those measures that satisfy the constraints $\G(\mu_{1}', \mu_{2}, \dots, \mu_{m}) \leq 0$.  Consequently, since for $k = 2, \dots, m$ and $j = 1, \dots, n_{k}$, $\E_{(\mu_{1}',\mu_{2},\dots,\mu_{m})}[g^{k}_{j}]$ is constant in $\mu_{1}'$, it follows from \cite[Theorem 2.1]{Winkler:1988} that the extreme set $\mathrm{ex}(\G_{1}) \subseteq \mathcal{M}(\mathcal{X}_{1})$ of the constraint set consists only of elements of $\Delta_{n_{1}+n'}(\mathcal{X}_{1})$.  In addition, von Weizs\"{a}cker and Winkler \cite[Corollary 3]{WeizsackerWinkler:1979} show that a Choquet theorem holds:  let $\mu'$ satisfy the constraints.  Then
	\[
		\mu'(B) =\int_{\mathrm{ex}(\G_{1})} \nu(B) \, \mathrm{d} p(\nu),
	\]
	for all Borel sets $B \subseteq \mathcal{X}_{1}$, where $p$ is a probability measure on the extreme set $\mathrm{ex}(\G_{1})$.
	
	According to Winkler, an extended-real-valued function $K$ on $\G_{1}$ is called \emph{measure affine} if it satisfies the barycentric formula
	\[
		K(\mu') = \int_{\mathrm{ex}(\G_{1})} K(\nu) \, \mathrm{d} p(\nu).
	\]
	When $K$ is measure affine, \cite[Theorem 3.2]{Winkler:1988} asserts that
	\[
		\sup_{\mu' \in \G_{1}} K(\mu') = \sup_{\nu \in \mathrm{ex}(\G_{1})} K(\nu),
	\]
	and so we conclude that
	\[
		\sup_{\mu' \in \G_{1}} K(\mu') = \sup_{\nu \in \mathrm{ex}(\G_{1})} K(\nu) \leq \sup_{\nu \in \Delta_{n_{1}+n'}(\mathcal{X}_{1}) \cap \G_{1}} K(\nu).
	\]
	However, since
	\[
		\sup_{\nu \in \Delta_{n_{1}+n'}(\mathcal{X}_{1}) \cap \G_{1} } K(\nu) \leq \sup_{\nu \in \G_{1} } K(\nu),
	\]
	it follows that
	\[
		\sup_{\mu' \in \G_{1}} K(\mu') = \sup_{\nu  \in \Delta_{n_{1}+n'}(\mathcal{X}_{1}) \cap \G_{1} } K(\nu).
	\]
	
	To apply this result, observe that \cite[proposition 3.1]{Winkler:1988} asserts that the evaluation function
	\[
		\mu_{1}' \mapsto \E_{\mu_{1}'} \left[ \E_{(\mu_{2},\dots,\mu_{m})} [r] \right]
	\]
	is measure affine.  Therefore,
	\begin{equation}
		\label{eq:winkler}
		\sup_{\substack{\mu_{1}' \in \mathcal{M}(\mathcal{X}_{1}) \\ \G(\mu_{1}',\mu_{2},\dots,\mu_{m})\leq 0}} \E_{(\mu_{1}', \mu_{2},\dots, \mu_{m})} [r] = \sup_{\substack{\mu_{1}' \in \Delta_{n_{1}+n'}(\mathcal{X}_{i}) \\ \G(\mu_{1}', \mu_{2}, \dots, \mu_{m}) \leq 0 }} \E_{(\mu_{1}',\mu_{2},\dots,\mu_{m})}[r].
	\end{equation}
	Now let $\varepsilon > 0$ and let $\mu_{1}^{\ast} \in \Delta_{n_{1}+n'}(\mathcal{X}_{1})$ be $\varepsilon$-suboptimal for the right-hand side of \eqref{eq:winkler}:  that is, $\G(\mu_{1}^{\ast}, \mu_{2}, \dots, \mu_{m}) \leq 0$, and
	\[
		\E_{(\mu^{\ast}_{1}, \mu_{2}, \dots, \mu_{m})} [r] \geq \sup_{\substack{\mu_{1}' \in \Delta_{n_{1}+n'}(\mathcal{X}_{i}) \\ \G(\mu_{1}',\mu_{2},\dots,\mu_{m}) \leq 0  }} \E_{(\mu_{1}',\mu_{2}, \dots, \mu_{m})} [r] -\varepsilon .
	\]
	Hence, by \eqref{eq:winkler},
	\[
		\E_{(\mu^{\ast}_{1},\mu_{2},\dots,\mu_{m})} [r] \geq \sup_{\substack{\mu_{1}'\in \mathcal{M}(\mathcal{X}_{1}) \\ \G(\mu_{1}',\mu_{2},\dots,\mu_{m}) \leq 0 }} \E_{(\mu_{1}',\mu_{2},\dots,\mu_{m})} [r] - \varepsilon \geq \E_{(\mu_{1},\mu_{2},\dots,\mu_{m})} [r] -\varepsilon.
	\]
	Consequently, the first component of $\mu$ by can be replaced some element of $\Delta_{n_{1}+n'}(\mathcal{X}_{1})$ to produce a feasible point $\mu' \in \mathcal{M}^{\G}$ without decreasing $\E[r]$ by more than $\varepsilon$.  By repeating this argument, it follows that for every point $\mu \in \mathcal{M}^{\G}$ there exists a $\mu' \in \mathcal{M}_{\Delta}$ such that
	\[
		\E_{\mu'}[r] \geq \E_{\mu}[r] - m \varepsilon.
	\]
	Since $\varepsilon$ was arbitrary the result follows.
\end{proof}

\begin{proof}[Proof of Corollary \ref{cor:ouqreduce}.]
	Simply use the identity
	\[
		\mathcal{U}(\mathcal{A}) = \sup_{(f,\mu) \in \mathcal{A}} \E_{\mu}[r_{f}] = \sup_{f \in \mathcal{G}}\sup_{\substack{\mu \in \mathcal{M}^{m}(\mathcal{X}) \\ \G(f,\mu) \leq 0}} \E_{\mu}[r_{f}]
	\]
	and then apply Theorem \ref{thm:baby_measure} to the inner supremum.
\end{proof}

\begin{proof}[Proof of Theorem \ref{thm:valuereduce}.]
	Corollary \ref{cor:ouqreduce} implies that $\mathcal{U}(\mathcal{A}) =\mathcal{U}(\mathcal{A}_{\Delta})$ where,
	\[
		\mathcal{A}_{\Delta} := \left\{ (f,\mu) \in \mathcal{G} \times \bigotimes_{i=1}^{m}{\Delta_{n}(\mathcal{X}_{i})} \,\middle|\,
		\E_{\mu}[g_{i} \circ f] \leq 0 \text{ for all } j = 1, \dots, n \right\}.
	\]
	For each $i=1,\dots,m$, the indexing of the Dirac masses pushes forward the measure $\mu^{i}$ with weights $\alpha^{i}_{k}, k=0,\dots,n$ to a measure  $\alpha^{i}$ on $\mathcal{N}$ with weights $\alpha^{i}_{k},k=0,\dots,n$. Let $\alpha :=\bigotimes_{i=1}^{m}{\alpha^{i}}$ denote the corresponding product measure on $\mathcal{D}=\mathcal{N}^{m}$. That is, we have a map
	\[
		\mathbb{A} \colon \bigotimes_{i=1}^{m} \Delta_{n}(\mathcal{X}_{i}) \to \mathcal{M}^{m}(\mathcal{D})
	\]
	and the product map
	\[
	\mathbb{F} \times \mathbb{A} \colon \mathcal{G} \times \bigotimes_{i=1}^{m} \Delta_{n}(\mathcal{X}_{i}) \to \mathcal{F}_\mathcal{D} \times \mathcal{M}^{m}(\mathcal{D}).
	\]

	Since for any function $g \colon \R \to \R$, we have $\mathbb{F}(g\circ f,\mu) = g \circ\mathbb{F}(f,\mu)$, it follows that for any
	$(f,\mu) \in \mathcal{F}\times \bigotimes_{i=1}^{m} \Delta_{n}(\mathcal{X}_{i})$ that
	\[
		\E_{\mu}[g\circ f] = \E_{\alpha_{\mu}} \left[ \mathbb{F}(g\circ f,\mu) \right] = \E_{\alpha_{\mu}} \left[ g \circ \mathbb{F}(f,\mu) \right].
	\]
	Consequently, with the function $\mathcal{R}^\mathcal{D} \colon \mathcal{F}_\mathcal{D} \times \mathcal{M}^{m}(\mathcal{D}) \to
	\R$ defined by
	\[
		\mathcal{R}^\mathcal{D}(h,\alpha) := \E_{\alpha}[r\circ h],
	\]
	and for each $j=1,\dots,n$, the functions $\G^\mathcal{D}_{j} \colon \mathcal{F}_\mathcal{D} \times \mathcal{M}^{m}(\mathcal{D}) \to \R$ defined by
	\[
		\G^\mathcal{N}_{j}(h,\alpha) := E_{\alpha} [g_{i}\circ h],
	\]
	we have that, for all $(f,\mu) \in \mathcal{F} \times \bigotimes_{i=1}^{m} \Delta_{n}(\mathcal{X}_{i})$,
	\begin{equation}
		\label{eq:r}
		\mathcal{R}(f,\mu)=\mathcal{R}^\mathcal{D}(\mathbb{F}(f,\mu),\alpha_{\mu}),
	\end{equation}
	and, for all $j = 1, \dots, n$ and all $(f,\mu) \in \mathcal{F}_\mathcal{D}\times \bigotimes_{i=1}^{m} \Delta_{n}(\mathcal{X}_{i})$,
	\begin{equation}
	\label{eq:g}
		\G_{j}(f,\mu)= \G^\mathcal{D}_{j}(\mathbb{F}(f,\mu),\alpha_{\mu}),
	\end{equation}
	where $\alpha_{\mu} := \mathbb{A}(\mu)$.

	That is,
	\[
		\mathcal{R} = \mathcal{R}^\mathcal{D} \circ \left( \mathbb{F}\times \mathbb{A} \right),
	\]
	\[
		\G_{j}=\G^\mathcal{D}_{j} \circ \left( \mathbb{F}\times \mathbb{A} \right) \text{ for each } j = 1, \dots, n.
	\]
	Consequently, any $(f,\mu) \in \mathcal{A}_{\Delta}$ is mapped by $\mathbb{F}\times \mathbb{A}$ to a point in $\mathcal{F} \times \mathcal{M}^{m}(\mathcal{D})$ that preserves the criterion value and the constraint, and by the assumption must lie in $\mathcal{G}_\mathcal{D} \times \mathcal{M}^{m}(\mathcal{D})$. This establishes $\mathcal{U}(\mathcal{A}_{\Delta}) \leq \mathcal{U}(\mathcal{A}_\mathcal{D})$.

	To obtain equality, consider  $(h,\alpha) \in \mathcal{A}_\mathcal{D}$.  By assumption, there exists an $(f,\mu) \in \mathcal{G}\times \bigotimes_{i=1}^{m} \Delta_{n}(\mathcal{X}_{i})$ such that $\mathbb{F}(f,\mu) = h$.  If we adjust the weights on $\mu$ so that $\mathbb{A}(\mu)=\alpha$, we still maintain $\mathbb{F}(f,\mu) = h$.  By \eqref{eq:r} and \eqref{eq:g}, this point has the same criterion value and satisfies the integral constraints of $\mathcal{A}_{\Delta}$.  The proof is finished.
\end{proof}

\begin{proof}[Proof of Proposition \ref{prop:mcd}.]
	Let $I := \eins_{[a,\infty)}$ be the indicator function and consider $r_{f}:= I \circ f$ so that $\mu[f \geq a] = \E_{\mu} [I \circ f]$.  Since $I \circ f$ is integrable for all $\mu \in \mathcal{M}^{m}(\mathcal{X})$ and we have one constraint $\E_{\mu}[f] \leq 0$, the result follows from Theorem \ref{thm:valuereduce}, provided that we have

	\[
		\mathbb{F} \left( \mathcal{G} \times \bigotimes_{i=1}^{m} \Delta_{1}(\mathcal{X}_{i}) \right) = \mathcal{G}_\mathcal{D}.
	\]
	To establish this, consider $f \in \mathcal{G}$ and observe that for all $\mu \in \bigotimes_{i=1}^{m} \Delta_{1}(\mathcal{X}_{i})$ it holds that $\mathbb{F}(f,\mu) \in \mathcal{G}_{D}$.  Therefore, we conclude that $\mathbb{F} \left(\mathcal{G} \times \bigotimes_{i=1}^{m} \Delta_{1}(\mathcal{X}_{i}) \right) \subseteq \mathcal{G}_\mathcal{D}$.  On the other hand, for any $h \in \mathcal{G}_\mathcal{D}$, we can choose a  measurable product partition of $\mathcal{X}$ dividing each $\mathcal{X}_{i}$ into $2$ cells.  We pull back the function $h$ to a function $f\in \mathcal{F}$ that  has the correct constant values in the partition cells, and place the  Dirac masses into the correct cells. Set the weights to any nonzero values.  It is easy to see that $f \in \mathcal{G}$. Moreover, we have a measure $\mu$ which satisfies $\mathbb{F}(f,\mu) =h$. Therefore, we conclude that $\mathbb{F} \left( \mathcal{G} \times \bigotimes_{i=1}^{m} \Delta_{1}(\mathcal{X}_{i}) \right) \supseteq \mathcal{G}_\mathcal{D}$.  This completes the proof.
\end{proof}

\begin{proof}[Proof of Theorem \ref{thm:C}.]
	First, observe that $\mathcal{G}_\mathcal{D}$ is a sub-lattice of $\mathcal{F}_\mathcal{D}$ in the usual lattice structure on functions.  That is, if $h_{1}, h_{2} \in \mathcal{G}_\mathcal{D}$, then it follows that both $\min (h_{1}, h_{2}) \in \mathcal{G}_\mathcal{D}$ and $\max (h_{1}, h_{2}) \in \mathcal{G}_\mathcal{D}$.  Therefore, for any admissible $(h,\alpha) \in \mathcal{A}_\mathcal{D}$, it follows that clipping $h$ at $a$ to $\min(h,a)$ produces an admissible $(\min(h,a),\alpha)$ and does not change the criterion value $\alpha[h \geq a]$.  Consequently, we can reduce to functions with maximum value $a$.  Moreover, since each function $h^{s}$ is in the sub-lattice $\mathcal{G}_{\mathcal{D}}$, it follows that $h^{C} \in \mathcal{G}_\mathcal{D},  C \in \mathcal{C}$.  For $C \in \mathcal{C}$, define the sub-lattice
	\[
		C_\mathcal{D} := \left\{ h \in \mathcal{F}_\mathcal{D} \,\middle|\, \{ s \mid h(s) = a \} = C\} \right\}
	\]
	of functions with value $a$ precisely on the set $C$.  Now, consider a function $h \in \mathcal{G}_\mathcal{D}$ such that $h \leq a$ and let $C$ be the set where $h=a$.  It follows that $h^{C} \leq h$, $h^{C} \in \mathcal{G}_\mathcal{D}$, and $h^{C}\in C_\mathcal{D}$.  Since $h^{C} \leq h$ implies that $\E_{\alpha}[h^{C}] \leq \E_{\alpha}[h]$ for all $\alpha$, it follows that replacing $(h,\alpha)$ by $(h^{C},\alpha)$ keeps it admissible, and $\alpha[h^{C} \geq a] =\alpha[h \geq a]$. The proof is finished.
\end{proof}

\subsection{Proofs for Section \ref{sec:Reduction-mcd}}
\label{sec-mcdproofs}

The proofs given in this subsection are direct applications of Theorem \ref{thm:C}. In particular, they are based on an analytical calculation of \eqref{eq:jkshdjshdjheer}.  Because Proposition \ref{prop:McD_reduced_explicit_allm} is fundamental to all the other results of the section, its proof will be given first.

\begin{proof}[Proof of Proposition \ref{prop:McD_reduced_explicit_allm}.]
	When non-ambiguous, we will use the notation $\E[h^{C_0}]$ for $\E_\alpha[h^{C_0}]$ and $\P[h^{C_0}\geq a]$ for $\alpha[h^{C_0} \geq a]$.  First, observe that, if $\sum_{j=1}^m D_j \leq a$, then $\min(h^{C_0})\geq 0$, and the constraint $\E[h^{C_0}]\leq 0$ imply $\P[h^{C_0}=0]=1$. This proves the first equation of \eqref{eq:McD_reduced_explicit_allm}.  Now, assume $a<\sum_{j=1}^m D_j$ and observe that
	\[
		h^{C_0}(s) = a - \sum_{j=1}^m (1-s_j) D_j.
	\]
	It follows that
	\begin{equation}
		\label{eq:hdjsgddhj}
		\E_{\alpha}[h^{C_0}] = a - \sum_{j=1}^m (1 - \alpha_j) D_j.
	\end{equation}
	If $D_m=0$, then the optimum is achieved on boundary of $[0,1]^m$ (i.e.\ by taking $\alpha_m=1$ since $C_0 = \{ (1, \dots, 1) \}$ and since $h^{C_0}$ does not depend on $s_m$) and the optimization reduces to an ($m-1$)-dimensional problem.  For that reason, we will assume in all of the proofs of the results given in this section that all the $D_i$s are strictly positive.  The statements of those results remain valid even if one or more of the $D_i$s are equal to zero.

	The condition $D_m > 0$ implies that $\min (D_{1}, \dots, D_{m}) > 0$ and that
	\begin{equation}
		\label{eq:jhgjshgdge}
		\alpha[h^{C_0} \geq a] = \prod_{j=1}^m \alpha_j.
	\end{equation}
	If the optimum in  $\alpha$ is achieved in the interior of the hypercube $[0,1]^m$,  then at that optimum the gradients of \eqref{eq:hdjsgddhj} and \eqref{eq:jhgjshgdge} are collinear.  Hence, in that case, there exists $\lambda\in \mathbb{R}$ such that for all $i\in \{ 1, \dots, m \}$,
	\begin{equation}
		\label{eq:shkjdhghjee}
		\frac{\prod_{j=1}^m \alpha_j}{\alpha_i}= \lambda D_i.
	\end{equation}
	Since $\alpha[h^{C_0}\geq a]$ is increasing in each $\alpha_j$, the optimum is achieved at $\E_\alpha[h^{C_0}]=0$.  Substitution of \eqref{eq:shkjdhghjee} into the equation $\E_\alpha[h^{C_0}]=0$ yields that
	\[
		\lambda = \frac{m \prod_{j=1}^m \alpha_j}{\sum_{j=1}^m D_j-a}
	\]
	and, hence,
	\begin{equation}
		\label{eq:ksdhdskjhdkdjhjd}
		\alpha_i=\frac{\sum_{j=1}^m D_j - a}{m D_i}.
	\end{equation}
	For all $i \in \{ 1, \dots, m \}$, the condition $0 < \alpha_i < 1$ is equivalent to $a < \sum_{j=1}^m D_j$ and
	\begin{equation}
		\label{eq:kdhkjhdjiee}
		\sum_{j=1}^m D_j < a + m D_i.
	\end{equation}
	It follows that for $\sum_{j=1}^m D_j - m D_m < a <\sum_{j=1}^m D_j $, the $\alpha$ defined by \eqref{eq:ksdhdskjhdkdjhjd} lies in the interior of $[0, 1]^{m}$ and satisfies
	\[
		\alpha[h^{C_0} \geq a] = \frac{\left( \sum_{j=1}^m D_j - a \right)^{m}}{m^m \prod_{j=1}^m D_j}.
	\]
	If $a \leq \sum_{j=1}^m D_j - m D_m$, then the optimum is achieved on boundary of $[0,1]^m$ (i.e.\ by taking $\alpha_m = 1$, since $C_0 = \{ (1, \dots, 1) \}$), and the optimization reduces to an $(m - 1)$-dimensional problem.

	To complete the proof, we use an induction.  Observe in particular that, for $k\leq m-1$,
	\[
		\frac{(\sum_{j=1}^k D_j-a)^{k}}{k^k \prod_{j=1}^k D_j} = \frac{(\sum_{j=1}^{k+1} D_j-a)^{k+1}}{(k+1)^{k+1} \prod_{j=1}^{k+1} D_j}
	\]
	for $a= \sum_{j=1}^{k+1}D_j - (k+1)D_{k+1}$, and that
	\begin{equation}
		\label{eq:ksdjejrfff}
		\frac{(\sum_{j=1}^k D_j-a)^{k}}{k^k \prod_{j=1}^k D_j} \leq \frac{(\sum_{j=1}^{k+1} D_j-a)^{k+1}}{(k+1)^{k+1} \prod_{j=1}^{k+1} D_j}
	\end{equation}
	is equivalent to $a \geq \sum_{j=1}^{k+1}D_j - (k+1)D_{k+1}$.  Indeed, writing $a = \sum_{j=1}^{k+1}D_j - (k + 1) D_{k + 1} + b$, equation \eqref{eq:ksdjejrfff} is equivalent to
	\[
		\left( 1 - \frac{b}{k D_{k+1}} \right)^k \leq \left( 1 - \frac{b}{(k+1) D_{k+1}} \right)^{k+1}.
	\]
	The function $f$ given by $f(x):=\big(1-\frac{y}{x}\big)^x$ is increasing in $x$ (for $0<y<x$):  simply examine the derivative of $\log f$, and use the elementary inequality
	\[
		\log (1 - z) + \frac{z}{1 - z} \geq 0 \text{ for }0 < z < 1.
	\]

	We will now give the outline of the induction.  It is trivial to obtain that equation \eqref{eq:McD_reduced_explicit_allm} is true for $m = 1$. Assume that it is true for $m = q-1$ and consider the case $m = q$.  Equation \eqref{eq:ksdhdskjhdkdjhjd} isolates the only potential optimizer $\alpha^q$, which is not on the boundary of $[0,1]^q$ (not ($q-1$)-dimensional).

	One obtains that equation \eqref{eq:McD_reduced_explicit_allm} holds for $m=q$ by comparing the value of $\alpha[h^{C_0}\geq a]$ at locations $\alpha$ isolated by equations \eqref{eq:ksdhdskjhdkdjhjd} and \eqref{eq:kdhkjhdjiee} with those isolated at step $q-1$. This comparison is performed via equation \eqref{eq:ksdjejrfff}.

	More precisely, if $\alpha^q$ (the candidate for the optimizer in $\alpha$ isolated by the previous paragraph) is not an optimum, then the optimum must lie in the boundary of $[0, 1]^q$.  Hence, the optimum must be achieved by taking $\alpha_i=1$ for some $i\in \{1,\dots,q\}$.  Observing that $\mathcal{U}(\mathcal{A}_{C_0})$ is increasing in each $D_i$, and since $D_q = \min_{i\in \{1,\dots,q\}} D_i$, that optimum can be achieved by taking $i=q$, which leads to computing $\mathcal{U}(\mathcal{A}_{C_0})$ with $(D_1,\dots,D_{q-1})$, where we can use the $(q-1)$-step  of the induction. Using equation \eqref{eq:ksdjejrfff} for $k=q-1$, we obtain that $\alpha^q$ is an optimum for $a\geq \sum_{j=1}^{q}D_j - qD_{q}$, and that, for $a\leq \sum_{j=1}^{q}D_j - qD_{q}$, the optimum is achieved by calculating $\mathcal{U}(\mathcal{A}_{C_0})$ with $q-1$ variables and $(D_1,\dots,D_{q-1})$. This finishes the proof by using the induction assumption (see formula \eqref{eq:McD_reduced_explicit_allm}).
\end{proof}

The following two lemmas illustrate simplifications that can be made using the symmetries of the hypercube:

\begin{lem}
	\label{lem:symmetry}
	Let $C_0 \in \mathcal{C}$. If $C_0$ is symmetric with respect to the hyperplane containing the center of the hypercube and normal to the direction $i$, then the optimum of \eqref{eq:sup_over_AC0} can be achieved by taking $\alpha_i=1$.
\end{lem}

\begin{proof}
	The proof follows by observing that if $C_0$ is symmetric with respect to the hyperplane containing the center of the hypercube and normal to the direction $i$, then $h^{C_0}(s)$ does not depend on the variable $s_i$.
\end{proof}

The following lemma is trivial:

\begin{lem}
	\label{lem:trivial}
	Let $(\alpha,C)$ be an optimizer of \eqref{eq:jkshdjshdjheer}. Then, the images of $(\alpha, C)$ by reflections with respect to hyperplanes containing the center of the hypercube and normal to its faces are also optimizers of \eqref{eq:jkshdjshdjheer}.
\end{lem}

The proofs of the remaining theorems now follow in the order that the results were stated in the main part of the paper.

\begin{proof}[Proof of Theorem \ref{thm:m1}.]
	The calculation of $\mathcal{U}(\mathcal{A}_{\mathcal{C}})$ for $m=1$ is trivial and also follows from Proposition \ref{prop:McD_reduced_explicit_allm}.
\end{proof}

\begin{figure}[tp]
	\begin{center}
		\scalebox{1}{
			\begin{pspicture}(0,0)(3.25,3.25)
				\definecolor{blue1}{rgb}{0.8,0.9,1.0}
				\pspolygon[linewidth=0.04,fillstyle=solid,fillcolor=blue1](0,0)(3,0)(3,3)(0,3)
				\pscircle[linewidth=0.04,fillstyle=solid,fillcolor=red](3,3){0.25}
			\end{pspicture}
		}
		\caption{For $m=2$, the optimum associated with $\mathcal{U}(\mathcal{A}_{\mathcal{C}})$ can be achieved with $C = \{ (1, 1) \}$. For that specific value of $C$, the linearity of $h^C(s) = a - D_1 (1 - s_1) - D_2 (1 - s_2)$ implies $\mathcal{U}(\mathcal{A}_{\mathrm{Hfd}}) = \mathcal{U}(\mathcal{A}_{\mathrm{McD}})$.}
		\label{fig:squareQm2}
	\end{center}
\end{figure}
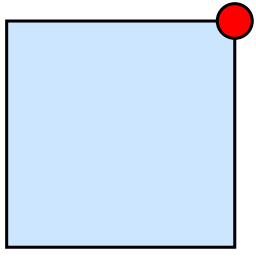

\begin{proof}[Proof of Theorem \ref{thm:m2}.]
	Write $C_1=\{ (1, 1) \}$ (see Figure \ref{fig:squareQm2}).  Theorem \ref{thm:m2} is a consequence of the following inequality:
	\begin{equation}
		\label{eq:jssjhffdgjhge}
		\max_{C_0 \in \mathcal{C}}\mathcal{U}(\mathcal{A}_{C_0})\leq \mathcal{U}(\mathcal{A}_{C_1})
	\end{equation}
	Assuming equation \eqref{eq:jssjhffdgjhge} to be true, equation \eqref{eq:McD_reduced_explicit_2d} is obtained by calculating
	$\mathcal{U}(\mathcal{A}_{C_1})$ from Proposition \ref{prop:McD_reduced_explicit_allm} with $m=2$. Let us now prove equation \eqref{eq:jssjhffdgjhge}.  Let $C_0 \in \mathcal{C}$;  we need to prove that
	\begin{equation}
		\label{eq:UAC0leqUAC1}
		\mathcal{U}(\mathcal{A}_{C_0}) \leq \mathcal{U}(\mathcal{A}_{C_1}).
	\end{equation}
	By symmetry (using Lemma \ref{lem:trivial}), it is no loss of generality to assume that $(1, 1) \in C_0$.  By Lemma \ref{lem:symmetry} the optima for $C_0 = \{ (1, 1), (1, 0) \}$ and $C_0 = \{ (1, 1), (0, 1) \}$ can be achieved with $C_1$ by taking $\alpha_1 = 1$ or $\alpha_2 = 1$.

	The case $C_0=\{(1,1);(1,0);(0,1);(0,0)\}$ is infeasible.

	For $C_0 = \{ (1, 1), (1, 0), (0, 1)\}$, we have $\P[h^{C_0} = a] = \beta$ and $\E[h^{C_0}] = a - (1-\beta) \min (D_1, D_2)$ with $\beta = 1 - (1 - \alpha_1) (1 - \alpha_2)$ (recall that $h^{C_0}$ is defined by equation \eqref{eq:jhdsjdgjhsgdgh}). Hence, at the optimum (in $\alpha$),
	\begin{equation}
		\label{eq:m2examlesolved}
		\P[h^{C_0} = a] =
		\begin{cases}
		1 - a / \min (D_1,D_2), & \text{ if } a < \min(D_1, D_2), \\
		0, & \text{ if } a\geq \min(D_1, D_2).
		\end{cases}
	\end{equation}
	Equation \eqref{eq:UAC0leqUAC1} then holds by observing that one always has both
	\[
		1 - \frac{a}{\min (D_1, D_2)} \leq 1- \frac{a}{\max (D_1, D_2)}
	\]
	and
	\[
		1 - \frac{a}{\min (D_1, D_2)} \leq \frac{(D_{1} + D_{2} -a)^{2}}{4 D_{1} D_{2}}.
	\]
	The last inequality is equivalent to $(D_1 - D_2 + a)^2 \geq 0$, which is always true.  The case $C_0=\{(1,1), (0,0)\}$ is bounded by the previous one since $\P[h^{C_0} = a] = \beta$ and $\E[h^{C_0}] = a \beta - (1 - \beta) \min (D_1, D_2)$ with $\beta = \alpha_1 \alpha_2 + (1 - \alpha_1) (1 - \alpha_2)$.  This finishes the proof.
\end{proof}

\begin{figure}[tp]
	\begin{center}
		\subfigure[$C_1$]{
			\scalebox{1}{
				\begin{pspicture}(0,0)(4.25,4.25)
					\definecolor{blue1}{rgb}{0.8,0.9,1.0}
					\definecolor{blue2}{rgb}{0.7,0.8,0.9}
					\pspolygon[linewidth=0.04,fillstyle=solid,fillcolor=blue1](0,0)(3,0)(3,3)(0,3)
					\pspolygon[linewidth=0.04,fillstyle=solid,fillcolor=blue1](0,3)(3,3)(4,4)(1,4)
					\pspolygon[linewidth=0.04,fillstyle=solid,fillcolor=blue2](3,0)(4,1)(4,4)(3,3)
					\pscircle[linewidth=0.04,fillstyle=solid,fillcolor=red](3,3){0.25}
				\end{pspicture}
			}
		}
		\subfigure[$C_2$]{
			\scalebox{1}{
				\begin{pspicture}(0,0)(4.25,4.25)
					\definecolor{blue1}{rgb}{0.8,0.9,1.0}
					\definecolor{blue2}{rgb}{0.7,0.8,0.9}
					\pspolygon[linewidth=0.04,fillstyle=solid,fillcolor=blue1](0,0)(3,0)(3,3)(0,3)
					\pspolygon[linewidth=0.04,fillstyle=solid,fillcolor=blue1](0,3)(3,3)(4,4)(1,4)
					\pspolygon[linewidth=0.04,fillstyle=solid,fillcolor=blue2](3,0)(4,1)(4,4)(3,3)
					\pscircle[linewidth=0.04,fillstyle=solid,fillcolor=red](0,3){0.25}
					\pscircle[linewidth=0.04,fillstyle=solid,fillcolor=red](3,3){0.25}
					\pscircle[linewidth=0.04,fillstyle=solid,fillcolor=red](3,0){0.25}
					\pscircle[linewidth=0.04,fillstyle=solid,fillcolor=red](4,4){0.25}
				\end{pspicture}
			}
		}
		\caption{For $m=3$, the optimum associated with $\mathcal{U}(\mathcal{A}_{\mathcal{C}})$ can be achieved with $C_1 = \{ (1, 1, 1) \}$ (leading to $\mathcal{F}_1$) or $C_2 = \{ (1, 1, 1), (0, 1, 1), (1, 0, 1), (1, 1, 0) \}$ (leading to $\mathcal{F}_2$).  The linearity of $h^{C_1}(s) = a - D_1 (1 - s_1) - D_2 (1 - s_2) - D_3 (1 - s_3)$ implies that $\mathcal{U}(\mathcal{A}_{\mathrm{Hfd}}) = \mathcal{U}(\mathcal{A}_{\mathrm{McD}})$ when $\mathcal{F}_1 \geq \mathcal{F}_2$.  Similarly, the nonlinearity of $h^{C_2}$ leads to $\mathcal{U}(\mathcal{A}_{\mathrm{Hfd}}) < \mathcal{U}(\mathcal{A}_{\mathrm{McD}})$ when $\mathcal{F}_1 < \mathcal{F}_2$.}
		\label{fig:cubem3}
	\end{center}
\end{figure}
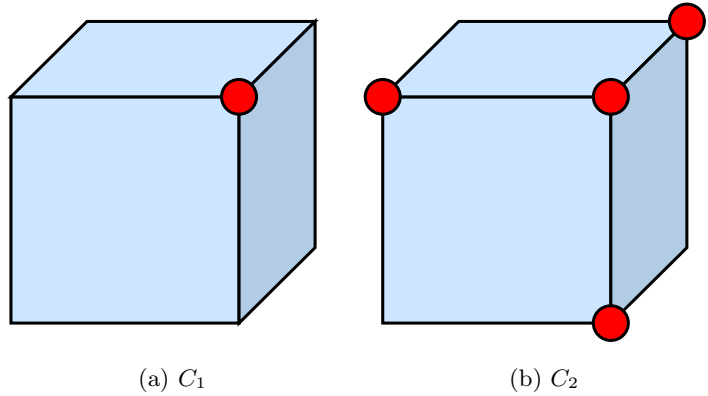

\begin{proof}[Proof of Theorem \ref{thm:m3}.]
	Recall that
	\[
		\mathcal{U}(\mathcal{A}_{\mathrm{McD}}) = \max_{C_0 \in \mathcal{C}}\mathcal{U}(\mathcal{A}_{C_0}).
	\]
	It follows from Proposition \ref{prop:McD_reduced_explicit_allm} that $\mathcal{F}_1$ corresponds to $\mathcal{U}(\mathcal{A}_{C_1})$
	with $C_1=\{(1,1,1)\}$. Write $C_2 = \{ (1, 1, 1), (0, 1, 1), (1, 0, 1), (1, 1, 0) \}$ (see Figure \ref{fig:cubem3}).  Let us now calculate $\mathcal{U}(\mathcal{A}_{C_2})$ ($\mathcal{F}_2$ corresponds to $\mathcal{U}(\mathcal{A}_{C_2})$, which is the optimum, and it is achieved in the interior of $[0,1]^3$).
	We have $\mathbb{P}[h^{C_{2}}=a]=\alpha_{2}\alpha_{3}+\alpha_{1}\alpha_{3}+\alpha_{1}\alpha_{2}-2\alpha_{1}\alpha_{2}\alpha_{3}$, and
	\[ \E[h^{C_{2}}] =a -D_{2}(1-\alpha_{1})(1-\alpha_{2})-D_{3}\left((1-\alpha_{2})(1-\alpha_{3})+(1-\alpha_{1})\alpha_{2}
	(1-\alpha_{3})\right).\]
	An internal optimal point $\alpha$ must satisfy, for some $\lambda \in \R$,
	\begin{subequations}
		\label{eq:lagrange}
		\begin{align}
			\alpha_{2}+\alpha_{3}-2\alpha_{2}\alpha_{3} &= \lambda \left( D_{2}(1-\alpha_{2})+D_{3}\alpha_{2}(1-\alpha_{3})\right), \\
			\alpha_{1}+\alpha_{3}-2\alpha_{1}\alpha_{3} &= \lambda \left( D_{2}(1-\alpha_{1})+D_{3}\alpha_{1}(1-\alpha_{3})\right), \\
			\alpha_{1}+\alpha_{2}-2\alpha_{1}\alpha_{2} &= \lambda \left( D_{3}(1-\alpha_{1}\alpha_{2})\right).
		\end{align}
	\end{subequations}
	If we multiply the first equation by $\alpha_{1}$ and subtract the second equation multiplied by
	$\alpha_{2}$, the we obtain that
	\[
		(\alpha_{1} - \alpha_{2}) \alpha_{3} = \lambda D_{2} (\alpha_{1} - \alpha_{2}),
	\]
	which implies that either  $\alpha_{1}=\alpha_{2}$ or $\alpha_{3}=\lambda D_{2}$.

	Suppose that $\alpha_{1}\neq \alpha_{2}$, so that $\alpha_{3}=\lambda D_{2}$. 	Subtraction of the second equation in \eqref{eq:lagrange} from the first yields
	\[
		(\alpha_{2} - \alpha_{1}) (1 - 2 \alpha_{3}) = \lambda \left( -D_{2} (\alpha_{2} - \alpha_{1}) + D_{3} (\alpha_{2} - \alpha_{1}) (1-\alpha_{3}) \right),
	\]
	which implies that either $\alpha_{1}=\alpha_{2}$ or
	\[
		1 - 2 \alpha_{3} = \lambda \left( - D_{2} + D_{3} (1 - \alpha_{3}) \right).
	\]
	Since $\alpha_{3} = \lambda D_{2}$, this becomes
	\[
		1 - \alpha_{3} = \frac{\alpha_{3} D_{3}}{D_{2}} (1 - \alpha_{3}),
	\]
	which implies, since $\alpha_{3} \neq 1$, that $\alpha_{3} = \frac{D_{2}}{D_{3}}$.  Therefore, $\lambda =\frac{1}{D_{3}}$.  Therefore, the third equation in \eqref{eq:lagrange} becomes
	\[
		\alpha_{1} + \alpha_{2} - 2 \alpha_{1} \alpha_{2} = \lambda \left( D_{3} (1 - \alpha_{1} \alpha_{2})\right) = 1 - \alpha_{1} \alpha_{2},
	\]
	which amounts to
	\[
		\alpha_{1} + \alpha_{2} - \alpha_{1} \alpha_{2} = 1,
	\]
	which in turn amounts to $\alpha_{1} (1 - \alpha_{2}) = 1 - \alpha_{2}$.  Since $\alpha_{2} \neq 1$, we conclude that $\alpha_{1} = 1$, contradicting the supposition that $\alpha$ is an interior point.  Therefore, $\alpha_{1} = \alpha_{2}$ and equations \eqref{eq:lagrange}, with $\alpha := \alpha_{1} = \alpha_{2}$, become
	\begin{subequations}
		\label{eq:lagrange2}
		\begin{align}
			\alpha+\alpha_{3}-2\alpha\alpha_{3} =& \lambda \left( D_{2}(1-\alpha)+D_{3}\alpha(1-\alpha_{3})\right) \\
			2\alpha-2\alpha^{2} &= \lambda \left( D_{3}(1-\alpha^{2}) \right).
		\end{align}
	\end{subequations}
	Hence,
	\begin{equation}
		\label{eq:bigm3}
		\mathbb{P}[h^{C_{2}}=a]=2\alpha\alpha_{3}+\alpha^{2}-2\alpha^{2}\alpha_{3}
	\end{equation}
	and
	\[
		\E[h^{C_{2}}] =a -D_{2}(1-\alpha)^{2}-D_{3} \left( (1-\alpha^{2})(1-\alpha_{3}) \right).
	\]
	The hypothesis that the optimum is not achieved on the boundary requires that
	\[
		D_3, 0 < \alpha < 1, D_2 + D_3 > a \text{ and } E[h^{C_{2}}]=0.
	\]
	The condition $\E[h^{C_{2}}]=0$ is required because equation \eqref{eq:bigm3} is strictly increasing along the direction $\alpha = \alpha_3$.

	Suppose that those conditions are satisfied.  The condition $\E[h^{C_{2}}]=0$ implies that
	\[
		1 - \alpha_3 = \frac{a}{D_3 (1-\alpha^2)} - \frac{D_2 (1-\alpha)}{D_3 (1+\alpha)},
	\]
	which in turns implies that
	\begin{equation}
		\label{eq:intm3}
		\alpha_3 = 1 - \frac{a}{D_3 (1 - \alpha^2)} + \frac{D_2 (1 - \alpha)}{D_3 (1 + \alpha)}.
	\end{equation}
	Substitution of \eqref{eq:intm3} into \eqref{eq:bigm3} yields that $\mathbb{P}[h^{C_{2}} = a] = \Psi(\alpha)$, with
	\[
		\Psi(\alpha) = \alpha^{2} + 2 (\alpha - \alpha^{2}) \left( 1 - \frac{a}{D_3} \frac{1}{(1 - \alpha^2)} + \frac{D_2}{D_3} \frac{1 - \alpha}{1 + \alpha} \right).
	\]
	Hence,
	\[
		\Psi(\alpha) = 2 \alpha - \alpha^2 - 2 \frac{a}{D_3} \frac{\alpha}{1 + \alpha} + 2 \frac{D_2}{D_3} \alpha \frac{(1 - \alpha)^2}{1 + \alpha}.
	\]
	$\Psi(\alpha)$ can be simplified using polynomial division.  In particular, using
	\[
		\alpha \frac{(1 - \alpha)^2}{1 + \alpha} = (1 - \alpha)^2 - \frac{(1 - \alpha)^2}{1 + \alpha},
	\]
	\[
		\alpha \frac{(1 - \alpha)^2}{1 + \alpha} = \alpha^2 + 1 - 2 \alpha - (1 + \alpha) + 4 - \frac{4}{1 + \alpha},
	\]
	where the last step is obtained from
	\[
		(1 - \alpha)^2 = (\alpha + 1 - 2)^2 = (\alpha + 1)^2 - 4 (1 + \alpha) + 4,
	\]
	we obtain that
	\[
		\Psi(\alpha) = 2 \alpha - \alpha^2 - 2 \frac{a}{D_3} \frac{\alpha}{1 + \alpha} + 2 \frac{D_2}{D_3} \left( 4 + \alpha^2 - 3 \alpha - \frac{4}{1+\alpha} \right).
	\]
	Therefore,
	\[
		\Psi(\alpha) = \alpha^2 \left( 2 \frac{D_2}{D_3} - 1 \right) - \frac{2 a}{D_3} \frac{\alpha}{1 + \alpha} + 2 \alpha \left( 1 - 3 \frac{D_2}{D_3} \right) + 8 \frac{D_2}{D_3} \frac{\alpha}{1 + \alpha}
	\]
	and
	\begin{equation}
		\label{eq:bigdfm3}
		\Psi(\alpha) = \alpha^2 \left( 2 \frac{D_2}{D_3} - 1 \right) - 2 \alpha \left( 3 \frac{D_2}{D_3} - 1 \right) + \frac{\alpha}{1+\alpha} \left( 8 \frac{D_2}{D_3} - 2 \frac{a}{D_3} \right).
	\end{equation}
	Equation \eqref{eq:bigdfm3} implies that
	\[
		D_3 \Psi'(\alpha) = 2 \alpha (2 D_2 - D_3) + 2 (D_3 - 3 D_2) - \frac{1}{(1 + \alpha)^2} (2 a - 8 D_2).
	\]
	The equation $\Psi'(\alpha)=0$ is equivalent to equation \eqref{eq:magic_cubic}.  An interior optimum requires the existence of an $\alpha \in (0, 1)$ such that $\Psi'(\alpha) = 0$ and $\alpha_3 \in (0, 1)$, which leads to the definition of $\mathcal{F}_2$.  This establishes the theorem for the $\mathcal{F}_2$ case.

	Next, using symmetries of the hypercube and through direct computation (as in the $m=2$ case), we obtain that
	\begin{equation}
		\label{eq:hfgfhfhhtf}
		C_0 \neq C_2 \implies \mathcal{U}(\mathcal{A}_{C_0}) \leq \mathcal{U}(\mathcal{A}_{C_1}).
	\end{equation}
	For the sake of concision, we will give the detailed proof of \eqref{eq:hfgfhfhhtf} only for
	\[
		C_3 =\{(1,1,1),  (0,1,1), (1,0,1)\}.
	\]
	This proof will give an illustration of generic reduction properties used in other cases.  To address all the symmetric transformations of $C_3$, we will give the proof without assuming that $D_1, D_2$ and $D_3$ are ordered.
	Let us now consider the $C_3$ scenario. If the optimum in $\alpha$ is achieved on the boundary of $[0, 1]^3$, then equation \eqref{eq:jssjhffdgjhge} implies $\mathcal{U}(\mathcal{A}_{C_3}) \leq \mathcal{U}(\mathcal{A}_{C_1})$.  Let us assume that the optimum is not achieved on the boundary of $[0,1]^3$.  Observe that
	\begin{equation}
		\label{eq:dskjdlksjhdkheh}
		h^{C_3}(s_1,s_2,0)=h^{C_3}(s_1,s_2,1)-D_3.
	\end{equation}
	Combining \eqref{eq:dskjdlksjhdkheh} with
	\[
		\E[h^{C_3}] = \alpha_3 \E[h^{C_3}(s_1, s_2, 1)] + (1 - \alpha_3) \E[h^{C_3}(s_1, s_2, 0)]
	\]
	implies that
	\[
		\E[h^{C_3}] = \E[h^{C_3}(s_1, s_2, 1)] - (1 - \alpha_3) D_3.
	\]
	Furthermore,
	\begin{equation}
		\label{eq:wlkwhclkhksdejckc}
		\P[h^{C_3} = a] = \alpha_3 \P[h^{C_3}(s_1, s_2, 1) = a].
	\end{equation}
	Maximizing \eqref{eq:wlkwhclkhksdejckc} with respect to $\alpha_3$ under the constraint $\E[h^{C_3}]\leq 0$ leads to $\E[h^{C_3}]=0$ (because $\P[h^{C_3}=a]$ and $\E[h^{C_3}]$ are linear in $\alpha_3$) and
	\begin{equation}
		\label{eq:alpha3}
		\alpha_3 = 1 - \frac{\E[h^{C_3}(s_1, s_2, 1)]}{D_3}.
	\end{equation}
Observe that the condition $\alpha_3 <1$ requires $\E[h^{C_3}(s_1, s_2, 1)] > 0$. If
 $\E[h^{C_3}(s_1, s_2, 1)]\leq 0$ then $\alpha_3 = 1$, and the optimum is achieved on the boundary of $[0,1]^3$.

	The maximization of $\P[h^{C_3}(s_1,s_2,1)=a]$ under the constraint $\E[h^{C_3}(s_1,s_2,1)]\leq E$ (where $E$ is a slack optimization variable) leads to (using the $m=2$ result)
	\[
		\P[h^{C_3}(s_1,s_2,1)=a]= 1-\frac{(a-E)}{\min(D_1,D_2)}
	\]
	if $a-E \leq \min(D_1,D_2)$, and $\P[h^{C_3}(s_1,s_2,1)=a]=0$ otherwise.  It follows from \eqref{eq:alpha3} and \eqref{eq:wlkwhclkhksdejckc} that if the optimum is achieved at an interior point, then the optimal value  of $\P[h^{C_3}=a]$ is achieved by taking the maximum of
	\[
		\P[h^{C_3}=a] = \left( 1-\frac{E}{D_3} \right) \left( 1-\frac{a-E}{\min(D_1,D_2)} \right)
	\]
	with respect to $E$ with the constraints $0\leq E \leq D_3$ and $a-\min(D_1,D_2) \leq E$. If the optimum is not achieved on the boundary of $[0,1]^3$ then one must have
	\[
		E = \frac{D_3+\min(D_1,D_2)-a}{2},
	\]
	which leads to
	\begin{equation}
		\label{eq:wlkwhclkhksdfdrejckc}
		\P[h^{C_3}=a]= \frac{\big(D_3+\min(D_1,D_2)-a\big)^2}{4D_3 \min(D_1,D_2)}.
	\end{equation}
	Comparison of \eqref{eq:wlkwhclkhksdfdrejckc} and \eqref{eq:McD_reduced_explicit_2d} implies that $\mathcal{U}(\mathcal{A}_{C_3}) \leq \mathcal{U}(\mathcal{A}_{C_1})$, by Proposition \ref{prop:McD_reduced_explicit_allm}.
\end{proof}

\begin{proof}[Proof of Proposition \ref{jhgjhgejeerrsddeedffr}.]
	The idea of the proof is to show that $h^C$ can be chosen so that $C$ contains only one vertex of the hypercube, in which case we have the explicit formula obtained in Proposition \ref{prop:McD_reduced_explicit_allm}.

	First, observe that if $a > \sum_{j=1}^{m-1} D_j$, then it is not possible to satisfy the constraint $\E_\alpha [h^C] \leq 0$ whenever $C$ contains two or more vertices of the hypercube.  Next, if $C$  contains two vertices $s^1, s^2$ of the hypercube, and the Hamming distance between those points is $1$, then $C$ is symmetric with respect to a hyperplane containing the center of the hypercube and normal to one of its faces, and the problem reduces to dimension $m-1$. It follows from Lemma \ref{lem:symmetry} that the optimum of \eqref{eq:sup_over_AC0} can be achieved by a  $C$ that has only one element.  If $C$ contains two vertices of the hypercube, and the Hamming distance between those points is greater than or equal to $2$, then the constraint $\E_\alpha [h^C] \leq 0$ is infeasible if $a >  \sum_{j=1}^{m-2} D_j +D_m$ (because $h^C>0$ in that case).  Therefore, we conclude using Proposition \ref{prop:McD_reduced_explicit_allm}.
\end{proof}

\begin{proof}[Proof of Theorem \ref{thm:Hfdm2}.]
	First, we observe that we always have
	\begin{equation}
		\label{shdhgdjhgd}
		\mathcal{U}(\mathcal{A}_{\mathrm{Hfd}})\leq \mathcal{U}(\mathcal{A}_{\mathrm{McD}}).
	\end{equation}
	We observe from equation \eqref{eq:jssjhffdgjhge} that the maximizer ($h^C$) of $\mathcal{U}(\mathcal{A}_{\mathrm{McD}})$ is linear (see Figure \ref{fig:squareQm2}), and hence is also an admissible function under
	$\mathcal{U}(\mathcal{A}_{\mathrm{Hfd}})$. This finishes the proof.
\end{proof}

\begin{proof}[Proof of Theorem \ref{thm:Hfdm3}.]
	Just as for $m=2$, equation \eqref{shdhgdjhgd} is always satisfied. Next, observing that $\mathcal{F}_1$, in Theorem \ref{thm:m3}, is associated with
	a linear maximizer $h^C$ (see Figure \ref{fig:cubem3}), we deduce that
	\[
		\mathcal{F}_1 \leq \mathcal{U}(\mathcal{A}_{\mathrm{Hfd}})\leq \max(\mathcal{F}_1,\mathcal{F}_2).
	\]
	This finishes the proof for equation \eqref{eq:McD_reduhdm3njfd03}. Let us now prove equation \eqref{eq:McD_reduhdm3njfd03bis}.  Assuming that $\mathcal{U}(\mathcal{A}_{\mathrm{Hfd}})=\mathcal{U}(\mathcal{A}_{\mathrm{McD}})$, it follows that $\mathcal{U}(\mathcal{A}_{\mathrm{McD}})$ can be achieved by a linear function $h_0$. Since at the optimum we must have $\mathbb{E}[h_0]=0$, and since $\min(h_0,a)$ is also a maximizer of $\mathcal{U}(\mathcal{A}_{\mathrm{McD}})$, it follows that $\min(h_0,a)=h_0$. Using the linearity of $h_0$, and the lattice structure of the set of functions in $\mathcal{U}(\mathcal{A}_{\mathrm{McD}})$, we deduce that $h_0=h^C$, where $C$ contains only one vertex of the cube. It follows that $\mathcal{F}_1\geq \mathcal{F}_2$. This finishes the proof.
\end{proof}

\addcontentsline{toc}{section}{References}
\bibliographystyle{plain}
\bibliography{./refs}

\end{document}